\documentclass[11pt]{article}
\usepackage{etex}

\usepackage[margin=1in]{geometry}
\usepackage{stackengine}

\usepackage{stmaryrd}

\usepackage{color}
\usepackage[latin1]{inputenc}
\usepackage[T1]{fontenc}
\usepackage[normalem]{ulem}
\usepackage[english]{babel}
\usepackage{verbatim}
\usepackage{graphicx}
\usepackage{enumerate}
\usepackage{amsmath,amssymb,amsfonts,amsthm,amscd,mathrsfs}
\usepackage{array}

\usepackage{dsfont}

\usepackage[T1]{fontenc}
\usepackage{babel}

\usepackage{url}

\usepackage{caption}
\usepackage{subcaption}
\usepackage[bookmarksopen, bookmarksnumbered]{hyperref}

\usepackage{tikz}
\usetikzlibrary{arrows}
\usetikzlibrary{arrows.meta}
\definecolor{wwhhii}{rgb}{1.,1.,1.}
\definecolor{rreedd}{rgb}{1.,0.,0.}
\definecolor{uuuuuu}{rgb}{0.26666666666666666,0.26666666666666666,0.26666666666666666}

\newtheorem{theorem}{Theorem}[section]

\newtheorem{lemma}[theorem]{Lemma}

\newtheorem{prop}[theorem]{Proposition}
\newtheorem{cor}[theorem]{Corollary}

\theoremstyle{definition}

\input xy
\xyoption{all}

\DeclareMathOperator{\argmax}{argmax}
\DeclareMathOperator{\Exp}{Exp}

\DeclareMathOperator{\argmin}{argmin}

\newcommand{\bu}{\bullet}

\newcommand{\E}{\mathbb E}
\newcommand{\PP}{\mathbb P}
\newcommand{\Z}{\mathbb Z}
\newcommand{\R}{\mathbb R}

\newcommand{\N}{\mathbb N}

\newcommand{\LL}{\mathbb L}
\newcommand{\VV}{\mathbb V}
\newcommand{\HH}{\mathbb H}

\newcommand{\don}{\mathds{1}}

\newcommand{\cD}{\mathcal D}

\newcommand{\cT}{\mathcal T}
\newcommand{\cB}{\mathcal B}

\newcommand{\cF}{\mathcal F}
\newcommand{\cE}{\mathcal E}
\newcommand{\cA}{\mathcal A}
\newcommand{\cU}{\mathcal U}
\newcommand{\cL}{\mathcal L}
\newcommand{\cM}{\mathcal M}

\newcommand{\cH}{\mathcal H}

\newcommand{\sS}{\mathscr S}

\newcommand{\sT}{\mathscr T}
\newcommand{\sM}{\mathscr M}

\newcommand{\fP}{\mathfrak P}
\newcommand{\fp}{\mathfrak p}
\newcommand{\sN}{\mathscr N}

\newcommand{\bc}{\mathbf c}

\newcommand{\bB}{\mathbf B}
\newcommand{\bG}{\mathbf G}

\newcommand{\bPsi}{\mathbf \Psi}
\newcommand{\hbPsi}{\Tilde{\bPsi}}
\newcommand{\brho}{\boldsymbol{\rho}}

\newcommand{\hh}{\hat{h}}

\newcommand{\br}{\mathbf r}
\newcommand{\boo}{\mathbf 0}
\newcommand{\bn}{\mathbf n}

\newcommand{\bfeta}{\boldsymbol \eta}
\newcommand{\boeta}{\boldsymbol \zeta}

\newcommand{\oL}{L}
\newcommand{\oI}{I}
\newcommand{\oxi}{\xi}

\newcommand{\oga}{\gamma}

\newcommand{\oY}{\overline{Y}}
\newcommand{\oR}{\overline{R}}

\newcommand{\ocT}{\mathbf{T}}

\newcommand{\heta}{\hat{\eta}}

\newcommand{\isd}{\overset{d}{=}}

\makeatletter
\renewcommand\tableofcontents{%
  \null\hfill\textbf{\Large\contentsname}\hfill\null\par
  \@mkboth{\MakeUppercase\contentsname}{\MakeUppercase\contentsname}%
  \@starttoc{toc}%
}

\g@addto@macro\normalsize{%
  \setlength\abovedisplayskip{5pt}
  \setlength\belowdisplayskip{5pt}
  \setlength\abovedisplayshortskip{3pt}
  \setlength\belowdisplayshortskip{3pt}
}

\makeatother

\numberwithin{equation}{section}

\begin{document}
\title{Convergence of the Environment Seen from Geodesics in Exponential Last-Passage Percolation}

\author{James B. Martin
\thanks{Department of Statistics, University of Oxford. e-mail: martin@stats.ox.ac.uk}
\and Allan Sly
\thanks{Department of Mathematics, Princeton University. e-mail: asly@princeton.edu} \and Lingfu Zhang
\thanks{Department of Mathematics, Princeton University, and Department of Statistics, University of California, Berkeley, and Division of Physics, Mathematics and Astronomy, California Institute of Technology. e-mail: lingfuz@caltech.edu}
}
\date{}

\maketitle

\begin{abstract}
A well-known question in planar first-passage percolation concerns the convergence of the empirical distribution of weights as seen along geodesics.
We demonstrate this convergence for an explicit model, directed last-passage percolation on $\Z^2$ with i.i.d.\;exponential weights, and provide explicit formulae for the limiting distributions, which depend on the asymptotic direction.
For example, for geodesics in the direction of the diagonal, the limiting weight distribution has density $(1/4+x/2+x^2/8)e^{-x}$, and so is a mixture of Gamma($1,1$), Gamma($2,1$) and Gamma($3,1$) distributions with weights $1/4$, $1/2$, and $1/4$ respectively.
More generally, we study the local environment as seen from vertices along geodesics 
(including information about the shape of the path and about the weights on and off the path
in a local neighborhood). 
We consider finite geodesics from $(0,0)$ to $n\brho$ for some vector $\brho$ in the first quadrant, in the limit as $n\to\infty$, as well as semi-infinite geodesics in direction $\brho$.
We show almost sure convergence of the empirical distributions of the environments along these geodesics, as well as convergence of the distributions of the environment around a typical point in these geodesics, to the same limiting distribution, for which we give an explicit description. 

We make extensive use of a correspondence
with TASEP as seen from an isolated second-class
particle for which we prove new results concerning
ergodicity and convergence to equilibrium. 
Our analysis relies  on geometric arguments involving estimates for last-passage times, available from the integrable probability literature.
\end{abstract}
\begin{figure}[hbt!]
    \centering
\begin{tikzpicture}[line cap=round,line join=round,>=triangle 45,x=3.2cm,y=3.2cm]
\clip(-0.5,-0.2) rectangle (2.5,2.2);

\fill[line width=0.pt,color=green,fill=green,fill opacity=0.2]
(-0.25,-0.15) -- (0.25,-0.15) -- (0.25,0.35) -- (-0.25,0.35) -- cycle;
\fill[line width=0.pt,color=green,fill=green,fill opacity=0.2]
(0.05,0.25) -- (0.55,0.25) -- (0.55,0.75) -- (0.05,0.75) -- cycle;
\fill[line width=0.pt,color=green,fill=green,fill opacity=0.2]
(0.65,0.55) -- (1.15,0.55) -- (1.15,1.05) -- (0.65,1.05) -- cycle;
\fill[line width=0.pt,color=green,fill=green,fill opacity=0.2]
(1.15,1.05) -- (1.65,1.05) -- (1.65,1.55) -- (1.15,1.55) -- cycle;
\fill[line width=0.pt,color=green,fill=green,fill opacity=0.2]
(1.75,1.55) -- (2.25,1.55) -- (2.25,2.05) -- (1.75,2.05) -- cycle;

\draw [line width=.1pt, opacity=0.3] (-0.6,-0.1) -- (2.6,-0.1);
\draw [line width=.1pt, opacity=0.3] (-0.6,0.) -- (2.6,0.);
\draw [line width=.1pt, opacity=0.3] (-0.6,0.1) -- (2.6,0.1);
\draw [line width=.1pt, opacity=0.3] (-0.6,0.2) -- (2.6,0.2);
\draw [line width=.1pt, opacity=0.3] (-0.6,0.3) -- (2.6,0.3);
\draw [line width=.1pt, opacity=0.3] (-0.6,0.4) -- (2.6,0.4);
\draw [line width=.1pt, opacity=0.3] (-0.6,0.5) -- (2.6,0.5);
\draw [line width=.1pt, opacity=0.3] (-0.6,0.6) -- (2.6,0.6);
\draw [line width=.1pt, opacity=0.3] (-0.6,0.7) -- (2.6,0.7);
\draw [line width=.1pt, opacity=0.3] (-0.6,0.8) -- (2.6,0.8);
\draw [line width=.1pt, opacity=0.3] (-0.6,0.9) -- (2.6,0.9);
\draw [line width=.1pt, opacity=0.3] (-0.6,1.) -- (2.6,1.);
\draw [line width=.1pt, opacity=0.3] (-0.6,1.1) -- (2.6,1.1);
\draw [line width=.1pt, opacity=0.3] (-0.6,1.2) -- (2.6,1.2);
\draw [line width=.1pt, opacity=0.3] (-0.6,1.3) -- (2.6,1.3);
\draw [line width=.1pt, opacity=0.3] (-0.6,1.4) -- (2.6,1.4);
\draw [line width=.1pt, opacity=0.3] (-0.6,1.5) -- (2.6,1.5);
\draw [line width=.1pt, opacity=0.3] (-0.6,1.6) -- (2.6,1.6);
\draw [line width=.1pt, opacity=0.3] (-0.6,1.7) -- (2.6,1.7);
\draw [line width=.1pt, opacity=0.3] (-0.6,1.8) -- (2.6,1.8);
\draw [line width=.1pt, opacity=0.3] (-0.6,1.9) -- (2.6,1.9);
\draw [line width=.1pt, opacity=0.3] (-0.6,2.) -- (2.6,2.);
\draw [line width=.1pt, opacity=0.3] (-0.6,2.1) -- (2.6,2.1);

\draw [line width=.1pt, opacity=0.3] (-0.4,-0.2) -- (-0.4,2.2);
\draw [line width=.1pt, opacity=0.3] (-0.3,-0.2) -- (-0.3,2.2);
\draw [line width=.1pt, opacity=0.3] (-0.2,-0.2) -- (-0.2,2.2);
\draw [line width=.1pt, opacity=0.3] (-0.1,-0.2) -- (-0.1,2.2);
\draw [line width=.1pt, opacity=0.3] (0.,-0.2) -- (0.,2.2);
\draw [line width=.1pt, opacity=0.3] (0.1,-0.2) -- (0.1,2.2);
\draw [line width=.1pt, opacity=0.3] (0.2,-0.2) -- (0.2,2.2);
\draw [line width=.1pt, opacity=0.3] (0.3,-0.2) -- (0.3,2.2);
\draw [line width=.1pt, opacity=0.3] (0.4,-0.2) -- (0.4,2.2);
\draw [line width=.1pt, opacity=0.3] (0.5,-0.2) -- (0.5,2.2);
\draw [line width=.1pt, opacity=0.3] (0.6,-0.2) -- (0.6,2.2);
\draw [line width=.1pt, opacity=0.3] (0.7,-0.2) -- (0.7,2.2);
\draw [line width=.1pt, opacity=0.3] (0.8,-0.2) -- (0.8,2.2);
\draw [line width=.1pt, opacity=0.3] (0.9,-0.2) -- (0.9,2.2);
\draw [line width=.1pt, opacity=0.3] (1.,-0.2) -- (1.,2.2);
\draw [line width=.1pt, opacity=0.3] (1.1,-0.2) -- (1.1,2.2);
\draw [line width=.1pt, opacity=0.3] (1.2,-0.2) -- (1.2,2.2);
\draw [line width=.1pt, opacity=0.3] (1.3,-0.2) -- (1.3,2.2);
\draw [line width=.1pt, opacity=0.3] (1.4,-0.2) -- (1.4,2.2);
\draw [line width=.1pt, opacity=0.3] (1.5,-0.2) -- (1.5,2.2);
\draw [line width=.1pt, opacity=0.3] (1.6,-0.2) -- (1.6,2.2);
\draw [line width=.1pt, opacity=0.3] (1.7,-0.2) -- (1.7,2.2);
\draw [line width=.1pt, opacity=0.3] (1.8,-0.2) -- (1.8,2.2);
\draw [line width=.1pt, opacity=0.3] (1.9,-0.2) -- (1.9,2.2);
\draw [line width=.1pt, opacity=0.3] (2.,-0.2) -- (2.,2.2);
\draw [line width=.1pt, opacity=0.3] (2.1,-0.2) -- (2.1,2.2);
\draw [line width=.1pt, opacity=0.3] (2.2,-0.2) -- (2.2,2.2);
\draw [line width=.1pt, opacity=0.3] (2.3,-0.2) -- (2.3,2.2);
\draw [line width=.1pt, opacity=0.3] (2.4,-0.2) -- (2.4,2.2);

\draw [red] plot coordinates {(-0.4,-0.1) (-0.3,-0.1) (-0.3,0.) (-0.3,0.1) (-0.2,0.1) (-0.1,0.1) (0.,0.1) (0.,0.2) (0.1,0.2) (0.1,0.3) (0.1,0.4) (0.2,0.4) (0.2,0.5) (0.3,0.5) (0.4,0.5) (0.4,0.6) (0.6,0.6) (0.6,0.7) (0.9,0.7) (0.9,0.9) (1.,0.9) (1.2,0.9) (1.2,1.) (1.3,1.) (1.3,1.3) (1.4,1.3) (1.4,1.5) (1.5,1.5) (1.5,1.7) (1.9,1.7) (1.9,1.8) (2.1,1.8) (2.1,2.) (2.3,2.) (2.3,2.1)  (2.4,2.1)};

\draw [fill=uuuuuu] (-0.4,-0.1) circle (1.0pt);
\draw [fill=uuuuuu] (2.4,2.1) circle (1.0pt);

\draw [fill=uuuuuu] (0.,0.1) circle (1.0pt);
\draw [fill=uuuuuu] (0.3,0.5) circle (1.0pt);
\draw [fill=uuuuuu] (0.9,0.8) circle (1.0pt);
\draw [fill=uuuuuu] (1.4,1.3) circle (1.0pt);
\draw [fill=uuuuuu] (2.,1.8) circle (1.0pt);

\end{tikzpicture}
\caption{An illustration of local environments along a finite geodesic.}  
\end{figure}

\setcounter{tocdepth}{2}
\tableofcontents

\section{Introduction}

In this article we study exactly solvable models of planar directed last-passage percolation (LPP), an instance of the more general Kardar-Parisi-Zhang (KPZ) universality class, which dates back to the seminal work of \cite{KPZ86}. 
The KPZ universality class has been a major topic of interest both in statistical physics and in probability theory in recent decades. In \cite{KPZ86}, the authors predicted universal scaling behaviour for a large number of planar random growth processes, including first-passage percolation (FPP) and corner growth processes; in particular, it is predicted that these models have length fluctuation exponent $1/3$ and transversal fluctuation exponent $2/3$.
Since then, rigorous progress has been made only in a handful of cases. 
The first breakthrough was made by Baik, Deift and Johansson \cite{BDJ99} when they established $n^{1/3}$ fluctuations of the length of the longest up-right path from $(0,0)$ to $(n,n)$ in a homogeneous Poissonian field on $\R^2$, and also established the GUE Tracy-Widom scaling limit.
Then Johansson proved a transversal fluctuation exponent of $2/3$ for that model, and also $n^{1/3}$ fluctuations and a Tracy-Widom scaling limit for LPP on $\Z^2$ with i.i.d.\;geometric or exponential weights \cite{J00, johansson2000shape}. 
For these models such results could be obtained due to their exact solvability, using exact distributional formulae from algebraic combinatorics, random matrix theory, or queueing theory in some cases. 
Since then there have been tremendous developments in achieving a detailed understanding of these exactly solvable models, with notable progress concerning scaling limits 
(see e.g.\ the recent works of \cite{matetski2016kpz, DOV}). For  surveys in this direction, see e.g.\ \cite{Corwin11, quastel2014airy, zygouras2018some}.

In another related direction, there has been great interest in studying FPP with general weights. In the 2D setting, such models are also conjectured to be in the KPZ universality class, but much less is known due to the lack of exact formulae. 
The geometry of the set of geodesics has been an important 
tool in the study of these models; 
see e.g. \cite{New95, auffinger201750}.
When trying to understand the behaviour of large finite or infinite geodesics, a well-known open question is whether the empirical distributions of weights as seen along geodesics converge; 
see e.g. \cite{AIM} where it was proposed by Hoffman during a 2015 American Institute of Mathematics workshop.
Recently, Bates gave an affirmative answer to this question for various abstract dense families of weight distributions \cite{bates2020empirical}. The proof uses a variational formula, and does not rely on any exactly solvable structure.

In this paper we study the limiting local behaviour for LPP in the exactly solvable case.
We focus on LPP on $\Z^2$ with i.i.d.\;exponential weights.
Rather than the weights along geodesics, we consider the more general `empirical environments' around vertices, along finite or semi-infinite geodesics, and we show that they converge to a deterministic measure. 
By the environment around a vertex, we mean the weights of nearby vertices,
and the path of the geodesic through them. 
In particular, this positively answers the question of Hoffman for a first explicit model.
Our approach is different from \cite{bates2020empirical} and relies on information provided by the exactly
solvable structure. 
In addition to proving convergence results, we also give an explicit description of the limiting distribution, which depends on the direction of the considered geodesics.
Using this description one can compute any limiting local statistics along the geodesics, and we give some first examples in this paper.

A particular exactly solvable input that we use is the connection 
between LPP and the totally asymmetric exclusion process (TASEP), dating back to \cite{rost1981non}. We use the correspondence between LPP semi-infinite geodesics and the trajectory of a second-class particle in TASEP, as developed in a series of works \cite{ferrari2005competition, ferrari2009phase, pimentel2016duality}. Then in order to understand local environments along LPP geodesics, we study stationary distributions of TASEP as seen from an isolated second-class particle.
Models involving second-class particles have been proved
powerful in understanding the evolution of TASEP
\cite{ferrari1991microscopic, ferrari1992shock, 
derrida1993exact, balazs2006cube, balazs2010order, schmid2021mixing, schmid2022mixing},
and stationary distributions for multi-type systems
have been widely studied 
\cite{derrida1993exact, speer1994two, ferrari1994invariant, angel2006stationary, ferrari2007stationary, evans2009matrix, arita2011}.
See also \cite{ferrari2018tasep} for a recent survey of related ideas.

Before formally stating our results, we remark that (besides this paper and \cite{bates2020empirical}) there are several other recent works on environments along geodesics in random planar geometry.
In \cite{dauvergne2020three}, the authors study geodesics in the directed landscape, the joint scaling limit of exponential LPP (see \cite{DOV,dauvergne2021scaling}).
They proved that when zooming in around a point in the geodesic,  the local environment converges to an object termed `the directed landscape with Brownian-Bessel boundary conditions' 
(\cite[Theorem 1.1]{dauvergne2020three}).
Back to the non-exactly solvable model of general weights FPP, tail estimates for the averaged empirical distribution of weights along geodesics have been obtained in \cite{janjigian2020tail}.
In \cite{basu2021environment}, convergence of the empirical distribution of environments along geodesics has been obtained in the Liouville Quantum Gravity setting.

\subsection{Model definition and main results}
We study the exponential weight planar directed last-passage percolation (LPP) model, which is defined as follows.
To each vertex $v \in \Z^2$ we associate an independent weight $\xi(v)$ with $\Exp(1)$ distribution.
For two vertices $u, v \in \Z^2$, we say $u\leq v$ if $u$ is coordinate-wise less than or equal to $v$.
For such $u, v$ and any up-right path $\gamma$ from $u$ to $v$, we define the \emph{passage time of the path} to be
\[
T(\gamma) := \sum_{w \in \gamma} \xi(w) .
\]
Then almost surely there is a unique up-right path from $u$ to $v$ that has the largest passage time.
We call this path the \emph{geodesic} $\Gamma_{u,v}$, and call $T_{u,v}:=T(\Gamma_{u,v})$ the \emph{(last-)passage time from $u$ to $v$}.
In this paper we always work under the event that there is a unique geodesic between any $u\le v$.

For any fixed $\rho\in (0,1)$, it is known that almost surely the following statements hold (see \cite{coupier2011multiple, ferrari2005competition}).
For each $u \in \Z^2$ there is a unique infinite up-right path from $u$ (called the \emph{semi-infinite geodesic} and denoted by $\Gamma_u^\rho$) asymptotically going to the $\brho:=((1-\rho)^2,\rho^2)$ direction, such that for any $v \le w$ contained in $\Gamma_u^\rho$, the part of $\Gamma_u^\rho$ between $v$ and $w$ is the geodesic $\Gamma_{v,w}$.
For any $u, v \in \Z^2$, the semi-infinite geodesics $\Gamma_u^\rho$ and $\Gamma_v^\rho$ coalesce; i.e.\;both $\Gamma_u^\rho\setminus \Gamma_v^\rho$ and $\Gamma_v^\rho\setminus \Gamma_u^\rho$ are finite.
Below and whenever we consider a specific $\rho$, we always work under the almost sure event where these statements hold.

Our main results concern the local behaviour around vertices along geodesics.
For each $v\in \Z^2$, we denote $\xi\{v\} := \{\xi(v+u)\}_{ u \in \Z^2}$.
For any (finite or semi-infinite) up-right path $\gamma$ we let $\gamma[i]$ be the $i$-th vertex in $\gamma$.

For any $u\le v\in \Z^2$, and each $w\in\Gamma_{u,v}$, we regard $(\xi\{w\}, \Gamma_{u,v}-w)$ as a point in $\R^{\Z^2}\times \{0,1\}^{\Z^2}$ (equipped with the product topology and the cylinder $\sigma$-algebra), and we define the \emph{empirical environment along $\Gamma_{u,v}$} as
\[
\mu_{u,v} := \frac{1}{|\Gamma_{u,v}|} \sum_{w\in \Gamma_{u,v}} \delta_{(\xi\{w\}, \Gamma_{u,v}-w)},
\]
where $\delta_{(\xi\{w\}, \Gamma_{u,v}-w)}$ is the dirac measure concentrated on $(\xi\{w\}, \Gamma_{u,v}-w)$.
Similarly, we define the \emph{empirical environment along the semi-infinite geodesic $\Gamma^\rho_v$} as
\[
\mu_{v;r}^\rho := \frac{1}{2r+1} \sum_{i=1}^{2r+1} \delta_{(\xi\{\Gamma^\rho_v[i]\}, \Gamma^\rho_v-\Gamma^\rho_v[i])},
\]
for any $v\in\Z^2$, $\rho\in(0,1)$, and $r\in\Z_{\ge 0}$.
We will show that these empirical environments converge.
For each $\rho$, there is a limiting measure $\nu^\rho$ on $\R^{\Z^2}\times \{0,1\}^{\Z^2}$, which is explicit and will be defined in Section \ref{sec:limit-dist}.

For any $n\in\Z$ we denote $\bn^\rho:=\left(\left\lfloor\frac{2(1-\rho)^2n}{\rho^2+(1-\rho)^2}\right\rfloor, \left\lceil\frac{2\rho^2n}{\rho^2+(1-\rho)^2}\right\rceil\right)$. We also denote $\boo:=(0,0)$.
For the following results we fix any $\rho\in(0,1)$.
\begin{theorem}  \label{thm:finite}
Almost surely the measures $\mu_{\boo,\bn^\rho}$ converge to $\nu^\rho$ weakly as $n\to \infty$.
\end{theorem}

\begin{theorem}  \label{thm:semi-infinite}
Almost surely the measures $\mu_{\boo;r}^\rho$ converge to $\nu^\rho$ weakly as $r\to \infty$.
\end{theorem}
In other words, for any bounded continuous function $f:\R^{\Z^2}\times \{0,1\}^{\Z^2}\to \R$ we have $\mu_{\boo,\bn^\rho}(f) \to \nu^\rho(f)$ as $n\to\infty$ almost surely, and $\mu_{\boo;r}^\rho(f) \to \nu^\rho(f)$ as $r\to\infty$ almost surely.

We also have convergence of distributions.
\begin{theorem}  \label{t:one-point-converg-infinite}
The laws of $(\xi\{\Gamma_\boo^\rho[i]\}, \Gamma_\boo^\rho-\Gamma_\boo^\rho[i])$ converge to $\nu^\rho$ weakly as $i\to\infty$.
\end{theorem}
\begin{theorem}  \label{t:one-point-converg-finite}
For each $0<\alpha<2$, the laws of $(\xi\{\Gamma_{\boo,\bn^\rho}[\lfloor \alpha n\rfloor]\}, \Gamma_{\boo,\bn^\rho}-\Gamma_{\boo,\bn^\rho}[\lfloor \alpha n\rfloor])$ converge to $\nu^\rho$ weakly as $n\to\infty$.
\end{theorem}
There results in particular imply that the marginal distribution of $\nu^\rho$ on $\R^{\Z^2}$ is singular to the i.i.d.\;$\Exp(1)$ distribution.
Specifically, these convergence results imply that for $(\xi,\gamma)\sim \nu^\rho$, the path $\gamma$ is a \emph{bigeodesic} for $\xi$, i.e.\;$\gamma$ is a bi-infinite up-right path such that for any $u\le v$ contained in $\gamma$, the part of $\gamma$ between $u$ and $v$ is the geodesic from $u$ to $v$, under the weights $\xi$. However, for $\xi$ being i.i.d.\;$\Exp(1)$, almost surely there is no bigeodesic, as proved in \cite{BBSnon, BHS}.

For the limiting measure $\nu^\rho$ to be defined in Section \ref{sec:limit-dist}, its construction is explicit, and from it one can compute any finite dimensional distributions of $\nu^\rho$, thus any limiting local statistics along exponential LPP geodesics.
Here we give a first example, which is the distribution function of $\xi(\boo)$ under $\nu^\rho$.
\begin{prop}  \label{prop:lawofone}
For $(\xi,\gamma)\sim \nu^\rho$, we have $\PP[\xi(\boo)>h]=\left(1+ \frac{\rho(1-\rho) h}{(1-\rho)^2+\rho^2} \right)(1+\rho(1-\rho) h)e^{-h}$. 
\end{prop}
The distribution of $\xi(\boo)$ given in Proposition \ref{prop:lawofone}
is a mixture of Gamma($1,1$), Gamma($2,1$) and Gamma($3,1$) distributions. 
In the case $\rho=1/2$, for example, the weights of this mixture 
are $1/4$, $1/2$, and $1/4$ respectively, and the distribution of $\xi(\boo)$
can be interpreted as that of $2\min(E_1+E_2, E_3+BE_4)$ with $B\sim\text{Bernoulli}(1/2)$ and
$(E_i)_{1\leq 1\leq 4}$ i.i.d.\;$\sim \Exp(1)$ independently of $B$.
Related but slightly less simple representations can be given for general $\rho$.
See the discussion after the proof of Proposition \ref{prop:turning} in Section \ref{sec:limit-dist}.

One interesting question in exponential LPP is to derive descriptions for geodesics.
They are known to be different from simple random walks, as their scaling limits are known be H\"{o}lder-$2/3^-$ regular \cite{DOV,dauvergne2020three}, and for $\Gamma_{\boo,\bn^\rho}$ its transversal fluctuation is in the order of $n^{2/3}$ \cite{balazs2006cube}.
Exact formulae for the geodesic one-point distribution are also obtained recently \cite{liu2022one}.
Our next result implies that $\Gamma_{\boo,\bn^\rho}$ is not like a simple random walk even at a small scale, by showing that one step is more likely to follow the same direction as the previous step than to make a `turning'.
It follows from the convergence results, and our explicit construction of the limiting measure $\nu^\rho$.
\begin{prop}  \label{prop:turning}
Denote by $N_{n,\rho}$ the number of `corners' along $\Gamma_{\boo,\bn^\rho}$; that is, the number of $v\in \Z^2$ such that $\{v-(1,0),v,v+(0,1)\} \subset \Gamma_{\boo,\bn^\rho}$, or $\{v-(0,1),v,v+(1,0)\} \subset \Gamma_{\boo,\bn^\rho}$.
Then almost surely we have $\frac{N_{n,\rho}}{2n}\to \frac{2\rho^2(1-\rho)^2(1+2\rho-2\rho^2)}{(1-\rho)^2+\rho^2}$, as $n\to\infty$.
\end{prop}
For example, for $\rho=1/2$, the proportion of steps which are `corners' converges to $3/8$.
For $(\xi, \gamma)\sim \nu^\rho$, the limiting path $\gamma$ can also be described as the `competition interface' in a growth process with some explicit random initial configurations. See the discussion at the end of Section \ref{sec:limit-dist}.

In our proofs of the above results we will use the connection between LPP and the totally asymmetric exclusion process (TASEP), which can be described as a Markov process $(\eta_t)_{t\in\R}$ on the space $\{0,1\}^\Z$ (also equipped with the product topology and the cylinder $\sigma$-algebra), where $\eta_t(x)=1$ means that there is a particle at site $x$ at time $t$, whereas 
$\eta_t(x)=0$ means that there is a hole at site $x$ at time $t$.
If there is a particle at site $x$ and a hole at site $x+1$, they switch at rate $1$, independently for all such $x$.
We shall consider TASEP with a single `second-class particle', which is denoted by $*$ and can switch with a hole to the right of it, or with a (normal) particle to the left of it.
We prove a corresponding result for TASEP with a single second-class particle as well, which may be of independent interest.
\begin{theorem}  \label{thm:cov-phi}
$\lim_{t\to\infty}\Phi_t^\rho = \Psi^\rho$ weakly.
\end{theorem}
Here $\Phi_t^\rho$ and $\Psi^\rho$ are measures on $\{0,1,*\}^\Z$ (with the product topology) to be defined in Section \ref{ssec:2cp-sta-def}, and we describe them here.
Consider TASEP with a single second-class particle, where initially the second-class particle is at the origin, and any other site has a (normal) particle with probability $\rho$ independently.
Then $\Phi_t^\rho$ is the law of such TASEP at time $t$, as seen from the only second-class particle. The measure $\Psi^\rho$ is the stationary distribution of TASEP as seen from an isolated second-class particle, with particle density $\rho$. 
In proving this theorem, we will also show that the corresponding stationary process (of TASEP as seen from an isolated second-class particle) is ergodic in time (Proposition \ref{prop:ergodic-2nd-class-tasep}).

\subsection{A roadmap of our arguments}
There are two main ingredients in our proofs of the above results: geometry of geodesics in exponential LPP, and TASEP as seen from an isolated second-class particle.

For each $\rho\in(0,1)$ there is a (density $\rho$) stationary distribution for TASEP, where for each site there is a particle with probability $\rho$ and a hole with probability $1-\rho$ independently (i.e.\;i.i.d.\;Bernoulli$(\rho)$).
Such i.i.d.\;Bernoulli TASEP corresponds to a growth process in $\Z^2$, which (when rotated by $\pi/4$) is a random walk at any time.
Dividing the interface into two competing clusters, this gives a competition interface which corresponds to a semi-infinite geodesic in LPP; see e.g. \cite{ferrari2009phase, pimentel2016duality}.
On the other hand, such a competition interface corresponds to a second-class particle in TASEP.
Thus, the environment seen from a semi-infinite geodesic corresponds to TASEP as seen from an isolated second-class particle. Connections between TASEP and LPP will be discussed in details in Section \ref{sec:prelim}.

We will construct the limiting measure $\nu^\rho$ in Section \ref{sec:limit-dist}, using the density $\rho$ stationary measure of TASEP as seen from an isolated second-class particle, as described in \cite{ferrari1994invariant} and to be studied in Section \ref{ssec:2cp-sta-def}; we then prove Propositions \ref{prop:lawofone} and \ref{prop:turning} in Section \ref{sec:limit-dist} assuming the convergence results.

For the convergence results we take the following approach.
For Theorem \ref{thm:cov-phi}, in Section \ref{ssec:cv-avg-tasep} we first prove a weak version of convergence in the averaged sense (Proposition \ref{prop:converge-2nd-class-tasep}), using a coupling argument of interacting particle systems.
In Section \ref{sec:cov-phi} we upgrade Proposition \ref{prop:converge-2nd-class-tasep} to Theorem \ref{thm:cov-phi} using LPP and geometric arguments.
In Section \ref{sec:weak-cov} we prove
weak versions of Theorems \ref{thm:finite} and \ref{thm:semi-infinite},
involving convergence in probability.
The in probability convergence along semi-infinite geodesics (Theorem \ref{thm:semi-infinite-in-prob}) is deduced from the TASEP convergence result Theorem \ref{thm:cov-phi} (or the averaged version Proposition \ref{prop:converge-2nd-class-tasep}) and ergodicity of the TASEP stationary process, which we have proved as Proposition \ref{prop:ergodic-2nd-class-tasep} in Section \ref{sec:semi-inf-cov}.
From then on we work completely in the LPP setting.
In Section \ref{sec:finite-cov} we prove the in probability convergence version of Theorem \ref{thm:finite} (Theorem \ref{thm:finite-slope}), by using Theorem \ref{thm:semi-infinite-in-prob} and covering a finite geodesic with an infinite one.

The next several sections rely on a generalization of Theorem \ref{thm:finite-slope}, which is Proposition \ref{prop:uniform-conv}, the main result of Theorem \ref{sec:paracov}. It says that for geodesics whose endpoints vary along two anti-diagonal segments, their empirical environments converge jointly (in probability).
The proof is via taking a finite (i.e.\;size not growing) dense family of geodesics, and showing that each geodesic connecting the two anti-diagonal segments can be mostly covered by one geodesic in the family.
Using this result, in Section  \ref{sec:one-point} we prove Theorems \ref{t:one-point-converg-infinite} and \ref{t:one-point-converg-finite}, by showing that environments of nearby vertices (along geodesics) are close to each other in distribution. In Section \ref{sec:exp-con}, by covering a long or semi-infinite geodesic by short ones, we prove that its empirical environment concentrates exponentially fast, and thus upgrade Theorem \ref{thm:semi-infinite-in-prob} to Theorem \ref{thm:semi-infinite} and Theorem \ref{thm:finite-slope} to Theorem \ref{thm:finite}.

At the end of this roadmap, we comment on how much our arguments rely on exact solvability. As mentioned above, while we do not work directly on formulae, we rely on the structure of the considered exponential LPP model.
The construction of $\nu^\rho$ in Section \ref{sec:limit-dist} uses the exact equivalence between exponential LPP and TASEP (as stated in Section \ref{sec:prelim}); and Section \ref{ssec:2cp-sta-def} contains purely interacting particle system arguments.
Most other proofs in this paper are via LPP geometric arguments, using basic estimates on passage times and geometric properties that have appeared in the literature (and are stated in Section \ref{ssec:elpps}).
For Section \ref{sec:cov-phi}, while we prove Theorem \ref{thm:cov-phi} which is about TASEP, the arguments are mainly via the connection with LPP and its geometry.
Starting from Section \ref{sec:weak-cov} all the proofs use only geometric arguments, except for the short Section \ref{ssec:semi-inf-weak} (where the in probability convergence of empirical environments along semi-infinite geodesics is quickly deduced using TASEP results).
We point out that the LPP geometric arguments throughout this paper are robust, with the only inputs from exact solvability being the passage time distribution tail estimates (Theorem \ref{t:onepoint} below), and that the so-called Busemann function (to be defined in Section \ref{ssec:buseman}) in an anti-diagonal is a random walk.

\subsection{Further applications and questions}
With the limiting measure $\nu^\rho$ one can get any local information along geodesics in LPP.
Before closing the introduction we discuss some questions, which can potentially be answered using our explicit description of $\nu^\rho$, either as direct applications or requiring some further analysis.

The first question is communicated to us by Alan Hammond.
Given that a vertex on a geodesic has a large weight, how would the local environment behave? For a vertex with a large weight, it would force the geodesic to go through it. Thus we expect that conditioned on this, weights of nearby vertices are distributed like i.i.d.\;$\Exp(1)$ random variables. From the TASEP aspect, a large weight corresponds to a long waiting time between two jumps of the second-class particle, and this is mostly due to a `jam' in TASEP, i.e.\;a consecutive sequence of particles to the right of the second-class particle, and a sequence of holes to the left. This resembles a `reversed' step initial condition.

A related question is about vertices near but off a geodesic.
For such vertices we have the following result.
\begin{lemma}
For $(\xi,\gamma)\sim \nu^\rho$, and any vertex $v\neq \boo$, the random variable $\xi(v)$ conditioned on $v\not\in\gamma$ is stochastically dominated by $\Exp(1)$.
\end{lemma}
\begin{proof}
For any vertices $u\le v$, any up-right path $\Gamma$ from $u$ to $v$, any vertex $w\not\in \Gamma$ with $u\le w \le$, and any $x>0$, the events $\Gamma_{u,v}=\Gamma$ and $\xi(w)>x$ are negatively correlated, by the FKG inequality.
Thus the law of $\xi(w)$ is stochastically dominated by $\Exp(1)$, conditioned on $\Gamma_{u,v}=\Gamma$.
This implies that for any vertex $v\neq \boo$, $n\in\N$, the random variable $\xi(\Gamma_{\boo,\bn^\rho}[n]+v)$ conditioned on that $v \not\in \Gamma_{\boo,\bn^\rho}-\Gamma_{\boo,\bn^\rho}[n]$ is stochastically dominated by $\Exp(1)$.
By Theorem \ref{t:one-point-converg-finite} and sending $n\to\infty$ we get the conclusion.
\end{proof}
It is then interesting to see if the distribution converges to $\Exp(1)$ as the distance of $v$ to the geodesic increases to infinity.

The next question is about a slightly different setting, that of LPP with i.i.d.\;geometric weights.
The main difference is that, due to that the weights are discrete, geodesics are not necessarily unique in this case.
However, one could still consider `rightmost' geodesics.
Geometric LPP corresponds to discrete-time TASEP, and one can similarly construct stationary measures for such TASEP as seen from an isolated second-class particle.
For a correspondence with rightmost geodesics, in discrete-time TASEP one takes a second-class particle which is prioritized to jump to the right rather than to the left.
One can similarly construct limiting measures, and thus get local information about the environment along rightmost geodesics.
One question that would be interesting to study is the proportion of `unique geodesic vertices'.
For fixed endpoints (or for one fixed endpoint and a fixed direction), take the intersection of all the geodesics, and call those vertices in that intersection `unique geodesic vertices'.
Do these unique geodesic vertices asymptotically make up a positive proportion of the geodesics? Furthermore, does the proportion converge in probability, and can we compute the limit explicitly? 
We think such questions are related
to the constructed limiting measures of the environment along rightmost geodesics, because we expect that 
a vertex in the geodesics is unlikely to be
`locally unique' without being a unique geodesic
vertex in the sense mentioned above. Anomalous 
`locally but not globally unique' vertices
should make up a vanishing proportion of the geodesics in the limit.

Another direction concerns the scaling limit of the measure $\nu^\rho$.
As mentioned above, in \cite{dauvergne2020three} the authors constructed the small scaling limit of the local environment around a vertex in the geodesic, in the directed landscape setting.
It is reasonable to expect that when zooming out, the measure $\nu^\rho$ would converge to the local environment constructed there.
Also, once this is established, we would like to see if our explicit description of $\nu^\rho$ could be used to get some explicit information about the local environment and geodesics in the directed landscape (see e.g. \cite[Problem 4]{dauvergne2020three}).
In fact, for the geodesic under $\nu^\rho$,  one can possibly obtain various information on its large scale behaviour using the description as a competition interface (see the end of Section \ref{sec:limit-dist}).

We expect that the LPP geometric arguments in this paper can be extended to get more information on environments along geodesics.
For example, it can be shown that, for any $0<\alpha<\beta<2$, the environments $(\xi\{\Gamma_{\boo,\bn^\rho}[\lfloor \alpha n\rfloor]\}, \Gamma_{\boo,\bn^\rho}-\Gamma_{\boo,\bn^\rho}[\lfloor \alpha n\rfloor])$ and $(\xi\{\Gamma_{\boo,\bn^\rho}[\lfloor \beta n\rfloor]\}, \Gamma_{\boo,\bn^\rho}-\Gamma_{\boo,\bn^\rho}[\lfloor \beta n\rfloor])$ are asymptotically independent, as $n\to\infty$.
A possible route to prove this statement is described as follows.
Consider the point-to-line profiles from $\boo$ to $\{(a,b):a+b=\lfloor (\alpha-\varepsilon) n \rfloor\}$ and from $\{(a,b):a+b=\lfloor (\alpha+\varepsilon) n \rfloor\}$ to $\bn^\rho$, i.e.\;consider the passage times $T_{\boo, u}$ and $T_{v,\bn^\rho}$, for $u,v$ varying in these two lines respectively. Here $\varepsilon>0$ is a small number.
These two point-to-line profiles are independent, each converges (after rescaling) to the so-called Airy$_2$ process \cite{BF08,BP08}, which is locally like a Brownian motion.
Then it can be shown that in small neighborhoods of the intersections of the geodesic $\Gamma_{\boo,\bn^\rho}$ with these two lines, the point-to-line profiles (after rescaling) are similar to two independent Brownian motions around the maximum of their sum, or equivalently $R-B, R+B$, where $R$ is a Bessel$_3$ process and $B$ is a Brownian motion. (In \cite{dauvergne2020three}, such behaviour is observed for geodesics in the directed landscape.)
Using coalescence of geodesics, such picture can be established even conditioned on the environment around $\Gamma_{\boo,\bn^\rho}[\lfloor \beta n\rfloor]$.
One can also show that the part of the geodesic $\Gamma_{\boo,\bn^\rho}$ between these lines is stable with respect to small perturbations of the point-to-line profiles. 
This implies that, no matter how the environment around $\Gamma_{\boo,\bn^\rho}[\lfloor \beta n\rfloor]$ behaves, conditioned on it the distribution of the environment around $\Gamma_{\boo,\bn^\rho}[\lfloor \alpha n\rfloor]$ remains roughly the same, as $n\to \infty$.
In fact, such asymptotic independence can be used to give an alternative proof of the convergence of $\mu_{\boo,\bn^\rho}$, without using any TASEP arguments or identifying the limit as $\nu^\rho$.
Indeed, it implies that for any bounded continuous $f$ the variance of $\mu_{\boo,\bn^\rho}(f)$ decays to zero. To upgrade such decay of variance to convergence, one needs to cover long geodesics with short ones, using arguments similar to those in Sections \ref{sec:weak-cov}--\ref{sec:exp-con}. \\

\noindent\textbf{Notations.}
Throughout the rest of this paper the following notations will be used.
For any $x,y\in\R\cup\{-\infty,\infty\}$ we denote $x\vee y = \max\{x,y\}$, and $x\wedge y =\min\{x,y\}$, and $\llbracket x, y\rrbracket$ is the set $[x, y]\cap \Z$.
For each $u=(a,b)\in \Z^2$, we denote $d(u)=a+b$ and $ad(u)=a-b$.
For $n\in\Z$ we denote $\LL_n=\{u\in\Z^2: d(u)=2n\}$.
Unless otherwise noted (mainly in Section \ref{ssec:elpps}), for the rest of this paper we always fix $\rho\in (0,1)$, and the choice of all other parameters and constants can depend on $\rho$.
We denote $\brho=((1-\rho)^2,\rho^2)$.
We also drop $\rho$ from some notations. Specifically, we write $\Gamma_u$ for $\Gamma_u^\rho$, $\mu_{v;r}$ for $\mu_{v;r}^\rho$, $\bn^{}$ for $\bn^\rho$, and $\nu$, $\Phi_t$, $\Psi$ for $\nu^\rho$, $\Phi_t^\rho$, $\Psi^\rho$.

\section{Stationary distribution of TASEP with a second-class particle} \label{ssec:2cp-sta-def}

We start with the totally asymmetric simple exclusion process (TASEP), which is a classical interacting particle system.
For TASEP with second-class particles, we represent it as a Markov process on $\{1,*,0\}^\Z$,
where the symbols $1$, $*$, and $0$ represent particles, 
second-class particles, and holes respectively. 
As in ordinary TASEP, any (normal) particle can switch with a hole to its right.
In addition, any second-class particle can switch with a hole to its right, and can switch with a (normal) particle to its left.
We consider TASEP as seen from an isolated second-class particle, which is related to LPP semi-infinite geodesics, as will be explained later in Section \ref{ssec:ci-hppair}.
Namely, suppose that $(\eta_t^*)_{t\in I}$ for some interval $I\subset \R$ is TASEP containing a single second-class particle, then the process $(\eta_t^*(l_t+\cdot))_{t\in I}$ is the corresponding TASEP as seen from an isolated second-class particle, where $l_t$ is the location of the second-class particle at time $t$.
There is a family of stationary distributions of TASEP as seen from an isolated second-class particle, constructed in \cite{ferrari1994invariant}.
In this section we study a particular one $\Psi=\Psi^\rho$, under which the configuration has the same asymptotic density $\rho$ of particles in both directions.

We first construct $\Psi$ following \cite{ferrari1994invariant}. We start by constructing a stationary distribution for TASEP with infinitely many second-class particles.

Let $Y_1(x), x\geq 1$ and $Y_2(x), x\geq 1$ be independent collections
of i.i.d.\;Bernoulli$(\rho)$ random variables. 
Let $R_1(x)=\sum_{y=1}^x Y_1(y)$ and $R_2(x)=\sum_{y=1}^x Y_2(y)$.
Then we can define a symmetric random walk $W$ by 
\begin{equation}\label{Wdef}
W(x)=R_2(x)-R_1(x)
\end{equation}
for $x\geq 0$.
We define also     
\begin{equation}\label{Mdef}
M(x)=\sup_{0\leq y\leq x} W(y),
\end{equation}
and $\cE=\{x\geq 1:M(x)>M(x-1)\}$. Then
$M(x)=|\cE\cap\llbracket 1,x \rrbracket|$.

Then we can see $M(x)-W(x)$ as a symmetric random walk with steps in $\{-1,0,1\}$ and forced to stay non-negative: if at one step this walk `tries' to go from $0$ to $-1$, it will be altered and stay $0$.
The points of $\cE$, i.e.\;the points of increase of $M$, are those steps where such alternation occurs. 
More precisely, we have that $x\in\cE$ if and only if $M(x-1)=W(x-1)$, and $Y_2(x)=1$, $Y_1(x)=0$.
By well-known properties of symmetric random walks, we can obtain that as $x\to\infty$,
$\PP[x\in\cE]$ decays like $x^{-1/2}$, while
$M(x)/x^{1/2}=|\cE\cap
\llbracket 1,x \rrbracket|/x^{1/2}$ converges in distribution to a random variable supported
on $(0,\infty)$.

Now we define a configuration $\sigma$ on $\Z_{\ge 0}$, by copying $Y_1$ except at points of $\cE$.
We set $\sigma(0)=*$ and, for $x\geq 1$,
\begin{equation}\label{eq:sigmadef}    
\sigma(x)=
\begin{cases}
1&\text{if }Y_1(x)=1\\
0&\text{if }Y_1(x)=0 \text{ and } x\notin\cE\\
*&\text{if }Y_1(x)=0 \text{ and } x\in\cE.
\end{cases}
\end{equation}
(There is a natural interpretation involving the departure process 
of a discrete-time $M/M/1$ queue -- see \cite{ferrari2007stationary}.)
We wish to extend this to give a configuration $\sigma(x)$ on the whole line $\Z$. We can do this in two equivalent ways:
\begin{enumerate}
    \item Note that $\sigma(x), x\geq 0$ is a renewal process
    with renewals at points $x$ where $\sigma(x)=*$, i.e.\;where
    $x\in\cE$. Between successive renewal points, we see an 
    i.i.d.\;sequence
    of finite strings in $\cup_{n\geq0}\{0,1\}^n$ (but the length of each string has an infinite expectation).
    We can extend $\sigma$ to a renewal process on the whole line by extending this sequence of i.i.d.\;strings, separated by stars, leftward from the origin also. 
    \item Alternatively, we can exploit the symmetry of TASEP under exchanging
    holes/particles and left/right. Write $\pi_\rho$ for the distribution defined above
    on $\sigma(x), x\geq 0$. Now generate another configuration $\tilde{\sigma}(x), x\geq 0$
    from $\pi_{1-\rho}$, independently of $\sigma$, and for $x\geq 1$ set
    \[
    \sigma(-x)=\begin{cases} 
    1&\text{if }\tilde{\sigma}(x)=0\\
    0&\text{if }\tilde{\sigma}(x)=1\\
    *&\text{if }\tilde{\sigma}(x)=*.
    \end{cases}
    \]
\end{enumerate}
The equivalence of these two definitions follows from 
the random walk construction above. If we look at the configuration
between $0$ and the first $*$ to the right of $0$, we obtain
a finite string of holes and particles whose distribution
is invariant under exchanging both left/right and hole/particle; 
this invariance comes from the invariance under reflection
of the random walk path beginning and ending at level $0$ which
is used to construct the configuration.

We also extend the definition of $\cE$ to the whole line, by saying $x\in\cE$
whenever $\sigma(x)=*$.

Now we have defined the distribution of $\{\sigma(x)\}_{x\in\Z}$. From the construction we note immediately that $\{\sigma(x)\}_{x\in\Z_+}$ is independent of $\{\sigma(x)\}_{x\in\Z_-}$.
Also if we consider the interval $\llbracket -x, x\rrbracket$, as $x\to \infty$ the density of $*$ in this interval converges to $0$ (since $\PP[x\in\cE]\to 0$ as $x\to\infty$), and the densities of $1$ and $0$ converge to $\rho$ and $1-\rho$ respectively.

This distribution is stationary for TASEP with second-class particles, as seen from one of the second-class particles.
\begin{prop}[\protect{\cite[Theorem 1]{ferrari1994invariant}}]  \label{prop:ferr94}
Let $(\sigma_t)_{t\ge 0}$ be TASEP with second-class particles, started from $\sigma_0=\sigma$. Suppose that at time $t\ge 0$, the second-class particle starting from the origin is at site $l_t$. Then $\sigma_t(l_t+\cdot)$ has the same distribution as $\sigma$.
\end{prop}

Given $\sigma$, there are two related projections of 
it which involve setting all the $*$ symbols except for the one at the origin
to be either $1$s or $0$s. 
\begin{enumerate}
    \item The simpler one consists of setting all $*$ symbols on positive sites (i.e.\;$\Z_+$) to be $0$,
    and all $*$ symbols on negative sites (i.e.\;$\Z_-$) to be $1$. This gives a configuration where the non-zero sites are i.i.d.\;
    Bernoulli$(\rho)$. 
    \item
    Alternatively, we can follow the opposite rule of setting all $*$ symbols on positive sites
    to be $1$ and all $*$ symbols on negative sites to be $0$. Specifically, define
    a configuration $\zeta^*$ by $\zeta^*(0)=*$ and for $x\ne0$, 
    \[
    \zeta^*(x)=
    \begin{cases}
    0&\text{if }\sigma(x)=0, \text{ or if } \sigma(x)=* \text{ and } x<0\\
    1&\text{if }\sigma(x)=1, \text{ or if } \sigma(x)=* \text{ and } x>0.
    \end{cases}
    \]
    This gives a configuration which, compared to the product measure of Bernoulli$(\rho)$,
    has a bias towards particles on positive sites and towards holes on negative sites.
    This bias decays as one gets further away from the origin.
\end{enumerate}
We define $\Psi$ to be the distribution of this $\zeta^*$.
Theorem 2 of \cite{ferrari1994invariant} says that it is 
stationary for TASEP as seen from an isolated second-class particle. The bias
above reflects the tendency created by the dynamics of the process for the second-class particle to get stuck behind particles and to get stuck in front of holes. 

The combination of the two projections above gives a coupling between the configuration $\zeta^*$ and the i.i.d.\;Bernoulli$(\rho)$ configuration in which the discrepancies are precisely the non-zero members of $\cE$. The fact that $|\cE\cap \llbracket 1,x \rrbracket|$ grows in the order of $\sqrt{x}$ implies that the product measure of Bernoulli$(\rho)$ and the stationary distribution of TASEP as seen from an isolated second-class particle are mutually singular. 

For later calculation, it will be useful to look at the position of 
the first hole to the right of the origin in $\zeta^* \sim \Psi$
(and similarly the first particle to the left). 

Let $X_+=\min\{x\geq 1: \zeta^*(x)=0\}$, which is also $\min\{x\geq 1:\sigma(x)=0\}$.
From the random walk construction of $\sigma(x), x>0$, one gets that
\[
X_+=\min\{x\geq 1: Y_1(x)=0, \text{ and for some }y\in\llbracket1,x\rrbracket, Y_2(y)=0\}.
\]
That is, to find $X_+$ we look for the first $0$ in the process $Y_2$, and then we look for the first $0$ in the process $Y_1$ from then on. Since all the variables $Y_1(x)$ and $Y_2(x)$ 
are i.i.d.\;Bernoulli$(\rho)$, this gives that $X_++1$ is the sum of two Geometric($1-\rho$) random variables, and so
\begin{equation}\label{eq:Xplus}
\PP[X_+=k]=k(1-\rho)^2\rho^{k-1}
\end{equation}
for $k\geq 1$.
Similarly if $X_-$ is the location of the first particle to the left of the origin,
then 
\begin{equation}\label{eq:Xminus}
\PP[X_-=-k]=k\rho^2(1-\rho)^{k-1}.
\end{equation}

For the next two subsections, we prove two properties of $\Psi$, respectively: (1) the corresponding stationary process of TASEP as seen from an isolated second-class particle is ergodic in time, and (2) convergence to $\Psi$ starting from the i.i.d.\;Bernoulli$(\rho)$ configuration, in the averaged sense (in other words, a weak version of Theorem \ref{thm:cov-phi}). These two properties will be key inputs to the rest of this paper.

\subsection{Ergodicity}   \label{sec:semi-inf-cov}

This subsection is devoted to proving the following ergodicity statement.
We let $(\zeta_t^*)_{t\in\R}$ denote the process of TASEP as seen from an isolated second-class particle, such that $\zeta_t^*\sim \Psi$ for each $t$.
\begin{prop}  \label{prop:ergodic-2nd-class-tasep}
The process $(\zeta_t^*)_{t\in\R}$ is ergodic in time.
\end{prop}
The key step is the following coupling between $\Psi$ and itself.
\begin{lemma}\label{l:couple-2-psi-rho}
For any $L\in \N$ and $\epsilon > 0$, there exist an integer $M>L$, and a coupling between $\Psi$ and itself, such that the following is true.
Let $\zeta^{(1)}$ and $\zeta^{(2)}$ be sampled from this coupling, then
\begin{enumerate}
    \item restricted to $\llbracket -L, L\rrbracket$, $\zeta^{(1)}$ and $\zeta^{(2)}$ are independent.
    \item with probability $>1-\epsilon$, $\zeta^{(1)}$ and $\zeta^{(2)}$ have the same number of particles in $\llbracket -M, -1 \rrbracket$ and in $\llbracket 1, M\rrbracket$, and $\zeta^{(1)}$ and $\zeta^{(2)}$ are identical on $\Z\setminus\llbracket -M, M\rrbracket$.
\end{enumerate}
\end{lemma}
To construct this coupling, we revisit the construction of $\Psi$.
For $\zeta^*\sim\Psi$, recall that we defined it on $\Z_+$ using two independent collections of i.i.d.\;Bernoulli$(\rho)$
random variables $Y_1(x), x\geq 1$ and $Y_2(x), x\geq 1$; and  $R_1(x)=\sum_{y=1}^x Y_1(y)$, $R_2(x)=\sum_{y=1}^x Y_2(y)$.
For $x\ge 1$, let
\begin{gather*}
    \oY_1(x)=\zeta^*(x)=
    \begin{cases}
        1,&Y_1(x)=1 \text{ or } x\in\cE\\
        0,&Y_1(x)=0 \text{ and } x\notin\cE
    \end{cases}
    \\
    \oY_2(x)=
    \begin{cases}
        0,&Y_2(x)=0 \text{ or } x\in\cE\\
        1,&Y_2(x)=1 \text{ and } x\notin\cE
    \end{cases}.
\end{gather*}
Namely, $\oY_1$ is just $\zeta^*$ on $\Z_+$, and $\oY_2$ is `paired with' $\oY_1$, such that $\oY_1+\oY_2=Y_1+Y_2$.
To see why we define $\oY_2$, consider
\begin{align*}
    \oR_1(x)&=\sum_{y=1}^x \oY_1(x) =R_1(x)+M(x)\\
    \oR_2(x)&=\sum_{y=1}^x \oY_2(x) =R_2(x)-M(x).
\end{align*}
We have that $\oR_1(x)-\oR_2(x)=2M(x)-W(x)$, 
where $W$ and $M$ are defined in (\ref{Wdef}) and (\ref{Mdef}). 
In particular we have $\oR_1(x)\geq \oR_2(x)$ for all $x$.
Note that $\oR_1(x)$ is the number of particles of $\zeta^*$
in the interval $\llbracket 1,x \rrbracket$. 
The process $\oR_1$ is certainly not Markovian; however, the process $(\oR_1(x), \oR_2(x)), x\geq 0$ is a Markov chain, and we will exploit this fact.

Consider the transition function $\ocT:\Z^2_{\ge 0}\times \Z^2_{\ge 0} \to [0,1]$ defined by
\[
\begin{split}
\ocT((a,b),(a+1,b+1)) &=
\rho^2,
\\ 
\ocT((a,b),(a+1,b)) &=
\rho(1-\rho) \frac{a-b+2}{a-b+1},
\\
\ocT((a,b),(a,b+1)) &=
\rho(1-\rho) \frac{a-b}{a-b+1},
\\
\ocT((a,b),(a,b)) &= 
(1-\rho)^2.
\end{split}
\]
\begin{lemma}  \label{lem:oS-Markov}
The process $(\oR_1,\oR_2)$ is a Markov chain in $\Z^2_{\ge 0}$ with transition probability $\ocT$.
\end{lemma}
\begin{proof}
For any $x\ge 0$, we show that
\begin{multline}  \label{eq:r-e-joint-distr}
\PP\big[\{\oR_1(y)\}_{y=0}^x=\{r_1(y)\}_{y=0}^x, \{\oR_2(y)\}_{y=0}^x=\{r_2(y)\}_{y=0}^x, M(x)=h\big]\\ = \rho^{r_1(x)+r_2(x)}(1-\rho)^{2x-r_1(x)-r_2(x)}
\end{multline}
for any integers $\{r_1(y)\}_{y=0}^x$, $\{r_2(y)\}_{y=0}^x$ and $h$ such that
\begin{enumerate}
    \item $r_1(0)=r_2(0)=0$,
    \item $r_1(y)-r_1(y-1), r_2(y)-r_2(y-1) \in \{0,1\}$, and $r_1(y)\ge r_2(y)$ for any $1\le y \le x$,
    \item $0\le h \le r_1(x)-r_2(x)$.
\end{enumerate}
We prove this by induction on $x$. The base case (of $x=0$) is trivial, and now we assume that it is true for $x$, and consider $x+1$.

Note that we have $x+1\in\cE$ if the following three conditions all hold:
(i) $M(x)=W(x)$ (i.e.\;$\oR_1(x)-\oR_2(x)=M(x)$); 
(ii) $Y_1(x+1)=0$; (iii) $Y_2(x+1)=1$. 
In that case we have $\oR_1(x+1)=\oR_1(x)+1$, 
$\oR_2(x+1)=\oR_2(x)$, and $M(x+1)=M(x)+1$.
In any other case we have $\oR_1(x+1)=\oR_1(x)+Y_1(x+1)$, $\oR_2(x+1)=\oR_2(x)+Y_2(x+1)$, and $M(x+1)=M(x)$.

Denote $y_1(x+1)=r_1(x+1)-r_1(x)$ and $y_2(x+1)=r_2(x+1)-r_2(x)$. From the above transition we have that when $h\le r_1(x)-r_2(x)$,
\begin{align*}
& \PP[\{\oR_1(y)\}_{y=0}^{x+1}=\{r_1(y)\}_{y=0}^{x+1}, \{\oR_2(y)\}_{y=0}^{x+1}=\{r_2(y)\}_{y=0}^{x+1}, M(x+1)=h] \\
=&
\PP[\{\oR_1(y)\}_{y=0}^x=\{r_1(y)\}_{y=0}^x, \{\oR_2(y)\}_{y=0}^x=\{r_2(y)\}_{y=0}^x, M(x)=h]\\
& \times \PP[Y_1(x+1)=y_1(x+1), Y_2(x+1)=y_2(x+1)],
\end{align*}
where the second probability on the right-hand side equals $\rho^{y_1(x+1)+y_2(x+1)}(1-\rho)^{2-y_1(x+1)-y_2(x+1)}$.
When $h> r_1(x)-r_2(x)$, we must have that $h=r_1(x)-r_2(x)+1$ and $y_1(x+1)=1$, $y_2(x+1)=0$, and that
\begin{align*}
& \PP[\{\oR_1(y)\}_{y=0}^{x+1}=\{r_1(y)\}_{y=0}^{x+1}, \{\oR_2(y)\}_{y=0}^{x+1}=\{r_2(y)\}_{y=0}^{x+1}, M(x+1)=h] \\
=&
\PP[\{\oR_1(y)\}_{y=0}^x=\{r_1(y)\}_{y=0}^x, \{\oR_2(y)\}_{y=0}^x=\{r_2(y)\}_{y=0}^x, M(x)=h-1]\\
& \times \PP[Y_1(x+1)=0, Y_2(x+1)=1],
\end{align*}
where the second probability on the right-hand side equals $\rho(1-\rho)$, which also equals $\rho^{y_1(x+1)+y_2(x+1)}(1-\rho)^{2-y_1(x+1)-y_2(x+1)}$.
Thus by the induction hypothesis (\eqref{eq:r-e-joint-distr} for $x$), we get \eqref{eq:r-e-joint-distr} for $x+1$.

Finally, by summing over all $h$, we conclude that
\begin{multline*}
\PP[\{\oR_1(y)\}_{y=0}^{x}=\{r_1(y)\}_{y=0}^{x}, \{\oR_2(y)\}_{y=0}^{x}=\{r_2(y)\}_{y=0}^{x}] \\
=
(r_1(x)-r_2(x)+1)\rho^{r_1(x)+r_2(x)}(1-\rho)^{2x-r_1(x)-r_2(x)}.
\end{multline*}
Using this we conclude that
\begin{multline*}
\PP[\oR_1(x+1)=r_1(x+1), \oR_2(x+1)=r_2(x+1) \mid \{\oR_1(y)\}_{y=0}^x=\{r_1(y)\}_{y=0}^x, \{\oR_2(y)\}_{y=0}^x=\{r_2(y)\}_{y=0}^x] \\
=
\frac{r_1(x+1)-r_2(x+1)+1}{r_1(x)-r_2(x)+1}
\rho^{y_1(x)+y_2(x)}(1-\rho)^{2-y_1(x)-y_2(x)},
\end{multline*}
which implies the conclusion.
\end{proof}

We have the following mixing property of this Markov chain.
\begin{lemma}   \label{l:mix-nointerrw}
For any $u,v\in\Z^2_{\ge 0}$, we have $\lim_{n\to\infty}\|\ocT^n(u,\cdot)-\ocT^n(v,\cdot)\|_1=0$.
\end{lemma}
\begin{proof}
Our strategy is to construct a coupling between two Markov chains, each with transition probability $\ocT$, starting from $u$ and $v$ respectively.

To construct the coupling, we recursively define a random process $(A^{(1)}, A^{(2)}, B^{(1)}, B^{(2)}):\Z_{\ge 0} \to \Z^2\times \Z_{\ge 0}^2$.
For $x\in\Z_{\ge 0}$, given $A^{(1)}(y)$, $B^{(1)}(y)$, $A^{(2)}(y)$, $B^{(2)}(y)$ for each $y\in\llbracket 0, x \rrbracket$, we define $A^{(1)}(x+1)$, $B^{(1)}(x+1)$, $A^{(2)}(x+1)$, $B^{(2)}(x+1)$ as follows.
First, we let $U_0$ be a Bernoulli$(2\rho(1-\rho))$ random variable.
\begin{enumerate}
    \item If $U_0=0$ we do the following. Let $B^{(1)}(x+1)=B^{(1)}(x)$ and $B^{(2)}(x+1)=B^{(2)}(x)$. If $A^{(1)}(x)=A^{(2)}(x)$, we let
$A^{(1)}(x+1)=A^{(2)}(x+1)=A^{(1)}(x)+2U_1-1=A^{(2)}(x)+2U_1-1$; otherwise we let $A^{(1)}(x+1)=A^{(1)}(x)+2U_1-1$ and $A^{(2)}(x+1)=A^{(2)}(x)+2U_2-1$. Here $U_1$ and $U_2$ are independent Bernoulli$\big(\frac{\rho^2}{\rho^2+(1-\rho)^2}\big)$ random variables, and are independent of $U_0$.
    \item If $U_0=1$ we do the following. Let $A^{(1)}(x+1)=A^{(1)}(x)$ and $A^{(2)}(x+1)=A^{(2)}(x)$. 
    \begin{itemize}
        \item If $B^{(1)}(x)=B^{(2)}(x)$, we let
$B^{(1)}(x+1)=B^{(2)}(x+1)=B^{(1)}(x)+2U_3-1=B^{(2)}(x)+2U_3-1$.
\item If $B^{(1)}(x)\neq B^{(2)}(x)$ and $\max_{0\le y \le x}B^{(1)}(y)\ge N$, we let $B^{(1)}(x+1)=B^{(1)}(x)+2U_3-1$ and $B^{(2)}(x+1)=B^{(2)}(x)+2U_4-1$.
\item If $B^{(1)}(x)< B^{(2)}(x)$ and $\max_{0\le y \le x}B^{(1)}(y)< N$, we let $B^{(1)}(x+1)=B^{(1)}(x)+2U_3-1$ and $B^{(2)}(x+1)=B^{(2)}(x)+2U_3U_5-1$.
\item If $B^{(1)}(x)> B^{(2)}(x)$ and $\max_{0\le y \le x}B^{(1)}(y)< N$, we let $B^{(1)}(x+1)=B^{(1)}(x)+2U_4U_6-1$ and $B^{(2)}(x+1)=B^{(2)}(x)+2U_4-1$.
    \end{itemize}
Here $N>0$ is a number to be determined; and $U_3$, $U_4$, $U_5$, $U_6$ are independent with distribution being Bernoulli$\big(\frac{B^{(1)}(x)+2}{2B^{(1)}(x)+2}\big)$, Bernoulli$\big(\frac{B^{(2)}(x)+2}{2B^{(2)}(x)+2}\big)$, Bernoulli$\big(\frac{B^{(2)}(x)+2}{2B^{(2)}(x)+2}\cdot \frac{2B^{(1)}(x)+2}{B^{(1)}(x)+2}\big)$, and Bernoulli$\big(\frac{2B^{(2)}(x)+2}{B^{(2)}(x)+2}\cdot \frac{B^{(1)}(x)+2}{2B^{(1)}(x)+2}\big)$, respectively; and they are independent of $U_0, U_1, U_2$.
\end{enumerate}
The reason behind the construction of  $(A^{(1)}, A^{(2)}, B^{(1)}, B^{(2)})$ is that, if we set the initial condition to be $A^{(1)}(0)=d(u)$, $B^{(1)}(0)=ad(u)$, and $A^{(2)}(0)=d(v)$, $B^{(2)}(0)=ad(v)$, for  $u,v\in\Z^2_{\ge 0}$, then it is straightforward to check that $x\mapsto\big(\frac{A^{(1)}(x)+B^{(1)}(x)+x}{2}, \frac{A^{(1)}(x)-B^{(1)}(x)+x}{2}\big)$ and $x\mapsto\big(\frac{A^{(2)}(x)+B^{(2)}(x)+x}{2}, \frac{A^{(2)}(x)-B^{(2)}(x)+x}{2}\big)$ are Markov chains with the same transition probability $\ocT$, starting from $u$ and $v$ respectively.
Indeed, from this construction, it is easy to check that for each $i=1, 2$, $(A^{(i)}, B^{(i)})$ is a Markov chain, with transition probability given by
\begin{equation}  \label{eq:ABtrans}
\begin{split}
\PP[A^{(i)}(x+1)=A^{(i)}(x)+1, B^{(i)}(x+1)=B^{(i)}(x)\mid A^{(i)}(x),B^{(i)}(x)] &=
\rho^2
\\ 
\PP[A^{(i)}(x+1)=A^{(i)}(x)-1, B^{(i)}(x+1)=B^{(i)}(x)\mid A^{(i)}(x),B^{(i)}(x)]  &=
(1-\rho)^2,
\\
\PP[A^{(i)}(x+1)=A^{(i)}(x), B^{(i)}(x+1)=B^{(i)}(x)+1\mid A^{(i)}(x),B^{(i)}(x)]
&=\rho(1-\rho)\frac{B^{(i)}(x)+2}{B^{(i)}(x)+1},
\\
\PP[A^{(i)}(x+1)=A^{(i)}(x), B^{(i)}(x+1)=B^{(i)}(x)-1\mid A^{(i)}(x),B^{(i)}(x)]
&=\rho(1-\rho)\frac{B^{(i)}(x)}{B^{(i)}(x)+1}.
\end{split}    
\end{equation}
From the construction above there are several other key properties to note.
\begin{enumerate}
    \item When $B^{(1)}(x)\neq B^{(2)}(x)$ and $\max_{0\le y \le x}B^{(1)}(y)<N$, there is always $|B^{(1)}(x+1)-B^{(2)}(x+1)|\le |B^{(1)}(x)- B^{(2)}(x)|$.
    \item If $A^{(1)}(x)= A^{(2)}(x)$ (resp.\;$B^{(1)}(x)= B^{(2)}(x)$), for any $y\ge x$ we must have $A^{(1)}(y)= A^{(2)}(y)$ (resp.\;$B^{(1)}(y)= B^{(2)}(y)$).
    \item The processes $A^{(1)}$ and $A^{(2)}$ are independent random walks until they are equal; starting from the first time when $B^{(1)}$ reaches $N$, the processes $B^{(1)}$ and $B^{(2)}$ are independent until they are equal.
\end{enumerate}
To show that $\lim_{n\to\infty}\|\ocT^n(u,\cdot)-\ocT^n(v,\cdot)\|_1=0$, it now suffices to show that
\begin{equation} \label{eq:cpsuff}
\liminf_{x\to\infty} \PP[A^{(1)}(x)=A^{(2)}(x),\; B^{(1)}(x)=B^{(2)}(x)] > 1-\epsilon,    
\end{equation}
for any $\epsilon>0$ and some choice of $N$.
First, we have that $A^{(1)}(x)=A^{(2)}(x)$ for all large enough $x$, by the third property above.

We next show that when $N$ is large enough depending on $u,v,\epsilon$, with probability at least $1-\epsilon$ we have $B^{(1)}(x)=B^{(2)}(x)$ for some large enough $x$ (thus for all large $x$, by the second property above).
Let $x_0 = \min\{x\in\Z_{\ge 0}: B^{(1)}(x)=N\}$.
We have $x_0<\infty$ almost surely, since $B^{(1)}$ dominates a simple random walk.

As stated in the third property above, given $B^{(1)}(x_0)$ and $B^{(2)}(x_0)$, the processes $B^{(1)}(x_0+x)$ and $B^{(2)}(x_0+x)$ for $x\ge 0$ are independent (until they are equal); and we further note that, when $N$ is taken large they should be very close to two independent random walks.
To make this more precise, we define proxies of $B^{(1)}$ and $B^{(2)}$.
For $i=1,2$, let $V^{(i)}:\Z_{\ge 0}\to \Z$ be a random walk satisfying $V^{(i)}(0)=B^{(i)}(x_0)$, and
\begin{equation}  \label{eq:Vitrans}
\begin{split}
\PP[V^{(i)}(x+1)=V^{(i)}(x)\mid V^{(i)}(x)] &=
\rho^2+(1-\rho)^2,
\\
\PP[V^{(i)}(x+1)=V^{(i)}(x)+1\mid V^{(i)}(x)]
&=\rho(1-\rho),
\\
\PP[V^{(i)}(x+1)=V^{(i)}(x)-1\mid V^{(i)}(x)]
&=\rho(1-\rho).
\end{split}    
\end{equation}
Also we let $V^{(1)}$ and $V^{(2)}$ be independent, until $V^{(1)}(x_1)=V^{(2)}(x_1)$ for some $x_1>0$, and let $V^{(1)}(x)=V^{(2)}(x)$ for all $x>x_1$.
For some $N_1$ large enough (depending on $u, v, \epsilon$) we have $\PP[x_1<N_1]>1-\epsilon/2$, thus
\begin{equation}  \label{eq:V12eq}
\PP[V^{(1)}(N_1)=V^{(2)}(N_1)] > 1-\epsilon/2.
\end{equation}
By comparing the transition probabilities \eqref{eq:ABtrans} and \eqref{eq:Vitrans}, we can couple $B^{(1)}, B^{(2)}$ with $V^{(1)}, V^{(2)}$, such that for any $x\ge 0$, given that $B^{(1)}(x_0+x)=V^{(1)}(x)$ and $B^{(2)}(x_0+x)=V^{(2)}(x)$, we have $B^{(1)}(x_0+x+1)=V^{(1)}(x+1)$ and $B^{(2)}(x_0+x+1)=V^{(2)}(x+1)$ with probability at least \[1-\rho(1-\rho)\bigg(\frac{1}{B^{(1)}(x_0+x)+1} + \frac{1}{B^{(2)}(x_0+x)+1}\bigg)> 1-\frac{2\rho(1-\rho)}{N-x-\|u-v\|_1}.\]
Here the inequality is due to that $B^{(1)}(x_0+x)\ge B^{(1)}(x_0)-x = N-x$, and $B^{(2)}(x_0+x) \ge B^{(1)}(x_0+x) - |B^{(1)}(0)-B^{(2)}(0)|\ge N-x-\|u-v\|_1$, using the first property above.
Under this coupling, by taking a union bound over $x$ we have that $V^{(1)}(x)=B^{(1)}(x_0+x)$ and $V^{(2)}(x)=B^{(2)}(x_0+x)$ for any $0\le x \le N_1$ with probability at least $1-\frac{2\rho(1-\rho)N_1}{N-N_1-\|u-v\|_1}$.
By taking $N$ large enough (depending on $N_1, \epsilon, u, v$) we can make this probability $>1-\epsilon/2$.
From this and \eqref{eq:V12eq}, we have $\PP[B^{(1)}(x_0+N_1)=B^{(2)}(x_0+N_1)]>1-\epsilon$. This implies \eqref{eq:cpsuff}, and the conclusion follows.
\end{proof}

We let $\sS$ denote the law of a Markov chain starting from $(0,0)$ with transition probability $\ocT$, i.e.\;the law of $(\oR_1,\oR_2)$.
From the above lemma we could construct a coupling between $\sS$ and itself, as follows.
\begin{lemma}\label{l:couple-2-S}
For any $L\in \N$ and $\epsilon > 0$, there exist an integer $M>L$, and a coupling between $\sS$ and itself, such that the following is true.
Let $(\oR_1^{(1)}, \oR_2^{(1)})$ and $(\oR_1^{(2)}, \oR_2^{(2)})$ be sampled from this coupling, then
\begin{enumerate}
    \item restricted to $\llbracket 0, L\rrbracket$, $(\oR_1^{(1)}, \oR_2^{(1)})$ and $(\oR_1^{(2)}, \oR_2^{(2)})$ are independent.
    \item $\PP[\oR_1^{(1)}(M)=\oR_1^{(2)}(M), \oR_2^{(1)}(M)=\oR_2^{(2)}(M)]>1-\epsilon$.
\end{enumerate}
\end{lemma}

\begin{proof}
We construct the coupling by first allowing $(\oR_1^{(1)}, \oR_2^{(1)})$ and $(\oR_1^{(2)}, \oR_2^{(2)})$ to evolve independently for the first $L$ steps.
Then conditioned on $(\oR_1^{(1)}(L), \oR_2^{(1)}(L))$ and $(\oR_1^{(2)}(L), \oR_2^{(2)}(L))$, 
we couple $(\oR_1^{(1)}(M), \oR_2^{(1)}(M))$ and $(\oR_1^{(2)}(M), \oR_2^{(2)}(M))$ to maximize the probability that they coincide.
The conclusion follows from Lemma \ref{l:mix-nointerrw} by taking $M$ large enough, since there are only finitely many possible values of $(\oR_1^{(1)}(L), \oR_2^{(1)}(L))$ and $(\oR_1^{(2)}(L), \oR_2^{(2)}(L))$.
\end{proof}

\begin{proof}[Proof of Lemma \ref{l:couple-2-psi-rho}]
From the coupling of two copies of $\oR_1$ given by Lemma \ref{l:couple-2-S}, we get a coupling between two copies of $\oY_1$, thus two copies of $\zeta^*\sim \Psi$ on $\Z_+$.
We can similarly couple two copies of $\zeta^*\sim \Psi$ on $\Z_-$.
As the measure $\Psi$ on $\Z_+$ and $\Z_-$ are independent, we get the desired coupling satisfying the statement of this lemma.
\end{proof}
We can now prove ergodicity of the stationary process of TASEP as seen from an isolated second-class particle, using the coupling given by Lemma \ref{l:couple-2-psi-rho}.
\begin{proof} [Proof of Proposition \ref{prop:ergodic-2nd-class-tasep}]
We assume the contrary. Then there is a measurable set $B\subset \{0,1,*\}^\Z$ invariant under the Markov process (of TASEP as seen for an isolated second-class particle), with $0<\Psi(B)<1$.
Let $\zeta^*\sim \Psi$.
For any $L\in \N$ we consider the random variable $\chi_L(\zeta^*)=\PP[\zeta^* \in B \mid \{\zeta^*(x)\}_{x\in\llbracket -L, L \rrbracket}]$. 
Note that this is a martingale in $L$, and almost surely converges to $\don[\zeta^*\in B]$.
Thus for any $\epsilon > 0$, we can take $L$ large enough, such that $\PP[|\chi_L(\zeta^*)-\don[\zeta^*\in B]| > \epsilon] < \epsilon$.

For the above $L$ and $\epsilon$, by Lemma \ref{l:couple-2-psi-rho} we can find $M>L$ and a coupling between $\Psi$ and itself.
Suppose that $\zeta^{(1)}, \zeta^{(2)}$ are sampled from this coupling.
By the first property of the coupling, and that $\chi_L$ only depends on the configuration in $\llbracket -L, L\rrbracket$, 
we have that $\chi_L(\zeta^{(1)})$ and $\chi_L(\zeta^{(2)})$ are independent. Thus
\[
\PP[\chi_L(\zeta^{(1)})>1-\epsilon,\;\chi_L(\zeta^{(2)})<\epsilon]
=
\PP[\chi_L(\zeta^{(1)})>1-\epsilon]\PP[\chi_L(\zeta^{(2)})<\epsilon].
\]
Note $\zeta^{(1)}\in B$ and $|\chi_L(\zeta^{(1)})-\don[\zeta^{(1)}\in B]| < \epsilon$ imply that $\chi_L(\zeta^{(1)})>1-\epsilon$, so we have
\[
\PP[\chi_L(\zeta^{(1)})>1-\epsilon] \ge \PP[\zeta^{(1)}\in B] - \PP[|\chi_L(\zeta^{(1)})-\don[\zeta^{(1)}\in B]| > \epsilon] > \Psi(B) - \epsilon,
\]
and similarly
\[
\PP[\chi_L(\zeta^{(2)})<\epsilon] \ge \PP[\zeta^{(2)}\not\in B] - \PP[|\chi_L(\zeta^{(2)})-\don[\zeta^{(2)}\in B]| > \epsilon] > 1-\Psi(B) - \epsilon.
\]
Combining the three above inequalities, we have
\[
\PP[\chi_L(\zeta^{(1)})>1-\epsilon,\;\chi_L(\zeta^{(2)})<\epsilon] >\Psi(B)(1-\Psi(B))-\epsilon.
\]
Using $\PP[|\chi_L(\zeta^{(1)})-\don[\zeta^{(1)}\in B]| > \epsilon] < \epsilon$ and $\PP[|\chi_L(\zeta^{(2)})-\don[\zeta^{(2)}\in B]| > \epsilon] < \epsilon$ again, we have
\[
\PP[\zeta^{(1)}\in B,\;\zeta^{(2)}\not\in B]>\Psi(B)(1-\Psi(B))-3\epsilon.
\]
Using the second property of the coupling (from Lemma \ref{l:couple-2-psi-rho}), and by taking $\epsilon$ small enough, we have that with probability $>\Psi(B)(1-\Psi(B))-4\epsilon>0$, all of the following conditions are satisfied: $\zeta^{(1)}\in B$ and $\zeta^{(2)}\not\in B$, and $\zeta^{(1)}$ and $\zeta^{(2)}$ are identical on $\Z\setminus\llbracket -M, M\rrbracket$, and they have the same number of particles in $\llbracket -M, -1 \rrbracket$ and in $\llbracket 1, M\rrbracket$.

Assuming that $\zeta^{(1)}$ and $\zeta^{(2)}$ satisfy the above conditions, we next couple two TASEPs starting from $\zeta^{(1)}$ and $\zeta^{(2)}$ at time $0$, such that switches happen between neighboring sites with the same Poisson clocks.
With positive probability the following happens: from time $0$ to time $1$, no switch happens between sites $x$ and $x+1$, for $x\in\{-M-1, -1, 0, M\}$;
and switches happen between sites $x$ and $x+1$, sequentially for $x=-M,\cdots, -2$ and for $x=1,\cdots, M-1$, and repeat this for $M$ times.
Then at time $1$ the two processes starting from $\zeta^{(1)}$ and $\zeta^{(2)}$ would be identical.
However, as $B$ and $\{0,1,*\}^\Z\setminus B$ are assumed to be invariant under the evolution of TASEP as seen from an isolated second-class particle, we get a subset of $\{0,1,*\}^\Z$ with positive $\Psi$ measure, and is contained (up to a zero measure set) in both  $B$ and $\{0,1,*\}^\Z\setminus B$. This is a contradiction.
\end{proof}

\subsection{Convergence in the averaged sense}  \label{ssec:cv-avg-tasep}

As indicated in the introduction, we consider the process $(\eta_t^*)_{t\ge 0}$, which is TASEP with a single second-class particle such that $\eta_0^*(x)$ are i.i.d Bernoulli$(\rho)$ for $x\in\Z\setminus\{0\}$ and $\eta_0^*(0)=*$.
We define $\Phi_t$ to be the law of $\eta_t^*(l_t+\cdot)$, where $l_t$ is the location of the second-class particle at time $t$.
In this subsection we prove a weak version of Theorem \ref{thm:cov-phi}, i.e.\;the convergence of $\Phi_t$ to $\Psi$ in the averaged sense.

\begin{prop}  \label{prop:converge-2nd-class-tasep}
We have $T^{-1}\int_0^T \Phi_t dt \to \Psi$ weakly, as $T\to\infty$.
\end{prop}
Our strategy to prove this is to construct a coupling between $\Phi_t$ and $\Psi$ using $\sigma$, the stationary configuration of TASEP with infinitely many second-class particles constructed in \eqref{eq:sigmadef}.

Recall that we have the following two projections of $\sigma$: first, if we set all $*$ symbols on positive sites to be $0$, and all $*$ symbols on negative sites to be $1$, we get i.i.d.\;Bernoulli$(\rho)$ on all non-zero sites;
second, if we set all $*$ symbols on positive sites to be $1$, and all $*$ symbols on negative sites to be $0$, we get a distribution which is stationary for TASEP as seen from an isolated second-class particle (see the discussion after Proposition \ref{prop:ferr94}).

Now let $(\sigma_t)_{t\ge 0}$ be TASEP with (infinitely many) second-class particles, and starting from $\sigma_0=\sigma$.
At time $0$, we label all the second-class particles with $\Z$ from right to left, such that the one at the origin is labeled $0$.
We consider two ways where the labels evolve.
\begin{itemize}
    \item Rule (a): for all second-class particles, the labels never change.
    \item Rule (b): for two second-class particles labeled $i>j$, if they are at sites $x$ and $x+1$, then with rate $1$ they exchange their labels.
\end{itemize}
We note that when forgetting the labels, the dynamic is unchanged.
For each $i\in \Z$ and $t\ge 0$, we denote $l_t^{a,i}$ as the location of the second-class particle labeled by $i$ at time $t$, under Rule (a).
Then for each $i\in\Z$ we have $l_t^{a,i}>l_t^{a,i+1}$, and there is no other second-class particle between sites $l_t^{a,i}$ and $l_t^{a,i+1}$.
We also denote $l_t^{b,i}$ as the location of the second-class particle labeled by $i$ at time $t$, under Rule (b).
Define $\sigma^{a,i}_t, \sigma^{b,i}_t:\Z\to\{0,1,*\}$ as $\sigma^{a,i}_t(x)=\sigma_t(x+l_t^{a,i})$ and $\sigma^{b,i}_t(x)=\sigma_t(x+l_t^{b,i})$,
which is $\sigma_t$ as seen from the second-class particle labeled by $i$, under each rule.

Our strategy to construct the coupling between $\Psi$ and $\Phi_t$ is to project $\sigma^{b,0}_t$ in two different ways, to get these two measures respectively (see Figure \ref{fig:2cp-cov}).
For the first way, we just look at the law of $\sigma^{b,0}_t$ without considering the labels.
As $\sigma$ is a renewal process, and that $\sigma$ is stationary (Proposition \ref{prop:ferr94}), we have that $\sigma^{a,i}_t$ has the same distribution as $\sigma$. We next show that the same is true for $\sigma^{b,i}_t$.

\begin{figure}[hbt!]
    \centering
\begin{tikzpicture}[line cap=round,line join=round,>=triangle 45,x=4cm,y=4cm]
\clip(-3.5,-0.23) rectangle (0.5,.6);

\fill[line width=0.pt,color=yellow,fill=yellow,fill opacity=0.8]
(-1.55,-0.35) -- (-1.55,-0.45) -- (-1.35,-0.45) -- (-1.35,-0.35) -- cycle;
\draw [line width=1.2pt, opacity=0.3] (-2.7,-0.4) -- (-0.3,-0.4);
\draw [fill=white] (-2.6,-0.4) circle (3.0pt);
\draw [fill=white] (-2.5,-0.4) circle (3.0pt);
\draw [fill=uuuuuu] (-2.4,-0.4) circle (3.0pt);
\draw [fill=white] (-2.3,-0.4) circle (3.0pt);
\draw [fill=uuuuuu] (-2.2,-0.4) circle (3.0pt);
\draw [fill=uuuuuu] (-2.1,-0.4) circle (3.0pt);
\draw [fill=white] (-2.,-0.4) circle (3.0pt);
\draw [fill=white] (-1.9,-0.4) circle (3.0pt);
\draw [fill=uuuuuu] (-1.8,-0.4) circle (3.0pt);
\draw [fill=white] (-1.7,-0.4) circle (3.0pt);
\draw [fill=white] (-1.6,-0.4) circle (3.0pt);
\draw [fill=white] (-1.5,-0.4) circle (3.0pt);
\draw [fill=uuuuuu] (-1.4,-0.4) circle (3.0pt);
\draw [fill=uuuuuu] (-1.3,-0.4) circle (3.0pt);
\draw [fill=uuuuuu] (-1.2,-0.4) circle (3.0pt);
\draw [fill=uuuuuu] (-1.1,-0.4) circle (3.0pt);
\draw [fill=white] (-1.,-0.4) circle (3.0pt);
\draw [fill=white] (-.9,-0.4) circle (3.0pt);
\draw [fill=uuuuuu] (-0.8,-0.4) circle (3.0pt);
\draw [fill=white] (-0.7,-0.4) circle (3.0pt);
\draw [fill=uuuuuu] (-0.6,-0.4) circle (3.0pt);
\draw [fill=uuuuuu] (-0.5,-0.4) circle (3.0pt);
\draw [fill=white] (-0.4,-0.4) circle (3.0pt);

\draw [line width=1.2pt, opacity=0.3] (-2.7,-0.1) -- (-0.3,-0.1);
\draw [fill=white] (-2.6,-0.1) circle (3.0pt);
\draw [fill=white] (-2.5,-0.1) circle (3.0pt);
\draw [fill=uuuuuu] (-2.4,-0.1) circle (3.0pt);
\draw [fill=white] (-2.3,-0.1) circle (3.0pt);
\draw [fill=uuuuuu] (-2.2,-0.1) circle (3.0pt);
\draw [fill=uuuuuu] (-2.1,-0.1) circle (3.0pt);
\draw [fill=white] (-2.,-0.1) circle (3.0pt);
\draw [fill=white] (-1.9,-0.1) circle (3.0pt);
\draw [fill=uuuuuu] (-1.8,-0.1) circle (3.0pt);
\draw [fill=white] (-1.7,-0.1) circle (3.0pt);
\draw [fill=white] (-1.6,-0.1) circle (3.0pt);
\draw [fill=yellow] (-1.5,-0.1) circle (3.0pt);
\draw [fill=uuuuuu] (-1.4,-0.1) circle (3.0pt);
\draw [fill=uuuuuu] (-1.3,-0.1) circle (3.0pt);
\draw [fill=uuuuuu] (-1.2,-0.1) circle (3.0pt);
\draw [fill=white] (-1.1,-0.1) circle (3.0pt);
\draw [fill=white] (-1.,-0.1) circle (3.0pt);
\draw [fill=uuuuuu] (-0.9,-0.1) circle (3.0pt);
\draw [fill=white] (-0.8,-0.1) circle (3.0pt);
\draw [fill=uuuuuu] (-0.7,-0.1) circle (3.0pt);
\draw [fill=uuuuuu] (-0.6,-0.1) circle (3.0pt);
\draw [fill=white] (-0.5,-0.1) circle (3.0pt);
\draw [fill=white] (-0.4,-0.1) circle (3.0pt);

\draw [line width=1.2pt, opacity=0.3] (-2.7,0.2) -- (-0.3,0.2);
\draw [fill=white] (-2.6,0.2) circle (3.0pt);
\draw [fill=yellow] (-2.5,0.2) circle (3.0pt);
\draw [fill=uuuuuu] (-2.4,0.2) circle (3.0pt);
\draw [fill=white] (-2.3,0.2) circle (3.0pt);
\draw [fill=uuuuuu] (-2.2,0.2) circle (3.0pt);
\draw [fill=uuuuuu] (-2.1,0.2) circle (3.0pt);
\draw [fill=white] (-2.,0.2) circle (3.0pt);
\draw [fill=yellow] (-1.9,0.2) circle (3.0pt);
\draw [fill=uuuuuu] (-1.8,0.2) circle (3.0pt);
\draw [fill=white] (-1.7,0.2) circle (3.0pt);
\draw [fill=white] (-1.6,0.2) circle (3.0pt);
\draw [fill=yellow] (-1.5,0.2) circle (3.0pt);
\draw [fill=uuuuuu] (-1.4,0.2) circle (3.0pt);
\draw [fill=uuuuuu] (-1.3,0.2) circle (3.0pt);
\draw [fill=yellow] (-1.2,0.2) circle (3.0pt);
\draw [fill=white] (-1.1,0.2) circle (3.0pt);
\draw [fill=white] (-1.,0.2) circle (3.0pt);
\draw [fill=uuuuuu] (-0.9,0.2) circle (3.0pt);
\draw [fill=white] (-0.8,0.2) circle (3.0pt);
\draw [fill=yellow] (-0.7,0.2) circle (3.0pt);
\draw [fill=yellow] (-0.6,0.2) circle (3.0pt);
\draw [fill=white] (-0.5,0.2) circle (3.0pt);
\draw [fill=white] (-0.4,0.2) circle (3.0pt);

\draw [line width=1.2pt, opacity=0.3] (-2.7,0.5) -- (-0.3,0.5);
\draw [fill=white] (-2.6,0.5) circle (3.0pt);
\draw [fill=uuuuuu] (-2.5,0.5) circle (3.0pt);
\draw [fill=uuuuuu] (-2.4,0.5) circle (3.0pt);
\draw [fill=white] (-2.3,0.5) circle (3.0pt);
\draw [fill=uuuuuu] (-2.2,0.5) circle (3.0pt);
\draw [fill=uuuuuu] (-2.1,0.5) circle (3.0pt);
\draw [fill=white] (-2.,0.5) circle (3.0pt);
\draw [fill=white] (-1.9,0.5) circle (3.0pt);
\draw [fill=uuuuuu] (-1.8,0.5) circle (3.0pt);
\draw [fill=white] (-1.7,0.5) circle (3.0pt);
\draw [fill=white] (-1.6,0.5) circle (3.0pt);
\draw [fill=yellow] (-1.5,0.5) circle (3.0pt);
\draw [fill=uuuuuu] (-1.4,0.5) circle (3.0pt);
\draw [fill=uuuuuu] (-1.3,0.5) circle (3.0pt);
\draw [fill=uuuuuu] (-1.2,0.5) circle (3.0pt);
\draw [fill=white] (-1.1,0.5) circle (3.0pt);
\draw [fill=white] (-1.,0.5) circle (3.0pt);
\draw [fill=uuuuuu] (-0.9,0.5) circle (3.0pt);
\draw [fill=white] (-0.8,0.5) circle (3.0pt);
\draw [fill=uuuuuu] (-0.7,0.5) circle (3.0pt);
\draw [fill=white] (-0.6,0.5) circle (3.0pt);
\draw [fill=white] (-0.5,0.5) circle (3.0pt);
\draw [fill=white] (-0.4,0.5) circle (3.0pt);

\fill[line width=0.pt,color=yellow,fill=yellow,fill opacity=0.8]
(-1.55,0.75) -- (-1.55,0.85) -- (-1.35,0.85) -- (-1.35,0.75) -- cycle;
\draw [line width=1.2pt, opacity=0.3] (-2.7,0.8) -- (-0.3,0.8);
\draw [fill=white] (-2.6,0.8) circle (3.0pt);
\draw [fill=uuuuuu] (-2.5,0.8) circle (3.0pt);
\draw [fill=uuuuuu] (-2.4,0.8) circle (3.0pt);
\draw [fill=white] (-2.3,0.8) circle (3.0pt);
\draw [fill=uuuuuu] (-2.2,0.8) circle (3.0pt);
\draw [fill=uuuuuu] (-2.1,0.8) circle (3.0pt);
\draw [fill=white] (-2.,0.8) circle (3.0pt);
\draw [fill=white] (-1.9,0.8) circle (3.0pt);
\draw [fill=uuuuuu] (-1.8,0.8) circle (3.0pt);
\draw [fill=white] (-1.7,0.8) circle (3.0pt);
\draw [fill=white] (-1.6,0.8) circle (3.0pt);
\draw [fill=white] (-1.5,0.8) circle (3.0pt);
\draw [fill=uuuuuu] (-1.4,0.8) circle (3.0pt);
\draw [fill=uuuuuu] (-1.3,0.8) circle (3.0pt);
\draw [fill=uuuuuu] (-1.2,0.8) circle (3.0pt);
\draw [fill=uuuuuu] (-1.1,0.8) circle (3.0pt);
\draw [fill=white] (-1.,0.8) circle (3.0pt);
\draw [fill=white] (-.9,0.8) circle (3.0pt);
\draw [fill=uuuuuu] (-0.8,0.8) circle (3.0pt);
\draw [fill=white] (-0.7,0.8) circle (3.0pt);
\draw [fill=uuuuuu] (-0.6,0.8) circle (3.0pt);
\draw [fill=white] (-0.5,0.8) circle (3.0pt);
\draw [fill=white] (-0.4,0.8) circle (3.0pt);

\begin{scriptsize}
\draw (-1.4,0.77) node[anchor=north]{$1$};
\draw (-1.5,0.77) node[anchor=north]{$0$};
\draw (-1.5,0.47) node[anchor=north]{$0$};
\draw (-1.5,0.17) node[anchor=north]{$0$};
\draw (-1.5,-0.13) node[anchor=north]{$0$};
\draw (-1.4,-0.43) node[anchor=north]{$1$};
\draw (-1.5,-0.43) node[anchor=north]{$0$};

\draw (-1.5,0.23) node[anchor=south, color=red]{$0$};
\draw (-1.2,0.23) node[anchor=south, color=red]{$3$};
\draw (-.7,0.23) node[anchor=south, color=red]{$4$};
\draw (-.6,0.23) node[anchor=south, color=red]{$-1$};
\draw (-1.9,0.23) node[anchor=south, color=red]{$-4$};
\draw (-2.5,0.23) node[anchor=south, color=red]{$2$};
\end{scriptsize}

\draw (-3.4,-0.1) node[anchor=west]{$\zeta^*\sim \Psi$};
\draw (-3.4,0.2) node[anchor=west]{$\sigma_t^{b,0}\stackrel{d}{=} \sigma$};
\draw (-3.4,0.5) node[anchor=west]{$\fp\sigma^{b,0}_t\sim \Phi_t$};

\end{tikzpicture}
\caption{A coupling between $\Psi$ and $\Phi_t$ via $\sigma_t^{b,0}$. The red numbers are labels of second-class particles.
Here $\zeta^*$ and $\fp\sigma^{b,0}_t$ are the same on $\llbracket -9,9 \rrbracket$.}  
\label{fig:2cp-cov}
\end{figure}

\begin{lemma}  \label{l:identify-sigmas-psi}
For each $i\in\Z$ and $t\ge 0$, $\sigma_t^{b,i}$ has the same distribution as $\sigma$.
\end{lemma}
\begin{proof}
Take any measurable set $B\subset \{0,1,*\}^\Z$, it suffices to show that $\PP[\sigma_t^{b,i}\in B] = \PP[\sigma\in B]$.

We fix $t\ge 0$.
As each second-class particle jumps with rate at most $1$, for any $\epsilon>0$ we can find $M>0$, such that $\PP[|l^{a,i}_t-l^{b,i}_t|>M]<\epsilon$ for any $i\in\Z$.
Take a large number $N\in \N$. For each $i$ with $|i|\le N-M$, if $|l^{b,i}_t-l^{a,i}_t|\le M$, we must have $l^{b,i}_t\in \{l^{a,j}_t:i-M\le j \le i+M\} \subset \{l^{a,j}_t:-N\le j \le N\}$, since the set $ \{l^{a,j}_t:i-M\le j \le i+M\}$ contains all locations of second-class particles in $\llbracket l^{a,i}_t-M, l^{a,i}_t+M\rrbracket$.
We then have that
\[
\begin{split}
&\E[|\{l^{b,i}_t:-N\le i \le N\}\setminus \{l^{a,i}_t:-N\le i \le N\} |]\\
= &\sum_{|i|\le N} \PP[l^{a,i}_t\not\in \{l^{a,j}_t:-N\le j \le N\}]\\
\le & 2M+\sum_{|i|\le N-M} \PP[l^{a,i}_t\not\in \{l^{a,j}_t:-N\le j \le N\}] \\
\le & 2M+\sum_{|i|\le N-M} \PP[|l^{b,i}_t-l^{a,i}_t|> M] \\
\le & 2M + 2N\epsilon.    
\end{split}
\]
Since both $|\{l^{b,i}_t:-N\le i \le N\}|$ and $|\{l^{a,i}_t:-N\le i \le N\} |$ equal $2N+1$, we have $|\{l^{b,i}_t:-N\le i \le N\}\setminus \{l^{a,i}_t:-N\le i \le N\} | = |\{l^{a,i}_t:-N\le i \le N\}\setminus \{l^{b,i}_t:-N\le i \le N\} |$, so
\[
\E[|\{l^{a,i}_t:-N\le i \le N\}\setminus \{l^{b,i}_t:-N\le i \le N\} |] \le 2M + 2N\epsilon.
\]
Thus since $\epsilon$ is arbitrarily taken, we have
\[
\lim_{N\to\infty}\frac{1}{2N+1}\big(\E[|\{-N\le i \le N: \sigma_t^{a,i}\in B\}|]-\E[|\{-N\le i \le N: \sigma_t^{b,i}\in B\}|]\big)=0.
\]
Since for each $i\in\Z$, $\sigma_t^{a,i}$ has the same distribution as $\sigma$,
we have 
\[
\lim_{N\to\infty}\frac{1}{2N+1}\E[|\{-N\le i \le N: \sigma_t^{a,i}\in B\}|]=
\lim_{N\to\infty}\frac{1}{2N+1}\sum_{i=-N}^N\PP[\sigma_t^{a,i}\in B] = \PP[\sigma\in B].
\]
By combining the above two equations, we have
\[
\lim_{N\to\infty}\frac{1}{2N+1}\E[|\{-N\le i \le N: \sigma_t^{b,i}\in B\}|]=
 \PP[\sigma\in B].
\]
Now that $\sigma$ is a renewal process, $\sigma_0^{b,i}$, thus $\sigma_t^{b,i}$, has the same distribution for all $i$. Thus the left-hand side in the previous equation equals $\PP[\sigma_t^{b,i}\in B]$ for each $i\in\Z$, and the conclusion follows.
\end{proof}
Now since $\sigma_t^{b,0}$ has the same distribution as $\sigma$, we can just identify all $*$ with $1$ in $\Z_+$, and identify all $*$ with $0$ in $\Z_-$, and get $\zeta^*\sim\Psi$ (by Lemma \ref{l:identify-sigmas-psi}).
For the other projection we need to look at the labels.
We define $\fp\sigma^{b,0}_t:\Z\to\{0,1,*\}$ from $\sigma^{b,0}_t$, by identifying all second-class particles whose labels are $<0$ with holes, and all second-class particles whose labels are $>0$ with particles.
Formally, we let $\fp\sigma^{b,0}_t(0)=*$, and $\fp\sigma^{b,0}_t(x)=1$ for any $x$ such that $\sigma_t(x+l_t^{b,0})=1$ or $x=l_t^{b,i}-l_t^{b,0}$ for some $i>0$; and $\fp\sigma^{b,0}_t(x)=0$ such that $\sigma_t(x+l_t^{b,0})=0$ or $x=l_t^{b,i}-l_t^{b,0}$ for some $i<0$.
See Figure \ref{fig:2cp-cov} for an illustration of $\fp\sigma^{b,0}_t$.

\begin{lemma}   \label{l:identify-etas-phi}
For each $t\ge 0$, we have $\fp\sigma^{b,0}_t\sim\Phi_t$.
\end{lemma}
\begin{proof}
We just need to check that $(\fp\sigma^{b,0}_t)_{t\ge 0}$ is TASEP as seen from an isolated second-class particle, and $\fp\sigma^{b,0}_0$ is i.i.d.\;Bernoulli$(\rho)$ on all non-zero sites.

We first consider the initial configuration $\fp\sigma^{b,0}_0$. It is obtained from $\sigma_0=\sigma$, by setting all $*$ symbols in $\Z_+$ to be $0$ and all $*$ symbols in $\Z_-$ to be $1$. This is because at $t=0$, the second-class particles in $\Z_+$ have negative labels, and the second-class particles in $\Z_-$ have positive labels.
Recall (from the discussion after Proposition \ref{prop:ferr94}) this implies that $\fp\sigma^{b,0}_0$ is i.i.d.\;Bernoulli$(\rho)$ on all non-zero sites.

We next consider the evolution of $(\fp\sigma^{b,0}_t)_{t\ge 0}$.
We now define $(\fp\sigma_t)_{t\ge 0}$ from $\sigma_t$, by identifying all second-class particles whose labels are $<0$ with holes and all second-class particles whose labels are $>0$ with particles. Then $\fp\sigma_t(l_t^{b,0})=*$, and $\fp\sigma_t(x)=1$ for any $x$ such that $\sigma_t(x)=1$, or $x=l_t^{b,i}$ for some $i>0$; and $\fp\sigma_t(x)=0$ such that $\sigma_t(x)=0$, or $x=l_t^{b,i}$ for some $i<0$.
Then $\fp\sigma_t$ is precisely $\fp\sigma^{b,0}_t$ shifted by $l_t^{b,i}$, and it suffices to check that the evolution of $(\fp\sigma_t)_{t\ge 0}$ is given by TASEP with a single second-class particle.
For $(\sigma_t)_{t\ge 0}$ and the labels evolving under Rule (b), recall that it is driven by the following generators, independently for all $x\in\Z$.
\begin{enumerate}
    \item[(1)] If $\sigma_t(x)=1$ and $\sigma_t(x+1)=0$, with rate $1$ we switch $\sigma_t(x)$ and $\sigma_t(x+1)$.
    \item[(2)] If $\sigma_t(x)=1$ and $\sigma_t(x+1)=*$ with $l_t^{b,i}=x+1$ for some $i\in\Z$, with rate $1$ we switch $\sigma_t(x)$ and $\sigma_t(x+1)$ and set $l_t^{b,i}=x$.
    \item[(3)] If $\sigma_t(x)=*$ with $l_t^{b,i}=x+1$ for some $i\in\Z$, and $\sigma_t(x+1)=0$, with rate $1$ we switch $\sigma_t(x)$ and $\sigma_t(x+1)$ and set $l_t^{b,i}=x+1$.
    \item[(4)] If $\sigma_t(x)=\sigma_t(x+1)=*$ with $l_t^{b,i}=x$ and $l_t^{b,j}=x+1$ for some $i>j$, with rate $1$ we set $l_t^{b,i}=x+1$ and $l_t^{b,j}=x$.
\end{enumerate}
From the definition of $(\fp\sigma_t)_{t\ge 0}$, these generators degenerate into that for each $x\in\Z$ we switch $\fp\sigma_t(x)$ and $\fp\sigma_t(x+1)$ with rate $1$, if one of the following happens:
\begin{enumerate}
    \item[(a)] $\fp\sigma_t(x)=1$ and $\fp\sigma_t(x+1)=0$.
    \item[(b)] $\fp\sigma_t(x)=1$ and $\fp\sigma_t(x+1)=*$.
    \item[(c)] $\fp\sigma_t(x)=*$ and $\fp\sigma_t(x+1)=0$.
\end{enumerate}
More precisely: (1) degenerates into (a); (2) degenerates into no change or (b) or (a), depending on whether $i>0$, $i=0$, or $i<0$; (3) degenerates into (a) or (c) or no change, depending on whether $i>0$, $i=0$, or $i<0$; (4) degenerates into (c) or (b) or (a) or no change, depending on whether $i=0$, $j=0$, $ij<0$, or $ij>0$.
These verify that $(\fp\sigma_t)_{t\ge 0}$ has the same generators as TASEP with a single second-class particle, so the conclusion follows.
\end{proof}

Now we finish the proof of Proposition \ref{prop:converge-2nd-class-tasep}, by the two projections of $\sigma^{b,0}_t$.
\begin{proof} [Proof of Proposition \ref{prop:converge-2nd-class-tasep}]
It suffices to take any cylinder set $B=B'\times  \{0,1\}^{\Z\setminus \llbracket -L, L\rrbracket}\subset \{0,1\}^\Z$, for some $L\in \N$ and $B'\subset \{0,1\}^{\llbracket -L, L\rrbracket}$, and show that $T^{-1}\int_0^T \Phi_t(B) dt \to \Psi(B)$.

By Lemma \ref{l:identify-sigmas-psi}, from $\sigma^{b,0}_t$, by identifying all $*$ with $1$ in $\Z_+$ and all $*$ with $0$ in $\Z_-$, we get $\zeta^*\sim\Psi$;
and by Lemma \ref{l:identify-etas-phi}, from $\sigma^{b,0}_t$ we get $\fp\sigma^{b,0}_t\sim \Phi_t$, by identifying all negatively labeled $*$ with $0$, and identifying all positively labeled $*$ with $1$ (see Figure \ref{fig:2cp-cov}). 
Then we have
\[
| \Phi_t(B) - \Psi(B)| \le 
\PP[\zeta^* \in B, \fp\sigma^{b,0}_t\not\in B] + \PP[\zeta^* \not\in B, \fp\sigma^{b,0}_t \in B] \le \PP[\zeta^*|_{\llbracket -L, L\rrbracket}\neq \fp\sigma^{b,0}_t|_{\llbracket -L, L\rrbracket}].
\]
The event in the right-hand side is equivalent to that, in $\sigma^{b,0}$, each $*$ in $\llbracket 1, L\rrbracket$ has a positive label and each $*$ in $\llbracket -L, -1\rrbracket$ has a negative label.
In other words, for any $i\in \Z$ with $l_t^{b,i} - l_t^{b,0} \in \llbracket -L, 0\rrbracket$, we must have $i\le 0$; and for any $i\in \Z$ with $l_t^{b,i} - l_t^{b,0} \in \llbracket 0, L \rrbracket$, we must have $i\ge 0$.
So we have
\begin{multline*}
| \Phi_t(B) - \Psi(B)| \le 1-\PP[\{l_t^{b,i} - l_t^{b,0}:i>0\}\cap \llbracket -L, 0\rrbracket = \{l_t^{b,i} - l_t^{b,0}:i<0\}\cap \llbracket 0, L\rrbracket = \emptyset]\\
\le \E[|\{l_t^{b,i} - l_t^{b,0}:i>0\}\cap \llbracket -L, 0\rrbracket|] + \E[|\{l_t^{b,i} - l_t^{b,0}:i<0\}\cap \llbracket 0, L \rrbracket|].    
\end{multline*}
By integrating over $t$ we have
\begin{equation}  \label{eq:psphbdt}
\int_0^T | \Phi_t(B) - \Psi(B)| dt 
\le 
\sum_{i\in\Z_+} \int_0^T \PP[l_t^{b,i} - l_t^{b,0}\in \llbracket -L, 0\rrbracket]dt
+
\sum_{i\in\Z_-} \int_0^T \PP[l_t^{b,i} - l_t^{b,0}\in \llbracket 0, L\rrbracket]dt.        
\end{equation}
We first bound the first term in the right-hand side.
For each $i\in \Z_+$ we recursively define a sequence of stopping times:
let $T_{i,1}=\inf\{t\ge 0: l_t^{b,i} - l_t^{b,0}\in \llbracket -L, 0\rrbracket\}\cup\{\infty\}$;
and given $T_{i,n}<\infty$, let $T_{i,n+1}=\inf\{t\ge T_{i,n}+1: l_t^{b,i} - l_t^{b,0}\in \llbracket -L,0\rrbracket\}\cup\{\infty\}$.

It is not difficult to see that there exists $\delta>0$ depending only on $L$, such that for any $t\ge 0$ and $n\in\N$ we have $\PP[l_{t+1}^{b,i} > l_{t+1}^{b,0} \mid T_{i,n}=t]>\delta$. Note that since $i>0$, if $l_{t_0}^{b,i} > l_{t_0}^{b,0}$ for some $t_0\ge 0$, we must have $l_{t}^{b,i} > l_{t}^{b,0}$ for any $t>t_0$. Thus the event $l_{T_{i,n}+1}^{b,i} > l_{T_{i,n}+1}^{b,0}$ implies that $T_{i,n+1}=\infty$. So for any $t\ge 0$ and $n \in \N$, we have
\[
\PP[T_{i,n+1}<\infty \mid T_{i,n}=t]<1-\delta,
\]
Then we have
\[
\PP[T_{i,n}<T] = \PP[T_{i,1}< T_{i,n}<T] \le \PP[T_{i,1}< T \text{ and } T_{i,n}<\infty] \le (1-\delta)^{n-1}\PP[T_{i,1}<T].
\]
Also note that $\int_0^T \don[l_t^{b,i} - l_t^{b,0}\in \llbracket -L, 0\rrbracket] dt \le \sum_{n=1}^\infty \don[T_{i,n}<T]$.
So we have
\begin{equation}  \label{eq:ttbd1}
\int_0^T \PP[l_t^{b,i} - l_t^{b,0}\in \llbracket -L, 0\rrbracket] dt \le \sum_{n=1}^\infty \PP[T_{i,n}<T] \le \sum_{n=1}^\infty (1-\delta)^{n-1}\PP[T_{i,1}<T] = \delta^{-1}\PP[T_{i,1}<T].
\end{equation}
Next we bound $\sum_{i\in\Z_+}\PP[T_{i,1}<T]$.
Take any $\epsilon>0$.
From the renewal construction of $\sigma$, we have that $l^{b,0}_0-l^{b,i}_0$ is the sum of $i$ i.i.d.\;positive random variables, each with infinite expectation.
Thus we have
\begin{equation} \label{eq:tltt}
\lim_{T\to\infty}\PP[l^{b,0}_0-l^{b,\lceil \epsilon T \rceil}_0 < 3T] = 0.    
\end{equation}
Given $\{l^{b,i}_0\}_{i\in\Z}$ satisfying $l^{b,0}_0-l^{b,\lceil \epsilon T \rceil}_0 \ge 3T$, for each $j\in\Z_{\ge 0}$, $\PP[T_{\lceil \epsilon T \rceil+j, 1}<T\mid \{l^{b,i}_0\}_{i\in\Z}]$ is bounded by the probability of the following event: there are two particles starting from $0$ and $-\lceil 3T \rceil - j$ respectively, jumping left and right respectively with rate $1$ independently, and the first time when they are within distance $L$ to each other is $<T$. 
This is just the probability that the sum of $\lceil 3T\rceil + j - L$ independent $\Exp(2)$ random variables is less than $T$ (since for the distance to decrease by $1$, the waiting time is the minimum of two independent $\Exp(1)$ random variables).
Summing up such probabilities for all $j$ and using \eqref{eq:tltt}, we get
\[
\lim_{T\to\infty}\sum_{i\ge \epsilon T}\PP[T_{i,1}<T] = 0.
\]
Plugging this into \eqref{eq:ttbd1} and summing over $i\in\Z_+$ there, we get
\[
\limsup_{T\to\infty}\sum_{i\in\Z_+}\int_0^T \PP[l_t^{b,i} - l_t^{b,0}\in \llbracket -L, 0\rrbracket] dt - \delta^{-1}\epsilon T \le 0.
\]
Similarly we have
\[
\limsup_{T\to\infty}\sum_{i\in\Z_-}\int_0^T \PP[l_t^{b,i} - l_t^{b,0}\in \llbracket 0, L\rrbracket] dt - \delta^{-1}\epsilon T \le 0.
\]
Adding them up and using \eqref{eq:psphbdt}, we get
\[
\limsup_{T\to\infty}T^{-1}\int_0^T | \Phi_t(B) - \Psi(B)| dt \le 2\delta^{-1}\epsilon.
\]
Since $\epsilon>0$ is arbitrarily taken, the conclusion follows. 
\end{proof}

\section{Coupling between TASEP and LPP}  \label{sec:prelim}

In this section we connect TASEP and LPP, and other objects such as an up-right growth process to be defined shortly.
These results are mostly from the literature, and will motivate the construction of the LPP limiting environment in Section \ref{sec:limit-dist}.

\subsection{Semi-infinite geodesics and the Busemann function}  \label{ssec:buseman}

We start by introducing a useful tool in studying LPP, namely, the Busemann function, and its beautiful duality property.

For any $u,v \in \Z^2$, we denote $\bB(u,v):=T_{u,\bc}-T_{v,\bc}$, where $\bc\in\Z^2$ is the coalescing point of $\Gamma_u$ and $\Gamma_v$;
i.e.\;$\bc$ is the vertex in $\Gamma_u \cap \Gamma_v$ with the smallest $d(\bc)$.
Such $\bB$ is called the \emph{Busemann function (in direction $\brho$)}.
We also write $\bG(u):=\bB(\boo,u)$.
The Busemann function satisfies the following properties.
\begin{enumerate}
\item For each $u,v,w\in\Z^2$, $\bB(u,v)+\bB(v,w)=\bB(u,w)$.
In particular, $\bB(u,v)=\bG(v)-\bG(u)$.
\item For each $u\in\Z^2$, $\bG(u)=\bG(u+(1,0))\wedge \bG(u+(0,1))-\xi(u)$.
\item For each $u\in\Z^2$, denote the \emph{dual weight} $\xi^\vee(u):= \bG(u) - \bG(u-(1,0))\vee \bG(u-(0,1))$, then its distribution is $\Exp(1)$.
\item For each $u\in \Z^2$, the distribution of $\bB(u,u+(0,1))$ is $\Exp(\rho)$, and the distribution of $\bB(u,u+(1,0))$ is $\Exp(1-\rho)$.
\item For any down-right path $\cU=\{u_k\}_{k\in\Z}$, let $\cU_- = \{u_k-(a,a): k \in \Z, a \in \N\}$ and $\cU_+ = \{u_k+(a,a): k \in \Z, a \in \N\}$.
Then the following random variables are independent: 
$\bB(u_k,u_{k-1})$ for each $k \in \Z$, $\xi(u)$ for each $u\in \cU_-$, and $\xi^\vee(u)$ for each $u\in \cU_+$.
\end{enumerate}
The first two properties are by definition.
The third property comes from \cite[Lemma 4.2]{ferrari2009phase} (see also \cite{balazs2006cube}).
For the last two properties, a proof can be found in \cite{seppalainen2020existence}.

All the semi-infinite geodesics (in direction $\brho$) can be characterized by the Busemann function $\bG$ and passage times.
\begin{lemma}  \label{lem:buse-opti}
For any $u \le v$ we have $\bB(u,v)=-\bG(u)+\bG(v)\ge T_{u,v}-\xi(v)$, and the equality holds if and only if $v\in \Gamma_u$.
\end{lemma}
\begin{proof}
Let $\bc$ be the coalescing point of $\Gamma_u$ and $\Gamma_v$. Then we have $\bB(u,v)=-\bG(u)+\bG(v)=T_{u,\bc}-T_{v,\bc}$.
From the definition of passage times, we have that $T_{u,\bc}\ge T_{u,v}+T_{v,\bc}-\xi(v)$, and the equality holds if and only if $v\in \Gamma_{u,\bc}$.
\end{proof}
In particular, by taking $v=u+(0,1)$ and $v=u+(1,0)$ in Lemma \ref{lem:buse-opti}, we must have that $\bG(u+(1,0))\neq \bG(u+(0,1))$ for any $u\in \Z^2$.
This is true as we have assumed the existence and uniqueness of all the finite geodesics, and the existence, uniqueness, and coalescence of all the semi-infinite geodesics in direction $\brho$. These properties are used in defining the Busemann function and in the proof of Lemma \ref{lem:buse-opti}.

The Busemann function $\bG$ actually contains all the information to reconstruct all the semi-infinite geodesics in direction $\brho$.
\begin{lemma}  \label{lem:reconssemiinf}
The semi-infinite geodesic $\Gamma_u$ for any $u\in \Z^2$ can be reconstructed recursively using $\bG$ as follows. We first let $\Gamma_u[1]=u$, and then we let
$\Gamma_u[i+1] = \argmin_{v \in \{\Gamma_u[i]+(1,0),\Gamma_u[i]+(0,1)\}} \bG(v)$ for each $i\in\N$.
\end{lemma}
This is proved by repeatedly using Lemma \ref{lem:buse-opti}, and we omit the details.

Using the dual weights $\xi^\vee$, which are also i.i.d.\;$\Exp(1)$ (by the third and last properties of the Busemann function), we define `downward semi-infinite geodesics'.
For any $u\in \Z^2$, we let $\Gamma_u^{\vee}$ be the semi-infinite geodesic from $u$ in direction $-\brho=(-(1-\rho)^2,-\rho^2)$, under the weights $\xi^{\vee}$.
Below we work under the almost sure event that such $\Gamma_u^{\vee}$ exists and is unique, and $\Gamma_u^\vee$ and $\Gamma_v^\vee$ coalesce, for any $u, v\in \Z^2$.
Such downward semi-infinite geodesics can also be constructed recursively using $\bG$.
More precisely, 
we let $\Gamma_u^{\vee}[1]=u$, and for each $i\in\N$ we let
\begin{equation}  \label{eq:veerecur}
\Gamma_u^{\vee}[i+1] = \argmax_{v \in \{\Gamma_u^{\vee}[i]-(1,0),\Gamma_u^{\vee}[i]-(0,1)\}} \bG(v).    
\end{equation}
By the definition of $\xi^\vee$ and using induction, it is straightforward to check that each finite segment of the path $\Gamma_u^\vee$ constructed from \eqref{eq:veerecur} is a geodesic under $\xi^\vee$. Also this path $\Gamma_u^\vee$ constructed from \eqref{eq:veerecur} has the same law as $-\Gamma_{-u}$ (since $\bG$ and $v\mapsto-\bG(-v)$ have the same law), so it has the desired asymptotic direction.

A quick observation is the following `non-crossing' property between semi-infinite geodesics and downward semi-infinite geodesics.
\begin{lemma}  \label{lem:non-cross}
For any $\Gamma_u$ and $\Gamma_v^{\vee}$ we cannot find $w \in \Z^2$ with $w,w-(1,0) \in \Gamma_u$ and $w,w+(0,1) \in \Gamma_v^{\vee}$ simultaneously, or $w,w-(0,1) \in \Gamma_u$ and $w,w+(1,0) \in \Gamma_v^{\vee}$ simultaneously.
This implies that the path $\Gamma_u+(1/2,1/2)$ divides $u+(\Z^2\setminus\Z_{\le 0}^2)$ into two parts, which are $\cup_{w\in \Gamma_u}(w+\Z_+\times \Z_{\le 0})$ and $\cup_{w\in \Gamma_u}(w+\Z_{\le 0}\times \Z_+)$, so that $\Gamma_v^\vee$ cannot intersect both of them.
Equivalently, the path $\Gamma_v^{\vee}-(1/2,1/2)$ divides $v+(\Z^2\setminus\Z_{\ge 0}^2)$ into two parts, which are $\cup_{w\in \Gamma_v^{\vee}}(w+\Z_-\times \Z_{\ge 0})$ and $\cup_{w\in \Gamma_v^{\vee}}(w+\Z_{\ge 0}\times \Z_-)$, so that $\Gamma_u$ cannot intersect both of them.
\end{lemma}
\begin{proof}
From the recursive constructions of $\Gamma_u$ and $\Gamma_v^{\vee}$, the event $w,w-(1,0) \in \Gamma_u$ implies that $\bG(w)< \bG(w+(-1,1))$, while $w,w+(0,1) \in \Gamma_v^{\vee}$ implies that $\bG(w)> \bG(w+(-1,1))$.
Thus the first statement holds.
The second statement follows similarly.
\end{proof}

\subsection{Growth process}  \label{ssec:gianddgeo}

The function $\bG$ can also be used to describe an up-right growth process.
For each $t\in\R$, we let $I_t:=\{u\in\Z^2: \bG(u)\le t\}$ be the set of vertices occupied by time $t$.
Then for any $u\in\Z^2$, the waiting time for it to be occupied (since the first time when both $u-(1,0)$ and $u-(0,1)$ are occupied) is $\xi^\vee(u)$, which is i.i.d.\;$\Exp(1)$ for all $u$.
Thus $(I_t)_{t\in\R}$ is a Markov process, such that given $I_t$, 
each vertex $u\not\in I_t$ with $u-(0,1), u-(1,0)\in I_t$ becomes occupied with rate $1$ independently.

We next define several objects that will be useful in proofs in later sections.
For any $t\in\R$ and $u\in\Z^2$, we denote  
\[
\xi^{\vee,t}(u):= \bG(u)\vee t - \bG(u-(1,0))\vee \bG(u-(0,1))\vee t.
\]
This can be understood as the waiting time for $u$ to be occupied, starting from $I_t$.
Note that for any $u$ such that $\{u-(1,0), u-(0,1)\} \not\subset I_t$, we have $\xi^{\vee}(u)=\xi^{\vee,t}(u)$.
A key property for $\xi^{\vee,t}$ is that it is still i.i.d.\;$\Exp(1)$ outside $I_t$.
\begin{lemma}  \label{lem:iidexpI}
For any $t\ge 0$, conditioned on $I_t$ and $\{\bG(u)\}_{u\in I_t}$, the random variables $\xi^{\vee,t}(u)$ are i.i.d.\;$\Exp(1)$ for all $u\not\in I_t$.
\end{lemma}
\begin{proof}
Take any down-right path $\cU=\{u_k\}_{k\in\Z}$, and denote $\cU_- = \{u_k-(a,a): k \in \Z, a \in \N\}$, $\cU_+ = \{u_k+(a,a): k \in \Z, a \in \N\}$. Let $\cU_c$ contain all $u\in\cU_+$ such that $\{u-(1,0),u-(0,1)\}\not\subset \cU_+$. Assume that $\boo \in \cU\cup \cU_-$.

We now consider the event $I_t=\cU\cup \cU_-$. It is equivalent to that $\bG(u_k)\le t$ for each $k\in\Z$, and $\bG(u)>t$ (or equivalently $\xi^{\vee}(u)>t-\bG(u-(1,0))\vee \bG(u-(0,1))$), for any $u\in\cU_c$.
We next study the distribution of $\{\xi^{\vee}(u)\}_{u\in \cU_+}$, conditioned on this event.

By the first two properties in Section \ref{ssec:buseman}, we know that $\{\bG(u)\}_{u\in \cU\cup \cU_-}$ determines $\{\bB(u_k,u_{k-1})\}_{k\in \Z}$ and $\{\xi(u)\}_{u\in \cU_-}$.
We next show that $\{\bG(u)\}_{u\in \cU\cup \cU_-}$ is also determined by $\{\bB(u_k,u_{k-1})\}_{k\in \Z}$ and $\{\xi(u)\}_{u\in \cU_-}$.
Indeed, by the first property in Section \ref{ssec:buseman}, for any $k \in \Z$ we have that $\bG(u_k)-\bG(u_0)=\bB(u_0,u_k)$ is determined by $\{\bB(u_k,u_{k-1})\}_{k\in \Z}$.
Then using the second property in Section \ref{ssec:buseman}, and the fact that $\boo \in \cU\cup \cU_-$, we have that for any $u\in\cU\cup \cU_-$, $\bG(u)-\bG(u_0)$ is determined by $\{\bB(u_k,u_{k-1})\}_{k\in \Z}$ and $\{\xi(u)\}_{u\in \cU_-}$, in particular for $u=\boo$. Since $\bG(\boo)=0$, we have that $\bG(u_0)$, thus $\bG(u)$ for any $u\in\cU\cup \cU_-$ is determined by $\{\bB(u_k,u_{k-1})\}_{k\in \Z}$ and $\{\xi(u)\}_{u\in \cU_-}$.

By the last three properties of the Busemann function $\bB$ in Section \ref{ssec:buseman}, $\{\bB(u_k,u_{k-1})\}_{k\in \Z}$, $\{\xi(u)\}_{u\in \cU_-}$, and $\{\xi^{\vee}(u)\}_{u\in \cU_+}$ are independent exponential random variables.
Thus conditioned on $\{\bG(u)\}_{u\in \cU\cup \cU_-}$ and the event $I_t=\cU\cup \cU_-$, we have
\begin{itemize}
    \item $\{\xi^{\vee}(u)\}_{u\in \cU_+}$ are independent random variables,
    \item $\xi^{\vee}(u)\sim \Exp(1)$ for each $u\in \cU_+\setminus \cU_c$,
    \item for each $u\in\cU_c$, $\xi^{\vee}(u)$ has the distribution of $\Exp(1)$ conditioned on $>t-\bG(u-(1,0))\vee \bG(u-(0,1))$.
\end{itemize}
Since an $\Exp(1)$ random variable conditioned on $>x$ for any $x\ge 0$ is just $x+\Exp(1)$, we have that \[\xi^{\vee}(u)-(t-\bG(u-(1,0))\vee \bG(u-(0,1)))\sim \Exp(1)\] for each $u\in\cU_c$.
We note that (still conditioned on $\{\bG(u)\}_{u\in \cU\cup \cU_-}$ and the event $I_t=\cU\cup \cU_-$) we have $\xi^{\vee,t}(u) = \xi^{\vee}(u)$ for any $u\in \cU_+\setminus \cU_c$ and $\xi^{\vee,t}(u)=\xi^{\vee}(u)-(t-\bG(u-(1,0))\vee \bG(u-(0,1)))$, so $\{\xi^{\vee,t}(u)\}_{u\in \cU_+}$ are i.i.d.\;$\Exp(1)$ random variables. Thus the conclusion follows.
\end{proof}

For any $t\in \R$ and $u\not\in I_t$, the path $\Gamma^{\vee}_u\setminus I_t$ can be constructed as the geodesic with boundary $I_t$, under the weights $\xi^{\vee,t}$.
For any $u\le v$, $u, v \not\in I_t$, let $T^{\vee,t}_{u,v}$ and $\Gamma^{\vee,t}_{u,v}$ denote the passage time and geodesic from $u$ to $v$ under the weights $\xi^{\vee,t}$.

\begin{lemma}  \label{lem:geo-part-cons}
For any $v\not\in I_t$ we have $\bG(v)-t = \max_{u\le v, u\not\in I_t} T^{\vee,t}_{u,v}$ and $\Gamma^\vee_v \setminus I_t = \Gamma^{\vee,t}_{u_*,v}$ for $u_*=\argmax_{u\le v, u\not\in I_t} T^{\vee,t}_{u,v}$.
\end{lemma}

The proof of this lemma is by a straightforward induction in $u$ in the up-right direction, and we omit the details here.

\subsection{The coupling and the competition interface}
\label{ssec:ci-hppair}
\begin{figure}[hbt!]
    \centering
\begin{tikzpicture}[line cap=round,line join=round,>=triangle 45,x=6cm,y=6cm]
\clip(-.5,-0.2) rectangle (1.19,1.49);

\fill[line width=0.pt,color=green,fill=green,fill opacity=0.2]
(-0.45,-0.05) -- (-0.15,-0.05) -- (-0.15,0.05)-- (-0.05,0.05) --(-0.05,0.25) --(0.05,0.25) --(0.15,0.25) --(0.15,0.45) --(0.25,0.45) --(0.25,0.55) -- (0.15,0.55) -- (0.15,0.85) -- (-0.05,0.85) -- (-0.05,0.95) -- (-0.15,0.95) -- (-0.15,1.15) -- (-0.45,1.15) -- cycle;
\fill[line width=0.pt,color=blue,fill=blue,fill opacity=0.2]
(-0.35,-0.15) -- (-0.35,-0.05) -- (-0.15,-0.05) -- (-0.15,0.05)-- (-0.05,0.05) --(-0.05,0.25) --(0.05,0.25) --(0.15,0.25) --(0.15,0.45) --(0.45,0.45) --(0.45,0.35) -- (0.65,0.35) -- (0.65,0.25) -- (0.75,0.25) -- (0.75,0.15) -- (0.95,0.15) -- (0.95,-0.05) -- (1.05,-0.05) -- (1.05,-0.15) -- cycle;

\draw [line width=.1pt, opacity=0.3] (-0.5,-0.1) -- (2.6,-0.1);
\draw [line width=.1pt, opacity=0.3] (-0.5,0.) -- (2.6,0.);
\draw [line width=.1pt, opacity=0.3] (-0.5,0.1) -- (2.6,0.1);
\draw [line width=.1pt, opacity=0.3] (-0.5,0.2) -- (2.6,0.2);
\draw [line width=.1pt, opacity=0.3] (-0.5,0.3) -- (2.6,0.3);
\draw [line width=.1pt, opacity=0.3] (-0.5,0.4) -- (2.6,0.4);
\draw [line width=.1pt, opacity=0.3] (-0.5,0.5) -- (2.6,0.5);
\draw [line width=.1pt, opacity=0.3] (-0.5,0.6) -- (2.6,0.6);
\draw [line width=.1pt, opacity=0.3] (-0.5,0.7) -- (2.6,0.7);
\draw [line width=.1pt, opacity=0.3] (-0.5,0.8) -- (2.6,0.8);
\draw [line width=.1pt, opacity=0.3] (-0.5,0.9) -- (2.6,0.9);
\draw [line width=.1pt, opacity=0.3] (-0.5,1.) -- (2.6,1.);
\draw [line width=.1pt, opacity=0.3] (-0.5,1.1) -- (2.6,1.1);
\draw [line width=.1pt, opacity=0.3] (-0.5,1.2) -- (2.6,1.2);
\draw [line width=.1pt, opacity=0.3] (-0.5,1.3) -- (2.6,1.3);
\draw [line width=.1pt, opacity=0.3] (-0.5,1.4) -- (2.6,1.4);
\draw [line width=.1pt, opacity=0.3] (-0.5,1.5) -- (2.6,1.5);
\draw [line width=.1pt, opacity=0.3] (-0.5,1.6) -- (2.6,1.6);
\draw [line width=.1pt, opacity=0.3] (-0.5,1.7) -- (2.6,1.7);
\draw [line width=.1pt, opacity=0.3] (-0.5,1.8) -- (2.6,1.8);
\draw [line width=.1pt, opacity=0.3] (-0.5,1.9) -- (2.6,1.9);
\draw [line width=.1pt, opacity=0.3] (-0.5,2.) -- (2.6,2.);
\draw [line width=.1pt, opacity=0.3] (-0.5,2.1) -- (2.6,2.1);
\draw [line width=.1pt, opacity=0.3] (-0.4,-0.2) -- (-0.4,2.2);
\draw [line width=.1pt, opacity=0.3] (-0.3,-0.2) -- (-0.3,2.2);
\draw [line width=.1pt, opacity=0.3] (-0.2,-0.2) -- (-0.2,2.2);
\draw [line width=.1pt, opacity=0.3] (-0.1,-0.2) -- (-0.1,2.2);
\draw [line width=.1pt, opacity=0.3] (0.,-0.2) -- (0.,2.2);
\draw [line width=.1pt, opacity=0.3] (0.1,-0.2) -- (0.1,2.2);
\draw [line width=.1pt, opacity=0.3] (0.2,-0.2) -- (0.2,2.2);
\draw [line width=.1pt, opacity=0.3] (0.3,-0.2) -- (0.3,2.2);
\draw [line width=.1pt, opacity=0.3] (0.4,-0.2) -- (0.4,2.2);
\draw [line width=.1pt, opacity=0.3] (0.5,-0.2) -- (0.5,2.2);
\draw [line width=.1pt, opacity=0.3] (0.6,-0.2) -- (0.6,2.2);
\draw [line width=.1pt, opacity=0.3] (0.7,-0.2) -- (0.7,2.2);
\draw [line width=.1pt, opacity=0.3] (0.8,-0.2) -- (0.8,2.2);
\draw [line width=.1pt, opacity=0.3] (0.9,-0.2) -- (0.9,2.2);
\draw [line width=.1pt, opacity=0.3] (1.,-0.2) -- (1.,2.2);
\draw [line width=.1pt, opacity=0.3] (1.1,-0.2) -- (1.1,2.2);
\draw [line width=.1pt, opacity=0.3] (1.2,-0.2) -- (1.2,2.2);
\draw [line width=.1pt, opacity=0.3] (1.3,-0.2) -- (1.3,2.2);
\draw [line width=.1pt, opacity=0.3] (1.4,-0.2) -- (1.4,2.2);
\draw [line width=.1pt, opacity=0.3] (1.5,-0.2) -- (1.5,2.2);
\draw [line width=.1pt, opacity=0.3] (1.6,-0.2) -- (1.6,2.2);
\draw [line width=.1pt, opacity=0.3] (1.7,-0.2) -- (1.7,2.2);
\draw [line width=.1pt, opacity=0.3] (1.8,-0.2) -- (1.8,2.2);
\draw [line width=.1pt, opacity=0.3] (1.9,-0.2) -- (1.9,2.2);
\draw [line width=.1pt, opacity=0.3] (2.,-0.2) -- (2.,2.2);
\draw [line width=.1pt, opacity=0.3] (2.1,-0.2) -- (2.1,2.2);
\draw [line width=.1pt, opacity=0.3] (2.2,-0.2) -- (2.2,2.2);
\draw [line width=.1pt, opacity=0.3] (2.3,-0.2) -- (2.3,2.2);
\draw [line width=.1pt, opacity=0.3] (2.4,-0.2) -- (2.4,2.2);

\draw [red] plot coordinates {(-0.4,-0.1) (-0.2,-0.1) (-0.2,0.) (-0.1,0.) (-0.1,0.2) (0.,0.2) (0.1,0.2) (0.1,0.3) (0.1,0.4) (0.2,0.4) (0.2,0.5) (0.3,0.5) (0.4,0.5) (0.4,0.6) (0.6,0.6) (0.6,0.7) (0.9,0.7) (0.9,0.9) (1.,0.9) (1.2,0.9) (1.2,1.) (1.3,1.) (1.3,1.3) (1.4,1.3) (1.4,1.5) (1.5,1.5) (1.5,1.7) (1.9,1.7) (1.9,1.8) (2.1,1.8) (2.1,2.) (2.3,2.) (2.3,2.1)  (2.4,2.1)};

\draw [blue] plot coordinates {(-0.35,-0.05) (-0.15,-0.05) (-0.15,0.05) (-0.05,0.05) (-0.05,0.25) (0.05,0.25) (0.15,0.25) (0.15,0.35) (0.15,0.45) (0.25,0.45) };

\draw [fill=uuuuuu] (-0.4,-0.1) circle (1.0pt);
\draw [fill=uuuuuu] (0.2,0.5) circle (1.0pt);
\draw (-0.4,-0.1) node[anchor=east]{$\boo$};
\draw (0.2,0.5) node[anchor=east]{$p_t$};
\draw (0.9,0.9) node[anchor=east, color=red]{$\Gamma_\boo$};
\draw (-0.05,0.25) node[anchor=south, color=blue]{$Z$};

\end{tikzpicture}
\caption{
An illustration of the growth process from LPP: the blue and green areas are the two clusters $C_1\cap I_t$ and $C_2\cap  I_t$ respectively, and the red curve is the semi-infinite geodesic $\Gamma_\boo$.
}  
\label{fig:interface}
\end{figure}

We now describe the coupling between LPP and TASEP (denoted as a Markov process on $\{0,1\}^\Z$).
In this subsection we let $(\eta_t)_{t\ge 0}$ denote TASEP with the following initial condition: let $\eta_0(0)=0$ and $\eta_0(1)=1$, and $\eta_0(x)$ be i.i.d.\;Bernoulli$(\rho)$ for all other $x$.
We label the holes by $\Z$ from left to right, with the one at site $0$ labeled $0$;
and label the particles by $\Z$ from right to left, with the one at site $1$ labeled $0$.
For any $(a, b) \in \Z^2$, if at time $0$ the particle labeled $b$ is to the right of the hole labeled $a$, we denote $L(a,b)=0$; otherwise, we denote $L(a,b)>0$ as the time when the particle switches with the hole. 
Then we have that $\{L(a,b)\}_{(a,b)\in \Z^2}$ has the same distribution as $\{\bG(a,b)\vee 0\}_{(a,b)\in \Z^2}$.
Indeed, using the last property of the Busemann function in Section \ref{ssec:buseman}, we can deduce that $I_0$ and $\{(a,b):L(a,b)=0\}$ have the same distribution; and given $\eta_0$, the random variables
\[
L(a,b)-L(a-1,b)\vee L(a,b-1)
\]
for all $(a,b)$ with $L(a,b)>0$ are i.i.d.\;$\Exp(1)$, because this is the waiting time for the particle labeled $b$ and the hole labeled $a$ to switch since the time they are next to each other.
Thus they have the same distribution as $\{\xi^{\vee,0}(u)\}_{u\not\in I_0}$ conditioned on $I_0$, according to Lemma \ref{lem:iidexpI}.
See e.g. \cite[Section 4.2]{ferrari2009phase} for more details on the equal in distribution between $L$ and $\bG\vee 0$.
We couple $(\eta_t)_{t\ge 0}$ with LPP so that $L=\bG\vee 0$ almost surely, and in the rest of this section we work under the event that this equality holds.
Then the TASEP configuration $\eta_t$ can be directly read from $I_t$ (see Figure \ref{fig:read}).
\begin{figure}[hbt!]
    \centering
\begin{tikzpicture}[line cap=round,line join=round,>=triangle 45,x=4.8cm,y=4.8cm]
\clip(-1.5,-0.4) rectangle (1.19,1.05);

\fill[line width=0.pt,color=yellow,fill=yellow,fill opacity=0.4]
(-0.05,0.25) --(0.15,0.25) --(0.15,0.45) --(0.25,0.45) --(0.25,0.55) -- (0.15,0.55) -- (0.15,0.85) -- (-0.05,0.85) -- cycle;
\fill[line width=0.pt,color=yellow,fill=yellow,fill opacity=0.4]
(-0.05,-0.15) -- (-0.05,0.05) --(-0.05,0.25) --(0.05,0.25) --(0.15,0.25) --(0.15,0.45) --(0.45,0.45) --(0.45,0.35) -- (0.65,0.35) -- (0.65,0.25) -- (0.75,0.25) -- (0.75,0.15) -- (0.95,0.15) -- (0.95,-0.05) -- (1.05,-0.05) -- (1.05,-0.15) -- cycle;

\foreach \i in {-1,...,19}
{
\draw [line width=.1pt, opacity=0.3] (-0.05,\i/10) -- (1.15,\i/10);
}

\foreach \i in {0,...,11}
{
\draw [line width=.1pt, opacity=0.3] (\i/10,-0.15) -- (\i/10,1.45);
}

\foreach \i in {-14,...,6}
{
\draw [line width=.05pt] [dotted] (\i/10+0.05,-0.5) -- (\i/10+2.05,1.5);
}

\draw [fill=uuuuuu] (0.1,0) circle (1.0pt);
\draw (0.1,0) node[anchor=east]{$\boo$};

\draw [line width=1.2pt, opacity=0.3] (-1.2,-0.3) -- (0.9,-0.3);
\draw [fill=white] (-1.15,-0.3) circle (2.8pt);
\draw [fill=white] (-1.05,-0.3) circle (2.8pt);
\draw [fill=uuuuuu] (-0.95,-0.3) circle (2.8pt);
\draw [fill=uuuuuu] (-0.85,-0.3) circle (2.8pt);
\draw [fill=uuuuuu] (-0.75,-0.3) circle (2.8pt);
\draw [fill=white] (-0.65,-0.3) circle (2.8pt);
\draw [fill=uuuuuu] (-0.55,-0.3) circle (2.8pt);
\draw [fill=white] (-0.45,-0.3) circle (2.8pt);
\draw [fill=white] (-0.35,-0.3) circle (2.8pt);
\draw [fill=uuuuuu] (-0.25,-0.3) circle (2.8pt);
\draw [fill=white] (-0.15,-0.3) circle (2.8pt);
\draw [fill=white] (-0.05,-0.3) circle (2.8pt);
\draw [fill=uuuuuu] (0.05,-0.3) circle (2.8pt);
\draw [fill=white] (0.15,-0.3) circle (2.8pt);
\draw [fill=uuuuuu] (0.25,-0.3) circle (2.8pt);
\draw [fill=white] (0.35,-0.3) circle (2.8pt);
\draw [fill=white] (0.45,-0.3) circle (2.8pt);
\draw [fill=uuuuuu] (0.55,-0.3) circle (2.8pt);
\draw [fill=uuuuuu] (0.65,-0.3) circle (2.8pt);
\draw [fill=white] (0.75,-0.3) circle (2.8pt);
\draw [fill=uuuuuu] (0.85,-0.3) circle (2.8pt);

\begin{tiny}

\foreach \i in {-9,...,11}
{
\draw (\i/10-0.25,-0.32) node[anchor=north]{$\i$};
}
\end{tiny}

\end{tikzpicture}

\caption{
An illustration of the correspondence between TASEP and the growth process.
}  
\label{fig:read}
\end{figure}

\begin{lemma}  \label{lem:Ideteta}
For any $t\ge 0$ and $x\in \Z$, $\eta_t(x)=0$ if and only if there is some $y \in \Z$ such that $(x+y,y)\in I_t$ and $(x+y,y+1)\not\in I_t$, and the hole at site $x$ has label $x+y$; and $\eta_t(x)=1$ if and only if there is some $y \in \Z$ such that $(x+y-1,y)\in I_t$ and $(x+y,y)\not\in I_t$, and the particle at site $x$ has label $y$.
Equivalently, if we let $f_t:\Z\to \Z$ be the function such that $f_t(x)$ is the largest integer with $(f_t(x)+x,f_t(x)) \in I_t$, then $f_t(x-1)-f_t(x)=\eta_t(x)$.
\end{lemma}
\begin{proof}
For simplicity of notations we denote $\cE_1$ as the event where there is $y \in \Z$ such that $(x+y,y)\in I_t$ and $(x+y,y+1)\not\in I_t$, and $\cE_2$ as the event where there is $y \in \Z$ such that $(x+y-1,y)\in I_t$ and $(x+y,y)\not\in I_t$.
Note that exactly one of $\cE_1$ and $\cE_2$ happens,
so it suffices to show that $\cE_1$ implies $\eta_t(x)=0$, since by symmetry we would have that $\cE_2$ implies $\eta_t(x)=1$, then the conclusion follows.

If $(x+y,y)\in I_t$ and $(x+y,y+1)\not\in I_t$, we have $L(x+y,y)\le t$ and $L(x+y,y+1)>t$ under the coupling.
This means that at time $t$, the hole labeled $x+y$ is to the left of the particle labeled $y$, but to the right of the particle labeled $y+1$.
Suppose that at time $0$, the hole labeled $x+y$ is at site $z$. 
Since the hole at site $0$ is labeled $0$, if $x+y>0$ we must have $z>1$, and there are $z-(x+y)$ particles between sites $0$ and $z$; and if $x+y<0$ we must have $z<0$, and there are $(x+y)-z$ particles between sites $z$ and $0$.
In either case, the nearest particle to the left of the hole labeled $x+y$ (at time $0$) must be labeled $x+y-z+1$ (since the particle at site $1$ is labeled $0$).
This means that at time $t$, the hole labeled $x+y$ has already swapped with $(y+1)-(x+y-z+1)=z-x$ particles. So at time $t$ it is at site $x$, and $\eta_t(x)=0$.
\end{proof}
We next consider the semi-infinite geodesic $\Gamma_\boo$ under this coupling.
It actually corresponds to the competition interface starting from $\boo$, which we describe now (see e.g. \cite{ferrari2005competition, ferrari2009phase}).
We define two clusters $C_1$ and $C_2$ for the growth process $(I_t)_{t\ge 0}$: let $\Z_+\times \{0\}\subset C_1$, and $\{0\} \times \Z_+ \subset C_2$.
For any $(a,b)\in\Z_+^2$ let its `parent' be either $(a-1,b)$ or $(a,b-1)$, whichever is occupied later;
then $(a,b)$ is in the same cluster as its parent.
Starting from any $u$ and by taking parent repeatedly, we can actually get $\Gamma_u^\vee\setminus I_0$, by \eqref{eq:veerecur}; thus we can equivalently define $C_1$ and $C_2$ such that for any $u\in \Z_{\ge 0}^2$, $u\neq \boo$, we let $u\in C_1$ if $\Gamma^\vee_u$ intersects $\Z_+\times \{0\}$, and $u\in C_2$ if $\Gamma^\vee_u$ intersects $\{0\}\times \Z_+$.
By Lemma \ref{lem:geo-part-cons}, such clusters are determined by $I_0$ and $\{\xi^{\vee,0}(u)\}_{u\not\in I_0}$, which are i.i.d.\;$\Exp(1)$ conditioned on $I_0$.
The competition interface $Z$ is defined to be the boundary of these clusters $C_1$ and $C_2$. Namely, we let $Z\subset (1/2,1/2)+\Z_{\ge 0}^2$, such that for any $v\in Z$, every vertex in $\Z_{\ge 0}^2$ to the upper-left of $v$ is in $C_2$, and every vertex in $\Z_{\ge 0}^2$ to the lower-right of $v$ is in $C_1$.
By Lemma \ref{lem:non-cross}, $Z=\Gamma_\boo+(1/2,1/2)$.
In words, the competition interface $Z$ defined from $I_0$ and $\{\xi^{\vee,0}(u)\}_{u\not\in I_0}$ is equivalent to the semi-infinite geodesic $\Gamma_\boo$ defined from $\{\xi(u)\}_{u\in\Z^2}$.
We also define the process $(p_t)_{t\ge 0}$, such that $p_t$ is the last vertex in $\Gamma_\boo\cap I_t$
(see Figure \ref{fig:interface}).

In the TASEP side, in $(\eta_t)_{t\ge 0}$ we keep track of a `hole-particle pair', which is a hole with a particle next to it in the right.
At $t=0$ it is the hole at site $0$ and particle at site $1$.
Whenever the particle is switched with a hole to the right, we move this pair to the right; and whenever the hole is switched with a particle to the left, we move this pair to the left (see Figure \ref{fig:pair} for an illustration).
We have the following lemma from \cite{ferrari2005competition}, which says that the trajectory of this `hole-particle pair' can be mapped to the competition interface.

\begin{figure}[hbt!]
    \centering
\begin{tikzpicture}[line cap=round,line join=round,>=triangle 45,x=4cm,y=4cm]
\clip(-2.7,-0.2) rectangle (-0.7,.63);

\begin{scriptsize}
\draw (-2.4,0.53) node[anchor=south]{$5$};
\draw (-2.3,0.53) node[anchor=south]{$-2$};
\draw (-2.2,0.53) node[anchor=south]{$4$};
\draw (-2.1,0.53) node[anchor=south]{$3$};
\draw (-2.,0.53) node[anchor=south]{$-1$};
\draw (-1.9,0.53) node[anchor=south]{$2$};
\draw (-1.8,0.53) node[anchor=south]{$1$};
\draw (-1.7,0.53) node[anchor=south]{$0$};
\draw (-1.6,0.53) node[anchor=south]{$0$};
\draw (-1.5,0.53) node[anchor=south]{$1$};
\draw (-1.4,0.53) node[anchor=south]{$2$};
\draw (-1.3,0.53) node[anchor=south]{$3$};
\draw (-1.2,0.53) node[anchor=south]{$-1$};
\draw (-1.1,0.53) node[anchor=south]{$4$};
\draw (-1.,0.53) node[anchor=south]{$-2$};
\draw (-.9,0.53) node[anchor=south]{$-3$};
\draw (-.8,0.53) node[anchor=south]{$5$};

\draw (-2.4,0.23) node[anchor=south]{$-2$};
\draw (-2.3,0.23) node[anchor=south]{$5$};
\draw (-2.2,0.23) node[anchor=south]{$4$};
\draw (-2.1,0.23) node[anchor=south]{$3$};
\draw (-2.,0.23) node[anchor=south]{$-1$};
\draw (-1.9,0.23) node[anchor=south]{$2$};
\draw (-1.8,0.23) node[anchor=south]{$1$};
\draw (-1.7,0.23) node[anchor=south]{$0$};
\draw (-1.6,0.23) node[anchor=south]{$1$};
\draw (-1.5,0.23) node[anchor=south]{$0$};
\draw (-1.4,0.23) node[anchor=south]{$2$};
\draw (-1.3,0.23) node[anchor=south]{$3$};
\draw (-1.2,0.23) node[anchor=south]{$-1$};
\draw (-1.1,0.23) node[anchor=south]{$4$};
\draw (-1.,0.23) node[anchor=south]{$-2$};
\draw (-.9,0.23) node[anchor=south]{$-3$};
\draw (-.8,0.23) node[anchor=south]{$5$};

\draw (-2.4,-0.07) node[anchor=south]{$-2$};
\draw (-2.3,-0.07) node[anchor=south]{$5$};
\draw (-2.2,-0.07) node[anchor=south]{$4$};
\draw (-2.1,-0.07) node[anchor=south]{$3$};
\draw (-2.,-0.07) node[anchor=south]{$-1$};
\draw (-1.9,-0.07) node[anchor=south]{$2$};
\draw (-1.8,-0.07) node[anchor=south]{$0$};
\draw (-1.7,-0.07) node[anchor=south]{$1$};
\draw (-1.6,-0.07) node[anchor=south]{$1$};
\draw (-1.5,-0.07) node[anchor=south]{$2$};
\draw (-1.4,-0.07) node[anchor=south]{$0$};
\draw (-1.3,-0.07) node[anchor=south]{$3$};
\draw (-1.2,-0.07) node[anchor=south]{$-1$};
\draw (-1.1,-0.07) node[anchor=south]{$4$};
\draw (-1.,-0.07) node[anchor=south]{$-2$};
\draw (-.9,-0.07) node[anchor=south]{$5$};
\draw (-.8,-0.07) node[anchor=south]{$-3$};
\end{scriptsize}

\fill[line width=0.pt,color=yellow,fill=yellow,fill opacity=0.8]
(-1.65,0.25) -- (-1.65,0.15) -- (-1.45,0.15) -- (-1.45,0.25) -- cycle;
\draw [line width=1.2pt, opacity=0.3] (-2.5,0.2) -- (-0.7,0.2);
\draw [fill=white] (-2.4,0.2) circle (3.0pt);
\draw [fill=uuuuuu] (-2.3,0.2) circle (3.0pt);
\draw [fill=uuuuuu] (-2.2,0.2) circle (3.0pt);
\draw [fill=uuuuuu] (-2.1,0.2) circle (3.0pt);
\draw [fill=white] (-2.,0.2) circle (3.0pt);
\draw [fill=uuuuuu] (-1.9,0.2) circle (3.0pt);
\draw [fill=uuuuuu] (-1.8,0.2) circle (3.0pt);
\draw [fill=white] (-1.7,0.2) circle (3.0pt);
\draw [fill=white] (-1.6,0.2) circle (3.0pt);
\draw [fill=uuuuuu] (-1.5,0.2) circle (3.0pt);
\draw [fill=white] (-1.4,0.2) circle (3.0pt);
\draw [fill=white] (-1.3,0.2) circle (3.0pt);
\draw [fill=uuuuuu] (-1.2,0.2) circle (3.0pt);
\draw [fill=white] (-1.1,0.2) circle (3.0pt);
\draw [fill=uuuuuu] (-1.,0.2) circle (3.0pt);
\draw [fill=uuuuuu] (-0.9,0.2) circle (3.0pt);
\draw [fill=white] (-0.8,0.2) circle (3.0pt);

\fill[line width=0.pt,color=yellow,fill=yellow,fill opacity=0.8]
(-1.55,-.15) -- (-1.55,-.05) -- (-1.35,-.05) -- (-1.35,-.15) -- cycle;
\draw [line width=1.2pt, opacity=0.3] (-2.5,-0.1) -- (-0.7,-0.1);
\draw [fill=white] (-2.4,-0.1) circle (3.0pt);
\draw [fill=uuuuuu] (-2.3,-0.1) circle (3.0pt);
\draw [fill=uuuuuu] (-2.2,-0.1) circle (3.0pt);
\draw [fill=uuuuuu] (-2.1,-0.1) circle (3.0pt);
\draw [fill=white] (-2.,-0.1) circle (3.0pt);
\draw [fill=uuuuuu] (-1.9,-0.1) circle (3.0pt);
\draw [fill=white] (-1.8,-0.1) circle (3.0pt);
\draw [fill=uuuuuu] (-1.7,-0.1) circle (3.0pt);
\draw [fill=white] (-1.6,-0.1) circle (3.0pt);
\draw [fill=white] (-1.5,-0.1) circle (3.0pt);
\draw [fill=uuuuuu] (-1.4,-0.1) circle (3.0pt);
\draw [fill=white] (-1.3,-0.1) circle (3.0pt);
\draw [fill=uuuuuu] (-1.2,-0.1) circle (3.0pt);
\draw [fill=white] (-1.1,-0.1) circle (3.0pt);
\draw [fill=uuuuuu] (-1.,-0.1) circle (3.0pt);
\draw [fill=white] (-0.9,-0.1) circle (3.0pt);
\draw [fill=uuuuuu] (-0.8,-0.1) circle (3.0pt);

\fill[line width=0.pt,color=yellow,fill=yellow,fill opacity=0.8]
(-1.75,0.55) -- (-1.75,0.45) -- (-1.55,0.45) -- (-1.55,0.55) -- cycle;
\draw [line width=1.2pt, opacity=0.3] (-2.5,0.5) -- (-0.7,0.5);
\draw [fill=uuuuuu] (-2.4,0.5) circle (3.0pt);
\draw [fill=white] (-2.3,0.5) circle (3.0pt);
\draw [fill=uuuuuu] (-2.2,0.5) circle (3.0pt);
\draw [fill=uuuuuu] (-2.1,0.5) circle (3.0pt);
\draw [fill=white] (-2.,0.5) circle (3.0pt);
\draw [fill=uuuuuu] (-1.9,0.5) circle (3.0pt);
\draw [fill=uuuuuu] (-1.8,0.5) circle (3.0pt);
\draw [fill=white] (-1.7,0.5) circle (3.0pt);
\draw [fill=uuuuuu] (-1.6,0.5) circle (3.0pt);
\draw [fill=white] (-1.5,0.5) circle (3.0pt);
\draw [fill=white] (-1.4,0.5) circle (3.0pt);
\draw [fill=white] (-1.3,0.5) circle (3.0pt);
\draw [fill=uuuuuu] (-1.2,0.5) circle (3.0pt);
\draw [fill=white] (-1.1,0.5) circle (3.0pt);
\draw [fill=uuuuuu] (-1.,0.5) circle (3.0pt);
\draw [fill=uuuuuu] (-0.9,0.5) circle (3.0pt);
\draw [fill=white] (-0.8,0.5) circle (3.0pt);

\draw (-2.5,0.5) node[anchor=east]{$\eta_0$};
\draw (-2.5,0.2) node[anchor=east]{$\eta_1$};
\draw (-2.5,-0.1) node[anchor=east]{$\eta_2$};

\end{tikzpicture}
\caption{
An illustration of the evolution of a hole-particle pair in $(\eta_t)_{\ge 0}$: the numbers above the particles/holes are the labels, which increase from left to right for holes, and decrease from left to right for particles. The yellow boxes indicate the tracked hole-particle pairs.
}  
\label{fig:pair}
\end{figure}

\begin{lemma}  \label{lem:compete-inter-2nd-class}
Under the above coupling between LPP and TASEP, for the hole-particle pair at time $t$, let $b_t$ be the label of the particle and $a_t$ be the label of the hole. Then $p_t=(a_t, b_t)$.
\end{lemma}

We note that this hole-particle pair can also be replaced by a second-class particle.
For this, note that $a_t$ is also the number of times the pair moved to the right up to time $t$, and $b_t$ is the number of times that pair moved to the left up to time $t$. Thus at time $t$ the hole-particle pair is at sites $a_t-b_t$ and $a_t-b_t+1$.
If we take $\eta_t^*(x)=\eta_t(x)$ for $x<a_t-b_t$, $\eta_t^*(x)=\eta_t(x+1)$ for $x>a_t-b_t$, and $\eta_t^*(a_t-b_t)=*$, we then have that $(\eta_t^*)_{t\ge 0}$ is TASEP with a second-class particle, starting from i.i.d.\;Bernoulli$(\rho)$ on $\Z\setminus\{0\}$.
So far we have seen that this process $(\eta_t^*)_{t\ge 0}$ contains the same information as $\bG\vee 0$ (by Lemma \ref{lem:Ideteta}), thus the same information as $I_0$ and $\{\xi^{\vee,0}(u)\}_{u\not\in I_0}$, and the trajectory of the second-class particle gives the semi-infinite geodesic $\Gamma_\boo$ (Lemma \ref{lem:compete-inter-2nd-class}).
Recall from Section \ref{ssec:2cp-sta-def} that $\Phi_t$ is the law of $\eta_t^*(a_t-b_t+\cdot)$, and $\Psi$ is the stationary distribution of TASEP as seen from an isolated second-class particle.
In light of the convergence of $\Phi_t$ to $\Psi$ as $t\to\infty$, stated in Theorem \ref{thm:cov-phi} or Proposition \ref{prop:converge-2nd-class-tasep}, the LPP limiting environment measure $\nu$ should be constructed from $\Psi$.
We give such construction in the next section.

\section{The LPP limiting environment}
\label{sec:limit-dist}

We are now ready to define $\nu$.
As before, we use $(\zeta_t^*)_{t\in\R}$ to denote the process of TASEP as seen from an isolated second-class particle, such that each $\zeta_t^*\sim\Psi$, the stationary distribution defined in Section \ref{ssec:2cp-sta-def}. 
The idea is to construct a growth process from $(\zeta_t^*)_{t\in\R}$, then take the environment around the origin. This would give the limiting environment along the geodesic $\Gamma_\boo$, as seen at a uniform time; i.e.\;it is the environment as seen from $p_t$ for a uniform $t$, where (recall that) $p_t$ is the last vertex in $\Gamma_\boo \cap I_t$. To get the environment $\nu$, which is as seen from a uniformly chosen vertex, we would do an extra reweighting.

We first replace the second-class particle in $(\zeta_t^*)_{t\in\R}$ by a hole-particle pair.
Namely, we let $(\zeta_t)_{t\in \R}$ be the process such that $\zeta_t(x)=\zeta_t^*(x)$ for $x<0$, $\zeta_t(x)=\zeta_t^*(x-1)$ for $x>1$, and $\zeta_t(0)=0$, $\zeta_t(1)=1$.
The process $(\zeta_t)_{t\in\R}$ is then the stationary process of TASEP as seen from a hole-particle pair.
We use $\hbPsi$ to denote the law of this process.

We next describe the procedure of obtaining the environment from $(\zeta_t)_{t\in\R}$.
We give the growth process in terms of the occupation time function, which we also denote by $\oL$ as a slight abuse of notation.
Similar to the i.i.d.\;Bernoulli initial setting in Section \ref{ssec:ci-hppair}, we label the particles from right to left, and the holes from left to right, such that at time $0$ the particle at site $1$ and the hole at site $0$ are both labeled $0$. 
Let $\oL(a,b)$ be the time when the particle labeled $b$ is switched with the hole labeled $a$.
Unlike the i.i.d.\;Bernoulli initial setting, here $(\zeta_t)_{t\in\R}$ is a stationary process and evolves from time $-\infty$ to $\infty$, so $\oL(a,b)$ may be negative and is well-defined for all $(a,b)\in\Z^2$.
We then use $\oL$ to define the limiting weights and path, which we denote by $\oxi$ and $\oga$ by slightly abusing these notations within this section.
We define $\xi$ via $\oxi(a,b)=\oL(a+1,b)\wedge \oL(a,b+1)-\oL(a,b)$, and define $\oga \subset \Z^2$ as the collection of all $(a,b)$, such that there is a time $t$ when the particle labeled $b$ is at site $1$ and the hole labeled $a$ is at site $0$ in $\eta_t$.
We let $\tilde{\nu}$ be the measure given by the law of such $(\xi, \oga)$ constructed from $(\zeta_t)_{t\in \R}\sim \hbPsi$.

We next do the reweighting.
We let $\bPsi$ be the measure $\hbPsi$ conditioned on $\oL(\boo)=0$, i.e.\;we let $d\bPsi = \lim_{\epsilon \to 0_+} \frac{\don[\oL(\boo)>-\epsilon]d\hbPsi}{\PP_{\hbPsi}[\oL(\boo)>-\epsilon]}$.
As $(\zeta_t)_{t\in\R}$ under $\hbPsi$ is a Markov process, the limit could be computed as $\hbPsi$ conditioned on that there is a jump of the hole-particle pair at time $0$; i.e.\;we first reweight $\hbPsi$ by $\don[\zeta_{0_-}(2)=0] + \don[\zeta_{0_-}(-1)=1]$, the events where a jump is allowed, then let the jump happen at time $0$.
More precisely, we can describe $\bPsi$ as follows. We have
\[
\bPsi = \frac{\PP_{\hbPsi}[\zeta_{0_-}(2)=0]\bPsi^{(1)} + \PP_{\hbPsi}[\zeta_{0_-}(-1)=1]\bPsi^{(2)}}{\PP_{\hbPsi}[\zeta_{0_-}(2)=0] + \PP_{\hbPsi}[\zeta_{0_-}(-1)=1]}
=\frac{(1-\rho)^2\bPsi^{(1)}+\rho^2\bPsi^{(2)}}{(1-\rho)^2+\rho^2}
,\]
where $\bPsi^{(1)}$ (resp.\;$\bPsi^{(2)}$) is $\hbPsi$ conditioned on that a jump of the hole-particle pair to the right (resp.\;to the left) happens at time $0$. 
More precisely, we define these measures as follows.
Let $(\zeta_t)_{t\in \R} \sim \bPsi^{(1)}$, then the negative time part $(\zeta_t)_{t<0}$ has distribution given by \[\frac{\don[\zeta_{0_-}(2)=0] d\hbPsi}{\PP_{\hbPsi}[\zeta_{0_-}(2)=0]};\]
and given $\zeta_{0_-}$ we let $\zeta_0$ be that $\zeta_0(-1)=\zeta_0(0)=0$, $\zeta_0(1)=1$, and $\zeta_0(x)=\zeta_{0_-}(x+1)$ for any $x\not\in \{-1,0,1\}$; and let $(\zeta_t)_{t\ge 0}$ be the Markov process of TASEP as seen from a hole-particle pair starting from $\zeta_0$.
Similarly, for $(\zeta_t)_{t\in \R} \sim \bPsi^{(2)}$, the negative part $(\zeta_t)_{t<0}$ has distribution given by \[\frac{\don[\zeta_{0_-}(-1)=1] d\hbPsi}{\PP_{\hbPsi}[\zeta_{0_-}(-1)=1]};\]
and given $\zeta_{0_-}$, we have $\zeta_0(0)=0$, $\zeta_0(1)=\zeta_0(2)=1$, and $\zeta_0(x)=\zeta_{0_-}(x-1)$ for any $x\not\in \{0,1,2\}$; and $(\zeta_t)_{t\ge 0}$ is the Markov process of TASEP as seen from a hole-particle pair starting from $\zeta_0$.

From this construction, the laws of $\zeta_0$ under $\bPsi^{(1)}$ and $\bPsi^{(2)}$ can also be described as follows.
Let $\Psi_{+}$ be the law of $\{\zeta^*(x)\}_{x\in\N}$ and $\Psi_{-}$ be the law of $\{\zeta^*(-x)\}_{x\in\N}$, for $\zeta^*\sim \Psi$.
Under $\bPsi^{(1)}$, there is $\zeta_0(-1)=\zeta_0(0)=0$, $\zeta_0(1)=1$, and $\{\zeta_0(x+1)\}_{x\in\N}\sim \Psi_+$ and $\{\zeta_0(-x-1)\}_{x\in\N}\sim \Psi_-$, and they are independent.
Under $\bPsi^{(2)}$, there is $\zeta_0(0)=0$, $\zeta_0(1)=\zeta_0(2)=1$, and $\{\zeta_0(x+2)\}_{x\in\N}\sim \Psi_+$,  $\{\zeta_0(-x)\}_{x\in\N}\sim \Psi_-$, and they are independent.

We define $\nu$ as the measure given by the law of $(\oxi, \oga)$, obtained using the procedure above from $(\zeta_t)_{t\in \R}\sim\bPsi$.
By Lemma \ref{lem:psi-L-comp} below we can see that $\oxi(\boo)$ has exponential tail under $\nu$, so $\E_{\nu}[\oxi(\boo)] < \infty$. We then show that $\tilde{\nu}$ is $\nu$ reweighted by $\oxi(\boo)$.
\begin{lemma}  \label{lem:reweight}
We have $d\tilde{\nu} = \frac{\oxi(\boo)d\nu}{\E_{\nu}[\oxi(\boo)]}$.
\end{lemma}
Let's explain why such a relation is expected.
Consider the sequence of times $\{\oL(u)\}_{u\in \oga}$ when the hole-particle pair jumps. Under $\hbPsi$ this is a stationary point process in $\R$.
Then $\nu$ corresponds to the environment
as seen from the hole-particle at a typical jump time. 
On the other hand, 
$\tilde{\nu}$ corresponds to the environment
as seen from the hole-particle at time $0$. 
Because of the `inspection effect', this is biased by the length of the interval in the point process containing time $0$, which is $\oxi(\boo)$.
\begin{proof}[Proof of Lemma \ref{lem:reweight}]
For each $s>0$, we let $\bPsi_{-s}$ be the measure of $\hbPsi$ conditioned on $\oL(\boo)=-s$, i.e.\;let $d\bPsi_{-s} = \lim_{\epsilon \to 0_+} \frac{\don[-s-\epsilon<\oL(\boo)<-s]d\hbPsi}{\PP_{\hbPsi}[-s-\epsilon<\oL(\boo)<-s]}$.
Note that under $\hbPsi$, almost surely $\oL(0,1), \oL(1,0) > 0$ and $\oL(\boo)<0$, since at time $0$ the following objects are ordered from left to right: the particle labeled $1$, the hole labeled $0$, the particle labeled $0$, and the hole labeled $1$. So $\don[-s-\epsilon<\oL(\boo)<-s] = \don[-s-\epsilon<\oL(\boo)<-s]\don[\oxi(\boo)>s]$.
Then since $\hbPsi$ is stationary, we have \[\don[-s-\epsilon<\oL(\boo)<-s]d\hbPsi = (\don[\oL(\boo)>-\epsilon]\don[\oxi(\boo)>s] d\hbPsi)\circ \sT_{-s},\] where $\sT_{-s}$ is the time translation operator: for any process $P=(P_w)_{w\in\R}$, we denote $\sT_{-s} P$ as the process $(P_{-s+w})_{w\in\R}$.
By multiplying $\epsilon^{-1}$ and sending $\epsilon\to 0_+$, we have
\begin{equation} \label{eq:hbpsid}
\PP_{\hbPsi}[\oL(\boo)=-s]d\bPsi_{-s} = \PP_{\hbPsi}[\oL(\boo)=0](\don[\oxi(\boo)>s] d\bPsi) \circ \sT_{-s},    
\end{equation}
where \[\PP_{\hbPsi}[\oL(\boo)=-s]=\lim_{\epsilon\to 0_+} \epsilon^{-1}\PP_{\hbPsi}[-s-\epsilon<\oL(\boo)<-s],\] \[\PP_{\hbPsi}[\oL(\boo)=0]=\lim_{\epsilon\to 0_+} \epsilon^{-1}\PP_{\hbPsi}[\oL(\boo)>-\epsilon]\] are the probability densities.
By integrating the left-hand side of \eqref{eq:hbpsid} over $s>0$ we get $d\hbPsi$, under which the law of $(\oxi, \oga)$ is $\tilde{\nu}$.
For the right-hand side of \eqref{eq:hbpsid}, we note that the laws of $(\oxi, \oga)$ are the same under $(\don[\oxi(\boo)>s]d\bPsi) \circ \sT_{-s}$ or $\don[\oxi(\boo)>s]d\bPsi$. So by integrating over $s>0$ and taking the law of $(\oxi, \oga)$, we get $\PP_{\hbPsi}[\oL(\boo)=0]\oxi(\boo)d\nu$. Thus we conclude that $d\tilde{\nu} = \PP_{\hbPsi}[\oL(\boo)=0]\oxi(\boo)d\nu$. Since $\tilde{\nu}$ and $\nu$ are probability measures, by integrating both sides we get $\PP_{\hbPsi}[\oL(\boo)=0]\E_{\nu}[\oxi(\boo)] = 1$, so the conclusion follows.
\end{proof}

The above construction allows us to explicitly compute finite dimensional distributions of $\nu$ and thus local geodesic statistics (assuming the main results of this paper).
For this rest of this section we illustrate such computations, and prove Propositions \ref{prop:lawofone} and \ref{prop:turning}.

We start with the following computations on the next jump times.
\begin{lemma}  \label{lem:psi-L-comp}
For any $h \ge 0$ we have
\[
\begin{split}
&\PP_{\bPsi^{(1)}}[\oL(1,0) > h] = (1+\rho(1-\rho) h)e^{-(1-\rho)h},\\
&\PP_{\bPsi^{(1)}}[\oL(0,1) > h] = (1+\rho h)e^{-\rho h},\\
&\PP_{\bPsi^{(2)}}[\oL(1,0) > h] = (1+(1-\rho)h)e^{-(1-\rho)h},\\
&\PP_{\bPsi^{(2)}}[\oL(0,1) > h] = (1+\rho(1-\rho) h)e^{-\rho h}.
\end{split}
\]
\end{lemma}
\begin{proof}
Let $D_+=\min\{x\geq 1: \zeta_0(x+1)=0\}$, the number of 
particles between the origin and the leftmost hole at a positive site.
Similarly let $D_-=\min\{x\geq 1: \zeta_0(-x)=1\}$, 
the number of holes to the right of the rightmost particle at a negative site, up to and including the origin. 

The distribution of $D_+$ under $\bPsi^{(1)}$
is that of $X_+$ given by
(\ref{eq:Xplus}), while the distribution of $D_+$
under $\bPsi^{(2)}$ is that of $X_++1$
(which is the distribution of the sum of two independent Geometric($1-\rho$) 
random variables). 

Similarly the distribution of $D_-$ under $\bPsi^{(2)}$
is that of $X_-$ at (\ref{eq:Xminus}), while the distribution 
of $D_-$ under $\bPsi^{(1)}$ is that of $X_-+1$. 

In order for the particle which is at site $1$ at time $0$ 
to jump, the hole starting at site $D_++1$ must switch
with each of the $D_+$ particles starting in $\llbracket 1,D_+\rrbracket$.
So given $D_+$, the distribution of $\oL(1,0)$ is 
the sum of $D_+$ independent $\Exp(1)$ random variables; 
that is, a Gamma($D_+, 1)$ distribution. A random 
variable $V$ with Gamma($k, 1$) distribution has $\E[e^{-sV}]=(1+s)^{-k}$,
and from this we obtain, for any $s>-1+\rho$, 
\[
\E_{\bPsi^{(1)}}[e^{-s(\oL(1,0))}] = 
\sum_{k=1}^\infty k(1-\rho)^2\rho^{k-1}(1+s)^{-k}
=\frac{(1+s)(1-\rho)^2}{(1+s-\rho)^2},
\]
which can be shown to match the expression for $\PP_{\bPsi^{(1)}}[\oL(1,0) > h]$ given in the statement.

Similarly, in order for the hole which is at site $0$ at time $0$ to jump,
the particle starting at site $-D_-$ must switch with each
of the $D_-$ holes starting in $\llbracket -D_-+1, 0\rrbracket$.
One obtains 
\[
\E_{\bPsi^{(1)}}[e^{-s(\oL(0,1))}] = 
\sum_{k=1}^\infty k\rho^2(1-\rho)^{k-1}(1+s)^{-(k+1)}
=\frac{\rho^2}{(\rho+s)^2},
\]
which matches the desired expression for 
$\PP_{\bPsi^{(1)}}[\oL(0,1) > h]$.

Analogous calculations give the probabilities under $\bPsi^{(2)}$.
\end{proof}
Now we compute the law of the weights on geodesics.
\begin{proof}[Proof of Proposition \ref{prop:lawofone}]
It suffices to compute the law of $\oL(1,0)\wedge \oL(0,1)$, under the measure $\bPsi=\frac{(1-\rho)^2\bPsi^{(1)}+\rho^2\bPsi^{(2)}}{(1-\rho)^2+\rho^2}$.
Note that under either $\bPsi^{(1)}$ or  $\bPsi^{(2)}$, the random variables $\oL(1,0)$ and $\oL(0,1)$ are independent.
Thus by Lemma \ref{lem:psi-L-comp} we get that \[\PP_{\bPsi^{(1)}}[\oL(1,0)\wedge \oL(0,1) > h] = (1+\rho h)(1+\rho(1-\rho) h)e^{-h},\] and \[\PP_{\bPsi^{(2)}}[\oL(1,0)\wedge \oL(0,1) > h] = (1+(1-\rho) h)(1+\rho(1-\rho) h)e^{-h}.\]
Thus the conclusion follows.
\end{proof}
Assuming Theorem \ref{thm:finite}, we can also compute the proportion of `corners' in geodesics.
\begin{proof}[Proof of Proposition \ref{prop:turning}]
Assuming Theorem \ref{thm:finite}, we have \[\frac{N_{n,\rho}}{2n} \to \PP_{\nu}[\{(0,0),(0,1),(-1,0)\}\subset \oga] + \PP_{\nu}[\{(0,0),(0,-1),(1,0)\}\subset \oga],\]
almost surely as $n\to\infty$.
From the construction of $\nu$, this equals
\[
\frac{(1-\rho)^2\PP_{\bPsi^{(1)}}[\oL(1,0)>\oL(0,1)]+\rho^2\PP_{\bPsi^{(2)}}[\oL(1,0)<\oL(0,1)]}{(1-\rho)^2+\rho^2}.
\]
Using that $\oL(1,0)-\oL(0,0)$ and $\oL(0,1)-\oL(0,0)$ are independent under either $\bPsi^{(1)}$ or $\bPsi^{(2)}$, by Lemma \ref{lem:psi-L-comp} we have
\[
\begin{split}
&\PP_{\bPsi^{(1)}}[\oL(1,0)>\oL(0,1)] = \rho^2(1+2\rho-2\rho^2),\\
&\PP_{\bPsi^{(2)}}[\oL(1,0)<\oL(0,1)] = (1-\rho)^2(1+2\rho-2\rho^2).
\end{split}
\]
Thus the conclusion follows.
\end{proof}

\noindent\textbf{An alternative representation of the weights on geodesics.} We also give an outline of alternative derivation of the 
formulae in Proposition \ref{prop:lawofone} and Proposition \ref{prop:turning},
which also leads to representations of the type mentioned after the statement
of Proposition \ref{prop:lawofone}. 

Note that under $\bPsi^{(2)}$, $D_+$ takes values in $\{2,3,\dots\}$ and has the distribution of the sum of two independent Geometric random variables 
with parameter $1-\rho$. Given $D_+$, the random variable $\oL(1,0)$ is the sum of $D_+$ independent $\Exp(1)$ random variables.
From this, $\oL(1,0)$ has the same distribution as
the sum of two $\Exp(1-\rho)$ random variables, or equivalently of $\frac{1}{1-\rho}(E_1+E_2)$ for $E_1, E_2$ i.i.d.\;$\sim\Exp(1)$.

Meanwhile under $\bPsi^{(2)}$, $D_-$ takes values in $\{1,2,\dots\}$ 
and has the distribution of the sum of two independent Geometric($\rho$) random variables minus $1$. Note that if $X\sim\text{Geometric}(\rho)$, then
$X-1\isd BX$ where $B\sim\text{Bernoulli}(\rho)$ independently of $X$. We obtain that $\oL(0,1)$ has the distribution of $\frac{1}{\rho}(E_3+B E_4)$, for 
$B\sim\text{Bernoulli}(\rho)$ and $E_3, E_4$ i.i.d.\;$\sim\Exp(1)$ independently of $B$.

Note $\oL(0,1)$ and $\oL(1,0)$ are independent under $\bPsi^{(2)}$.
So we can combine the previous two paragraphs to get that 
the distribution of $\oxi(\boo)=\oL(0,1) \wedge \oL(1,0)$ 
under $\bPsi^{(2)}$ is that of 
\[
\frac{1}{1-\rho}(E_1+E_2)\wedge \frac{1}{\rho}(E_3+B E_4),
\]
for $B\sim\text{Bernoulli}(\rho)$ and $(E_i)_{1\leq 1\leq 4}$ i.i.d.\;$\sim\Exp(1)$ independently of $B$.

We continue in the particular case $\rho=1/2$. Then the distribution of $\oxi(\boo)$ is the same under $\bPsi^{(1)}$ as under $\bPsi^{(2)}$,
and so its distribution under $\bPsi$ is again the same, that of 
$2((E_1+E_2)\wedge (E_3+B E_4))$ for $B\sim\text{Bernoulli}(1/2)$ and $(E_i)_{1\leq 1\leq 4}$ i.i.d.\;$\sim\Exp(1)$ independently of $B$.

By elementary arguments involving the memoryless property of exponentials, 
this distribution can be seen to be a $(1/4, 1/2, 1/4)$ mixture
of Gamma($1,1$), Gamma($2,1$) and Gamma($3,1$) distributions. 

A similar but slightly more involved argument can be made for the case of general $\rho$, to give that the distribution of $\oxi(\boo)$ is again a mixture of Gamma($1,1$), Gamma($2,1$) and Gamma($3,1$) distributions,
now with weights
\[
\left(
\frac{\rho^4+(1-\rho)^4}{\rho^2+(1-\rho)^2},
2\rho(1-\rho),
\frac{2\rho^2(1-\rho)^2}{\rho^2+(1-\rho)^2}
\right).
\]
As a function of $\rho\in (0,1)$,
this distribution is stochastically increasing 
on $(0,1/2]$, and symmetric around $1/2$.\\

\noindent\textbf{The path as a competition interface.} 
As $\Gamma_\boo$ in the i.i.d.\;$\Exp(1)$ random field, the path $\oga$ under $\nu$ can also be described as a competition interface.
For $\oL$ under $\bPsi$, by slightly abusing the notations we let $\oI_0:=\{u\in \Z^2: \oL(u)\le 0\}$ and \[
\oxi^{\vee,0}(u):= \oL(u)\vee 0 - \oL(u-(1,0))\vee \oL(u-(0,1))\vee 0,
\]
for each $u\in\Z^2$.
Then like Lemma \ref{lem:Ideteta}, we can show that $\oI_0$ contains the same information as $\zeta_0$ (whose law under $\bPsi$ is explicitly described using $\Psi_+$ and $\Psi_-$ above). Namely, we have $(0,0)\in\oI_0$ and $(0,1), (1,0)\not\in \oI_0$; and for any $x\in \Z$, $\zeta_0(x)=0$ if and only if there is some $y \in \Z$ such that $(x+y,y)\in \oI_0$ and $(x+y,y+1)\not\in \oI_0$, and $\zeta_0(x)=1$ if and only if there is some $y \in \Z$ such that $(x+y-1,y)\in \oI_0$ and $(x+y,y)\not\in \oI_0$.

Under $\bPsi$ and conditioned on $\oI_0$, the weights $\{\oxi^{\vee,0}(u)\}_{u\not\in \oI_0}$ are i.i.d.\;$\Exp(1)$. This is because, for any $(a, b)\in\Z^2$, $\oxi^{\vee,0}(a,b)$ is the waiting time for the particle labeled $b$ and the hole labeled $a$ to switch, since they are next to each other; and that is i.i.d.\;$\Exp(1)$ for all $(a, b)\not\in \oI_0$, given $\zeta_0$.

From $\oI_0$ and $\oxi^{\vee,0}$ under $\bPsi$, we define a competition interface, similar to how the competition interface is defined in Section \ref{ssec:ci-hppair}.
Specifically, for any $u\le v$, $u, v \not\in \oI_0$, let $T^{\vee,0}_{u,v}$ and $\Gamma^{\vee,0}_{u,v}$ be the passage time and geodesic from $u$ to $v$ under the weights $\oxi^{\vee,0}$.
For any $v\in \Z_{\ge 0}^2\setminus \{\boo\}$, we consider the vertex $u_*\not\in I_0$ with the maximum $T^{\vee,0}_{u_*,v}$.
If $\Gamma^{\vee,0}_{u_*,v}$ intersects $\Z_+\times \{0\}$ we let $v\in C_1$; otherwise $\Gamma^{\vee,0}_{u_*,v}$ intersects $\{0\}\times \Z_+$ and we let $v\in C_2$.
We then have that $(\oga\cap \Z_{\ge 0}^2)+(1/2,1/2)$  is the boundary between $C_1$ and $C_2$, using analogues of Lemmas \ref{lem:non-cross} and \ref{lem:geo-part-cons}.

Using this formulation of $\oga$ and the explicit description of $\zeta_0$ under $\bPsi$, and passage time estimates (e.g. Theorem \ref{t:onepoint} below) or the convergence of the passage time point-to-line profile to the so-called Airy$_2$ process (see Theorem \ref{thm:profile-a2} below), one can possibly show that $\oga[i]$ has transversal fluctuation in the order of $i^{2/3}$ for large $i$ (here $\oga[i]$ denotes the $i$-th vertex in $\oga\cap \Z_{\ge 0}^2$), and even obtain exact formulae for the distribution of its scaling limit. We leave these to future explorations.

\section{Geometric estimates for LPP}  \label{ssec:elpps}
While so far most arguments are on TASEP and use interacting particle system techniques, for the rest of this paper we will mainly use various LPP geometric arguments.
In this section we do some preparations, by providing some useful tools.
Most results in this section have appeared in the literature.

In this section, we do not fix $\rho\in (0,1)$, and all constants are not assumed to depend on $\rho$, unless otherwise stated.
For $a,b\in\Z$ and $\rho \in (0,1)$, we denote \[\langle a, b\rangle_\rho := 
\left(\left\lfloor\frac{2(1-\rho)^2a}{\rho^2+(1-\rho)^2}\right\rfloor+b, \left\lceil\frac{2\rho^2a}{\rho^2+(1-\rho)^2}\right\rceil-b\right).\]
Then recall that $\bn^\rho=\langle n, 0\rangle_\rho$ for any $n\in \Z$.
We also write $\langle 0, b\rangle := (b,-b)$ for any $b\in\Z$.

We start with a basic geometric property called `ordering of geodesics'.
Note that for any $\Z^2$ vertices $u\le v$, if $u'\le v'$ and $u',v'\in\Gamma_{u,v}$, we must have that $\Gamma_{u',v'}$ is the part of $\Gamma_{u,v}$ between $u'$ and $v'$ (including $u', v'$).
This immediately leads to the following result.
\begin{lemma}  \label{l:ordering}
For any $(a_1,b_1)$ and $(a_2,b_2)$, we say $(a_1,b_1)\preceq(a_2,b_2)$, if $a_1\le a_2$ and $b_1\ge b_2$. 
For any $u_1$, $u_2$ and $v_1$, $v_2$ such that $u_1\le v_1$, $u_2\le v_2$ and $u_1\preceq u_2$, $v_1\preceq v_2$,
and any $w_1\in\Gamma_{u_1,v_1}$, $w_2\in\Gamma_{u_2,v_2}$, we cannot have $w_2\preceq w_1$ unless $w_1=w_2$.
\end{lemma}

We next give estimates on passage times.
We have that $T_{\boo,(m,n)}$ has the same law as the largest eigenvalue of $X^*X$ where $X$ is an $(m+1)\times (n+1)$ matrix of i.i.d.\;standard complex Gaussian entries (see \cite[Proposition 1.4]{johansson2000shape}).
Hence we get the following one-point estimates from \cite[Theorem 2]{LR10}. 
\begin{theorem}
\label{t:onepoint}
There exist constants $c,C>0$, such that for any $m\geq n \geq 1$ and $x>0$, we have
\begin{equation}  \label{e:wslope}
\PP[T_{\boo, (m,n)}-(\sqrt{m}+\sqrt{n})^{2} \geq xm^{1/2}n^{-1/6}] \leq Ce^{-cx}.    
\end{equation}
In addition, for each $\psi>1$, there exist $C',c'>0$ depending on $\psi$ such that if $\frac{m}{n}< \psi$, we have
\begin{equation}  \label{e:slope}
\begin{split}
&\PP[T_{\boo, (m,n)}-(\sqrt{m}+\sqrt{n})^{2} \geq xn^{1/3}] \leq C'e^{-c'(x^{3/2}\wedge xn^{1/3})},\\
&\PP[T_{\boo, (m,n)}-(\sqrt{m}+\sqrt{n})^{2} \leq -xn^{1/3}] \leq C'e^{-c'x^3},
\end{split}
\end{equation}
and as a consequence
\begin{equation}
\label{e:mean}
|\E[T_{\boo, (m,n)}] -(\sqrt{m}+\sqrt{n})^2|\leq C'n^{1/3}.
\end{equation}
\end{theorem}
Below we will frequently use parallelograms in $\R^2$.
For simplicity of notations, in the rest of this section, for any vertices $u\le v$ and $x>0$, we let $U_{u,v}^x$ denote the parallelogram whose one pair of opposite sides have length $2x$, parallel to the anti-diagonal, and have midpoints $u$ and $v$ respectively.
Formally, we let
\[
U_{u,v}^x=\{ u+\alpha(v-u) + (y,-y): \alpha\in [0,1], y\in [-x, x] \}.
\]
We next state the following parallelogram estimate.
\begin{prop}  [\protect{\cite[Theorem 4.2]{basu2019temporal}}]
\label{t:seg-to-seg}
Let $U_1, U_2$ be the part of $U_{(-m,m),(n,n)}^{n^{2/3}}$ below $\LL_{\lfloor n/3 \rfloor}$ and above $\LL_{\lceil 2n/3 \rceil}$ respectively.
For each $\psi<1$, there exist constants $c,C>0$ depending only on $\psi$, such that when $|m|<\psi n$,
\[
\PP\Big[\sup_{u\in U_1, v\in U_2} |T_{u,v}-\E[T_{u,v}]| \geq x n^{1/3}\Big]\leq Ce^{-c(x^{3/2}\wedge xn^{1/3})}.
\]
\end{prop}
Such a result is first proved as \cite[Proposition 10.1, 10.5]{basu2014last}, in the setting of Poissonian LPP.
In the setting of exponential LPP a proof is given in \cite[Appendix C]{basu2019temporal}, following the ideas in \cite{basu2014last}.

We will also need the following estimate on the coalescing probability of two geodesics, for finite and semi-infinite geodesics respectively.
\begin{prop} [\protect{\cite{zhang2019optimal}}] \label{prop:coalesce}
For each $\psi \in (0,1)$, there exists $C>0$, such that 
\[
\PP[\Gamma_{\boo,\langle n,b_1\rangle_{1/2}} \cap \LL_{n-m} = \Gamma_{\boo,\langle n,b_2\rangle_{1/2}} \cap \LL_{n-m}] > 1 - C m^{-2/3}|b_1-b_2| 
\]
for any $n,m\in\N, b_1, b_2 \in \Z$, such that $m<n/3$, $|b_1|, |b_2|<\psi n$.
\end{prop}
\begin{prop}[\protect{\cite[Theorem 2]{basu2019coalescence}}]  \label{prop:coal-semi-inf-1}
For any $\rho\in(0,1)$, there is a constant $C>0$, such that for any $r\in\N$, and $k>1$, we have $\PP[\Gamma_\boo^\rho\cap \LL_{\lfloor r^{3/2}k\rfloor} \neq \Gamma_{\langle 0, r \rangle}^\rho\cap \LL_{\lfloor r^{3/2}k\rfloor}] < Ck^{-2/3}$.
\end{prop}
\begin{figure}[hbt!]
     \centering
     \begin{subfigure}[t]{0.45\textwidth}
         \centering
\begin{tikzpicture}[line cap=round,line join=round,>=triangle 45,x=5cm,y=5cm]
\clip(-0.15,-0.15) rectangle (1.15,0.8);

\fill[line width=0.pt,color=green,fill=green,fill opacity=0.35]
(0.19,0.21) -- (0.29,0.11) -- (0.05,-0.05) -- (-0.05,0.05) -- cycle;

\draw (0.5,-0.04) node[anchor=north east]{$\LL_r$};

\draw [line width=.4pt] (-0.35,0.75) -- (0.75,-0.35);

\draw (0,0) node[anchor=north]{$\boo$};
\draw (.24,.16) node[anchor=west]{$\langle r,0\rangle_\rho$};
\draw (0.78,0.73) node[anchor=north west]{$\Gamma_\boo^\rho$};

\draw [red] plot [smooth] coordinates {(0.9,.85) (0.87,0.82) (0.84,0.79) (0.78,0.73) (0.71,0.71) (0.68,0.68) (0.66,0.64) (0.6,0.59) (0.58,0.58) (0.52,0.56) (0.46,0.51) (0.39,0.45) (0.35,0.39) (0.31, 0.36) (0.27,0.34) (0.25,0.27) (0.23,0.19) (0.18,0.14) (0.14,0.09) (0.07,0.04) (0,0)};
\draw [fill=uuuuuu] (0.,0.) circle (1.0pt);
\draw [fill=uuuuuu] (.24,.16) circle (1.0pt);
\end{tikzpicture}

         \caption{For the semi-infinite geodesic $\Gamma_\boo^\rho$: Lemma \ref{lem:semi-inf-trans} stats that with probability $>1-Ce^{-cx^3}$, its intersection with $\LL_r$ is within distance $xr^{2/3}$ to $\langle r,0\rangle_\rho$; Corollary \ref{cor:trans-semi-inf} stats that with probability $>1-Ce^{-cx^3}$, below $\LL_r$ it is contained in $U_{\boo,\langle r,0\rangle_\rho}^{xr^{2/3}}$.}
         \label{fig:62p}
     \end{subfigure}
     \hfill
     \begin{subfigure}[t]{0.5\textwidth}
         \centering
\begin{tikzpicture}[line cap=round,line join=round,>=triangle 45,x=5cm,y=5cm]
\clip(-0.15,-0.15) rectangle (1.25,0.8);

\fill[line width=0.pt,color=green,fill=green,fill opacity=0.35]
(0.19,0.21) -- (0.29,0.11) -- (0.05,-0.05) -- (-0.05,0.05) -- cycle;

\draw (0.5,-0.06) node[anchor=north east]{$\LL_r$};

\draw [line width=.4pt] (-0.35,0.75) -- (0.75,-0.35);

\draw [line width=.3pt] [dotted] (0,0) -- (0.9,0.6);

\draw (0,0) node[anchor=north]{$\boo$};
\draw (0.9,.71) node[anchor=north west]{$\langle n,b\rangle_{1/2}$};
\draw (.24,.16) node[anchor=west]{$\langle r,b'\rangle_{1/2}$};

\draw (0.6,.46) node[anchor=south east]{$\Gamma_{\boo, \langle n,b\rangle_{1/2}}$};

\draw [red] plot [smooth] coordinates {(0.9,.6) (0.87,0.57) (0.84,0.56) (0.78,0.48) (0.71,0.47) (0.68,0.455)  (0.6,0.44) (0.58,0.38) (0.46,0.34) (0.39,0.33) (0.35,0.25) (0.31, 0.22) (0.27,0.2) (0.25,0.15) (0.23,0.11) (0.18,0.1) (0.14,0.05) (0.07,0.04) (0,0)};
\draw [fill=uuuuuu] (0.,0.) circle (1.0pt);
\draw [fill=uuuuuu] (.24,.16) circle (1.0pt);
\draw [fill=uuuuuu] (0.9,.6) circle (1.0pt);
\end{tikzpicture}

         \caption{For the infinite geodesic $\Gamma_{\boo, \langle n,b\rangle_{1/2}}$: Corollary \ref{cor:trans-fluc-comb} stats that with probability $>1-Ce^{-cx^3}$, below $\LL_r$ it is contained in $U_{\boo,\langle r,b'\rangle_{1/2}}^{xr^{2/3}}$.}
         \label{fig:62f}
     \end{subfigure}
        \caption{Illustrations of the transversal fluctuation estimates.}
        \label{fig:62}
\end{figure}
Similar coalescence estimates have also been obtained in various other papers, such as \cite{seppalainen2020coalescence, balazs2021local}.

We next give some estimates on transversal fluctuations of geodesics (see Figure \ref{fig:62}).
Such geodesic fluctuation estimates have been proved using various methods in the literature \cite{basu2014last,basu2021time, basu2019temporal,basu2019coalescence, hammond2020modulus, zhang2019optimal, busani2022universality, emrah2020right, bhatia2020moderate}.
We start with an estimate for semi-infinite geodesics, which is illustrated by Figure \ref{fig:62p}.
\begin{lemma}  \label{lem:semi-inf-trans}
For any $\psi\in(0,1)$, there exist $c,C>0$ such that the following is true.
Let $\rho\in (\psi, 1-\psi)$, and $r, b_r\in\Z$ such that $\Gamma_\boo^\rho[2r+1]=\langle r,b_r\rangle_\rho$. Then $\PP[|b_r|>xr^{2/3}] < Ce^{-cx^3}$ for any $x>0$.
\end{lemma}
This bound can be quickly deduced from \cite[Theorem 3.1]{emrah2020right} or \cite[Theorem 2.4]{bhatia2020moderate}.
Here we give a proof using the above passage time estimates, and properties of the Busemann function.
\begin{proof}[Proof of Lemma \ref{lem:semi-inf-trans}]
In this proof we let $c,C>0$ denote small and large constants that depend on $\psi$, and the values can change from line to line.
We assume that $r$ and $x$ are large enough (depending on $\psi$), since otherwise the conclusion follows obviously.

For simplicity of notations we denote $T_{u,v}^\bu = T_{u,v}-\xi(v)$ for any vertices $u\le v$.
Let $\bB$ be the Busemann function in direction $\brho$, as defined in Section \ref{ssec:buseman}.
By Lemma \ref{lem:buse-opti}, the event $|b_r|>xr^{2/3}$ implies that there exists $b\in\Z$ such that $|b|>xr^{2/3}$, and $T_{\boo,\langle r,b\rangle_\rho}^\bu + \bB(\langle r,b\rangle_\rho,\br^\rho) > T_{\boo,\br^\rho}^\bu$.
To bound this event, we denote $L_j:=\{\langle r,b\rangle_\rho: |b-2j\lfloor r^{2/3} \rfloor|\le r^{2/3}\}$ for $j\in\Z$.
For each $j$ such that $L_j$ intersects $\Z_{\ge 0}^2$, we have
\begin{equation}
\begin{split}  \label{eq:thr-term}
&\PP\big[\max_{u\in L_j\cap \Z_{\ge 0}^2} T_{\boo,u}^\bu + \bB(u,\br^\rho) > T_{\boo,\br^\rho}^\bu\big] \\ <& \PP\big[\max_{u\in L_j\cap \Z_{\ge 0}^2} T_{\boo,u}^\bu - \E[T_{\boo,u}^\bu] > c_0j^2r^{1/3}\big] 
+ \PP\big[T_{\boo,\br^\rho}^\bu -\E[T_{\boo,\br^\rho}^\bu] < -c_0j^2r^{1/3}\big]
\\
&+ \PP\big[\max_{u\in L_j\cap \Z_{\ge 0}^2}\bB(u,\br^\rho) + \E[T_{\boo,u}^\bu] - \E[T_{\boo,\br^\rho}^\bu] > -2c_0j^2r^{1/3}\big],
\end{split}    
\end{equation}
where $c_0>0$ is a small constant depending only on $\psi$, and satisfies several conditions to be specified below. 
We next show that each term in the right-hand side of \eqref{eq:thr-term} is bounded by $Ce^{-c|j|^3}$; then by summing over $j\in \Z$ such that $2|j|+1>x$ and $L_j$ intersects $\Z_{\ge 0}^2$ (note that the later implies that $|j|\le 2r (2\lfloor r^{2/3} \rfloor)^{-1} < 2r^{1/3}$) we get the conclusion.
\begin{itemize}
    \item For the first term, we let $\psi'>0$ be a small number (depending on $c_0$ and $\psi$ and to be determined). When $L_j$ is contained in $\{\langle r,b\rangle_{1/2}: |b|<(1-\psi')r\}$, we consider the parallelogram $U_{\boo,\langle r, 2j\lfloor r^{2/3} \rfloor \rangle_\rho}^{r^{2/3}}$.
    Using Proposition \ref{t:seg-to-seg} with this parallelogram we get the desired bound.
    When $L_j$ is not contained in $\{\langle r,b\rangle_{1/2}: |b|<(1-\psi')r\}$ we cannot directly apply Proposition \ref{t:seg-to-seg}, since the above parallelogram may not satisfy the slope condition there.
    Instead, we take some small $\alpha>0$ (depending on $c_0$ and $\psi$ and to be determined), and consider the parallelogram $U_{-\langle \lfloor \alpha r\rfloor, 0\rangle_{1/2}, \langle r, 2j\lfloor r^{2/3} \rfloor \rangle_\rho}^{r^{2/3}}$.
    Using Proposition \ref{t:seg-to-seg} with this parallelogram we get
    \begin{equation}  \label{eq:weak-1}
    \PP\big[\max_{u\in L_j\cap \Z_{\ge 0}^2} T_{-\langle \lfloor \alpha r\rfloor, 0\rangle_{1/2},u}^\bu - \E[T_{-\langle \lfloor \alpha r\rfloor, 0\rangle_{1/2},u}^\bu] > 2^{-1}c_0j^2r^{1/3}\big] < Ce^{-c|j|^3}.    
    \end{equation}
    For any $u\in L_j\cap \Z_{\ge 0}^2$ we have $T_{\boo,u}^\bu \le T_{\langle -\lfloor \alpha r\rfloor, 0\rangle_{1/2},u}^\bu$, and
    \[
    \E[T_{\boo,u}^\bu] > \E[T_{\langle -\lfloor \alpha r\rfloor, 0\rangle_{1/2},u}^\bu] - 200^{-1}c_0 \psi^2 r > \E[T_{\langle -\lfloor \alpha r\rfloor, 0\rangle_{1/2},u}^\bu] - 2^{-1}c_0j^2r^{1/3},
    \]
    where the two inequalities are due to the following reasons.
    The first inequality is by $\E[T_{\boo,u}^\bu]\ge 2r$ and $\E[T_{\langle -\lfloor \alpha r\rfloor, 0\rangle_{1/2},u}^\bu]<2r+200^{-1}c_0 \psi^2 r$, which is due to \eqref{e:mean} and the fact that $L_j$ is not contained in $\{\langle r,b\rangle_{1/2}: |b|<(1-\psi')r\}$, and taking $\psi'$ and $\alpha$ small enough (depending on $\psi$ and $c_0$). The second inequality is by $|j|>0.1\psi r^{1/3}$, which is implied by the fact that $L_j$ is not contained in $\{\langle r,b\rangle_{1/2}: |b|<(1-\psi')r\}$.
    
    Thus we have 
    \[
    \max_{u\in L_j\cap \Z_{\ge 0}^2} T_{-\langle \lfloor \alpha r\rfloor, 0\rangle_{1/2},u}^\bu - \E[T_{-\langle \lfloor \alpha r\rfloor, 0\rangle_{1/2},u}^\bu] > \max_{u\in L_j\cap \Z_{\ge 0}^2} T_{\boo,u}^\bu - \E[T_{\boo,u}^\bu] - 2^{-1}c_0j^2r^{1/3},\]
    so the first term in the right-hand side of \eqref{eq:thr-term} is bounded as desired by \eqref{eq:weak-1}.
    \item For the second term we apply Theorem \ref{t:onepoint}.
    \item For the last term, by \eqref{e:wslope} and \eqref{e:mean} from Theorem \ref{t:onepoint}, we have $\E[T_{\boo,\langle r, b\rangle_\rho}^\bu] - \E[T_{\boo,\br^\rho}^\bu] \le Cr^{1/3}-b(\rho^{-1}-(1-\rho)^{-1}) - c_1b^2 r^{-1}$ for any $\langle r,b\rangle_\rho\in\Z_{\ge 0}^2$,
    where $c_1>0$ is determined by $\psi$. 
    By taking $c_0<c_1$, and assuming that both $r$ and $|j|$ are large enough, we have
    $c_1b^2r^{-1}-Cr^{1/3}-2c_0j^2r^{1/3}>c_0j^2r^{1/3}$ when $|b-2j\lfloor r^{2/3} \rfloor|\le r^{2/3}$.
    Then the last term in the right-hand side of \eqref{eq:thr-term} is bounded by
    \[
    \PP\big[\max_{ |b-2j\lfloor r^{2/3} \rfloor|\le r^{2/3} }\bB(\langle r, b\rangle_\rho,\br^\rho) -b(\rho^{-1}-(1-\rho)^{-1}) > c_0j^2r^{1/3} \big].
    \] 
    Note that $b\mapsto \bB(\langle r,b\rangle_\rho, \br^\rho)-b(\rho^{-1}-(1-\rho)^{-1})$ is a (two-sided) centered random walk, where each step has exponential tail determined by $\rho$ (see Section \ref{ssec:buseman}).
    We can apply concentration inequalities to get the desired bound.\\
    (For example, we can do a Chernoff type estimate as follows. Take any $c_2>0$, small enough depending on $\psi$. Without loss of generality we assume $j>0$. Write the random walk as $b\mapsto \sum_{i=1}^b X_i$ for $b> 0$, where each $X_i$ has exponential tail determined by $\rho$. Consider $e^{c_2jr^{-1/3}\sum_{i=1}^b X_i}$, which is a supermartingale in $b$.
    Let $\tau$ be the first time after $(2j-1)\lfloor r^{2/3}\rfloor$ when this supermartingale is at least $e^{c_2^{3/2}j^3}$, or $(2j+1)\lfloor r^{2/3}\rfloor+1$, whichever is smaller.
    Then we have
    \[\begin{split}
    &\PP\bigg[\max_{|b-2j\lfloor r^{2/3}\rfloor|\le r^{2/3}}\sum_{i=1}^b X_i > c_2^{1/2}j^2r^{1/3}\bigg]
    \\
    \le &\PP\bigg[\sum_{i=1}^{\tau} X_i > c_2^{1/2}j^2r^{1/3}\bigg] = \PP[ e^{c_2jr^{-1/3}\sum_{i=1}^{\tau} X_i} > e^{c_2^{3/2}j^3} ] \\
    \le & e^{-c_2^{3/2}j^3}\E[e^{c_2jr^{-1/3}\sum_{i=1}^{\tau} X_i}] \\
    \le & e^{-c_2^{3/2}j^3}\E[e^{c_2jr^{-1/3}\sum_{i=1}^{(2j+1)\lfloor r^{2/3}\rfloor+1} X_i}] = e^{-c_2^{3/2}j^3}\E[e^{c_2jr^{-1/3}X_1}]^{(2j+1)\lfloor r^{2/3}\rfloor+1},    
    \end{split}\]
    and this is bounded by $e^{-c_2^{3/2}j^3/2}$ when $c_2$ is small enough, since $\E[e^{c_2jr^{-1/3}X_1}] < e^{Cc_2^2j^2r^{-2/3}}$.)
\end{itemize}
By these bounds the conclusion follows.
\end{proof}

In addition to the above one-point bound, we also quote the following uniform bound on transversal fluctuations of geodesics.
\begin{lemma} [\protect{\cite[Proposition C.9]{basu2019temporal}}] \label{lem:trans-fluc-uni}
For each $\psi\in(0,1)$ there exist constants $c,C>0$ such that the following is true.
For $x>0$, $n\in\N$, and $|b|<\psi n$, the probability that the geodesic $\Gamma_{\boo,\langle n, b\rangle_{1/2}}$ exits $U_{\boo,\langle n,b\rangle_{1/2}}^{xn^{2/3}}$ is at most $Ce^{-cx^3}$.
\end{lemma}
The following result for semi-infinite geodesics follows immediately by combining Lemma \ref{lem:semi-inf-trans} and Lemma \ref{lem:trans-fluc-uni}. See also Figure \ref{fig:62p}.
\begin{cor} \label{cor:trans-semi-inf}
For each $\psi\in(0,1)$, there exist constants $c,C>0$ such that the following is true.
Take any $r\in\N$ large enough, any $x>0$, and any $\rho\in (\psi, 1-\psi)$.
Then with probability at least $1-Ce^{-cx^3}$, the part of the geodesic $\Gamma_\boo^\rho$ below $\LL_r$ is contained in $U_{\boo,\langle r,0\rangle_\rho}^{xr^{2/3}}$.
\end{cor}
We also have such near-end transversal fluctuation estimate for finite geodesics (see Figure \ref{fig:62f}).
\begin{cor} \label{cor:trans-fluc-comb}
For each $\psi\in(0,1)$, there exist constants $c,C>0$ such that the following is true.
Take any integers $0<r<n$ that are large enough, any $x>0$, and any integer $b$ with $|b|<\psi n$.
Let $\langle r, b'\rangle_{1/2}$ be the vertex in $\LL_r$ that is closest to the straight line connecting $\boo$ and $\langle n, b\rangle_{1/2}$.
Then with probability at least $1-Ce^{-cx^3}$, the geodesic $\Gamma_{\boo,\langle n, b\rangle_{1/2}}$ below $\LL_r$ is contained in $U_{\boo,\langle r,b'\rangle_{1/2}}^{xr^{2/3}}$.
\end{cor}

\begin{proof}
Let $c,C>0$ denote small and large constants depending only on $\psi$, and the values can change from line to line.
When $x>2r^{1/3}$ the conclusion follows obviously, so we can assume that $x\le 2r^{1/3}$.
We now take $\rho_-, \rho_+ \in (0,1)$, such that $\bn^{\rho_-}=\langle n, b-\lfloor cxn^{1/3} \rfloor\rangle_{1/2}$ and $\bn^{\rho_+}=\langle n, b+\lfloor cxn^{1/3} \rfloor\rangle_{1/2}$.
Take $b_-, b_+\in\Z$ such that $\langle n, b_-\rangle_{1/2} \in \Gamma_\boo^{\rho_-}$ and $\langle n, b_+\rangle_{1/2} \in \Gamma_\boo^{\rho_+}$.
By Lemma \ref{lem:semi-inf-trans}, with probability at least $1-Ce^{-cx^3}$ we have $b_-<b<b_+$, thus $\Gamma_{\boo,\langle n, b\rangle_{1/2}}$ is sandwiched between $\Gamma_\boo^{\rho_-}$ and $\Gamma_\boo^{\rho_+}$ below $\LL_r$ by ordering of geodesics (Lemma \ref{l:ordering}).
By Corollary \ref{cor:trans-semi-inf}, with probability at least $1-Ce^{-cx^3}$, $\Gamma_\boo^{\rho_-}$ and $\Gamma_\boo^{\rho_+}$ below $\LL_r$ are both contained in $U_{\boo,\langle r,b'\rangle_{1/2}}^{xr^{2/3}}$, so the conclusion follows.
\end{proof}

Finally, we have the following estimate on the passage time along a semi-infinite geodesic.
\begin{figure}[hbt!]
\centering
\begin{tikzpicture}[line cap=round,line join=round,>=triangle 45,x=7cm,y=7cm]
\clip(-0.15,-0.18) rectangle (1.1,0.8);

\fill[line width=0.pt,color=yellow,fill=yellow,fill opacity=0.5]
(0.64,0.51) -- (0.74,0.41) -- (0.11,-0.01) -- (0.01,0.09) -- cycle;

\draw (0.2,-0.06) node[anchor=north east]{$\LL_{\lfloor cl \rfloor}$};
\draw (0.9,0.3) node[anchor=north east]{$\LL_{\lfloor Cl \rfloor}$};
\draw (0.55,-0.06) node[anchor=north east]{$\HH_l^\rho$};

\draw [line width=.4pt] (-0.35,0.45) -- (0.45,-0.35);
\draw [line width=.4pt] (0.15,1) -- (1,.15);

\draw [line width=.4pt] (0.7,-0.2) -- (-0.2,.55);
\draw [line width=1pt] [green] (0.31,0.125) -- (0.19+0.012,.225-0.01);

\draw (0,0) node[anchor=north]{$\boo$};
\draw (0.25+0.006,0.17) node[anchor=east]{$u_*$};

\draw (0.45,.36) node[anchor=south east]{$\Gamma_\boo^\rho$};

\draw [red] plot [smooth] coordinates {(0.9,.6) (0.87,0.57) (0.84,0.56) (0.78,0.48) (0.71,0.47) (0.68,0.455)  (0.6,0.44) (0.58,0.38) (0.46,0.34) (0.39,0.33) (0.35,0.25) (0.31, 0.22) (0.27,0.2) (0.25,0.15) (0.23,0.11) (0.18,0.1) (0.14,0.05) (0.07,0.04) (0,0)};
\draw [fill=uuuuuu] (0.,0.) circle (1.0pt);
\draw [fill=uuuuuu] (0.25+0.006,0.17) circle (1.0pt);
\end{tikzpicture}

         \caption{An illustration of Lemma \ref{lem:semi-inf-passt} and its proof: the yellow region is the parallelogram $U=U_{\langle \lfloor cl\rfloor,0\rangle_\rho, \langle \lfloor Cl\rfloor,0\rangle_\rho}^{cx^{1/2}l^{2/3}}$, and $V$ is the set of vertices within distance $1$ from the green segment. When $c$ is small and $C$ is large (depending on $\rho$), if $\Gamma_\boo^\rho$ between $\LL_{\lfloor cl\rfloor}$ and $\LL_{\lfloor Cl\rfloor}$ is contained in $U$, it must intersect the line $\HH_l^\rho$ inside $U$.}
         \label{fig:217}
\end{figure}
For simplicity of notations, below we denote (recall that $\brho=((1-\rho)^2,\rho^2)$)
\[\HH_x^\rho:=\{x\brho+y((1-\rho),-\rho):y\in\R\},\]
for any $\rho\in(0,1)$ and $x\in \R$. Note that $\HH_x$ intersect the axes at $(0, x\rho)$ and $(x(1-\rho), 0)$.
\begin{lemma}  \label{lem:semi-inf-passt}
For each $\psi\in(0,1)$, there exist constants $c,C>0$ such that the following is true.
Take any $\rho\in (\psi, 1-\psi)$ and $l>0$.
Let $u_*$ be the first vertex in $\Gamma_\boo^\rho$ above the line $\HH_l^\rho$.
Then $\PP[|T_{\boo,u_*}-l|>xl^{1/3}]< Ce^{-cx}$ for any $0<x<cl^{2/3}$.
\end{lemma}
\begin{proof}
Let $c,C>0$ denote small and large constants depending only on $\psi$, and below the values can change from line to line.
Let $U=U_{\langle \lfloor cl\rfloor,0\rangle_\rho, \langle \lfloor Cl\rfloor,0\rangle_\rho}^{cx^{1/2}l^{2/3}}$.
Let $V$ be the set of all $v\in U$ that is within distance $1$ to the line $\HH_l^\rho$.
By Corollary \ref{cor:trans-semi-inf}, with probability at least $1-Ce^{-cx^{3/2}}$, the geodesic $\Gamma_\boo^\rho$ between $\LL_{\lfloor cl\rfloor}$ and $\LL_{\lfloor Cl\rfloor}$ is contained in $U$, thus $u_*\in U$ and $u_* \in V$, since (when $c$ is small and $C$ is large) we must have that $\LL_{\lfloor cl\rfloor}\cap \Z_{\ge 0}^2$ is below $\HH_l^\rho$ and $\LL_{\lfloor Cl\rfloor}\cap \Z_{\ge 0}^2$ is above $\HH_l^\rho$ (see Figure \ref{fig:217}).

By \eqref{e:mean} in Theorem \ref{t:onepoint}, we have $|\E[T_{\boo,v}] - l|<cxl^{1/3}$ for any $v\in V$.
It remains to show that
\begin{equation} \label{eq:semi-inf-pass}
\PP\bigg[\max_{v\in V} |T_{\boo,v}-\E[T_{\boo,v}]|>cxl^{1/3}\bigg]< Ce^{-cx}.    
\end{equation}
For this, we split $V$ into $\lceil x^{1/2} \rceil$ sets $V_1, \ldots, V_{\lceil x^{1/2} \rceil}$, each with diameter $<cl^{2/3}$. Since $cx^{1/2}l^{2/3}<cl$, the slope of any line passing through $\boo$ and intersecting $V$ is bounded away from $0$ and $\infty$. So we can apply Proposition \ref{t:seg-to-seg} to each $V_i$ and conclude that
\[
\PP\bigg[\max_{v\in V_i} |T_{\boo,v}-\E[T_{\boo,v}]|>cxl^{1/3}\bigg]< Ce^{-c(x^{3/2}
\wedge xl^{2/3})}.
\]
By a union bound over $i$ we get \eqref{eq:semi-inf-pass}, so the conclusion follows.
\end{proof}
Combining Corollary \ref{cor:trans-semi-inf} and Lemma \ref{lem:semi-inf-passt} we get the following (see Figure \ref{fig:218}).
\begin{cor}  \label{cor:lem:semi-inf-loc}
For each $\psi\in(0,1)$, there exist constants $c,C>0$ such that the following is true.
Take any $\rho\in (\psi, 1-\psi)$ and $l>0$, and let $u_*$ be the last vertex in $\Gamma_\boo^\rho$ with $T_{\boo,u_*}\le l$.
Then for any $0<x<cl^{2/9}$, with probability $>1-Ce^{-cx^3}$ the vertex $u_*$ is between the lines $\HH_{l-x^3l^{1/3}}^\rho$ and $\HH_{l+x^3l^{1/3}}^\rho$, and $\Gamma_{\boo,u_*}\subset U_{\boo,\langle l, 0\rangle_\rho}^{xl^{2/3}}$.
\end{cor}

\begin{figure}[hbt!]
\centering
\begin{tikzpicture}[line cap=round,line join=round,>=triangle 45,x=9cm,y=9cm]
\clip(-0.15,-0.1) rectangle (0.6,0.58);

\fill[line width=0.pt,color=yellow,fill=yellow,fill opacity=0.5]
(0.42,0.43) -- (0.56,0.29) -- (0.07,-0.07) -- (-0.07,0.07) -- cycle;

\draw (0.5,-0.09) node[anchor=south]{$\HH_{l-x^3l^{1/3}}^\rho$};
\draw (0.5,0.07) node[anchor=south]{$\HH_{l+x^3l^{1/3}}^\rho$};
\draw (0.35,.5) node[anchor=east]{$\LL_l$};

\draw [line width=.4pt] (0.8,-0.2) -- (-0.1,.5);
\draw [line width=.4pt] (0.6,-0.2) -- (-0.3,.5);
\draw [line width=.4pt] (0.8,0.05) -- (0.3,.55);

\draw (0,0) node[anchor=north]{$\boo$};
\draw (0.25+0.006,0.17) node[anchor=east]{$u_*$};

\draw (0.55,.36) node[anchor=south east]{$\Gamma_\boo^\rho$};

\draw [red] plot [smooth] coordinates {(0.9,.6) (0.87,0.57) (0.84,0.56) (0.78,0.48) (0.71,0.47) (0.68,0.455)  (0.6,0.44) (0.58,0.38) (0.46,0.34) (0.39,0.33) (0.35,0.25) (0.31, 0.22) (0.27,0.2) (0.25,0.15) (0.23,0.11) (0.18,0.1) (0.14,0.05) (0.07,0.04) (0,0)};
\draw [fill=uuuuuu] (0.,0.) circle (1.0pt);
\draw [fill=uuuuuu] (0.25+0.006,0.17) circle (1.0pt);
\end{tikzpicture}

         \caption{An illustration of Corollary \ref{cor:lem:semi-inf-loc}: the yellow region is the parallelogram $U_{\boo,\langle l, 0\rangle_\rho}^{xl^{2/3}}$. When $x<cl^{2/9}$, the intersections between $U_{\boo,\langle l, 0\rangle_\rho}^{xl^{2/3}}$ and the lines $\HH_{l-x^3l^{1/3}}^\rho$ and $\HH_{l+x^3l^{1/3}}^\rho$ are strictly below $\LL_l$. Thus if $\Gamma_\boo^\rho$ below $\LL_l$ is contained in $U_{\boo,\langle l, 0\rangle_\rho}^{xl^{2/3}}$, the part of $\Gamma_\boo^\rho$ between $\HH_{l-x^3l^{1/3}}^\rho$ and $\HH_{l+x^3l^{1/3}}^\rho$ must also be contained in $U_{\boo,\langle l, 0\rangle_\rho}^{xl^{2/3}}$; and if in addition $u_*$ is between $\HH_{l-x^3l^{1/3}}^\rho$ and $\HH_{l+x^3l^{1/3}}^\rho$, we must have $\Gamma_{\boo,u_*}\subset U_{\boo,\langle l, 0\rangle_\rho}^{xl^{2/3}}$.}
         \label{fig:218}
\end{figure}

\section{Convergence of TASEP as seen from an isolated second-class particle}   \label{sec:cov-phi}
Starting from this section, we again always fix $\rho\in (0,1)$, and the choice of all other parameters and constants can depend on $\rho$ unless otherwise stated.

Using geometric arguments and estimates from Section \ref{ssec:elpps}, in this section we upgrade Proposition \ref{prop:converge-2nd-class-tasep} to Theorem \ref{thm:cov-phi}.
The general idea is to show that $\Phi_t$ and $\Phi_{t+s}$ are close when $s$ is much smaller than $t$.
\begin{prop}\label{prop:phi-close}
For any $N\in\N$, there is a constant $C>0$ such that the following is true.
Take any $s,t>C$ with $t<s^{1.01}$, and any continuous function $f:\{0,1,*\}^{\llbracket -N, N \rrbracket}\to [0,1]$, regarded as a function on $\{0,1,*\}^\Z$, we have
$|\Phi_t(f)-\Phi_{t+s}(f)| < C(s/t)^{1/30}$.
\end{prop}
Using this we can deduce Theorem \ref{thm:cov-phi}.
\begin{proof}[Proof of Theorem \ref{thm:cov-phi}]
Take any $N\in\N$ and $f:\{0,1,*\}^{\llbracket -N, N \rrbracket}\to [0,1]$, regarded as a function on $\{0,1,*\}^\Z$, it suffices to show that
\begin{equation}  \label{eq:cov-phi-pf}
\lim_{t\to\infty}\Phi_t(f) = \Psi(f).
\end{equation}
Take any $\delta>0$.
By Proposition \ref{prop:converge-2nd-class-tasep} we have that $(\delta t)^{-1}\int_0^{\delta t}\Phi_{t+s}(f)ds \to \Psi(f)$ as $t\to\infty$.
By Proposition \ref{prop:phi-close} we have for any $t>C$,
\begin{multline*}
\left|\Phi_t(f)-(\delta t)^{-1}\int_0^{\delta t}\Phi_{t+s}(f)ds\right| \le (\delta t)^{-1}\int_{t^{1/1.01}}^{\delta t}|\Phi_t(f)-\Phi_{t+s}(f)|ds + (\delta t)^{-1}t^{1/1.01} \\ < C\delta^{1/30} + (\delta t)^{-1}t^{1/1.01},    
\end{multline*}
where $C$ is a constant depending on $N$.
Thus $\limsup_{t\to\infty}|\Phi_t(f) - \Psi(f)| \le C\delta^{1/30}$, and by sending $\delta \to 0$ we get \eqref{eq:cov-phi-pf}.
\end{proof}
To prove Proposition \ref{prop:phi-close}, we construct a coupling between $\Phi_t$ and $\Phi_{t+s}$.
For this we recall the setup of TASEP as seen from a hole-particle pair (or equivalently a second-class particle).

Let $(\eta_t^-)_{t\ge 0}$ and $(\eta_t^+)_{t\ge 0}$ be two copies of TASEP, with $\eta_0^-(0)=\eta_0^+(0)=0$ and $\eta_0^-(1)=\eta_0^+(1)=1$; and all $\eta_0^-(x), \eta_0^+(x)$ for $x\in\Z\setminus\{0,1\}$ are i.i.d.\;Bernoulli$(\rho)$.
In both copies, we label the holes by $\Z$ from left to right, with the hole at site $0$ at time $0$ labeled $0$;
and label the particles by $\Z$ from right to left, with the particle at site $1$ at time $0$ labeled $0$.
Keeping track of the hole-particle pair starting from sites $0$ and $1$, as described in Section \ref{ssec:ci-hppair}, we denote $p_t^-=(a_t^-, b_t^-)$ and  $p_t^+=(a_t^+, b_t^+)$ as the labels of the tracked hole and particle at time $t$ (in $(\eta_t^-)_{t\ge 0}$ and $(\eta_t^+)_{t\ge 0}$ respectively). 

For notational convenience, we also denote $\heta^-_t=\eta^-_t(x+a_t^--b_t^-+\cdot)$, $\heta^+_t=\eta^+_t(x+a_t^+-b_t^++\cdot)$ for any $t\ge 0$.
Then $(\heta_t^-)_{t\ge 0}$ and $(\heta_t^+)_{t\ge 0}$ are TASEPs as seen from a hole-particle pair, and by replacing the hole-particle pair in $\heta_t^-$ or $\heta_t^+$ by a second-class particle we get the distribution $\Phi_t$ (defined in Section \ref{ssec:cv-avg-tasep}).

Below we fix $s>0$.
Our general strategy to couple the processes $(\heta_t^-)_{t\ge 0}$ and $(\heta_{t+s}^+)_{t\ge 0}$ so that they evolve with the same set of waiting times (to be explained shortly).
We implement this via coupling TASEP and LPP as described at the beginning of Section \ref{ssec:ci-hppair}, and coupling the corresponding LPP models.
For this, let's set up some useful notations.
\begin{itemize}
    \item For any $a,b\in\Z$, if in ${\eta_0^-}$ (resp.\;${\eta_0^+}$) the particle with label $b$ is to the left of the hole with label $a$, we denote $L^-(a,b)$ (resp.\;$L^+(a,b)$) as the time when they switch; otherwise we set $L^-(a,b)=0$ (resp.\;$L^+(a,b)=0$).
Let $\{\xi^-(u)\}_{u\in \Z^2}$ (resp.\;$\{\xi^+(u)\}_{u\in \Z^2}$) be i.i.d.\;$\Exp(1)$ weights, and below we work under the almost sure event that there is a unique geodesic between any $\Z^2$ vertices $u\le v$ under these weights, and from any $u\in\Z^2$ there is a unique semi-infinite geodesic in direction $\brho$ under these random weights, and all these semi-infinite geodesics coalesce.
We let $\bG^-$ (resp.\;$\bG^+$) be the LPP Busemann function in direction $\brho$ under these random weights (defined like $\bG$ in Section \ref{ssec:buseman}).
We couple $(\eta_t^-)_{t\ge 0}$ (resp.\;$(\eta_t^+)_{t\ge 0}$) with $\{\xi^-(u)\}_{u\in \Z^2}$ (resp.\;$\{\xi^+(u)\}_{u\in \Z^2}$) so that $\bG^- \vee 0=L^-$ (resp.\;$\bG^+ \vee 0=L^+$) almost surely, and below we work under the event that this equality holds.
\item We use $I_t^-$, $\xi^{-,\vee}$, $\xi^{-,\vee,t}$, $T_{u,v}^-$, $\Gamma_{u,v}^-$, $\Gamma_u^-$, $\Gamma_u^{-,\vee}$ (resp.\;$I_t^+$, $\xi^{+,\vee}$, $\xi^{+,\vee,t}$, $T_{u,v}^+$, $\Gamma_{u,v}^+$, $\Gamma_u^+$, $\Gamma_u^{+,\vee}$) to denote the objects $I_t$ (growth process), $\xi^{\vee}$ (dual weights), $\xi^{\vee,t}$ (dual weights from $I_t$), $T_{u,v}$ (passage time), $\Gamma_{u,v}$ (finite geodesic), $\Gamma_u$ (semi-infinite geodesic), $\Gamma_u^{\vee}$ (downward semi-infinite geodesic) defined in the introduction, and Section \ref{ssec:buseman}, Section \ref{ssec:gianddgeo}, under the weights $\xi^-$ (resp.\;$\xi^+$). 
We shall also work under the almost sure event that these downward semi-infinite geodesics (under these weights) exist and are uniqueness and coalescence.
In addition, for any $t\in\R$ we let
\[
\begin{split}
& \partial I_t^+:=\{u\in I_t^+:\bG^+(u+(1,0))\vee \bG^+(u+(0,1))> t\},\\
& \partial I_t^-:=\{u\in I_t^-:\bG^-(u+(1,0))\vee \bG^-(u+(0,1))> t\}.
\end{split}
\]
    \item For $t\ge 0$,
$p_t^-=(a_t^-, b_t^-)$ (resp.\;$p_t^+=(a_t^+, b_t^+)$) is the last vertex in $\Gamma_\boo^-\cap I_t^-$ (resp.\;$\Gamma_\boo^+\cap I_t^+$), by Lemma \ref{lem:compete-inter-2nd-class}.
\end{itemize}

We first describe the coupling between $\heta_0^-$ and $\heta_s^+$.
Take $r\in \N$. For any coupling between $\heta_0^-$ and $\heta_s^+$, we denote $\cA$ as the event where
\[
\heta_0^-(x)=\heta_s^+(x), \forall x\in\Z, |x|> r;\quad \sum_{x=-r}^r \heta_0^-(x)=\sum_{x=-r}^r\heta_s^+(x).
\]
By Lemma \ref{lem:Ideteta}, under $\cA$ we can find a (unique) $p^*\in\Z^2$, such that $I_0^-\cap \{u\in\Z^2:|ad(u)|>r\} = (I_s^+-p^*)\cap \{u\in\Z^2:|ad(u)|>r\}$, and $ad(p^*)=ad(p_s^+)$.
In particular, this implies that 
\begin{equation} \label{eq:soutsi}
I_0^- \setminus \{u\in\Z^2: |ad(u)|, |d(u)| \le r\}
=(I_s^+-p^*) \setminus \{u\in\Z^2: |ad(u)|, |d(u)| \le r\}.    
\end{equation}
\begin{lemma}  \label{lem:couple-s}
There is a coupling of $\heta_0^-$ and $\heta_s^+$ such that $\PP[\cA]>1-C(rs^{-2/3})^{-1/10}$ when $Cs^{2/3}<r<s^{2/3+0.01}$ and  $s>C$, where $C>0$ is a constant.
\end{lemma}
We leave the construction of this coupling to the next subsection, and proceed to couple $(\heta_t^-)_{t\ge 0}$ and $(\heta_{s+t}^+)_{t\ge 0}$.
The idea is actually straightforward: we just couple the exponential waiting times. Namely, we note that for any $(a,b)\not \in I_0^-$, $\xi^{-,\vee,0}(a,b)$ is precisely the waiting time for the hole labeled $a$ to switch with the particle labeled $b$ since they are next to each other; and conditioned on $I_0^-$, $\{\xi^{-,\vee,0}(u)\}_{u\in \Z^2\setminus I_0^-}$ are i.i.d.\;$\Exp(1)$ (see also Lemma \ref{lem:iidexpI}). The same is true for $\{\xi^{+,\vee,s}(u)\}_{u\in \Z^2\setminus I_s^+}$ conditioned on $I_s^+$.
So we just couple these two sets of waiting times (as much as possible), up to a translation by $p^*$.

We note that, for $(\heta_t^-)_{t\ge 0}$ and $(\heta_{s+t}^+)_{t\ge 0}$, we need to couple them conditioned on $\heta_0^-$ and $\heta_s^+$.
We next show that, under $\cA$, $\heta_0^-$ and $\heta_s^+$ and $p^*$ contain precisely the same information as $I_0^-$ and $I_s^+$. (Then conditioned on $\heta_0^-$ and $\heta_s^+$ and $p^*$, the waiting times $\{\xi^{-,\vee,0}(u)\}_{u\in \Z^2\setminus I_0^-}$ and $\{\xi^{+,\vee,s}(u)\}_{u\in \Z^2\setminus I_s^+}$ are i.i.d.\;$\Exp(1)$.)
Indeed, $\heta_0^-$ is just $\eta_0^-$, which determines $I_0^-$ according to Lemma \ref{lem:Ideteta}.
Using Lemma \ref{lem:Ideteta} we also get that $\heta_s^+$ determines $I_s^+$, up to a translation of $\Z^2$; and the translation can be uniquely determined using $p^*$ and the fact that $I_0^-$ and $I_s^+-p^*$ are the same outside a compact set.
In the other direction, given $I_s^+$ and $I_0^-$ we can uniquely find $p^*$, and (by Lemma \ref{lem:Ideteta}) $\eta_0^-=\heta_0^-$ is determined by $I_0^-$, and $\eta_s^+$ is determined by $I_s^+$. To get $\heta_s^+$ we just shift $\eta_s^+$ by $ad(p^*)$.

We now couple $(\heta_t^-)_{t\ge 0}$ and $(\heta_{s+t}^+)_{t\ge 0}$. We let $\heta_0^-$ and $\heta_s^+$ be coupled using the coupling from Lemma \ref{lem:couple-s}. If $\cA$ does not hold, we just couple $(\heta_t^-)_{t\ge 0}$ and $(\heta_{s+t}^+)_{t\ge 0}$ (conditioned on $\heta_0^-$ and $\heta_s^+$) arbitrarily. 
If $\cA$ holds, we couple $(\heta_t^-)_{t\ge 0}$ and $(\heta_{s+t}^+)_{t\ge 0}$ (conditioned on $\heta_0^-$ and $\heta_s^+$ and $p^*$) so that the event
\begin{equation}  \label{eq:couplexi}
\xi^{-,\vee,0}(u)=\xi^{+,\vee,s}(u+p^*),\quad \forall u\in \Z^2\setminus (I_0^-\cup (I_s^+-p^*))
\end{equation}
holds with probability $1$. Below we work under this event whenever $\cA$ holds.

We bound the total variation distance between the $\{0,1,*\}^{\llbracket -N ,N\rrbracket}$ marginals of $\Phi_t$ and $\Phi_{t+s}$ (thus prove Proposition \ref{prop:phi-close}), by bounding the probability that $\heta_t^-$ and $\heta_{t+s}^+$ are different in $\llbracket -N ,N+1\rrbracket$ under this coupling. 

In the LPP setting, this is to show that with high probability, for any $u$ around $p^-_t$ we have $L^-(u)=L^+(u+p^*)-s$.
By Lemma \ref{lem:geo-part-cons}, this is implied by that, for such $u$ the paths $\Gamma^{-,\vee}_u\setminus I_0^-$ and  $\Gamma^{+,\vee}_{u+p^*}\setminus I_s^- - p^*$ are the same, and is contained in the area $\Z^2\setminus (I_0^-\cup (I_s^+-p^*))$ where the weights are coupled.
Using the non-crossing property (Lemma \ref{lem:non-cross}), this is ensured by coalescence of (upward) semi-finite geodesics.
More precisely we consider the following events (which are also illustrated in Figure \ref{fig:lem:phi-close-event}).

Take $m, r\in\N$ with $r<m$.
\begin{itemize}
    \item Let $\cB_-$ be the event where 
\begin{multline*}
\exists u_{-,1}, u_{-,2} \in \partial I_0^-,\quad ad(u_{-,1})<-r, ad(u_{-,2})>r, \; d(u_{-,1}), d(u_{-,2}) < 2m,\\  \Gamma_{u_{-,1}}^-\cap \LL_m = \Gamma_{u_{-,2}}^-\cap \LL_m.    
\end{multline*}
\item Let $\cB_+$ be the event where $\cA$ happens (with the same $r$), and in addition
\begin{multline*}
\exists u_{+,1}, u_{+,2} \in \partial I_s^+,\quad ad(u_{+,1}-p^*)<-r,\; ad(u_{+,2}-p^*)>r,\; d(u_{+,1}-p*), d(u_{+,2}-p*) < 2m,\\ (\Gamma_{u_{+,1}}^- - p^*)\cap \LL_m = (\Gamma_{u_{+,2}}^- - p^*)\cap \LL_m.
\end{multline*}
\item For any $t>0$ we let $\cF_t$ be the event where $d(p_t^-)>2m+2N+2$. 
\end{itemize}

\begin{lemma}  \label{lem:phi-close-event-0}
For any $t>0$, under $\cB_-\cap\cB_+\cap\cF_t$ we have that $\heta_t^-$ equals $\heta_{t+s}^+$ in $\llbracket -N, N+1\rrbracket$.
\end{lemma}
To prove this, we mainly need to establish the following result.
\begin{lemma}  \label{lem:phi-close-event}
Under $\cB_-\cap\cB_+$, we have 
\begin{enumerate}
    \item $L^-(u)=L^+(u+p^*)-s$ for any $u\in\Z^2$ with $d(u)>2m$, $u\not\in I_0^-$, $u+p^*\not\in I_s^+$.
    \item $p_t^-=p_{t+s}^+-p^*$ for any $t>0$ with $d(p_t^-)\ge 2m$.
\end{enumerate}
\end{lemma}
\begin{figure}[hbt!]
    \centering
\begin{tikzpicture}[line cap=round,line join=round,>=triangle 45,x=4cm,y=4cm]
\clip(-1.1,-0.6) rectangle (2.4,2.3);

\fill[line width=0.pt,color=red,fill=red,fill opacity=0.15]
(-0.9,-0.6) -- (1.55,-0.6) -- (1.55,-0.35) -- (1.35,-0.35) -- (1.35,-0.15) -- (1.05,-0.15) -- (1.05,-0.05) -- (0.95,-0.05) -- (0.95,0.15) -- (0.75,0.15) -- (0.75,0.25) -- (0.65,0.25) -- (0.65,0.35) -- (0.45,0.35) -- (0.45,0.45) -- (0.15,0.45) --(0.25,0.45) --(0.25,0.55) -- (0.15,0.55) -- (0.15,0.85) -- (-0.05,0.85) -- (-0.05,0.95) -- (-0.15,0.95) -- (-0.15,1.15) -- (-0.55,1.15) -- (-0.55,1.25) -- (-0.75,1.25) -- (-0.75,1.45) -- (-0.9,1.45) -- cycle;
\fill[line width=0.pt,color=blue,fill=blue,fill opacity=0.1]
(-0.9,-0.6) -- (1.55,-0.6) -- (1.55,-0.35) -- (1.35,-0.35) -- (1.35,-0.15) -- (1.05,-0.15) -- (1.05,-0.05) -- (0.95,-0.05) -- (0.95,0.15) -- (0.55,0.15) -- (0.55,0.45) --(0.45,0.45) -- (0.45,0.55) -- (0.35,0.55) -- (0.35,0.65) -- (0.15,0.65) -- (0.15,0.85) -- (-0.05,0.85) -- (-0.05,0.95) -- (-0.15,0.95) -- (-0.15,1.15) -- (-0.55,1.15) -- (-0.55,1.25) -- (-0.75,1.25) -- (-0.75,1.45) -- (-0.9,1.45) -- cycle;

\draw [line width=.1pt, opacity=0.3] (-0.9,-0.5) -- (2.6,-0.5);
\draw [line width=.1pt, opacity=0.3] (-0.9,-0.4) -- (2.6,-0.4);
\draw [line width=.1pt, opacity=0.3] (-0.9,-0.3) -- (2.6,-0.3);
\draw [line width=.1pt, opacity=0.3] (-0.9,-0.2) -- (2.6,-0.2);
\draw [line width=.1pt, opacity=0.3] (-0.9,-0.1) -- (2.6,-0.1);
\draw [line width=.1pt, opacity=0.3] (-0.9,0.) -- (2.6,0.);
\draw [line width=.1pt, opacity=0.3] (-0.9,0.1) -- (2.6,0.1);
\draw [line width=.1pt, opacity=0.3] (-0.9,0.2) -- (2.6,0.2);
\draw [line width=.1pt, opacity=0.3] (-0.9,0.3) -- (2.6,0.3);
\draw [line width=.1pt, opacity=0.3] (-0.9,0.4) -- (2.6,0.4);
\draw [line width=.1pt, opacity=0.3] (-0.9,0.5) -- (2.6,0.5);
\draw [line width=.1pt, opacity=0.3] (-0.9,0.6) -- (2.6,0.6);
\draw [line width=.1pt, opacity=0.3] (-0.9,0.7) -- (2.6,0.7);
\draw [line width=.1pt, opacity=0.3] (-0.9,0.8) -- (2.6,0.8);
\draw [line width=.1pt, opacity=0.3] (-0.9,0.9) -- (2.6,0.9);
\draw [line width=.1pt, opacity=0.3] (-0.9,1.) -- (2.6,1.);
\draw [line width=.1pt, opacity=0.3] (-0.9,1.1) -- (2.6,1.1);
\draw [line width=.1pt, opacity=0.3] (-0.9,1.2) -- (2.6,1.2);
\draw [line width=.1pt, opacity=0.3] (-0.9,1.3) -- (2.6,1.3);
\draw [line width=.1pt, opacity=0.3] (-0.9,1.4) -- (2.6,1.4);
\draw [line width=.1pt, opacity=0.3] (-0.9,1.5) -- (2.6,1.5);
\draw [line width=.1pt, opacity=0.3] (-0.9,1.6) -- (2.6,1.6);
\draw [line width=.1pt, opacity=0.3] (-0.9,1.7) -- (2.6,1.7);
\draw [line width=.1pt, opacity=0.3] (-0.9,1.8) -- (2.6,1.8);
\draw [line width=.1pt, opacity=0.3] (-0.9,1.9) -- (2.6,1.9);
\draw [line width=.1pt, opacity=0.3] (-0.9,2.) -- (2.6,2.);
\draw [line width=.1pt, opacity=0.3] (-0.9,2.1) -- (2.6,2.1);

\draw [line width=.1pt, opacity=0.3] (-0.8,-0.6) -- (-0.8,2.2);
\draw [line width=.1pt, opacity=0.3] (-0.7,-0.6) -- (-0.7,2.2);
\draw [line width=.1pt, opacity=0.3] (-0.6,-0.6) -- (-0.6,2.2);
\draw [line width=.1pt, opacity=0.3] (-0.5,-0.6) -- (-0.5,2.2);
\draw [line width=.1pt, opacity=0.3] (-0.4,-0.6) -- (-0.4,2.2);
\draw [line width=.1pt, opacity=0.3] (-0.3,-0.6) -- (-0.3,2.2);
\draw [line width=.1pt, opacity=0.3] (-0.2,-0.6) -- (-0.2,2.2);
\draw [line width=.1pt, opacity=0.3] (-0.1,-0.6) -- (-0.1,2.2);
\draw [line width=.1pt, opacity=0.3] (0.,-0.6) -- (0.,2.2);
\draw [line width=.1pt, opacity=0.3] (0.1,-0.6) -- (0.1,2.2);
\draw [line width=.1pt, opacity=0.3] (0.2,-0.6) -- (0.2,2.2);
\draw [line width=.1pt, opacity=0.3] (0.3,-0.6) -- (0.3,2.2);
\draw [line width=.1pt, opacity=0.3] (0.4,-0.6) -- (0.4,2.2);
\draw [line width=.1pt, opacity=0.3] (0.5,-0.6) -- (0.5,2.2);
\draw [line width=.1pt, opacity=0.3] (0.6,-0.6) -- (0.6,2.2);
\draw [line width=.1pt, opacity=0.3] (0.7,-0.6) -- (0.7,2.2);
\draw [line width=.1pt, opacity=0.3] (0.8,-0.6) -- (0.8,2.2);
\draw [line width=.1pt, opacity=0.3] (0.9,-0.6) -- (0.9,2.2);
\draw [line width=.1pt, opacity=0.3] (1.,-0.6) -- (1.,2.2);
\draw [line width=.1pt, opacity=0.3] (1.1,-0.6) -- (1.1,2.2);
\draw [line width=.1pt, opacity=0.3] (1.2,-0.6) -- (1.2,2.2);
\draw [line width=.1pt, opacity=0.3] (1.3,-0.6) -- (1.3,2.2);
\draw [line width=.1pt, opacity=0.3] (1.4,-0.6) -- (1.4,2.2);
\draw [line width=.1pt, opacity=0.3] (1.5,-0.6) -- (1.5,2.2);
\draw [line width=.1pt, opacity=0.3] (1.6,-0.6) -- (1.6,2.2);
\draw [line width=.1pt, opacity=0.3] (1.7,-0.6) -- (1.7,2.2);
\draw [line width=.1pt, opacity=0.3] (1.8,-0.6) -- (1.8,2.2);
\draw [line width=.1pt, opacity=0.3] (1.9,-0.6) -- (1.9,2.2);
\draw [line width=.1pt, opacity=0.3] (2.,-0.6) -- (2.,2.2);
\draw [line width=.1pt, opacity=0.3] (2.1,-0.6) -- (2.1,2.2);
\draw [line width=.1pt, opacity=0.3] (2.2,-0.6) -- (2.2,2.2);
\draw [line width=.1pt, opacity=0.3] (2.3,-0.6) -- (2.3,2.2);

\draw [line width=.3pt, color=red] (2.4,1.3) -- (1.5,2.2);

\draw [red] plot coordinates {(0.2,0.5) (0.3,0.5) (0.4,0.5) (0.4,0.6) (0.6,0.6) (0.6,0.7) (0.9,0.7) (0.9,0.9) (1.,0.9) (1.2,0.9) (1.2,1.) (1.3,1.) (1.3,1.3) (1.4,1.3) (1.4,1.5) (1.5,1.5) (1.5,1.7) (1.9,1.7) (1.9,1.8) (2.1,1.8) (2.1,2.) (2.3,2.) (2.3,2.1)  (2.4,2.1)};

\draw [red] plot coordinates {(0.1,0.8) (0.4,0.8) (0.4,0.9) (0.7,0.9) (0.7,1.0) (0.9,1.0) (0.9,1.1) (1.3,1.1) (1.3,1.3) (1.4,1.3) (1.4,1.5) (1.5,1.5) (1.5,1.7) (1.9,1.7) (1.9,1.8) (2.1,1.8) (2.1,2.) (2.3,2.) (2.3,2.1)  (2.4,2.1)};

\draw [red] plot coordinates {(0.9,0.1) (0.9,0.5) (1.2,0.5) (1.2,1.) (1.3,1.) (1.3,1.3) (1.4,1.3) (1.4,1.5) (1.5,1.5) (1.5,1.7) (1.9,1.7) (1.9,1.8) (2.1,1.8) (2.1,2.) (2.3,2.) (2.3,2.1)  (2.4,2.1)};

\draw [blue] plot coordinates {(-0.7,-0.4) (-0.7,-0.2) (-0.5,-0.2) (-0.5,-0.1) (-0.2,-0.1) (-0.2,0.) (-0.1,0.) (-0.1,0.2) (0.2,0.2) (0.2,0.3) (0.3,0.3) (0.3,0.6) (0.8,0.6) (0.8,0.8) (0.9,0.8) (1.1,0.8) (1.1,1.2) (1.2,1.2) (1.2,1.3) (1.3,1.3) (1.4,1.3) (1.4,1.5) (1.5,1.5) (1.5,1.7) (1.9,1.7) (1.9,1.8) (2.1,1.8) (2.1,2.) (2.3,2.) (2.3,2.1)  (2.4,2.1)};

\draw [blue] plot coordinates {(0.8,0.1) (1.0,0.1) (1.0,0.6) (1.1,0.6) (1.1,0.7) (1.3,0.7) (1.3,0.8) (1.4,0.8) (1.4,1.5) (1.5,1.5) (1.5,1.7) (1.9,1.7) (1.9,1.8) (2.1,1.8) (2.1,2.) (2.3,2.) (2.3,2.1)  (2.4,2.1)};

\draw [blue] plot coordinates {(-0.1,0.9) (0.3,0.9) (0.3,1.0) (0.5,1.0) (0.5,1.2) (1.,1.2) (1.,1.3) (1.1,1.3) (1.1,1.4) (1.4,1.4) (1.4,1.5) (1.5,1.5) (1.5,1.7) (1.9,1.7) (1.9,1.8) (2.1,1.8) (2.1,2.) (2.3,2.) (2.3,2.1)  (2.4,2.1)};

\draw [purple] plot coordinates {(1.3,1.3) (1.4,1.3) (1.4,1.5) (1.5,1.5) (1.5,1.7) (1.9,1.7) (1.9,1.8) (2.1,1.8) (2.1,2.) (2.3,2.) (2.3,2.1)  (2.4,2.1)};
\draw [purple] plot coordinates {(1.3,1.3) (1.4,1.3) (1.4,1.5) (1.5,1.5) (1.5,1.7) (1.9,1.7) (1.9,1.8) (2.1,1.8) (2.1,2.) (2.3,2.) (2.3,2.1)  (2.4,2.1)};
\draw [purple] plot coordinates {(1.3,1.3) (1.4,1.3) (1.4,1.5) (1.5,1.5) (1.5,1.7) (1.9,1.7) (1.9,1.8) (2.1,1.8) (2.1,2.) (2.3,2.) (2.3,2.1)  (2.4,2.1)};

\draw [fill=blue,color=blue] (-0.7,-0.4) circle (1.0pt);
\draw [fill=blue,color=blue] (0.8,0.1) circle (1.0pt);
\draw [fill=blue,color=blue] (-0.1,0.9) circle (1.0pt);
\draw [fill=red,color=red] (0.2,0.5) circle (1.0pt);
\draw [fill=red,color=red] (0.9,0.1) circle (1.0pt);
\draw [fill=red,color=red] (0.1,0.8) circle (1.0pt);
\draw (-0.7,-0.4) node[anchor=east,color=blue]{$\boo$};
\draw (0.8,0.1) node[anchor=east,color=blue]{$u_{+,2}$};
\draw (-0.1,0.9) node[anchor=east,color=blue]{$u_{+,1}$};
\draw (0.2,0.5) node[anchor=east,color=red]{$\boo$};
\draw (0.1,0.8)
node[anchor=east,color=red]{$u_{-,1}$};
\draw (0.9,0.1) node[anchor=north,color=red]{$u_{-,2}$};
\draw (2.3,1.38) node[anchor=north,color=red]{$\LL_m$};

\draw (-0.7,1.0) node[anchor=east,color=blue]{$I_s^+$};
\draw (-0.55,1.0) node[anchor=east,color=red]{$I_0^-$};

\end{tikzpicture}
\caption{An illustration of the events $\cB_-$ and $\cB_+$, translated and superposed together.
The red objects are for $\cB_-$, and are constructed from $({\eta_t^-})_{t\ge 0}$; and the blue objects are for $\cB_+$ and are constructed from $({\eta_t^+})_{t\ge 0}$. The difference between the red $\boo$ and blue $\boo$ is $p^*$.
}  
\label{fig:lem:phi-close-event}
\end{figure}
\begin{proof}
Since $r<m$, under $\cA$ we have $\{u\in\Z^2:d(u)>2m\} \cap I_0^- = \{u\in\Z^2:d(u)>2m\} \cap (I_s^+-p^*)$ by \eqref{eq:soutsi}.
Denote $U=\Z^2 \setminus ( I_0^- \cup ( I_s^+ -p^*))$.

We first show that, under $\cA\cap\cB_-$, we must have $\Gamma_u^{-,\vee}\setminus I_0^- \subset U$ for any $u\in U$ with $d(u)>2m$.
Indeed, by the non-crossing property (Lemma \ref{lem:non-cross}), the path $\Gamma_{u_{-,1}}^-+(1/2,1/2)$ divides $u_{-,1}+(\Z^2\setminus\Z_{\le 0}^2)$ into two parts $P_{1,\downarrow}$ (the lower-right part) and $P_{1,\uparrow}$ (the upper-left part) such that $\Gamma_u^{-,\vee}$ intersects at most one of them, and the path $\Gamma_{u_{-,2}}^-+(1/2,1/2)$ divides $u_{-,2}+(\Z^2\setminus\Z_{\le 0}^2)$ into two parts $P_{2,\downarrow}$ (the lower-right part) and $P_{2,\uparrow}$ (the upper-left part) such that $\Gamma_u^{-,\vee}$ intersects at most one of them.
Since $\Gamma_{u_{-,1}}^-\cap \LL_m = \Gamma_{u_{-,2}}^-\cap \LL_m$, we must have
\[\{u:d(u)>2m\}\cap P_{1,\downarrow}=\{u:d(u)>2m\}\cap P_{2,\downarrow},\; \{u:d(u)>2m\}\cap P_{1,\uparrow}=\{u:d(u)>2m\}\cap P_{2,\uparrow}.\]
For any $u\in U$ with $d(u)>2m$, depending on whether $u\in P_{1,\uparrow}, P_{2,\uparrow}$ or $u\in P_{1,\downarrow}, P_{2,\downarrow}$, we must have that $\Gamma_u^{-,\vee}\setminus I_0^- \subset P_{1,\uparrow}$ or $\Gamma_u^{-,\vee}\setminus I_0^- \subset P_{2,\downarrow}$.
Denote the lower endpoint of $\Gamma_u^{-,\vee}\setminus I_0^-$ as $v_0$.
By Lemma \ref{lem:geo-part-cons}, we have $v_0-(0,1), v_0-(1,0) \in I_0^-$, so $v_0\in u_{-,1}+\Z_{\le 0}\times \Z_+$ (if $v_0\in P_{1,\uparrow}$) or $v_0\in u_{-,2}+\Z_+\times \Z_{\le 0}$ (if $v_0\in P_{2,\downarrow}$).
Thus we have $v_0\not\in I_s^+-p^*$, by \eqref{eq:soutsi} and the fact that $ad(u_{-,1})<-r$, $ad(u_{-,2})>r$.
So $v_0\in U$, which implies that $\Gamma_u^{-,\vee}\setminus I_0^- \subset U$.

Similarly, under $\cB_+$ we have that $\Gamma_{u+p^*}^{+,\vee}\setminus I_s^+ -p^*\subset U$ for any $u\in U$ with $d(u)>2m$.

According to \eqref{eq:couplexi} we have $\xi^{-,\vee,0}(u)=\xi^{+,\vee,s}(u+p^*)$ for any $u\in U$.
Then by Lemma \ref{lem:geo-part-cons}, for any $u\in U$ with $d(u)>2m$, since we have shown that $\Gamma_u^{-,\vee}\setminus I_0^- \subset U$ and $\Gamma_{u+p^*}^{+,\vee} \setminus I_s^+ -p^*\subset U$, we must have that 
\begin{equation} \label{eq:gammaeq1}
\Gamma_u^{-,\vee}\setminus I_0^- = \Gamma_{u+p^*}^{+,\vee}\setminus I_s^+-p^*,    
\end{equation}
and $L^-(u)=L^+(u+p^*)-s$.
Thus the first statement holds.

We next prove the second statement. Below we always assume $\cB_-\cap\cB_+$.
Using that $p_t^-$ is the last vertex in $\Gamma_\boo^-\cap I_t^-$ and $p_{t+s}^+$ is the last vertex in $\Gamma_\boo^+\cap I_{t+s}^+$, and the first statement,
it suffices to show that 
\begin{equation}  \label{eq:phi-close-event2}
\Gamma_\boo^-\cap \{u\in\Z^2:d(u)>2m\}+p^*= \Gamma_\boo^+\cap \{u\in\Z^2:d(u)>2m+d(p^*)\}.    
\end{equation}
By the non-crossing property (Lemma \ref{lem:non-cross}), $\Gamma_\boo^-\cap\{u\in\Z^2:d(u)>2m\}$ is determined by $\Gamma_u^{-,\vee}$ for all $u\in U$ with $d(u)\ge 2m$.
More precisely, for all $u\in U$ with $d(u)>2m$, we divide them into two parts, depending on whether the lower endpoint of $\Gamma_u^{-,\vee}\setminus I_0^-$ is in $\Z_{\le 0}\times \Z_+$ or $\Z_+\times\Z_{\le 0}$, and $\Gamma_\boo^-\cap\{u\in\Z^2:d(u)\ge 2m\}+(1/2,1/2)$ is the boundary of these two parts.
Similarly, we can also divide all $u\in U$ with $d(u)>2m$ into two parts, depending on whether the lower endpoint of $\Gamma_{u+p^*}^{+,\vee}\setminus I_s^+$ is in $p_s^++\Z_{\le 0}\times \Z_+$ or $p_s^++\Z_+\times\Z_{\le 0}$, and $(\Gamma_\boo^+-p^*)\cap\{u\in\Z^2:d(u)\ge 2m\}+(1/2,1/2)$ is the boundary of these two parts.

To prove \eqref{eq:phi-close-event2}, it then remains to show that, for any $u\in U$ with $d(u)>2m$, 
\begin{equation}  \label{eq:path-d-equiv}
\begin{split}
&\text{the lower endpoint of } \Gamma_u^{-,\vee}\setminus I_0^- \text{ is in }  \Z_{\le 0}\times \Z_+ \\ &\text{if and only if  } \\ &\text{the lower endpoint of }  \Gamma_{u+p^*}^{+,\vee}\setminus I_s^+ \text{ is in }  p_s^++\Z_{\le 0}\times \Z_+.
\end{split}
\end{equation}
For this, we denote the lower endpoint of $\Gamma_u^{-,\vee}\setminus I_0^-$ as $v_0$.
By \eqref{eq:gammaeq1}, $v_0+p^*$ is the lower endpoint of $\Gamma_{u+p^*}^{+,\vee}\setminus I_s^+$.
Recall that we have $ad(p^+_s-p^*)=0$, so $v_0$ is either to in both $\Z_{\le 0}\times \Z_+$ and $p^+_s-p^*+\Z_{\le 0}\times \Z_+$, or in both $\Z_+\times\Z_{\le 0}$ and $p^+_s-p^*+\Z_+\times\Z_{\le 0}$.
Thus we get \eqref{eq:path-d-equiv}, which implies \eqref{eq:phi-close-event2} and that the conclusion follows.
\end{proof}

\begin{proof}[Proof of Lemma \ref{lem:phi-close-event-0}]
The event that $\heta_t^-$ equals $\heta_{t+s}^+$ in $\llbracket -N, N+1\rrbracket$ can be written as
\[
\{\eta_t^- (x+ad(p_t^-))\}_{x\in \llbracket -N, N+1\rrbracket} = \{\eta_{t+s}^+(x+ad(p_{t+s}^+))\}_{x\in \llbracket -N, N+1\rrbracket}.
\]
By Lemma \ref{lem:Ideteta}, this is implied by
\begin{equation} \label{eq:phi-close-equa}
(I_t^- - p_t^-) \cap \llbracket -N-1, N+1\rrbracket^2 = (I_{t+s}^+ - p_{t+s}^+) \cap \llbracket -N-1, N+1\rrbracket^2.    
\end{equation}
Below we assume $\cB_-\cap\cB_+\cap\cF_t$. For any $u\in \llbracket -N-1, N+1\rrbracket^2 + p_t^-$, we have $d(u)>2m$ by $\cF_t$.
Since under $\cA$ the sets $I_0^-$ and $I_s^+-p^*$ are the same outside $\{u\in\Z^2: |ad(u)|, |d(u)| \le r\}$ (as stated by \eqref{eq:soutsi}), and $m>r$, we have either $u\in I_0^-\cap(I_s^+-p^*)$ or $u\not\in I_0^-\cup(I_s^+-p^*)$.
In the later case we have $L^-(u)=L^+(u+p^*)-s$, by the first statement of Lemma \ref{lem:phi-close-event}.
Thus we always has that either $u\in I_t^-\cap(I_{t+s}^+-p^*)$, or $u\not\in I_t^-\cup(I_{t+s}^+-p^*)$.
Thus we have
\[
(\llbracket -N-1, N+1\rrbracket^2 + p_t^- ) \cap I_t^- = (\llbracket -N-1, N+1\rrbracket^2 + p_t^- ) \cap (I_{t+s}^+-p^*).
\]
By the second statement of Lemma \ref{lem:phi-close-event} we have $p_t^-=p_{t+s}^+-p^*$, so we get \eqref{eq:phi-close-equa}, and the conclusion follows.
\end{proof}

To prove Proposition \ref{prop:phi-close}, it remains to lower bound $\PP[\cB_-]$, $\PP[\cB_+]$, and $\PP[\cF_t]$.
For this we set up some additional notations.
Recall that $\brho=((1-\rho)^2,\rho^2)$. As in Section \ref{ssec:elpps} (but omitting $\rho$ from the notation), we denote
\[\HH_x:=\{x\brho+y((1-\rho),-\rho):y\in\R\},\]
for any $x\in\R$.
We also denote
\[\VV_x:=\{(x,-x)+y\brho:y\in\R\},\]
and for any set $A\subset \R$ we denote $\VV_A:=\cup_{x\in A}\VV_x$ and $\HH_A:=\cup_{x\in A}\HH_x$.

\begin{figure}[hbt!]
     \centering
     \begin{subfigure}[t]{0.4\textwidth}
         \centering
\begin{tikzpicture}[line cap=round,line join=round,>=triangle 45,x=0.05cm,y=0.05cm]
\clip(0,0) rectangle (135,150);

\fill[line width=0.pt,color=yellow,fill=yellow,fill opacity=0.45]
(40,50) --(50,40) --(40,30) -- (30,40) -- cycle;

\fill[line width=0.pt,color=yellow,fill=yellow,fill opacity=0.45]
(0,70) -- (40,70) --(40,300) -- (0,300) -- cycle;
\fill[line width=0.pt,color=yellow,fill=yellow,fill opacity=0.45]
(90,0) -- (90,40) --(300,40) -- (300,0) -- cycle;

\fill[line width=0.pt,color=blue,fill=blue,fill opacity=0.25]
(0,70) --(5,68) --(10,62) --(15,58) --(20,57) --(25,52) --(30,48) --(35,41) --(40,40) --(45,37) --(50,29) --(55,22) --(60,19) --(65,16) --(70,14) --(75,10) --(80,5) --(85,2) --(90,0) -- (0,0) -- cycle;

\draw [line width=.4pt] (220,0) -- (0,220);
\draw [red] plot [smooth] coordinates {(10,40) (25,52) (55,62) (85,80) (100,82) (115,88) (145,112) };
\draw [red] plot [smooth] coordinates {(40,10) (55,22) (70,28) (95,65) (100,80) (115,88) (145,112) };

\draw [fill=uuuuuu] (40,40) circle (1.0pt);
\draw [fill=uuuuuu] (10,40) circle (1.0pt);
\draw [fill=uuuuuu] (40,10) circle (1.0pt);
\draw [fill=uuuuuu] (25,52) circle (1.0pt);
\draw [fill=uuuuuu] (55,22) circle (1.0pt);

\draw [fill=uuuuuu] (10,40) circle (1.0pt);
\draw [fill=uuuuuu] (40,10) circle (1.0pt);

\begin{scriptsize}

\draw (40,40) node[anchor=north]{$\boo$};
\draw (10,40) node[anchor=north]{$v_1$};
\draw (40,10) node[anchor=north]{$v_2$};
\draw (25,52) node[anchor=south]{$u_{-,1}$};
\draw (55,22) node[anchor=south]{$u_{-,2}$};
\draw (55,8) node[anchor=west]{$I_0^-$};
\draw (70,150) node[anchor=north]{$\LL_m$};
\draw (55,62) node[anchor=south]{$\Gamma_{v_1}^-$};
\draw (90,55) node[anchor=west]{$\Gamma_{v_2}^-$};
\end{scriptsize}

\end{tikzpicture}

         \caption{Estimating $\PP[\cB_-]$: by Corollary \ref{cor:trans-semi-inf}, with probability $>1-Ce^{-cr}$, the geodesics $\Gamma_{v_1}^-$ and $\Gamma_{v_2}^-$ are disjoint from the sets $\{u\in\Z^2: |d(u)|, |ad(u)|\le r\}$ and $\Z_{\le 0}\times \llbracket 2r(1-\rho)^{-4},\infty\rrbracket$ and $\llbracket 2r\rho^{-4},\infty\rrbracket \times \Z_{\le 0}$.}
         \label{fig:p61m}
     \end{subfigure}
     \hfill
     \begin{subfigure}[t]{0.55\textwidth}
         \centering
\begin{tikzpicture}[line cap=round,line join=round,>=triangle 45,x=0.05cm,y=0.05cm]
\clip(-40,-20) rectangle (135,150);

\fill[line width=0.pt,color=green,fill=green,fill opacity=0.5]
(43,50) --(53,43) --(40,33) -- (30,40) -- cycle;
\fill[line width=0.pt,color=green,fill=green,fill opacity=0.2]
(43,55) --(53,48) -- (58,43) -- (53,38) --(40,28) -- (30,35) -- (25,40) -- (30,45) -- cycle;

\fill[line width=0.pt,color=green,fill=green,fill opacity=0.2]
(-100,73) -- (40,73) -- (53,83) --(40,300) -- (-100,300) -- cycle;
\fill[line width=0.pt,color=green,fill=green,fill opacity=0.2]
(90,-100) -- (90,40) -- (103,50) --(300,50) -- (300,-100) -- cycle;

\fill[line width=0.pt,color=blue,fill=blue,fill opacity=0.25]
(-40,96) -- (-30,88) -- (-18,82) -- (-11,74) -- (0,70) --(5,68) --(10,62) --(15,58) --(20,57) --(25,55) --(30,48) --(35,41) --(40,40) --(45,37) --(50,29) --(55,22) --(60,19) --(65,16) --(70,14) --(75,10) --(80,5) --(85,2) --(90,0) -- (98,-5) -- (110,-12) -- (125,-20) -- (-40,-20) -- cycle;

\draw [line width=.4pt] (220,0) -- (0,220);
\draw [line width=.4pt] (160,0) -- (0,160);
\draw [red] plot [smooth] coordinates {(-30,20) (-10,34) (10,40) (25,55) (55,62) (77,71) (100,82) (115,92) (145,106) };
\draw [red] plot [smooth] coordinates {(0,20) (20,31) (40,40) (55,50) (62,63) (77,71) (100,82) (115,92) (145,106) };
\draw [red] plot [smooth] coordinates {(0,-10) (18,2) (40,10) (55,22) (70,51) (77,70) (100,82) (115,92) (145,106) };

\draw [fill=uuuuuu] (40,40) circle (1.0pt);
\draw [fill=uuuuuu] (25,55) circle (1.0pt);
\draw [fill=uuuuuu] (55,22) circle (1.0pt);
\draw [fill=uuuuuu] (-30,20) circle (1.0pt);
\draw [fill=uuuuuu] (0,-10) circle (1.0pt);
\draw [fill=uuuuuu] (0,20) circle (1.0pt);

\begin{scriptsize}
\draw (60,43) node[anchor=east]{$A'$};
\draw (51,43) node[anchor=east]{$A$};
\draw (40,40) node[anchor=north]{$p_s^+$};
\draw (0,20) node[anchor=north]{$\boo$};
\draw (-30,20) node[anchor=north]{$v_1$};
\draw (0,-10) node[anchor=north]{$v_2$};
\draw (25,55) node[anchor=south]{$u_{+,1}$};
\draw (55,22) node[anchor=south]{$u_{+,2}$};
\draw (55+30,8-20) node[anchor=west]{$I_s^+$};
\draw (60,150) node[anchor=north]{$\LL_m+p^*$};
\draw (10,150) node[anchor=north]{$\LL_m$};
\draw (50,60) node[anchor=south]{$\Gamma_{v_1}^+$};
\draw (70,55) node[anchor=west]{$\Gamma_{v_2}^+$};
\draw (20,30) node[anchor=north]{$\Gamma_{\boo}^+$};
\end{scriptsize}

\end{tikzpicture}

         \caption{Estimating $\PP[\cA\setminus\cB_+]$: by Corollary \ref{cor:trans-semi-inf} and Corollary \ref{cor:lem:semi-inf-loc}, with probability $>1-Ce^{-cr^3s^{-2}}$ we have $p_s^+\in A$, and the geodesics $\Gamma_{v_1}^+$ and $\Gamma_{v_2}^+$ are disjoint from the sets $A'$ and $A+\Z_{\le 0}\times \llbracket 2r(1-\rho)^{-4},\infty\rrbracket$ and $A+\llbracket 2r\rho^{-4},\infty\rrbracket \times \Z_{\le 0}$.
         }
         \label{fig:p61p}
     \end{subfigure}
        \caption{Illustrations of the proof of Proposition \ref{prop:phi-close}. The probabilities of the coalescence events are estimated using Proposition \ref{prop:coal-semi-inf-1}.}
        \label{fig:p61}
\end{figure}

\begin{proof}[Proof of Proposition \ref{prop:phi-close}]
In this proof we let $c,C>0$ be small and large constants which depend on $N$, and the values can change from line to line.

We will show that $\heta_t^-$ equals $\heta_{t+s}^+$ in $\llbracket -N, N+1\rrbracket$ with probability $>1-C(s/t)^{1/30}$, assuming that $t,s$ are large enough depending on $N$. We could also assume that $t/s$ is large enough depending on $N$, since otherwise we would have $1-C(s/t)^{1/30} < 0$.
For the parameters in the definition of the events $\cA$, $\cB_-$, $\cB_+$, and $\cF_t$, we take $m=\lfloor t/10 \rfloor$ and $r=\lfloor s^{1/3}t^{1/3}\rfloor$. 
Denote $v_{1}=\left(-\lfloor r\rho^{-2}(1-\rho)^{-2}\rfloor, 0\right)$ and $v_{2}=\left(0,-\lfloor r\rho^{-2}(1-\rho)^{-2}\rfloor\right)$.

We first lower bound $\PP[\cB_-]$.
We take $u_{-,1}$ to be the last vertex in $\Gamma_{v_{1}}^-\cap I_0^-$, and $u_{-,2}$ to be the last vertex in $\Gamma_{v_{2}}^-\cap I_0^-$.
Then $u_{-,1}\in\Z_{\le 0}\times \Z_{\ge 0}$, and $u_{-,2}\in\Z_{\ge 0}\times \Z_{\le 0}$.
By Corollary \ref{cor:trans-semi-inf}, we have 
\[\PP[ad(u_{-,1})<-r],\; \PP[ad(u_{-,2})>r]>1-Ce^{-cr},\]
and
\[
\PP[u_{-,1}\in \Z_{\le 0}\times \llbracket 0, 2r(1-\rho)^{-4} \rrbracket],\; \PP[u_{-,2}\in \llbracket 0, 2r\rho^{-4} \rrbracket \times \Z_{\le 0}] >1-Ce^{-cr}.
\]
See Figure \ref{fig:p61m} for an illustration.
Since $2r(1-\rho)^{-4}, 2r\rho^{-4} < 2m$ (as $t, s, t/s$ are large enough), we have $\PP[d(u_{-,1}), d(u_{-,2}) < 2m]>1-Ce^{-cr}$.
By Proposition \ref{prop:coal-semi-inf-1} we have $\PP[\Gamma_{v_{1}}^- \cap \LL_m = \Gamma_{v_{2}}^- \cap \LL_m]>1-Crm^{-2/3}$.
Thus we conclude that $\PP[\cB_-]>1-Crm^{-2/3}-Ce^{-cr}$.

We next bound $\PP[\cA\setminus\cB_+]$ (see Figure \ref{fig:p61p} for several sets in $\Z^2$ to be defined). We take $u_{+,1}$ to be the last vertex in $\Gamma_{v_{1}}^+\cap I_s^+$, and $u_{+,2}$ to be the last vertex in $\Gamma_{v_{2}}^+\cap I_s^+$.
Then by ordering of geodesics (Lemma \ref{l:ordering}), and that all of $p_s^+$, $u_{+,1}$, and $u_{+,2}$ are in $\partial I_s^+$, we must have that $u_{+,1} \in p_s^++\Z_{\le 0}\times\Z_{\ge 0}$ and $u_{+,2} \in p_s^++\Z_{\ge 0}\times\Z_{\le 0}$.

Let $A=\VV_{(-r,r)}\cap \HH_{(s-(r^3s^{-2})s^{1/3}, s+(r^3s^{-2})s^{1/3})}$.
Note that $p_s^+$ is the last vertex in $\{u\in \Gamma_\boo^+:\bG^+(u)\le s\}$, so by Lemma \ref{lem:buse-opti} we have that $p_s^+-(1,0)$ or $p_s^+-(0,1)$ is the last vertex in $\{u\in \Gamma_\boo^+:T_{\boo,u}^+\le s\}$.
Then by Corollary \ref{cor:lem:semi-inf-loc} we have $\PP[p_s^+\in A]>1-Ce^{-cr^3s^{-2}}$.

\begin{itemize}
    \item 
When $p_s^+\in A$, we must have $ad(u_{+,1})<ad(p_s^+)-r$ and $ad(u_{+,2})>ad(p_s^+)+r$, unless $u_{+,1}\in A'$ or $u_{+,2}\in A'$, where \[A'=A+\{u\in\R^2: |d(u)|, |ad(u)| \le r\} \subset \VV_{(-2r,2r)}\cap \HH_{(s-Cr, s+Cr)}.\]
By Corollary \ref{cor:trans-semi-inf}, we have 
$\PP[u_{+,1}\in A'], \PP[u_{+,2}\in A']<Ce^{-cr^3s^{-2}}$.
Recall that under $\cA$ we have $ad(p_s^+)=ad(p^*)$, we now conclude that
\[
\PP[\{ad(u_{+,1})>ad(p^*)-r\}\cap\cA],\; \PP[\{ad(u_{+,2})<ad(p^*)+r\}\cap\cA]<Ce^{-cr^3s^{-2}}.
\]
\item
When $p_s^+\in A$, we must also have $d(u_{+,1})<d(p_s^+)+2r(1-\rho)^{-4}$, unless $u_{+,1}\in A+\Z_{\le 0}\times \llbracket 2r(1-\rho)^{-4},\infty\rrbracket$.
By Corollary \ref{cor:trans-semi-inf} this happens with probability $<Ce^{-cr^3s^{-2}}$, so we have
\[ \PP[d(u_{+,1}-p_s^+)\ge 2r(1-\rho)^{-4}]<Ce^{-cr^3s^{-2}}.\] 
When $\cA$ holds, we can find some $w_1, w_2\in \Z^2$, $|ad(w_1)|, |ad(w_2)|\le r+1$, such that $w_1\in (\Z_-\times\Z_+) \cap (p_s^+-p^*+\Z_-\times\Z_+)$, and $w_2\in (\Z_+\times\Z_-) \cap (p_s^+-p^*+\Z_+\times\Z_-)$. This implies that $|d(p_s^+-p^*)|\le 2r$.
Thus we conclude that
\[
\PP[\{d(u_{+,1}-p^*)\ge 2m\}\cap \cA]\le  \PP[d(u_{+,1}-p_s^+)\ge 2r(1-\rho)^{-4}] <  Ce^{-cr^3s^{-2}}.
\]
Here the first inequality is by $2m-2r\ge 2r(1-\rho)^{-4}$ (due to taking $s,t,t/s$ large).
Similarly we have $\PP[\{d(u_{+,2}-p^*)\ge 2m\}\cap \cA] <  Ce^{-cr^3s^{-2}}$.
\item
We have shown that $\cA$ implies $|d(p_s^+-p^*)|\le 2r$. If in addition $p_s^+\in A$, we must have that $d(p^*)\ge d(p_s^+)-2r>0$ (since $s, t, t/s$ are taken to be large). So using Proposition \ref{prop:coal-semi-inf-1} we have 
\begin{multline*}
\PP[\{(\Gamma_{v_{1}}^+-p^*) \cap \LL_m \neq (\Gamma_{v_{2}}^+-p^*) \cap \LL_m\}\cap \cA]
\le
\PP[\Gamma_{v_{1}}^+ \cap \LL_m \neq \Gamma_{v_{2}}^+ \cap \LL_m] + \PP[p_s^+\not\in A] \\ <Crm^{-2/3}+Ce^{-cr^3s^{-2}}.    
\end{multline*}
\end{itemize}
Thus we conclude that $\PP[\cA\setminus \cB_+]<Crm^{-2/3}+Ce^{-cr^3s^{-2}}$.

Finally we consider $\PP[\cF_t]$. 
Since $p_t^-$ is the last vertex in $\{u\in \Gamma_\boo^-:\bG^-(u)\le t\}$, by Lemma \ref{lem:buse-opti} we have that $p_t^--(1,0)$ or $p_t^--(0,1)$ is the last vertex in $\{u\in \Gamma_\boo^-:T_{\boo,u}^-\le t\}$.
Then by Corollary \ref{cor:lem:semi-inf-loc}, with probability $>1-Ce^{-ct^{2/3}}$ we have that $p_t^-\in \HH_{(t/2,2t)}\cap \VV_{(-t^{8/9}, t^{8/9})}$, thus $d(p_t^-)\ge ((1-\rho)^2+\rho^2)t/2-t^{8/9}\ge t/4-t^{8/9}$.
Note that since $m=\lfloor t/10 \rfloor$ and $t$ is taken large enough depending on $N$, we have $t/4-t^{8/9}>2m+2N+2$. So we have $\PP[\cF_t]>1-Ce^{-ct^{2/3}}$.

By Lemma \ref{lem:phi-close-event-0}, $\heta_t^-$ equals $\heta_{t+s}^+$ in $\llbracket -N, N+1\rrbracket$ with probability at least
\begin{multline*}
\PP[\cB_-]+\PP[\cA]-\PP[\cA\setminus \cB_+] + \PP[\cF_t] - 2 \\
>
1-(Crm^{-2/3}+Ce^{-cr})-C(rs^{-2/3})^{-1/10}-(Crm^{-2/3}+Ce^{-cr^3s^{-2}})-Ce^{-ct^{2/3}} > 1-C(s/t)^{1/30},
\end{multline*}
where the first inequality uses the estimates of $\PP[\cB_-]$, $\PP[\cA\setminus \cB_+]$, and $\PP[\cA]$ above, and the estimates on $\PP[\cA]$ from Lemma \ref{lem:couple-s}.
Thus the conclusion follows.
\end{proof}

\subsection{The initial step coupling}
This subsection is devoted to proving Lemma \ref{lem:couple-s}.

We define $(\tau_t)_{t\in\R}$ as the process of stationary TASEP with density $\rho$, i.e.\;for any $t\in\R$, we have $\tau_t(x)$ being Bernoulli$(\rho)$ for each $x \in \Z$ independently.
Our strategy is to construct a coupling between the processes $({\eta_t^+})_{t \ge 0}$ and $(\tau_t)_{t\ge 0}$, where (with high probability) $\heta_s^+$ and $\tau_s$ are identical outside $\llbracket -r, r\rrbracket$, and have the same number of particles in $\llbracket -r, r\rrbracket$.
It would be straightforward to couple $\heta_0^-$ and $\tau_s$ since both are Bernoulli$(\rho)$ on $\Z\setminus\{0,1\}$.

We denote $\alpha = (rs^{-2/3})^{1/5}$, and $r_i=\alpha^i s^{2/3}$, for $i=1,2,3,4$.
Below we assume that $\alpha$ and $s$ are large enough, and also $\alpha < r^{0.01}$.
Recall the notations $L^-$, $\bG^-$, $I_t^-$, $\partial I_t^-$, $\xi^{-,\vee}$, $\xi^{-,\vee,t}$, $T_{u,v}^-$, $\Gamma_{u,v}^-$, $\Gamma_u^-$, $\Gamma_u^{-,\vee}$, $p_t^-$ (resp.\;$L^+$, $\bG^+$, $I_t^+$, $\partial I_t^+$, $\xi^{+,\vee}$, $\xi^{+,\vee,t}$, $T_{u,v}^+$, $\Gamma_{u,v}^+$, $\Gamma_u^+$, $\Gamma_u^{+,\vee}$, $p_t^+$) for LPP with weights $\xi^-$ (resp.\;$\xi^+$).
Also recall the notations $\VV_x$, $\HH_x$ and $\VV_A$, $\HH_A$, for $x\in\R$ and $A\subset \R$.

We now explain the coupling between the processes $({\eta_t^+})_{t \ge 0}$ and $(\tau_t)_{t\ge 0}$.
One straightforward way is to couple $\eta_0^+$ and $\tau_0$ so that they are the same outside a compact interval, and let them evolve using the same exponential waiting times (just like how $({\eta_t^-})_{t \ge 0}$ and $({\eta_{s+t}^+})_{t \ge 0}$ are coupled).
One can show that under this coupling, with high probability $\eta_s^+$ and $\tau_s$ are the same, like how Proposition \ref{prop:phi-close} is proved assuming Lemma \ref{lem:couple-s}.
However, we need to compare $\heta_s^+$ and $\tau_s$ instead.
For this, we first shift ${\eta_0^+}$ by $ad(p_s^+)$, and then couple it with $\tau_0$, so that they are the same outside a compact interval, and then let them evolve using the same exponential waiting times.
One problem is that the number $ad(p_s^+)$ depends on the evolution of $(\eta_t^+)_{t\ge 0}$.
To solve this, we exploit the fact that $ad(p_s^+)$ mostly depends on the evolution around the hole-particle pair.
Specifically, we take the following approach: we first sample the evolution of $({\eta_t^+})_{t \ge 0}$ around the pair (which corresponds to sampling the waiting times $\xi^{+,\vee,0}$ in a tube $\VV_{(-r_1,r_1)}$) to get a proxy of $p_s^+$, which equals $p_s^+$ with high probability.
Using that we could shift $\tau_0$ and couple the rest of the waiting times $\xi^{+,\vee,0}$ with the waiting times of $(\tau_t)_{t\ge 0}$.

We start by defining the proxy of $p_s^+$.
Denote $P=\VV_{(-r_1,r_1)}\cap\Z^2$.
First we define $L^P$, by letting $L^P(u)=0$ for $u\in I_0^+\cup (\Z^2\setminus P)$, and setting $L^P(u)=L^P(u-(1,0))\vee L^P(u-(0,1)) + \xi^{+,\vee,0}(u)$ recursively for each $u \in P\setminus I_0^+$.
Like Lemma \ref{lem:geo-part-cons}, $L^P(u)$ can also be defined as the maximum passage time in $\Z^2\setminus I_0^+$ to $u$, under the weights $\{\xi^{+,\vee,0}(v)\don[v\in P]\}_{v\in\Z^2\setminus I_0^+}$.
Then with probability $1$ we have that $L^P(u)$ for $u\in P\setminus I_0^+$ are mutually different, and below we work under this event.
Analogue to the inductive construction of $\Gamma^+_\boo$ (see Lemma \ref{lem:reconssemiinf}), we define $\Gamma^P_\boo$ by letting $\Gamma^P_\boo[1]=\boo$ and 
\begin{equation}  \label{eq:ind-gp}
\Gamma^P_\boo[i+1]=\argmin_{v\in\{\Gamma^P_\boo[i]+(1,0),\Gamma^P_\boo[i]+(0,1)\}\cap P}L^P(v)    
\end{equation}
for each $i\in\N$.
Recall that $p_s^+$ is the last vertex in $\Gamma^+_\boo\cap I_s^+$.
We let $p^P$ be the last vertex in $\{u\in\Gamma^P_\boo: L^P(u)\le s\}$.
Denote $M=ad(p^P)$. Then $M$ is determined by $I_0^+\cap P$ and $\{\xi^{+,\vee,0}(u)\}_{u\in P\setminus I_0^+}$.

We next show that this proxy $p^P$ equals $p_s^+$ with high probability.
\begin{lemma}  \label{lem:prob-restrict-ab}
$\PP[p^P=p_s^+]>1-Ce^{-c\alpha^3}$ for some constants $c,C>0$.
\end{lemma}
\begin{proof}
According to Lemma \ref{lem:geo-part-cons} and as above stated, $L^+(u)$ and $L^P(u)$ are the maximum passage times from a vertex in $\Z^2\setminus I_0^+$ to $u$, under the weights $\{\xi^{+,\vee,0}(v)\}_{v\in\Z^2\setminus I_0^+}$ and $\{\xi^{+,\vee,0}(v)\don[v\in P]\}_{v\in\Z^2\setminus I_0^+}$ respectively. Also Lemma \ref{lem:geo-part-cons} states that the path with the maximum passage time to $u$ under the weights $\{\xi^{+,\vee,0}(v)\}_{v\in\Z^2\setminus I_0^+}$ is precisely $\Gamma_u^{+,\vee}\setminus I_0^+$.
Then we have that $L^+(u)=L^P(u)$ for any $u$ with $\Gamma_u^{+,\vee}\setminus I_0^+ \subset P$. 

We take 
\[u_1=\left(\left\lfloor-\frac{3((1-\rho)^2+\rho^2)r_1}{4\rho^2}\right\rfloor,0\right),\quad u_2=\left(0, \left\lfloor-\frac{3((1-\rho)^2+\rho^2)r_1}{4(1-\rho)^2}\right\rfloor\right).\]
We let $\cD_1$ be the event where 
\[
\Gamma_{u_1}^+\cap \HH_{(-\infty,2s)} \subset \VV_{(-r_1,-r_1/2)},\;\; \Gamma_{u_2}^+\cap \HH_{(-\infty,2s)} \subset \VV_{(r_1/2,r_1)}.
\]
Assuming that $\cD_1$ holds, we denote
\[
S=\big(\cup_{v\in \Gamma_{u_1}^+} (v+\Z_+\times\Z_{\le 0})\big) \cap \big(\cup_{v\in \Gamma_{u_2}^+} (v+\Z_{\le 0}\times\Z_+)\big) \cap (\HH_{(-\infty,2s)}\setminus I_0^+).
\]
In other words, $S$ is the set consisting of vertices in $\HH_{(-\infty,2s)}\setminus I_0^+$ between $\Gamma_{u_1}^+$ and $\Gamma_{u_2}^+$.
Take any $u\in S$, by the non-crossing property (Lemma \ref{lem:non-cross}), we must have that $\Gamma_u^{+,\vee}$ is disjoint from $\cup_{v\in \Gamma_{u_1}^+} (v+\Z_{\le 0}\times\Z_+)$.
As $u_1+\Z_{\le 0}^2$, $\cup_{v\in \Gamma_{u_1}^+} (v+\Z_{\le 0}\times\Z_+)$, $\cup_{v\in \Gamma_{u_1}^+} (v+\Z_+\times\Z_{\le 0})$ is a disjoint partition of $\Z^2$, and $u_1+\Z_{\le 0}^2\subset I_0^+$, we must have that $\Gamma_u^{+,\vee}\setminus I_0^+ \subset \cup_{v\in \Gamma_{u_1}^+} (v+\Z_+\times\Z_{\le 0})$.
Since $u\in\HH_{(-\infty,2s)}$, we further have that $\Gamma_u^{+,\vee}\setminus I_0^+ \subset \cup_{v\in \Gamma_{u_1}^+\cap \HH_{(-\infty,2s)}} (v+\Z_+\times\Z_{\le 0})$.
Similarly we have $\Gamma_u^{+,\vee}\setminus I_0^+ \subset \cup_{v\in \Gamma_{u_2}^+\cap \HH_{(-\infty,2s)}} (v+\Z_{\le 0}\times\Z_+)$.
Then by $\cD_1$ we must have $\Gamma_u^{+,\vee}\setminus I_0^+ \subset P$, so $L^+(u)=L^P(u)$.

Let $\cD_2$ be the event where $\Gamma_\boo^+ \cap I_s^+ \subset \VV_{(-r_1/3,r_1/3)}\cap \HH_{(-\infty,3s/2)}$.
Under $\cD_1$ we have $\VV_{(-r_1/3-1,r_1/3+1)}\cap \HH_{(-\infty,2s)}\setminus I_0^+ \subset S$, so under $\cD_1\cap\cD_2$ we have $(\Gamma_\boo^+\cap I_s^+) + \{(1,0), (0,1)\} \subset S$.
Then by the inductive constructions of $\Gamma^+_\boo$ and $\Gamma_\boo^P$ (Lemma \ref{lem:reconssemiinf} and \eqref{eq:ind-gp}), one can inductively show that $\Gamma_\boo^+[i]=\Gamma_\boo^P[i]$ and $L^+(\Gamma_\boo^+[i]+(1,0))=L^P(\Gamma_\boo^P[i]+(1,0))$, $L^+(\Gamma_\boo^+[i]+(0,1))=L^P(\Gamma_\boo^P[i]+(0,1))$, for all $i\in\N$ such that $L^+(\Gamma_\boo^+[i]) \le s$.
By considering the largest such $i$ we conclude that $p^P=p_s^+$.

Now we have $\PP[p^P=p_s^+] \ge \PP[\cD_1\cap \cD_2]$, and it remains to lower bound $\PP[\cD_1]$ and $\PP[\cD_2]$.
By Corollary \ref{cor:trans-semi-inf} we have $\PP[\cD_1]>1 -  Ce^{-cr_1^3s^{-2}}$.
For $\cD_2$, as $p_s^+$ is the last vertex in $\{u\in \Gamma_\boo^+:\bG^+(u)\le s\}$, by Lemma \ref{lem:buse-opti} we have that $p_s^+-(1,0)$ or $p_s^+-(0,1)$ is the last vertex in $\{u\in \Gamma_\boo^+:T_{\boo,u}^+\le s\}$.
Then Corollary \ref{cor:lem:semi-inf-loc} implies that $\PP[\cD_2]>1-Ce^{-cr_1^3s^{-2}}$ (noting that $3s/2>s+C(r_1^3s^{-2})s^{1/3}$ by our choice of the parameters). Thus the conclusion follows.
\end{proof}

We next couple $\tau_0$ and $({\eta_t^+})_{t\ge 0}$.
We state the coupling by constructing $\tau_0$ conditioned on $({\eta_t^+})_{t\ge 0}$, using the following steps.
\begin{enumerate}
\item
Let $\tau_0(x-M)$ be i.i.d.\;Bernoulli$(\rho)$ for each $x\in \llbracket -r_2,r_2\rrbracket$.
\item
For each $x=\lfloor - r_2 \rfloor, \lfloor - r_2 \rfloor - 1, \ldots$, we take $\tau_0(x-M)$ to be i.i.d.\;Bernoulli$(\rho)$, until the first $x_*\in\Z$ with $\sum_{x=x_*}^{0}{\eta_0^+}(x) - \tau_0(x-M) = 0$.
Then for each $x<x_*$ we take $\tau_0(x-M)={\eta_0^+}(x)$.
\item
For each $x=\lceil r_2 \rceil, \lceil r_2 \rceil+1, \ldots$, we take $\tau_0(x-M)$ to be i.i.d.\;Bernoulli$(\rho)$, until the first $x^*\in\Z$ with $\sum_{x=1}^{x^*}{\eta_0^+}(x) - \tau_0(x-M) = 0$.
Then for each $x>x^*$ we take $\tau_0(x-M)={\eta_0^+}(x)$.
\end{enumerate}
As $\alpha=r_2/r_1$ is large enough, the set $P\cap(\Z_{\ge 0}\times \Z_{\le 0}\cup \Z_{\le 0}\times \Z_{\ge 0})$ is contained in $\{u\in\Z^2:|ad(u)| < r_2\}$.
Also note that $I_0^+\cap P$ is determined by $I_0^+\cap P \cap(\Z_{\ge 0}\times \Z_{\le 0}\cup \Z_{\le 0}\times \Z_{\ge 0})$, as $P\cap \Z_-^2 \subset I_0^+$ and $P\cap \Z_+^2$ is disjoint from $I_0^+$.
Then $I_0^+\cap P$ is determined by $I_0^+\cap\{u\in\Z^2:|ad(u)| < r_2\}$, which (by Lemma \ref{lem:Ideteta}) is determined by $\{{\eta_0^+}(x)\}_{|x|< r_2}$, thus is independent of $\{{\eta_0^+}(x)\}_{|x|\ge r_2}$.
Since $M$ is determined by $I_0^+\cap P$ and $\{\xi^{+,\vee,0}(u)\}_{u\in P\setminus I_0^+}$, we have that $M$ is also independent of $\{{\eta_0^+}(x)\}_{|x|\ge r_2}$. Then from the construction of $\tau_0$, and that $\{{\eta_0^+}(x)\}_{|x|\ge r_2}$ are i.i.d.\;Bernoulli$(\rho)$, we have that $\{\tau_0(x-M)\}_{x\in\Z}$ are i.i.d.\;Bernoulli$(\rho)$ conditioned on $M$.
Thus $\{\tau_0(x)\}_{x\in\Z}$ are i.i.d.\;Bernoulli$(\rho)$ and independent of $M$.
Conditioned on $I_0^+$, both $\tau_0$ and $M$ are independent of $\{\xi^{+,\vee,0}(u)\}_{u\in \Z^2\setminus (P\cup I_0^+)}$.

We would let $(\tau_t)_{t\ge 0}$ evolve using the waiting times $\{\xi^{+,\vee,0}(u)\}_{u\in \Z^2\setminus (P\cup I_0^+)}$.
For this, in $\tau_0$ we label the holes by $\Z$ from left to right, and the particles by $\Z$ from right to left, such that for any $x \in \Z\setminus \llbracket x_*, x^* \rrbracket$, the particle (or hole) at site $x-M$ has the same label as the particle (or hole) at site $x$ in ${\eta_0^+}$.
This can be achieved since that $\tau_0$ and $\eta_0^+$ are the same outside $\llbracket x_*, x^* \rrbracket$, and they have the same number of particles in $\llbracket x_*, 0 \rrbracket$ and $\llbracket 1, x^* \rrbracket$, respectively.

Let $L^\tau(a,b)$ be the time when particle labeled $b$ switches with the hole labeled $a$ if in $\tau_0$ this particle is to the left of this hole, and let $L^\tau(a,b)=0$ otherwise. Note that unlike $L^+$, this function $L^\tau$ does not have the same distribution as the Busemann function in LPP. However, we can still define a growth process from it. For each $t\ge 0$ denote $I_t^\tau:=\{u\in \Z^2:L^\tau(u)\le t\}$, and $\partial I_t^\tau:=\{u\in I_t^\tau:L^\tau(u+(1,0))\vee L^\tau(u+(0,1))> t\}$.
Then $I_0^\tau$ is the same as $I_0^+$ outside a compact set.
\begin{lemma}  \label{lem:adcons}
For any $u\in (I_0^\tau\setminus I_0^+)\cup (I_0^+\setminus I_0^\tau)$ we have $ad(u)\in \llbracket x_*, x^* \rrbracket$
\end{lemma}
\begin{proof}
Write $u=(a,b)$.
If $u\not\in I_0^+$, there is some $x\le 0$ such that $(a-x,b-x)\not\in I_0^+$, $(a-x-1,b-x)\in I_0^+$, or $(a-x-1,b-x)\not\in I_0^+$, $(a-x-1,b-x-1)\in I_0^+$.
Then by Lemma \ref{lem:Ideteta}, either $\eta_0^+(a-b)=1$ and the particle at site $a-b$ (in $\eta_0^+$) has label $b-x$, or $\eta_0^+(a-b)=0$ and the hole at site $a-b$ (in $\eta_0^+$) has label $a-x-1$.
If $u\in I_0^\tau$, we can similar deduce that there is some $y\ge 0$ such that $(a+y+1,b+y+1)\not\in I_0^\tau$, $(a+y,b+y+1)\in I_0^\tau$, or $(a+y,b+y+1)\not\in I_0^\tau$, $(a+y,b+y)\in I_0^\tau$.
Then by (an analogue of) Lemma \ref{lem:Ideteta}, either $\tau_0(a-b-M)=1$ and the particle at site $a-b-M$ (in $\tau_0$) has label $b+y+1$, or $\tau_0(a-b-M)=0$ and the hole at site $a-b-M$ (in $\tau_0$) has label $a+y$.

Thus if $u\in I_0^\tau\setminus I_0^+$, we must have $\eta_0^+(a-b)\neq \tau_0(a-b-M)$ or the particles/holes do not have the same label.
So from the coupling between $\eta_0^+$ and $\tau_0$ we have that $a-b\in \llbracket x_*, x^*\rrbracket$, and the conclusion follows. The case where $u\in I_0^+\setminus I_0^\tau$ follows from similar arguments.
\end{proof}

We can also define the waiting times, by letting
\[
\xi^{\tau,\vee}(u)=L^\tau(u) - L^\tau(u-(1,0))\vee L^\tau(u-(0,1))
\]
for any $u\in\Z^2\setminus I_0^\tau$.
Given $I_0^\tau$ (equivalently, $\tau_0$ and the labels), we have that $\{\xi^{\tau,\vee}(u)\}_{u\not\in I_0^\tau}$ are i.i.d.\;$\Exp(1)$, since they are precisely the waiting times for certain particles and holes to switch.
Then almost surely $L^\tau(u)$ are mutually different for all $u \in \Z^2\setminus I_0^\tau$, and below we assume this event.

We now couple $\{\xi^{\tau,\vee}(u)\}_{u\not\in I_0^\tau}$ with $(\eta_t^+)_{\ge 0}$, such that conditioned on $\tau_0$ and $(\eta_t^+)_{t\ge 0}$, we have $\xi^{\tau,\vee}(u)=\xi^{+,\vee,0}(u)$ for any $u\in \Z^2\setminus (I^+_0\cup I_0^\tau \cup P)$ and $\xi^{\tau,\vee}(u)$ for $u\in (P\cup I^+_0)\setminus I_0^\tau$ are i.i.d.\;$\Exp(1)$.
Under this coupling, we denote $\cE_1$ as the event where for any $x<-r$, $\tau_s(x)=\heta_s^+(x)=\eta_s^+(x+ad(p_s^+))$,
and the particle or hole at site $x$ has the same label for $\tau_s$ and $\heta_s^+$; denote $\cE_2$ as the event where the same is true for any $x>r$.
\begin{lemma} \label{lem:bdce12}
We have $\PP[\cE_1], \PP[\cE_2]>1-C\alpha^{-1/2}$ when $Cs^{2/3}<r<s^{2/3+0.01}$ and  $s>C$, where $C>0$ is a constant.
\end{lemma}
We can now prove Lemma \ref{lem:couple-s} assuming Lemma \ref{lem:bdce12}.
\begin{proof}[Proof of Lemma \ref{lem:couple-s}]
Under $\cE_1\cap\cE_2$, we have $\tau_s(x)=\heta_s^+(x)$ for any $x\in\Z, |x|> r$.
Also, note $1+\sum_{x=-r}^r \heta_s^+(x)$ and $1+\sum_{x=-r}^r\tau_s(x)$ are precisely the difference between the label of the first particle to the left of $-r$ and the label of the first particle to the right of $r$ in $\heta_s^+$ and $\tau_s$ respectively, so we have that $\sum_{x=-r}^r \heta_s^+(x)=\sum_{x=-r}^r\tau_s(x)$ under $\cE_1\cap\cE_2$.

We can couple $\heta_0^-$ with $\tau_s$ as follows.
Conditioned on $\heta_0^-$, we let $\tau_s(x)$ for $x=1,2,\ldots,$ be i.i.d.\;Bernoulli$(\rho)$, until some $y^*\in\Z$ such that $\sum_{x=1}^{y^*}\tau_s(x) - \heta_0^-(x) = 0$, and for any $x>y^*$ we let $\tau_s(x)=\heta_0^-(x)$; we also let 
$\tau_s(x)$ for $x=0,-1,\ldots,$ be i.i.d.\;Bernoulli$(\rho)$, until some $y_*\in\Z$ such that $\sum_{x=y_*}^{0}\tau_s(x) - \heta_0^-(x) = 0$, and for any $x<y_*$ we let $\tau_s(x)=\heta_0^-(x)$.
Let $\cE_*$ be the event where $|y_*|, |y^*|\le r$.
Then we have $\PP[\cE_*] > 1-Cr^{-1/2}$ for some constant $C>0$, since $\heta_0^-$ is Bernoulli$(\rho)$ on $\Z\setminus\{0,1\}$.
On the other hand, under $\cE_*$ we have $\heta_0^-(x)=\tau_s(x)$ for any $x\in\Z, |x|> r$ and $\sum_{x=-r}^r \heta_0^-(x)=\sum_{x=-r}^r\tau_s(x)$.
Thus $\cE_*\cap \cE_1 \cap \cE_2$ implies $\cA$, and $\PP[\cA] > \PP[\cE_1]+\PP[\cE_2]+\PP[\cE_*]-2$.
Using $\PP[\cE_*] > 1-Cr^{-1/2}$ and Lemma \ref{lem:bdce12}, the conclusion follows.
\end{proof}

For the rest of this section we prove Lemma \ref{lem:bdce12}.

For any $u\in (\Z^2\setminus I_0^\tau) \cup \partial I_0^\tau$, we also define the `semi-infinite geodesic' $\Gamma_u^\tau$ recursively, by letting $\Gamma_u^\tau[1]=u$, and for each $i\in\N$ letting $\Gamma_u^\tau[i+1] = \argmin_{v \in \{\Gamma_u^\tau[i]+(1,0),\Gamma_u^\tau[i]+(0,1)\}} L^\tau(v)$.
Note that since $L^\tau$ is not coupled with the LPP Busemann function, these $\Gamma_u^\tau$ are not actual geodesics.

We consider the following events (see Figure \ref{fig:2event-inc} for an illustration of the geometric objects).
\begin{itemize}
    \item[$\cE_3:$] there exists a vertex $u_+\in \partial I_0^+$, such that $ad(u_+)<x_*$ and $\Gamma_{u_+}^+\cap I_s^+ \subset \VV_{(-6r_3,-r_1)}$, and $a_+'>(1-\rho)^2s-r_4$ for $u_+'=(a_+',b_+')$ being the last vertex in $\Gamma_{u_+}^+\cap I_s^+$.
    \item[$\cE_3^\tau:$] there exists a vertex $u_\tau\in \partial I_0^\tau$ such that $ad(u_\tau)<x_*$ and $\Gamma_{u_\tau}^\tau \cap I_s^\tau\subset \VV_{(-6r_3,-r_1)}$, and $a_\tau'
>(1-\rho)^2s-r_4$ for $u_\tau'=(a_\tau',b_\tau')$ being the last vertex in $\Gamma_{u_\tau}^+\cap I_s^\tau$.
    \item[$\cE_4:$] for each $u=(a,b)\in \partial I_s^+$ with $ad(u)\le M-r$, we have $a<(1-\rho)^2s-r_4-1$, and $u+(1,0)\in \VV_{(-\infty,-6r_3)}$.
\end{itemize}
\begin{figure}[hbt!]
    \centering
\begin{tikzpicture}[line cap=round,line join=round,>=triangle 45,x=0.10cm,y=0.088cm]
\clip(0.0,10) rectangle (125.,125.);

\fill[line width=0.pt,color=uuuuuu,fill=uuuuuu,fill opacity=0.2]
(0,40) -- (51,80.5) --(51,125) -- (0,125) -- cycle;

\fill[line width=0.pt,color=blue,fill=blue,fill opacity=0.25]
(0,70) --(5,68) --(10,62) --(15,58) --(20,57) --(25,52) --(30,48) --(35,41) --(40,37) --(45,34) --(50,29) --(55,22) --(60,19) --(65,16) --(70,14) --(75,10) --(80,5) --(85,2) --(90,0) -- (0,0) -- cycle;
\fill[line width=0.pt,color=red,fill=red,fill opacity=0.25]
(0,70) --(5,68) --(10,62) --(15,58) --(20,57) --(25,52) --(30,48) --(35,38) --(40,34) --(45,32) --(50,32) --(55,26) --(60,19) --(65,16) --(70,14) --(75,10) --(80,5) --(85,2) --(90,0) -- (0,0) -- cycle;

\draw [red] plot coordinates {(30,48) (35,51) (40,55) (45,64) (50,67) (55,69) (60,73) (63,79) (67,83) (75,88) (80, 91) (85,97) (90,100) (95,104) (100,110) (105,113) (110,117) (118,125)};
\draw [color=red,fill=red] (30,48) circle (1.0pt);
\draw [color=red,fill=red] (63,79) circle (1.0pt);

\draw [blue] plot coordinates {(27.5,50) (35,53) (40,56) (45,61) (50,64) (55,71) (60,72) (67.5,74) (71,83) (75,88) (80, 91) (85,97) (90,100) (95,104) (100,110) (105,113) (110,117) (118,125)};

\draw [color=blue,fill=blue] (27.5,50) circle (1.0pt);
\draw [color=blue,fill=blue] (67.5,74) circle (1.0pt);
\draw [color=blue,fill=blue] (43.5,34.9) circle (1.0pt);

\draw (43.5,34.9) node[anchor=east,blue]{$\boo$};
\draw (27.5,50) node[anchor=east,blue]{$u_+$};
\draw (67.5,74) node[anchor=west,blue]{$u_+'$};
\draw (30,48) node[anchor=west,red]{$u_\tau$};
\draw (63,79) node[anchor=east,red]{$u_\tau'$};

\draw (54,10) node[anchor=south,red]{$I_0^\tau$};
\draw (60,10) node[anchor=south,blue]{$I_0^+$};
\draw (110,10) node[anchor=south,red]{$I_s^\tau$};
\draw (116,10) node[anchor=south,blue]{$I_s^+$};

\draw (0,100) node[anchor=north west]{$Q$};

\begin{scriptsize}
\draw (62,125) node[anchor=north east]{$(a,b):a=(1-\rho)^2s-r_4$};
\draw (66,125) node[anchor=north west]{$u:ad(u)=M-r$};
\draw (15,107) node[anchor=west,purple]{$\partial I_s^+ \cap \partial I_s^\tau$};
\end{scriptsize}

\fill[line width=0.pt,color=blue,fill=blue,fill opacity=0.1]
(3,125) -- (8,120)-- (15,113) -- (20,110)-- (25, 106)-- (30,102) -- (35,100) --(40,96) --(45,91) --(50,89) --(55,86) --(60,82) --(65,77) --(70,71) --(75,68) --(80,64) --(85,62) --(90,59) --(95,58) --(100,51) --(105,49) --(110,46) --(115,40) --(120,35) --(125,32) -- (125,0) -- (0,0) -- (0,125) -- cycle;
\fill[line width=0.pt,color=red,fill=red,fill opacity=0.1]
(3,125) -- (8,120)-- (15,113) -- (20,110)-- (25, 106)-- (30,102) -- (35,100) --(40,96) --(45,91) --(50,89) --(55,86) --(60,82) --(65,77) --(70,67) --(75,67) --(80,66) --(85,64) --(90,59) --(95,58) --(100,51) --(105,49) --(110,46) --(115,40) --(120,35) --(125,32) -- (125,0) -- (0,0) -- (0,125) -- cycle;

\draw [purple] plot coordinates {(3,125) (8,120) (15,113) (20,110) (25, 106) (30,102) (35,100) (40,96) (43,93) };

\draw [line width=.1pt,dashed] (0,40) -- (125,140);
\draw [line width=.1pt,dashed] (0,10) -- (125,110);
\draw [line width=.1pt,dashed] (0,-10) -- (125,90);

\draw [line width=.1pt,dashed] (51,0) -- (51,125);

\draw [line width=.1pt,dashed] (-38,0) -- (74,128);

\draw (127,89) node[anchor=north east]{$\VV_{r_1}$};
\draw (126,106) node[anchor=north east]{$\VV_{-r_1}$};
\draw (110,125) node[anchor=north east]{$\VV_{-6r_3}$};
\end{tikzpicture}
\caption{The events $\cE_3, \cE_3^\tau,\cE_4$, assuming $p^P=p_s^+$. The shaded region is $Q$ in the proof of Lemma \ref{lem:2event-inc}.}   \label{fig:2event-inc}
\end{figure}

The purpose of these events is as follows. For $\cE_3$ and $\cE_3^\tau$, they ensure that for $u$ in a certain region (around $\{u:\Z^2:ad(u)\le M-r\}\cap \partial I_s^+$), the downward geodesics $\Gamma_u^{+,\vee}$ and $\Gamma_u^{\tau,\vee}$ are disjoint from $P$.
Thus using Lemma \ref{lem:adcons} and the coupling between $\xi^{\tau,\vee}$ and $\xi^{+,\vee,0}$ we have $L^+(u)=L^\tau(u)$ for these $u$.
Then we can deduce that $\{u:\Z^2:ad(u)\le M-r\}\cap \partial I_s^+$ is the same as $\{u:\Z^2:ad(u)\le M-r\}\cap \partial I_s^\tau$.
Then using $p^P=p_s^+$ and Lemma \ref{lem:Ideteta}, we have that $\cE_1$ holds.
The event $\cE_4$ is to define this `certain region'.
In summary, we have the following statement.
\begin{lemma}   \label{lem:2event-inc}
$\{p^P=p_s^+\}\cap\cE_3\cap \cE_3^\tau\cap \cE_4  \subset \cE_1$.
\end{lemma}

\begin{proof}
Below we assume that $\cE_3 \cap \cE_3^\tau\cap \cE_4$ holds and $p^P=p_s^+$.
Denote
\[
Q=\{(a,b) \in \Z^2: a<(1-\rho)^2s-r_4, (a,b) \in \VV_{(-\infty,-6r_3)}\}.
\]
See Figure \ref{fig:2event-inc}.
Then we must have that $Q\setminus I_0^+=Q\setminus I_0^\tau$. Otherwise we can find some $u\in Q$ with $u\in \partial I_0^\tau$ and $u\not\in I_0^+$, or $u\in \partial I_0^+$ and $u\not\in I_0^\tau$.
In the first case, $u\in u_\tau+\Z_{\le 0}\times\Z_+$ since $u\in\VV_{(-\infty,-6r_3)}$, $u_\tau \in \VV_{(-6r_3,-r_1)}$, and both $u, u_\tau\in \partial I_0^\tau$. So we must have $ad(u)<ad(u_\tau)<x_*$. But $ad(u)\in \llbracket x_*, x^*\rrbracket$ by Lemma \ref{lem:adcons}, so we get a contradiction.
A similar contradiction can be obtained in the second case.

Now take any $u\in Q\setminus I_0^+=Q\setminus I_0^\tau$, and we next show that $L^+(u)\le L^\tau(u)$.
By Lemma \ref{lem:geo-part-cons} we have that $L^+(u)=\sum_{v\in \Gamma^{+,\vee}_u\setminus I_0^+} \xi^{+,\vee,0}(v)$, and this is the maximum passage time to $u$ from a vertex in $\Z^2\setminus I_0^+$, under the weights $\xi^{+,\vee,0}$.
Analogously, $L^\tau(u)$ equals the maximum passage time to $u$ from a vertex in $\Z^2\setminus I_0^\tau$, under the weights $\xi^{\tau,\vee}$.
It then suffices to show that $\Gamma^{+,\vee}_u\setminus I_0^+$ is disjoint from $P$ and $I_0^\tau$, since then $\Gamma^{+,\vee}_u\setminus I_0^+$ is an up-right path from a vertex in $\Z^2\setminus I_0^\tau$ to $u$, and for any $v \in \Gamma^{+,\vee}_u\setminus I_0^+$ we have $\xi^{\tau,\vee}(v)=\xi^{+,\vee,0}(v)$, thus
\[
L^+(u)=\sum_{v\in \Gamma^{+,\vee}_u\setminus I_0^+} \xi^{+,\vee,0}(v) =\sum_{v\in \Gamma^{+,\vee}_u\setminus I_0^+} \xi^{\tau,\vee}(v) \le L^\tau(u).
\]
We show that $(\Gamma^{+,\vee}_u\setminus I_0^+)\cap (P\cup I_0^\tau)=\emptyset$, using the following steps.
\begin{itemize}
    \item[\textbf{Step 1.}] 
    By Lemma \ref{lem:non-cross}, the path $\Gamma_{u_+}^++(1/2,1/2)$ divides $u_++(\Z^2\setminus\Z_{\le 0}^2)$ into two parts, such that $\Gamma^{+,\vee}_u$ intersects at most one of them.
    By $\cE_3$ and $u\in Q\setminus I_0^+$, we know that $u$ must be in the upper-left part, so $\Gamma^{+,\vee}_u$ cannot intersect the lower-right part.
    In particular, $\Gamma^{+,\vee}_u$ is disjoint from $u_++\Z_+\times \Z_{\le 0}$.
    Also $u_++\Z_{\le 0}^2\subset I_0^+$ since $u_+\in I_0^+$, so $\Gamma^{+,\vee}_u\setminus I_0^+ \subset u_++\Z\times \Z_+$.
    \item[\textbf{Step 2.}] If $\Gamma^{+,\vee}_u\setminus I_0^+$ is not disjoint from $I_0^\tau$, take any $v\in (\Gamma^{+,\vee}_u\setminus I_0^+) \cap I_0^\tau$. Then $v\in u_++\Z\times \Z_+$ according to the previous step. By Lemma \ref{lem:adcons} we have $ad(v)\ge x_*$, and $\cE_3$ states that $ad(u_+)<x_*$. So $ad(v)>ad(u_+)$, thus $v-u_+\in \Z_+^2$. But this implies that $u_++(1,1)\in I_0^\tau\setminus I_0^+$, which contradicts with Lemma \ref{lem:adcons} since $ad(u_++(1,1))=ad(u_+)<x_*$.
    \item[\textbf{Step 3.}] Since $u_+\subset \VV_{(-6r_3,-r_1)}$ (by $\cE_3$), we have $P\cap (u_+ + \Z_{\le 0}\times \Z_{\ge 0})=\emptyset$. We also have that $u_++\Z_{\le 0}^2\subset I_0^+$ since $u_+\in I_0^+$. Thus we have $P\setminus I_0^+\subset u_+ + \Z_+\times \Z$.
    Take any $(a,b) \in P\setminus I_0^+$.
    If $a\ge (1-\rho)^2s-r_4$, we cannot have $(a,b) \in \Gamma^{+,\vee}_u$ since $u\in Q$.
    If $a< (1-\rho)^2s-r_4$, by $\cE_3$ we must have that $(a,b)$ is in the lower-right part from Step 1, so still we must have that $(a,b) \not\in \Gamma^{+,\vee}_u$.
    So $P\setminus I_0^+$ is disjoint from $\Gamma^{+,\vee}_u$, and equivalently $\Gamma^{+,\vee}_u\setminus I_0^+$ is disjoint from $P$.
\end{itemize}
So far we have shown $L^+(u)\le L^\tau(u)$. We can also show $L^+(u)\le L^\tau(u)$ with essentially verbatim arguments, using $\cE_3^\tau$ instead of $\cE_3$.
We then conclude that $L^+(u)=L^\tau(u)$ for any $u\in Q\setminus I_0^+=Q\setminus I_0^\tau$.

We then show that $\cE_1$ holds, using Lemma \ref{lem:Ideteta}.
Specifically, take any $x\in \Z$ with $x<-r$, we next show that $\heta_s^+(x)=\tau(x)$, and the particles (or holes) have the same labels.

We first assume that $\heta_s^+(x)=1$.
Then ${\eta_s^+}(x+ad(p_s^+))=1$.
By $p^P=p_s^+$ we have $M=ad(p_s^+)$, so ${\eta_s^+}(x+M)=1$.
By Lemma \ref{lem:Ideteta}, there is some $y\in\Z$, such that $(M+x+y-1,y)\in I_s^+$ and $(M+x+y,y)\not\in I_s^+$, and the particle at $x$ in $\heta_s^+$ has label $y$.
Since $ad(M+x+y-1,y)=M+x-1\le M-r$ and $(M+x+y-1,y)\in \partial I_s^+$, by $\cE_4$ we have $M+x+y<(1-\rho)^2s-r_4$ and $(M+x+y,y)\in \VV_{(-\infty,-6r_3)}$. 
Thus $(M+x+y,y), (M+x+y-1,y)\in Q$, and
\[
L^\tau(M+x+y,y)= L^+(M+x+y,y) > s \ge L^+(M+x+y-1,y)= L^\tau(M+x+y-1,y).
\]
This implies that $(M+x+y-1,y) \in I_s^\tau$ and $(M+x+y,y) \not\in I_s^\tau$.
Then by (an analogue of) Lemma \ref{lem:Ideteta} we have $\tau_s(x)=1$, and the particle at $x$ in $\tau_s$ has label $y$.

Similarly, if we assume that $\heta_s^+(x)=0$, we can deduce that $\tau_s(x)=0$, and the holes have the same label.
By taking $x$ over all integers $<-r$, we conclude that $\cE_1$ holds assuming $p^P=p_s^+$ and $\cE_3\cap \cE_4 \cap \cE_3^\tau$.
\end{proof}

It now suffices to lower bound the probabilities of the events $\cE_3$, $\cE_3^\tau$, $\cE_4$.
\begin{lemma}   \label{lem:lbd-2event}
$\PP[\cE_3], \PP[\cE_3^\tau]>1-C\alpha^{-1/2}$ for constants $c,C>0$.
\end{lemma}

\begin{lemma}   \label{lem:lbd-2event4}
$\PP[\cE_4]>1-Ce^{-c\alpha}$ for constants $c,C>0$.
\end{lemma}

Using Lemma \ref{lem:2event-inc} and Lemmas \ref{lem:lbd-2event}, \ref{lem:lbd-2event4}, we get lower bound for $\PP[\cE_1]$. We can lower bound $\PP[\cE_2]$ similarly. Thus Lemma \ref{lem:bdce12} follows.

To prove these estimates we introduce some other setups. 
For the convenience of notations, we extend $\xi^{+,\vee,0}$ from $\Z^2\setminus I_0^+$ to $\Z^2$, so that conditioned on $I_0^+$, $\{\xi^{+,\vee,0}(v)\}_{v\in I_0^+}$ are i.i.d.\;$\Exp(1)$ and are independent of everything else.
For each $u\le v$, we let $T^{+,\vee,0}_{u,v}$ and $\Gamma^{+,\vee,0}_{u,v}$ be the passage time and geodesic from $u$ to $v$ under the weights $\xi^{+,\vee,0}$.
For any $v\not\in I_0^+$ we denote $\Gamma^{+,\vee,0}_{I,v}=\Gamma^{+,\vee,0}_{u_*,v}$ and $T^{+,\vee,0}_{I,v}=T^{+,\vee,0}_{u_*,v}$, where $u_*=\argmax_{u\le v,u\not\in I_0^+}T^{+,\vee,0}_{u,v}$.
In words, $\Gamma^{+,\vee,0}_{I,v}$ and $T^{+,\vee,0}_{I,v}$ are the geodesic and passage time from boundary $I_0^+$ to $v$, under the weights $\xi^{+,\vee,0}$. By Lemma \ref{lem:geo-part-cons} we have $\Gamma^{+,\vee,0}_{I,v}=\Gamma^{+,\vee}_v\setminus I_0^+$ and $T^{+,\vee,0}_{I,v}=L^+(v)$.

\begin{proof}[Proof of Lemma \ref{lem:lbd-2event}]
We shall write the proof for the estimates of $\PP[\cE_3]$, and the approach we take here applies to $\PP[\cE_3^\tau]$ essentially verbatim.
We will use $c,C>0$ to denote small and large enough constants, and their values can change from line to line.

We consider the following events (see Figure \ref{fig:lbd-2event}).
\begin{enumerate}
\item[$\cE_5:$] $x_*>-r_3$.
\item[$\cE_6:$] $\VV_{(-jr_4,jr_4)}\cap \partial I_0^+\subset \HH_{(-j\alpha r_4^{1/2},j\alpha r_4^{1/2})}$ for each $j\in\N$.
\item[$\cE_7:$] $\VV_{(-6r_3,-r_3)}\cap \partial I_s^+\subset \HH_{(s-2\alpha r_4^{1/2},s+2\alpha r_4^{1/2})}$.
\item[$\cE_8:$] Let $u_1$ be the intersection of $\HH_{2s}$ with $\VV_{-5r_3}$ and $u_2$ be the intersection of $\HH_{2s}$ with $\VV_{-2r_3}$ (rounded to the nearest lattice vertex).
Then $\Gamma_{I,u_1}^{+,\vee,0}\subset \VV_{(-6r_3,-4r_3)}$ and $\Gamma_{I,u_2}^{+,\vee,0}\subset \VV_{(-3r_3,-r_3)}$.
\end{enumerate}
The events $\cE_6$ and $\cE_7$ just say that $\partial I_0^+$ and $\partial I_s^+$ behave `typically' in certain regions. The event $\cE_8$ is to bound the transversal fluctuation of $\Gamma_{u_+}^+$ (for some $u_+\in\partial I_0^+$), using the non-crossing property of downward and upward semi-infinite geodesics (Lemma \ref{lem:non-cross}).
\begin{figure}[hbt!]
    \centering
\begin{tikzpicture}[line cap=round,line join=round,>=triangle 45,x=0.08cm,y=0.066cm]
\clip(-10,5) rectangle (180.,128.);

\fill[line width=0.pt,color=yellow,fill=yellow,fill opacity=0.6]
(5,60) -- (11.25,65) -- (81.25,9) -- (75,4) -- cycle;
\fill[line width=0.pt,color=yellow,fill=yellow,fill opacity=0.6]
(1.875,57.25) -- (14.375,67.5) -- (14.375-100,67.5+80) -- (1.875-100,57.25+80) -- cycle;
\fill[line width=0.pt,color=yellow,fill=yellow,fill opacity=0.6]
(71.725+100,1.5-80) -- (84.375+100,11.25-80) -- (84.375,11.25) -- (71.725,1.5) -- cycle;

\fill[line width=0.pt,color=yellow,fill=yellow,fill opacity=0.6]
(43.75,83) -- (56.25,93) -- (81.25,73) -- (68.75,63) -- cycle;

\draw [line width=.04pt,dashed] (-30,0) -- (220,200);
\draw [line width=.04pt,dashed] (-40,0) -- (210,200);
\draw [line width=.04pt,dashed] (-10,0) -- (240,200);
\draw [line width=.04pt,dashed] (10,0) -- (260,200);
\draw [line width=.04pt,dashed] (-60,0) -- (190,200);

\draw [line width=.9pt,color=blue] (-20,85) --(-15,79) --(-10,77)--
(-5,73) --(0,70)-- (5,68)-- (10,62)-- (15,58)-- (20,57)-- (25,52) --(30,48)-- (35,41) --(40,37) --(45,34) --(50,29)-- (55,22) --(60,19) --(65,16)-- (70,14)-- (75,10)-- (80,5)-- (85,2)-- (90,0);

\draw [line width=.9pt,color=blue] (3,125)-- (8,120)-- (15,113)-- (20,110)-- (25, 106)-- (30,102)-- (35,100)-- (40,96)-- (45,91)-- (50,89)-- (55,86)-- (60,82)-- (65,77)-- (70,71)-- (75,68)-- (80,64)-- (85,62)-- (90,59)-- (95,58)-- (100,51)-- (105,49)-- (110,46)-- (115,40)-- (120,35)-- (125,32);

\draw [line width=.4pt,color=blue] (27.5,50)-- (35,53)-- (40,56)-- (45,61)-- (50,64)-- (55,71)-- (60,72)-- (67.5,74)-- (71,83)-- (75,88)-- (80, 91)-- (85,99)-- (90,102)-- (95,110)-- (100,114)-- (105,116)-- (110,118)-- (118,123) -- (123,128);

\draw [color=blue,fill=blue] (27.5,50) circle (1.0pt);
\draw [color=blue,fill=blue] (67.5,74) circle (1.0pt);
\draw [color=blue,fill=blue] (43.5,34.9) circle (1.0pt);

\draw [line width=.6pt,color=brown] (20,57) -- (25,62) -- (30,64)-- (35,67)-- (40,71)-- (45,74)-- (50,78)-- (55,81)-- (60,88)-- (65,92)-- (70,98)-- (75,99)-- (80, 102)-- (85,105) -- (90,112);

\draw [line width=.6pt,color=brown] (35,41)-- (40,45)-- (45,48)-- (50,54)-- (55,57)-- (60,62)-- (65,67)-- (70,69)-- (75,74)-- (80,77)-- (85,85) -- (90,88) -- (95,96) -- (100,98) -- (105,100);

\draw [color=brown,fill=brown] (90,112) circle (1.0pt);
\draw [color=brown,fill=brown] (105,100) circle (1.0pt);
\draw (90,112) node[anchor=east,brown]{$u_1$};
\draw (105,100) node[anchor=west,brown]{$u_2$};

\draw (43.5,34.9) node[anchor=north east,blue]{$\boo$};
\draw (27.5,50) node[anchor=north,blue]{$u_+$};
\draw (67.5,74) node[anchor=west,blue]{$u_+'$};

\draw (60,6) node[anchor=south,blue]{$\partial I_0^+$};
\draw (110,44) node[anchor=east,blue]{$\partial I_s^+$};

\begin{scriptsize}
\draw (175,128) node[anchor=north east]{$\VV_{r_3}$};
\draw (155,128) node[anchor=north east]{$\VV_{-r_3}$};
\draw (135,128) node[anchor=north east]{$\VV_{-3r_3}$};
\draw (115,128) node[anchor=north east]{$\VV_{-4r_3}$};
\draw (95,128) node[anchor=north east]{$\VV_{-6r_3}$};
\draw (120,120) node[anchor=north east, color=blue]{$\Gamma_{u_+}^+$};
\draw (77,107) node[anchor=north east, color=brown]{$\Gamma_{I,u_1}^{+,\vee,0}$};
\draw (92,95) node[anchor=north west, color=brown]{$\Gamma_{I,u_2}^{+,\vee,0}$};
\draw (80,65) node[anchor=west]{$\VV_{(-6r_3,-r_3)}\cap \HH_{(s-2\alpha r_4^{1/2},s+2\alpha r_4^{1/2})}$};
\draw (-5,75) node[anchor=west]{$S_*\setminus S^*$};
\end{scriptsize}

\end{tikzpicture}
\caption{The events to lower bound $\PP[\cE_3]$.}   \label{fig:lbd-2event}
\end{figure}

We next show that $\cE_5\cap\cE_6\cap\cE_7\cap\cE_8\subset\cE_3$.
For this, we take any $u_+ \in \partial I_0^+\cap \VV_{(-4r_3,-3r_3)}$, and let $u_+'=(a_+',b_+')$ be the last vertex in $\Gamma_{u_+}^+\cap I_s^+$.
Then we need to show that $ad(u_+)<x_*$, $\Gamma_{u_+}^+\cap I_s^+ \subset \VV_{(-6r_3,-r_3)}$, and $a_+'>(1-\rho)^2s-r_4$, assuming $\cE_5\cap\cE_6\cap\cE_7\cap\cE_8$.
\begin{itemize}
    \item By $\cE_6$, and note that $r_3>C\alpha r_4^{1/2}$ by our choice of the parameters, we have $ad(u_+)<-r_3$. So under $\cE_5\cap\cE_6$ we have $ad(u_+)<x_*$.
    \item Under $\cE_8$, the path $\Gamma_{I,u_1}^{+,\vee,0}-(1/2,1/2)$ divides $(u_1+(\Z^2\setminus \Z_{\ge 0}^2)) \setminus I_s^+$ into two parts: $(\cup_{v\in \Gamma_{I,u_1}^{+,\vee,0}} v+\Z_-\times \Z_{\ge 0}) \setminus I_s^+$ and $(\cup_{v\in \Gamma_{I,u_1}^{+,\vee,0}} v+\Z_{\ge 0}\times \Z_-) \setminus I_s^+$.
    Then $u_+$ must be in the second part, so by Lemma \ref{lem:non-cross} $\Gamma_{u_+}^+$ must be disjoint from the first part, thus $\Gamma_{u_+}^+\cap \VV_{(-\infty, -6r_3]}\subset u_1+\Z_{\ge 0}^2$. 
    If $\Gamma_{u_+}^+\cap \VV_{(-\infty, -6r_3]}\cap I_s^+$ is not empty, we must have $u_1 \in I_s^+$, which contradicts $\cE_7$.
    So under $\cE_7\cap \cE_8$ we must have that $\Gamma_{u_+}^+ \cap I_s^+$ is disjoint from  $\VV_{(-\infty, -6r_3]}$, and similarly it is also disjoint from $\VV_{[-r_3,\infty)}$. These mean that $\Gamma_{u_+}^+\cap I_s^+ \subset \VV_{(-6r_3,-r_3)}$.
    \item $\{\Gamma_{u_+}^+\cap I_s^+ \subset \VV_{(-6r_3,-r_3)}\}\cap\cE_7$ implies that $u_+'\in \HH_{(s-2\alpha r_4^{1/2},s+2\alpha r_4^{1/2})}\cap \VV_{(-6r_3,-r_3)}$. 
    Thus we get $a_+'>(1-\rho)^2s-r_4$ since $r_4>Cr_3, C\alpha r_4^{1/2}$ by our choice of the parameters.
\end{itemize}

It remains to estimate the probabilities of these events and take a union bound.\\

\noindent\textbf{Bound $\PP[\cE_5]$.} By the coupling between $\tau_0$ and $({\eta_t^+})_{t\ge 0}$ (stated after the proof of Lemma \ref{lem:prob-restrict-ab}), the number $-x_*$ is just the time of a symmetric random walk hitting $0$ after $r_2$. Thus $\PP[\cE_5]\ge 1-Cr_2^{1/2}r_3^{-1/2}=1-C\alpha^{-1/2}$.\\

\noindent\textbf{Bound $\PP[\cE_6]$.} The event $\cE_6$ is again on the hitting probability of a random walk. 
Indeed, by Lemma \ref{lem:Ideteta}, if we let $f(x)$ be the largest integer with $(f(x)+x,f(x))\in I_0^+$, we must have $f(0)=0$, $f(x)=\sum_{y=1}^x -\eta_0^+(x)$ for any $x\ge 1$, and $f(x)=\sum_{y=x+1}^0 \eta_0^+(x)$ for any $x\le -1$;
and $\{\eta_0^+(x)\}_{x\in\Z\setminus\{0,1\}}$ are i.i.d.\;Bernoulli$(\rho)$.
Thus for each $j\in\N$ we have $\PP[\VV_{(-jr_4,jr_4)}\cap \partial I_0^+\subset \HH_{(-j\alpha r_4^{1/2},j\alpha r_4^{1/2})}] > 1-Ce^{-cj\alpha^2}$, so when $\alpha > C$ we have $\PP[\cE_6] \ge 1-Ce^{-c\alpha^2}$.\\

For $\PP[\cE_7]$ and $\PP[\cE_8]$, we reduce them to estimates on last-passage times and geodesic transversal fluctuations under the weights $\xi^{+,\vee,0}$, and use results from Section \ref{ssec:elpps}.\\

\noindent\textbf{Bound $\PP[\cE_7]$.}
We note that $\cE_7$ is implied by the following two events:
\begin{itemize}
    \item $T^{+,\vee,0}_{I,v}=L^+(v)>s$ whenever $v \in \VV_{(-6r_3,-r_3)} \cap \HH_{[s+2\alpha r_4^{1/2},\infty)}\cap \Z^2$,
    \item $T^{+,\vee,0}_{I,v}=L^+(v)\le s$ whenever $v-(1,1) \in \VV_{(-6r_3,-r_3)} \cap \HH_{(-\infty, s-2\alpha r_4^{1/2}]}\cap \Z^2$.
\end{itemize}
These two events imply that $\VV_{(-6r_3,-r_3)} \cap \HH_{(-\infty, s-2\alpha r_4^{1/2}] \cup [s+2\alpha r_4^{1/2},\infty)}$ is disjoint from $\partial I_s^+$, so $\cE_7$ holds.

To estimate the probabilities of these events, we need to bound the passage times under $\xi^{+,\vee,0}$.
For this, we set up the following notations.
For each $j\in\N$ we let
\[
S_j=\VV_{(-jr_4,-(j-1)r_4]\cup [(j-1)r_4,jr_4)}\cap \HH_{(-j\alpha r_4^{1/2},\infty)},\]
\[S^j=\VV_{(-jr_4,-(j-1)r_4]\cup [(j-1)r_4,jr_4)}\cap \HH_{[j\alpha r_4^{1/2},\infty)}.\]
Let $S_*=\cup_{j\in\N}S_j$ and $S^*=\cup_{j\in\N}S^j$.
Then the event $\cE_6$ precisely says that $\partial I_0^+\subset S_*\setminus S^*$ (see Figure \ref{fig:lbd-2event}), and implies that $S^*\subset \Z^2\setminus I_0^+ \subset S_*$.

We consider the following events:
\begin{enumerate}
\item[$\cE_7':$] $T^{+,\vee,0}_{u,v} \le s$ for any vertices $u\in S_*$ and $v\in \VV_{(-6r_3,-r_3)} \cap \HH_{(-\infty,s-2\alpha r_4^{1/2}]}$ with $u\le v$.
\item[$\cE_7'':$] For any $v \in \VV_{(-6r_3,-r_3)} \cap \HH_{[s+2\alpha r_4^{1/2},\infty)}\cap \Z^2$, there exists $u\in \VV_{(-6r_3,-r_3)} \cap \HH_{(\alpha r_4^{1/2},\infty)}\cap \Z^2$ such that $T^{+,\vee,0}_{u,v}>s$.
\end{enumerate}
Then under $\cE_7'\cap\cE_7''\cap\cE_6$, the two events above hold, thus $\cE_7$ holds.

\begin{figure}[hbt!]
     \centering
     \begin{subfigure}[t]{0.45\textwidth}
         \centering
\begin{tikzpicture}[line cap=round,line join=round,>=triangle 45,x=0.037cm,y=0.029cm]
\clip(-55,-45) rectangle (145.,140.);

\fill[line width=0.pt,color=yellow,fill=yellow,fill opacity=0.4]
(140,-56) -- (105,-28) -- (107.5,-26) -- (72.5,2) -- (75,4) -- (5,60) -- (2.5,58) -- (-32.5,86) -- (-35,84) -- (-70,112) -- (-70,300) -- (300,300) -- (300,-56) -- cycle;

\fill[line width=0.pt,color=yellow,fill=yellow,fill opacity=1]
(155,-44) -- (120,-16) -- (117.5,-18) -- (82.5,10) -- (80,8) -- (10,64) -- (12.5,66) -- (-22.5,94) -- (-20,96) -- (-55,124) -- (-70,300) -- (300,300) -- (300,-56) -- cycle;

\draw [line width=.8pt] [red] (5,60) -- (75,4);
\draw [line width=.8pt] [red] (140,-56) -- (105,-28);
\draw [line width=.8pt] [red] (72.5,2) -- (107.5,-26);
\draw [line width=.8pt] [red] (2.5,58) -- (-32.5,86);
\draw [line width=.8pt] [red] (-70,112) -- (-35,84);
\draw [line width=.8pt] [blue] (20+6.5,50+5.2) -- (40+4,34+7.2) ;

\draw [line width=.8pt] [blue] (120,130) -- (140-2.5,114+2) ;
\draw [line width=.8pt] [red] (110,122) -- (130-2.5,106+2) ;

\draw [fill=uuuuuu] (42.5,34) circle (1.0pt);

\begin{scriptsize}

\draw (42.5,34) node[anchor=west]{$\boo$};
\draw (30,42) node[anchor=north east] [red]{$P_1$};
\draw (85,-6) node[anchor=north east] [red]{$P_2$};
\draw (-5,66) node[anchor=north east] [red]{$P_2$};
\draw (110,-30) node[anchor=north east] [red]{$P_3$};
\draw (-35,90) node[anchor=north east] [red]{$P_3$};
\draw (115,125) node[anchor=north east] [red]{$P_*$};

\draw (-35,105) node[anchor=north] {$S_*\setminus S^*$};
\draw (-35,135) node[anchor=north] {$S^*$};
\end{scriptsize}

\end{tikzpicture}

         \caption{$\PP[\cE_7]$: assuming $\cE_6$, the event $\cE_7$ is implied by $\cE_7'\cap \cE_7''$, about passage times under the weights $\xi^{+,\vee,0}$. For $\PP[\cE_7']$, we need to upper bound the passage times from $\cup_{j\in\N}P_j$ to $P_*$; for $\PP[\cE_7'']$, we need to lower bound the passage times from around $\VV_{(-6r_3,-r_3)} \cap \HH_{3\alpha r_4^{1/2}/2}$ to $\VV_{(-6r_3,-r_3)} \cap \HH_{s+2\alpha r_4^{1/2}}$ (the blue segments).}
         \label{fig:ce7}
     \end{subfigure}
     \hfill
     \begin{subfigure}[t]{0.5\textwidth}
         \centering
\begin{tikzpicture}[line cap=round,line join=round,>=triangle 45,x=0.18cm,y=0.1cm]
\clip(5,24) rectangle (53,83);

\fill[line width=0.pt,color=yellow,fill=yellow,fill opacity=0.6]
(8-50,64+40) -- (12-50,64+40) -- (82+50,8-40) -- (78+50,8-40) -- cycle;

\draw [line width=.04pt,dashed] (-42,0) -- (208,200);
\draw [line width=.04pt,dashed] (-58,0) -- (192,200);
\draw [line width=.04pt,dashed] (-74,0) -- (176,200);
\draw [line width=.04pt,dashed] (-26,0) -- (224,200);

\draw [line width=.04pt,dashed] (14-50,64+40) -- (84+50,8-40);

\draw [line width=.9pt,color=brown] (20,57) -- (25,62) -- (30,64)-- (35,67)-- (40,71)-- (45,76) -- (50,80);
\draw [line width=.9pt,color=brown] (11,61.6) -- (16,63) -- (21,64.5) -- (27,67)-- (35,68)-- (40,71)-- (45,76) -- (50,80);
\draw [line width=.9pt,color=brown] (27,48.8) -- (29,53) -- (32,56) -- (35,61)-- (37,64)--  (40,71)-- (45,76) -- (50,80);

\draw [line width=.9pt,color=blue] (-20,85) --(-15,81) --(-10,79)--
(-5,75) --(0,70)-- (5,69)-- (10,64)-- (15,59)-- (20,57)-- (25,52) --(30,49)-- (35,43) --(40,40) --(45,35) --(50,32)-- (55,27) --(60,21) --(65,18)-- (70,14)-- (75,10)-- (80,5)-- (85,2)-- (90,0);

\draw [fill=uuuuuu] (50,32) circle (1.0pt);
\draw [color=brown,fill=brown] (50,80) circle (1.0pt);
\draw [color=brown,fill=brown] (20,57) circle (1.0pt);
\draw [color=brown,fill=brown] (22,57.6) circle (1.0pt);
\draw [color=brown,fill=brown] (11,61.6) circle (1.0pt);
\draw [color=brown,fill=brown] (27,48.8) circle (1.0pt);

\begin{scriptsize}

\draw (50,32) node[anchor=south]{$\boo$};
\draw (50,80) node[anchor=south,brown]{$u_1$};
\draw (22,57.6) node[anchor=west,brown]{$u_1'$};
\draw (20,57) node[anchor=north,brown]{$u_3$};
\draw (11,61.6) node[anchor=north,brown]{$u_4$};
\draw (27,48.8) node[anchor=north,brown]{$u_5$};

\draw (40,40) node[anchor=north,blue]{$\partial I_0^+$};
\draw (5,66) node[anchor=north west]{$\VV_{-5r_3-3r_2}$};
\draw (5,66-12.8) node[anchor=north west]{$\VV_{-5r_3-r_2}$};
\draw (5,66-25.6) node[anchor=north west]{$\VV_{-5r_3+r_2}$};
\draw (5,66-37.4) node[anchor=north west]{$\VV_{-5r_3+3r_2}$};
\draw (45,36) node[anchor=south west]{$\HH_{2\alpha r_4^{1/2}}$};
\draw (50,32) node[anchor=north, orange]{$S_*\setminus S^*$};
\end{scriptsize}

\end{tikzpicture}

         \caption{$\PP[\cE_8]$: under $\cE_8'\cap \cE_6$, we must have that $u_3$ (the lower endpoint of $\Gamma_{I,u_1}^{+,\vee,0}$) is in $\VV_{(-5r_3-r_2, -5r_3+r_2)}$. Then if $\Gamma^{+,\vee,0}_{u_4,u_1}$ and $\Gamma^{+,\vee,0}_{u_5,u_1}$ below $\HH_{2\alpha r_4^{1/2}}$ are contained in $\VV_{(-5r_3-3r_2,-5r_3-r_2)}$ and $\VV_{(-5r_3+r_2,-5r_3+3r_2)}$, respectively, $\Gamma_{I,u_1}^{+,\vee,0}$ is sandwiched between them, and the transversal fluctuation of $\Gamma_{I,u_1}^{+,\vee,0}$ is controlled by $\Gamma^{+,\vee,0}_{u_4,u_1}$ and $\Gamma^{+,\vee,0}_{u_5,u_1}$.
         }
         \label{fig:ce8}
     \end{subfigure}
        \caption{Illustrations of bounding $\PP[\cE_7]$ and $\PP[\cE_8]$ in the proof of Lemma \ref{lem:lbd-2event}.}
        \label{fig:ce78}
\end{figure}

We next lower bound the probabilities $\PP[\cE_7']$ and $\PP[\cE_7'']$. These bounds are deduced from the estimates of Theorem \ref{t:onepoint} and Proposition \ref{t:seg-to-seg}.

We first consider $\PP[\cE_7']$. For each $j\in\N$, we let $P_j$ be the collection of all vertices in $\Z^2$ that are within distance $1$ from $\VV_{(-jr_4,-(j-1)r_4]\cup [(j-1)r_4,jr_4)}\cap \HH_{-j\alpha r_4^{1/2}}$, and let $P_*$ be the collection of all vertices in $\Z^2$ that are within distance $1$ from $\VV_{(-6r_3,-r_3)} \cap \HH_{s-2\alpha r_4^{1/2}}$.
To lower bound $\PP[\cE_7']$, we just need to consider $T^{+,\vee,0}_{u,v}$, for all $u\in \cup_{j\in\N}P_j$ and $v \in P_*$ (see Figure \ref{fig:ce7}).
We note that for any $j\in\N$ and any $u\in P_j$, $v\in P_*$ with $u\le v$, if we write $(a,b)=v-u$ we have 
\begin{equation}  \label{eq:bdee7}
(\sqrt{a}+\sqrt{b})^2<s-c\alpha r_4^{1/2}-c(j-1)^2r_4^2s^{-1}.
\end{equation}
For $j>csr_4^{-1}+1$, we apply \eqref{e:wslope} in Theorem \ref{t:onepoint} to each $u\in P_j$ and $v\in P_*$ and take a union bound, to conclude that
\[
\PP[T^{+,\vee,0}_{u,v} \le s,\; \forall u\in P_j, v\in P_*, u\le v, j>csr_4^{-1}+1] > 1- Csr_3e^{-c\sqrt{s}}.
\]
For any $j\le csr_4^{-1}+1$, the slope of $v-u$ for any $u\in P_j$ and $v\in P_*$ is bounded away from $0$ and $\infty$.
Thus we can split $P_j$ and $P_*$ into $Cr_4s^{-2/3}$ and $Cr_3s^{-2/3}$ segments of length $<Cs^{2/3}$, and apply Proposition \ref{t:seg-to-seg}.
Note that using \eqref{e:mean} from Theorem \ref{t:onepoint} and \eqref{eq:bdee7}, we have that $\E[T^{+,\vee,0}_{u,v}]<s-cj^2\alpha^3 s^{1/3}$ for any $u\in P_j$ and $v\in P_*$.
We then conclude that
\[
\PP[T^{+,\vee,0}_{u,v} \le s,\; \forall u\in P_j, v\in P_*] > 1-C(r_3s^{-2/3})(r_4s^{-2/3})e^{-cj^2\alpha^3}.
\]
Thus we have that \[\PP[\cE_7']>1- Csr_3e^{-c\sqrt{s}}-Cr_3r_4s^{-4/3}\sum_{j\in \N}e^{-cj^2\alpha^3}=1- Cs^{5/3}\alpha^3 e^{-c\sqrt{s}}-C\alpha^7\sum_{j\in \N}e^{-cj^2\alpha^3}.\]

For $\PP[\cE_7'']$, we need to consider $T^{+,\vee,0}_{u,v}$, for all $v\in \Z^2\cap \VV_{(-6r_3,-r_3)}$ within distance $1$ from $\HH_{s+2\alpha r_4^{1/2}}$, and $u\in \VV_{(-6r_3,-r_3)} \cap \HH_{(\alpha r_4^{1/2},\infty)}\cap \Z^2$ within distance $1$ from $v-(s+\alpha r_4^{1/2}/2)\brho$ (see Figure \ref{fig:ce7}).
For such $u$ and $v$, the slope of $v-u$ is bounded away from $0$ and $\infty$.
By \eqref{e:mean} we have $\E[T^{+,\vee,0}_{u,v}] > s+c\alpha^3 s^{1/3}$.
We then apply Proposition \ref{t:seg-to-seg} by covering all such $u,v$ with $Cr_3s^{-2/3}$ parallelograms of size $Cs\times Cs^{2/3}$, and we conclude that $\PP[\cE_7'']>1-Cr_3s^{-2/3}e^{-c\alpha^3}=1-C\alpha^3e^{-c\alpha^3}$.

In summary and using the fact that $Cs^{2/3}<r<s^{2/3+0.01}$ from the statement of Lemma \ref{lem:couple-s} (thus $\alpha>C$ and $\alpha<s^{0.002}$), we have $\PP[\cE_6\setminus \cE_7]<Ce^{-c\alpha^3}$.
\\

\noindent\textbf{Bound $\PP[\cE_8]$.} We denote $u_3$ as the lower endpoint of $\Gamma_{I,u_1}^{+,\vee,0}$.
Consider the event $\cE_8'$, where 
\begin{itemize}
    \item for any $u\in (S_* \setminus \VV_{(-5r_3-r_2, -5r_3+r_2)})\cap \Z^2$, there is $T^{+,\vee,0}_{u,u_1} < 2s-4\alpha r_4^{1/2}$,
    \item $T^{+,\vee,0}_{u_1',u_1} > 2s-4\alpha r_4^{1/2}$, where $u_1'$ is the intersection of $\HH_{2\alpha r_4^{1/2}}$ with $\VV_{-5r_3}$ (rounded to the nearest lattice vertex). Note that $u_1'\in S^* \cap \VV_{(-5r_3-r_2, -5r_3+r_2)}$.
\end{itemize}
Under $\cE_8'\cap \cE_6$ we must have $u_3 \in \VV_{(-5r_3-r_2, -5r_3+r_2)}$, since $T_{u_3,u_1}^{+,\vee,0}$ is the maximum passage time from $I_0^+$ to $u_1$ (see Figure \ref{fig:ce8}).
We can deduce that $\PP[\cE_8']>1-Ce^{-c\alpha^3}$ similar to how $\PP[\cE_7']$ and $\PP[\cE_7'']$ are bounded above using Theorem \ref{t:onepoint} and Proposition \ref{t:seg-to-seg}, and we omit the details.

Now we take $u_4, u_5$ as the intersection of $\HH_{-\alpha r_4^{1/2}}$ with $\VV_{-5r_3-2r_2}$ and $\VV_{-5r_3+2r_2}$, respectively (rounded to the nearest lattice vertex, see Figure \ref{fig:ce8}).
Consider $\Gamma^{+,\vee,0}_{u_4,u_1}$ and $\Gamma^{+,\vee,0}_{u_5,u_1}$.
By Corollary \ref{cor:trans-fluc-comb} we have
\begin{equation}  \label{eq:trfc1}
\PP[\Gamma^{+,\vee,0}_{u_4,u_1} \cap \HH_{(-\alpha r_4^{1/2}, 2\alpha r_4^{1/2})} \subset \VV_{(-5r_3-3r_2,-5r_3-r_2)}] > 1-Ce^{-c r_2^3 \alpha^{-2}r_4^{-1}},    
\end{equation}
\begin{equation}  \label{eq:trfc2}
\PP[\Gamma^{+,\vee,0}_{u_5,u_1} \cap \HH_{(-\alpha r_4^{1/2}, 2\alpha r_4^{1/2})} \subset \VV_{(-5r_3+r_2,-5r_3+3r_2)}] > 1-Ce^{-c r_2^3 \alpha^{-2}r_4^{-1}},    
\end{equation}
and by Lemma \ref{lem:trans-fluc-uni} we have
\begin{equation}  \label{eq:trfc3}
\PP[\Gamma^{+,\vee,0}_{u_4,u_1} \subset \VV_{(-6r_3,-4r_3)}], \;\PP[\Gamma^{+,\vee,0}_{u_5,u_1} \subset \VV_{(-6r_3,-4r_3)}] > 1-Ce^{-c r_3^3 s^{-2}}. \end{equation}
When $\cE_8'\cap \cE_6$ happens, we have $u_3 \in \VV_{(-5r_3-r_2, -5r_3+r_2)} \cap \HH_{(-\alpha r_4^{1/2}, 2\alpha r_4^{1/2})}$.
If the events in the left-hand side of \eqref{eq:trfc1} and \eqref{eq:trfc2} also happen, 
we must have that $\Gamma_{I,u_1}^{+,\vee,0}=\Gamma^{+,\vee,0}_{u_3,u_1}$ is between $\Gamma^{+,\vee,0}_{u_4,u_1}$ and $\Gamma^{+,\vee,0}_{u_5,u_1}$, by ordering of geodesics (Lemma \ref{l:ordering}).
If in addition the event in the left-hand side of \eqref{eq:trfc3} happens, we have
$\Gamma_{I,u_1}^{+,\vee,0}=\Gamma^{+,\vee,0}_{u_3,u_1} \subset \VV_{(-6r_3,-4r_3)}$.
We can use similar arguments to study the event  $\Gamma_{I,u_2}^{+,\vee,0} \subset \VV_{(-3r_3,-r_3)}$.
Then with $\PP[\cE_8']>1-Ce^{-c\alpha^3}$ we conclude that $\PP[\cE_6\setminus\cE_8]<Ce^{-c\alpha^3}+Ce^{-c r_2^3 \alpha^{-2}r_4^{-1}} + Ce^{-c r_3^3 s^{-2}} = Ce^{-c\alpha^3}+Ce^{-c s^{4/3}} + Ce^{-c \alpha^9}$.
Using $Cs^{2/3}<r<s^{2/3+0.01}$ (from the statement of Lemma \ref{lem:couple-s}), this is bounded by $Ce^{-c\alpha^3}$.
\\

Putting together the bounds for $\PP[\cE_5], \PP[\cE_6], \PP[\cE_6\setminus \cE_7], \PP[\cE_6\setminus\cE_8]$ we conclude that $\PP[\cE_3]>1-C\alpha^{-1/2}$.
\end{proof}

\begin{proof}[Proof of Lemma \ref{lem:lbd-2event4}]
We again use $c,C>0$ to denote small and large enough constants, and their values can change from line to line.
We consider three events.
\begin{enumerate}
\item[$\cE_9:$] $|M-(1-2\rho)s|<r_2$.
\item[$\cE_{10}:$] for any $u=(a,b)$ with $ad(u)=a-b<(1-2\rho)s+r_2-r$ and $a\ge (1-\rho)^2s-r_4-1$, we have $u\not\in I_s^+$.
\item[$\cE_{11}:$] for any $u=(a,b)$ with $a< (1-\rho)^2s-r_4$ and $u-(0,1)\in \VV_{[-6r_3,\infty)}$, we have $u\in I_s^+$.
\end{enumerate}
We have that $\cE_9\cap\cE_{10}\cap\cE_{11}\subset \cE_4$.
Indeed, $\cE_9\cap \cE_{10}$ implies that $\partial I_s^+$ is disjoint from
\[
\{(a,b)\in\Z^2:\; a-b\le M-r,\; a\ge (1-\rho)^2s-r_4-1\};
\]
and $\cE_{11}$ implies that $\partial I_s^+$ is disjoint from
\[
\{(a,b)\in\Z^2:\; (a+1,b) \in \VV_{[-6r_3,\infty)}, \; a< (1-\rho)^2s-r_4-1\},
\]
since this set shifted by $(1,1)$ is contained in $I_s^+$ by $\cE_{11}$.
See Figure \ref{fig:2event-inc} for an illustration of these regions.
Thus under $\cE_9\cap\cE_{10}\cap\cE_{11}$, for any  $u=(a,b)\in\partial I_s^+$ with $ad(u)\le M-r$ we must have $a< (1-\rho)^2s-r_4-1$ and $u+(1,0) \in \VV_{(-\infty, -6r_3)}$.
So we conclude that $\cE_9\cap\cE_{10}\cap\cE_{11}\subset \cE_4$, and it remains to lower bound $\PP[\cE_9]$, $\PP[\cE_{10}]$,  and $\PP[\cE_{11}]$.
\\

\noindent\textbf{Bound $\PP[\cE_9]$.}
By Lemma \ref{lem:buse-opti}, $p_s^+-(1,0)$ or $p_s^+-(0,1)$ is the last vertex in $\{u\in \Gamma_\boo^+:T_{\boo,u}^+\le s\}$ since $p_s^+$ is the last vertex in $\{u\in \Gamma_\boo^+:\bG^+(u)\le s\}$.
So by Corollary \ref{cor:lem:semi-inf-loc} we have $\PP[|ad(p_s^+)-(1-2\rho)s|<r_2]>1-Ce^{-cr_2^3s^{-2}}$.
Then by Lemma \ref{lem:prob-restrict-ab} we have that $\PP[\cE_9]>1-Ce^{-cr_2^3s^{-2}}-Ce^{-c\alpha^3}=1-Ce^{-c\alpha^6}-Ce^{-c\alpha^3}$.\\

To bound $\PP[\cE_{10}]$ and $\PP[\cE_{11}]$, we just need to bound the function $L^+$ at certain vertices.
For this, we recall the event $\cE_6$ and sets $S_*$, $S^*$ from the proof of Lemma \ref{lem:lbd-2event}.\\

\noindent\textbf{Bound $\PP[\cE_{10}]$.}
We take $u^*=(a^*,b^*)$ where $a^*=\lceil (1-\rho)^2s-r_4-1\rceil$ and $b^*=a^*-\lceil (1-2\rho)s+r_2-r \rceil$, then $\cE_{10}$ is equivalent to $L^+(u^*)=T^{+,\vee,0}_{I,u^*}>s$.
Denote ${u^*_-}=(\lfloor r_4 \rfloor,\lfloor r_4\rfloor)$. As $u^*_-\in S^*$, under $\cE_6$ we have $u_-^* \not\in I_0^+$.
Thus under $\cE_6\setminus \cE_{10}$ we have 
$T^{+,\vee,0}_{{u^*_-},u^*}\le s$. Then $\PP[\cE_6\setminus \cE_{10}] \le \PP[T^{+,\vee,0}_{{u^*_-},u^*}\le s] < Ce^{-cr^3/s}$, where the last inequality is by the fact that $(\sqrt{a_*-\lfloor r_4 \rfloor}+\sqrt{b_*-\lfloor r_4 \rfloor})^2>s+cr$ and \eqref{e:slope} in Theorem \ref{t:onepoint}.\\

\noindent\textbf{Bound $\PP[\cE_{11}]$.}
Let $u_*=(a_*,b_*)$ where $a_*=\lceil (1-\rho)^2s-r_4-1\rceil$, and $b_*$ is the largest integer such that $u_*-(0,1)\in \VV_{[-6r_3,\infty)}$. Then $\cE_{11}$ is equivalent to that $L^+(u)=T^{+,v,0}_{I,u_*}\le s$.
Under $\cE_6\setminus \cE_{11}$ we have that there is some $u\in S_*$, $u\le u_*$, such that $T^{+,v,0}_{u,u_*}>s$.
We note that for any $u\in S_*$ with $u\le u_*$, if we let $(a,b)=u_*-u$, there is $(\sqrt{a}+\sqrt{b})^2<s-cr_4$.
Then by \eqref{e:wslope} in Theorem \ref{t:onepoint} and a union bound (over all $u\in S_*$ such that $u\in u_*+\Z_{\le 0}^2$ and $u-(1,1)\not\in S_*$), we have $\PP[\cE_6\setminus \cE_{11}]<Cse^{-cr_4s^{-1/2}}$.\\

Putting together the bounds for $\PP[\cE_9], \PP[\cE_6\setminus \cE_{10}], \PP[\cE_6\setminus \cE_{11}]$ and the bound for $\PP[\cE_6]$ in the proof of Lemma \ref{lem:lbd-2event}, and using the fact that $Cs^{2/3}<r<s^{2/3+0.01}$ from the statement of Lemma \ref{lem:couple-s} (thus $\alpha>C$ and $ Ce^{-cr^3/s}, Cse^{-cr_4s^{-1/2}}<Ce^{-c\alpha}$), we conclude that $\PP[\cE_4]>1-Ce^{-c\alpha}$.
\end{proof}

\section{In probability convergence of empirical environments}  \label{sec:weak-cov}
In this section we prove in probability convergence versions of the main results Theorem \ref{thm:finite} and Theorem \ref{thm:semi-infinite}.
The semi-infinite geodesic one (Theorem \ref{thm:semi-infinite-in-prob} below) follows quickly from the convergence of TASEP as seen from an isolated second-class particle (Proposition \ref{prop:converge-2nd-class-tasep} or Theorem \ref{thm:cov-phi}), and ergodicity of the stationary process (Proposition \ref{prop:ergodic-2nd-class-tasep}).
The finite geodesic one (Theorem \ref{thm:finite-slope}) is via geometric arguments, specifically, covering finite geodesics by semi-infinite geodesics.

\subsection{Semi-infinite geodesics}  \label{ssec:semi-inf-weak}

We start with convergence along semi-infinite geodesics and giving a weak version of Theorem \ref{thm:semi-infinite}.
\begin{theorem}  \label{thm:semi-infinite-in-prob}
For any bounded continuous function $f:\R^{\Z^2}\times \{0,1\}^{\Z^2}\to \R$, we have $\mu_{\boo;r}(f) \to \nu(f)$ in probability as $r\to \infty$.
\end{theorem}
We let $(\eta_t^*)_{t\ge 0}$ be the process of TASEP starting from i.i.d.\;Bernoulli$(\rho)$ on $\Z\setminus\{0\}$, and $\eta_0^*(0)=*$.
Then recall (from Section \ref{ssec:cv-avg-tasep}) that $\eta_t^*(l_t+\cdot)\sim \Phi_t$, for $l_t$ being the location of the second-class particle at time $t$.
We also let $\boeta^*=(\zeta^*_t)_{t\in\R}$ to be the stationary process of TASEP as seen from an isolated second-class, i.e.\;for each $t$ we have $\zeta^*_t\sim \Psi$ (defined in Section \ref{ssec:2cp-sta-def}).

For any process $P=(P_w)_{w\in\R}$ and $t\in\R$, we denote $\sT_t P$ as the process $(P_{t+w})_{w\in\R}$.
By Lemmas \ref{lem:compete-inter-2nd-class} and \ref{lem:reweight}, we can deduce Theorem \ref{thm:semi-infinite-in-prob} from the following result.
To make things well-defined, we let $\bfeta^*=(\eta_t^*(l_t+\cdot))_{t\in\R}$, such that  $\eta_t^* = \eta_0^*$ and $l_t=0$ for each $t<0$.
Let $\{0,1,*\}^{\Z\times\R}$ be equipped with the product topology.
\begin{prop}  \label{prop:tasep-in-prob}
For any bounded and continuous function $f:\{0,1,*\}^{\Z\times\R}\to\R$, we have
\[
T^{-1}\int_0^T f(\sT_t \bfeta^*) dt \to \E[f(\boeta^*)]
\]
in probability as $T\to\infty$.
\end{prop}
By Birkhoff's Ergodic Theorem, this proposition follows from Proposition \ref{prop:converge-2nd-class-tasep} or Theorem \ref{thm:cov-phi}, and Proposition \ref{prop:ergodic-2nd-class-tasep}.
\begin{proof} [Proof of Proposition \ref{prop:tasep-in-prob}]
Without loss of generality we assume that $0\le f \le 1$, and for some $s>0$ it is measurable with respect to the $\sigma$-algebra generated by  $A\times \{0,1\}^{\Z\times (-\infty, -s)\cup (s,\infty)}$ for all measurable $A\subset \{0,1\}^{\Z\times [-s,s]}$.
Take any $\delta>0$, then by Birkhoff's Ergodic Theorem and Proposition \ref{prop:ergodic-2nd-class-tasep}, we can find $r$ large enough such that
$\PP\left[\left|r^{-1}\int_0^{r} f(\sT_t \boeta^*) dt -\E[f(\boeta^*)]\right| > \delta\right] <\delta$.

For each $t\ge 0$, denote $\chi_t = \don\left[\left|r^{-1}\int_{t+s}^{t+s+r} f(\sT_w \bfeta^*) dw -\E[f(\boeta^*)]\right| \ge \delta\right]$.
Let $F: \eta\mapsto \E[\chi_t \mid \eta_t^*=\eta]$, then this $F$ is the same for all $t\ge 0$, and is an upper semi-continuous function on the space $\{\eta: \eta(0)=*, \eta(x)\in\{0,1\}, \forall x\neq 0\}\subset \{0,1,*\}^\Z$ since $f$ is continuous.
Then by Theorem \ref{thm:cov-phi} we have
\begin{multline*}
\limsup_{N\to\infty}N^{-1}\sum_{i=0}^{N-1} \E[\chi_{ir}]
=\limsup_{N\to\infty}N^{-1}\sum_{i=0}^{N-1} \E[F(\eta_{ir}^*)] \le \E[F(\zeta_0^*)] \\
= \PP\left[\left|r^{-1}\int_0^{r} f(\sT_t \boeta^*) dt -\E[f(\boeta^*)]\right| \ge \delta\right] <\delta.    
\end{multline*}
This implies that for any $N$ large enough, we have $\PP[\sum_{i=0}^{N-1} \chi_{ir} > \sqrt{\delta}N] <\sqrt{\delta}$, thus
\[
\PP\left[\left| (Nr)^{-1}\int_s^{Nr+s} f(\sT_t \bfeta^*) dt -\E[f(\boeta^*)]\right| > \sqrt{\delta} + \delta \right] < \sqrt{\delta},
\]
which implies our conclusion since $\delta>0$ is arbitrary.
\end{proof}

\subsection{From semi-infinite geodesics to point-to-point geodesics}  \label{sec:finite-cov}
From the in probability convergence along semi-infinite geodesics (Theorem \ref{thm:semi-infinite-in-prob}), we deduce the following in probability convergence along finite geodesics. It can be viewed as a weak version of Theorem \ref{thm:finite}.

Recall (from Section \ref{ssec:elpps}) that we let $\langle a, b\rangle_\rho = \left(\left\lfloor\frac{2(1-\rho)^2a}{\rho^2+(1-\rho)^2}\right\rfloor+b, \left\lceil\frac{2\rho^2a}{\rho^2+(1-\rho)^2}\right\rceil-b\right)$.
Since $\rho$ is fixed, for the rest of this paper we also write $\langle a, b\rangle=\langle a, b\rangle_\rho$.
\begin{theorem}  \label{thm:finite-slope}
Let $\{b_n\}_{n\in\N}$ be a sequence of integers such that $\limsup_{n\to\infty} n^{-2/3}|b_n| < \infty$.
Then for any bounded continuous function $f:\R^{\Z^2}\times \{0,1\}^{\Z^2}\to \R$, we have $\mu_{\boo,\langle n, b_n\rangle_{}}(f) \to \nu(f)$ in probability as $n\to\infty$.
\end{theorem}
We now explain the strategy of proving this theorem.
The general idea is to cover the finite geodesic $\Gamma_{\boo,\langle n, b_n\rangle_{}}$ with a semi-infinite geodesic. More precisely, for any $\epsilon>0$, we construct an event that depends only on the i.i.d.\;random weights $\xi$ on or above $\LL_n$, such that (1) this event happens with positive probability (lower bounded uniformly in $n$) and (2) assuming this event, with high probability a $1-\epsilon$ portion of $\Gamma_{\boo,\langle n, b_n\rangle_{}}$ is contained in $\Gamma_\boo$.
Then by Theorem \ref{thm:semi-infinite-in-prob}, conditioned on this event the empirical environment $\mu_{\boo,\langle n, b_n\rangle_{}}$ would be `$\epsilon$-close' to $\nu$, with high probability for $n$ large enough. 
On the other hand, since $\mu_{\boo,\langle n, b_n\rangle_{}}$ depends mainly on the random i.i.d.\;weights $\xi$ below $\LL_n$, it is roughly `independent' of the constructed event, so it would always be close to $\nu$, with high probability for $n$ large enough. 

We start by describing the event. Recall $\bB$ (and also $\bG$), the Busemann function in direction $\brho$. The event basically says that the Busemann function $\bB(\langle n, b_n+b \rangle_{}, \langle n,b_n\rangle_{})$ decays fast when $b$ is slightly away from $0$. By Lemma \ref{lem:buse-opti} this can force $\Gamma_\boo$ to intersect $\LL_n$ near $\langle n, b_n+b \rangle_{}$, and that $\Gamma_{\boo,\langle n, b_n\rangle_{}}$ overlaps with $\Gamma_\boo$ can be deduced using coalescence and 
ordering of geodesics (Proposition \ref{prop:coalesce} and Lemma \ref{l:ordering}).

We now formally define this event and study its probability. For simplicity of notations, we shift it by $-\langle n,b_n\rangle_{}$ and look at the Busemann function on $\LL_0$.
Let $\cE_{h,n}$ denote the following event: for any $b \in \Z$ with $h^{-1}n^{2/3}<|b|<hn^{2/3}$, there is $\bG(\langle 0,b\rangle) + b(\rho^{-1}-(1-\rho)^{-1}) > hn^{1/3}$;
and for $b\in\Z$ with $|b|\ge hn^{2/3}$, there is
$\bG(\langle 0,b\rangle) + b(\rho^{-1}-(1-\rho)^{-1}) > -|b|n^{-1/3}$.
We show that its probability is lower bounded uniformly in $n$,
\begin{lemma}\label{lem:randomwalk-shape-lowbd}
For any $h>1$, there is $\delta>0$ such that $\PP[\cE_{h,n}]>\delta$ for all $n$ large enough.
\end{lemma}

\begin{proof}
Denote $F(b)= -\bG(\langle 0,b\rangle)-b(\rho^{-1}-(1-\rho)^{-1})$, then $F$ is a (two-sided) random walk, where each step is centered with exponential tail.
By independence of all the steps, we have
\[
\begin{split}
\PP[\cE_{h,n}]
\ge & 
\PP\left[\max_{h^{-1}n^{2/3}<|b|<hn^{2/3}} F(b) < -hn^{1/3}\right]\\
&\times\PP\left[\max_{b\ge hn^{2/3}}F(b) -F(\lfloor hn^{2/3}\rfloor) - bn^{-1/3}<hn^{1/3}\right]
\\
&\times\PP\left[\max_{b\le -hn^{2/3}}F(b) -F(-\lfloor hn^{2/3}\rfloor) + bn^{-1/3}<hn^{1/3}\right].
\end{split}
\]
As the process $F$ converges to a (two-sided) Brownian motion (weakly in the uniform topology) in compact sets, the first factor in the right-hand side is lower bounded by a positive constant.
We next lower bound the factor in the second line, and the third line could be lower bounded in a similar way.
The second line is at least
\[
\begin{split}
\PP\left[\max_{b\in\N}F(b)- bn^{-1/3}<hn^{1/3}\right]
\ge & \PP\left[\max_{b \in \llbracket 0, In^{2/3}\rrbracket}F(b) - bn^{-1/3}<hn^{1/3}\right]
\\
&- \sum_{i=I}^\infty 
\PP\left[\max_{b\in\llbracket in^{2/3}, (i+1)n^{2/3}\rrbracket}F(b) \ge (i+h)n^{1/3}\right].    
\end{split}
\]
where $I$ is a large integer.
As $n\to\infty$, the first term in the right-hand side converges to the probability that a Brownian motion is bounded below a (sloped) line in $[0,I]$, and such probability is lower bounded uniformly in $I$.
For the sum in the second line, the term for each $i$ is upper bounded by
\[
\begin{split}
&
\PP\left[F(\lceil in^{2/3}\rceil) \ge (i+h)n^{1/3}/2\right] + \PP\left[\max_{b \in\llbracket 0,  n^{2/3}\rrbracket}F(b) \ge (i+h)n^{1/3}/2\right]
\\
\le &
\PP\left[F(\lceil in^{2/3}\rceil) \ge (i+h)n^{1/3}/2\right] + 2\PP\left[F(\lfloor n^{2/3}\rfloor) \ge (i+h)n^{1/3}/2\right],
\end{split}
\]
where the inequality is by the reflection principle.
By a Bernstein type estimate for the sum of independent random variables with exponential tails, this could be bounded by $Ce^{-ci}$ for some $c,C>0$, independent of $n$. Thus by taking $I$ large enough the conclusion follows.
\end{proof}

\begin{figure}[hbt!]
    \centering
\begin{tikzpicture}[line cap=round,line join=round,>=triangle 45,x=6cm,y=6cm]
\clip(-0.55,-0.15) rectangle (1.55,1.55);

\fill[line width=0.pt,color=yellow,fill=yellow,fill opacity=0.6]
(0.15,1.55) -- (1.55,0.15) -- (1.55,1.55) -- cycle;

\draw [line width=.4pt] (0.15,1.55) -- (1.55,0.15);
\draw [line width=.4pt] (-0.05,1.35) -- (1.35,-0.05);

\draw (1.,1.55) node[anchor=north east]{$\cE_{h,n}'$};
\draw (-0.15,0.7) node[anchor=west]{$\cA_{h,n}, \cB_{h,n}$};

\begin{scriptsize}
\draw (1.25,0.1) node[anchor=north east]{$\LL_{\lfloor (1-h^{-1})n\rfloor}$};
\draw (1.45,0.3) node[anchor=north east]{$\LL_n$};

\draw (0,0) node[anchor=east]{$\boo$};

\draw (0.85,0.85) node[anchor=west]{$\langle n, b_n\rangle_{}$};

\draw (0.73,0.9) node[anchor=east]{$\Gamma_{\boo,\langle n, \lfloor b_n-h^{-1}n^{2/3} \rfloor\rangle_{}}$};
\draw (0.77,0.7) node[anchor=west]{$\Gamma_{\boo,\langle n, \lceil b_n+h^{-1}n^{2/3} \rceil\rangle_{}}$};
\draw (1.05,1.05) node[anchor=south east]{$\Gamma_\boo$};
\end{scriptsize}

\draw [red] plot [smooth] coordinates {(1.56,1.48) (1.44,1.42) (1.34,1.37) (1.26,1.29) (1.17,1.17) (1.09,1.14) (1.05,1.05) (0.97,1.02) (0.91,0.97) (0.85,0.94) (0.78,0.89) (0.71,0.85) (0.68,0.78) (0.66,0.75) (0.6,0.69) (0.58,0.64)};
\draw [red] plot [smooth] coordinates {(0.85,0.85) (0.79,0.82) (0.73,0.81) (0.68,0.78) (0.66,0.75) (0.6,0.69) (0.58,0.64)};

\draw [blue] plot [smooth] coordinates {(0.77,0.93) (0.73,0.89) (0.7,0.86) (0.66,0.78) (0.65,0.75) (0.6,0.69) (0.58,0.64) (0.52,0.6) (0.46,0.57) (0.39,0.55) (0.35,0.49) (0.31, 0.41) (0.27,0.36) (0.25,0.27) (0.23,0.19) (0.18,0.14) (0.14,0.09) (0.07,0.04) (0,0)};
\draw [blue] plot [smooth] coordinates {(0.93,0.77) (0.82,0.76) (0.78,0.75) (0.7,0.75) (0.65,0.74) (0.6,0.69) (0.58,0.64) (0.52,0.6) (0.46,0.57) (0.39,0.55) (0.35,0.49) (0.31, 0.41) (0.27,0.36) (0.25,0.27) (0.23,0.19) (0.18,0.14) (0.14,0.09) (0.07,0.04) (0,0)};
\draw [fill=uuuuuu] (0.,0.) circle (1.4pt);
\draw [fill=uuuuuu] (0.85,0.85) circle (1.4pt);

\draw [fill=uuuuuu] (0.93,0.77) circle (1.4pt);
\draw [fill=uuuuuu] (0.77,0.93) circle (1.4pt);

\end{tikzpicture}
\caption{An illustration of the proof of Theorem \ref{thm:finite-slope}. The event $\cE_{h,n}'$ is on the spiky behaviour of the Busemann function, the event $\cA_{h,n}$ is on passage times from $\boo$ to $\LL_n$, and the event $\cB_{h,n}$ is on coalescence of geodesics.
Under their intersection, most of $\Gamma_{\boo,\langle n,b_n\rangle_{}}$ is also in $\Gamma_\boo$.
The events $\cA_{h,n}$ and $\cB_{h,n}$ happen with high probability, and $\cE_{h,n}'$ happens with positive probability lower bounded uniformly in $n$.
The event $\cE_{h,n}'$ depends only on $\xi$ in the yellow region, while $\cA_{h,n}$ and $\cB_{h,n}$ depend only on $\xi$ in the remaining region (and roughly so does $\mu_{\boo,\langle n,b_n\rangle_{}}$).
}  \label{fig:fin-slo}
\end{figure}

\begin{proof}[Proof of Theorem \ref{thm:finite-slope}]
It suffices to show that, for any $s\in\N$ and any continuous $f:\R^{\llbracket -s,s \rrbracket^2}\times \{0,1\}^{\llbracket -s,s \rrbracket^2} \to [0,1]$, regarded as a function on $\R^{\Z^2}\times \{0,1\}^{\Z^2}$, there is $\mu_{\boo,\langle n, b_n\rangle_{}}(f) \to \nu(f)$ in probability.

In this proof we use $c, C>0$ to denote small and large enough constants, whose values may change from line to line.
We then have that $|b_n|<Cn^{2/3}$ for any $n\in\N$.
For simplicity of notations we denote $T_{u,v}^\bu = T_{u,v}-\xi(v)$ for any vertices $u\le v$.

We denote $\cE_{h,n}'$ as $\cE_{h,n}$ translated by $\langle n, b_n \rangle_{}$, i.e.\;$\cE_{h,n}'$ is the event where 
\[
\begin{split}
   \bB(\langle n, b \rangle_{}, \langle n,b_n\rangle_{}) <& (b-b_n)(\rho^{-1}-(1-\rho)^{-1}) - hn^{1/3}, \; \text{ for any } h^{-1}n^{2/3}<|b-b_n|<hn^{2/3}, \\
   \bB(\langle n, b \rangle_{}, \langle n,b_n\rangle_{}) <& (b-b_n)(\rho^{-1}-(1-\rho)^{-1}) + |b-b_n|n^{-1/3}, \; \text{ for any } |b-b_n|\ge hn^{2/3}.
\end{split}
\]
Denote $\langle n, b_n'\rangle_{}$ as the intersection of $\Gamma_\boo$ with $\LL_{n}$.
Take any $\epsilon>0$.
By Theorem \ref{thm:semi-infinite-in-prob}, for any $n$ large enough (depending on $\epsilon, f$), we have
\[
\PP[|\mu_{\boo,\langle n, b_n'\rangle_{}}(f) - \nu(f)| < \epsilon] > 1-\epsilon.
\]
By Lemma \ref{lem:randomwalk-shape-lowbd}, when $\epsilon$ is taken small enough depending on $h$, we have
\begin{equation}  \label{eq:nlgep}
\PP[|\mu_{\boo,\langle n, b_n'\rangle_{}}(f) - \nu(f)| < \epsilon
\mid \cE_{h,n}'] > 1-\sqrt{\epsilon}    
\end{equation}
for any $n$ large enough (depending on $h, \epsilon, f$).

We next study the overlap between $\Gamma_\boo$ and $\Gamma_{\boo,\langle n, b_n\rangle_{}}$, under the event $\cE_{h,n}'$.
We denote $\cA_{h,n}$ as the following event:
for any $b\in\Z$, we have
\begin{itemize}
    \item $T_{\boo,\langle n, b\rangle_{}}^\bu + b(\rho^{-1}-(1-\rho)^{-1}) > \E[T_{\boo,\langle n, 0\rangle_{}}]-hn^{1/3}/2$, if $|b-b_n|\le h^{-1}n^{2/3}$;
    \item $T_{\boo,\langle n, b\rangle_{}}^\bu + b(\rho^{-1}-(1-\rho)^{-1}) < \E[T_{\boo,\langle n, 0\rangle_{}}]+hn^{1/3}/2$, if $h^{-1}n^{2/3}<|b-b_n|< hn^{2/3}$;
    \item $T_{\boo,\langle n, b\rangle_{}}^\bu + b(\rho^{-1}-(1-\rho)^{-1}) < \E[T_{\boo,\langle n, 0\rangle_{}}]-hn^{1/3}/2-|b-b_n|n^{-1/3}$, if $|b-b_n|\ge hn^{2/3}$.
\end{itemize}
We have $\PP[\cA_{h,n}]>1 - e^{-ch}$ for $n$ and $h$ large enough.
This can be deduced by applying \eqref{e:wslope} in Theorem \ref{t:onepoint} to $T_{\boo,\langle n, b\rangle_{}}$ for each $b\in \llbracket -n, n\rrbracket$ with $|b|>(\rho^2\wedge (1-\rho)^2)n$, and splitting $\{\langle n,b\rangle : b \in \llbracket -(\rho^2\wedge (1-\rho)^2)n, (\rho^2\wedge (1-\rho)^2)n\rrbracket\}$ into segments of length $n^{2/3}$ and using Proposition \ref{t:seg-to-seg} with each one of them.

We also denote $\cB_{h,n}$ as the following event:
\[
\Gamma_{\boo,\langle n, \lfloor b_n-h^{-1}n^{2/3} \rfloor\rangle_{}} \cap \LL_{\lfloor (1-h^{-1})n\rfloor} = \Gamma_{\boo,\langle n, \lceil b_n+h^{-1}n^{2/3} \rceil\rangle_{}} \cap \LL_{\lfloor (1-h^{-1})n\rfloor}.
\]
By Proposition \ref{prop:coalesce}, we have $\PP[\cB_{h,n}]>1-Ch^{-1/3}$, for $h< cn^{2/3}$ and $h$ large enough.

Note that $\cA_{h,n}$ and $\cB_{h,n}$ only depend on the i.i.d.\;random weights $\xi$ below $\LL_n$, and $\cE_{h,n}'$ only depends on $\xi$ on or above $\LL_n$, so the events $\cA_{h,n}, \cB_{h,n}$ are independent of $\cE_{h,n}'$ (see Figure \ref{fig:fin-slo}).
Using that $\PP[\cA_{h,n}]>1 - e^{-ch}$, $\PP[\cB_{h,n}]>1-Ch^{-1/3}$, and \eqref{eq:nlgep}, for $n$ large enough (depending on $h, \epsilon, f$) we have
\[
\PP[\cA_{h,n}, \cB_{h,n}, |\mu_{\boo,\langle n, b_n'\rangle_{}}(f) - \nu(f)| < \epsilon 
\mid \cE_{h,n}']  > 1-\sqrt{\epsilon} - e^{-ch} -Ch^{-1/3}.
\]
Under $\cA_{h,n}\cap \cE_{h,n}'$,
we have
\[
T_{\boo,\langle n, b\rangle_{}}^\bu + \bB(\langle n, b \rangle_{}, \langle n,b_n\rangle_{}) < T_{\boo,\langle n, b_n\rangle_{}}^\bu,
\]
for any $b\in \Z$, $|b-b_n|>h^{-1}n^{2/3}$.
Thus there must be $|b_n-b_n'|\le h^{-1}n^{2/3}$ by Lemma \ref{lem:buse-opti}.
Then under $\cA_{h,n}\cap \cB_{h,n}\cap \cE_{h,n}'$, we must have $\Gamma_{\boo,\langle n, b_n\rangle_{}} \cap \LL_{\lfloor (1-h^{-1})n\rfloor} = \Gamma_{\boo,\langle n, b_n'\rangle_{}} \cap \LL_{\lfloor (1-h^{-1})n\rfloor}$ by ordering of geodesics (Lemma \ref{l:ordering}), and $|\mu_{\boo,\langle n, b_n'\rangle_{}}(f) - \nu(f)|<\epsilon$ implies that
\[
|\mu_{\boo,\langle n, b_n\rangle_{}}(f) - \nu(f)| < \epsilon + h^{-1}.
\]
So we have
\[
\PP[|\mu_{\boo,\langle n, b_n\rangle_{}}(f) - \nu(f)| < \epsilon + h^{-1}
\mid \cE_{h,n}'] > 1-\sqrt{\epsilon} - e^{-ch} -Ch^{-1/3}.
\]
Note that $\Gamma_{\boo,\langle n, b_n\rangle_{}}$ is determined by the weights $\xi$ below $\LL_n$, so it is independent of $\cE_{h,n}'$. For each $v \in \Gamma_{\boo,\langle n, b_n\rangle_{}}$ with $d(v)<2n-2s$, $f(v)$ is determined by the weights $\xi$ in $v+\llbracket -s, s\rrbracket^2$, so it is also independent of $\cE_{h,n}'$.
Thus we conclude that \[\PP[|\mu_{\boo,\langle n, b_n\rangle_{}}(f) - \nu(f)| < \epsilon + h^{-1}+s/n] > 1-\sqrt{\epsilon} - e^{-ch} -Ch^{-1/3}\]
for any $n$ large enough (depending on $h, \epsilon, f$). Since $h$ can be taken arbitrarily large and $\epsilon$ is any number small enough depending on $h$,
we conclude that $\mu_{\boo,\langle n, b_n\rangle_{}}(f) \to \nu(f)$ in probability.
\end{proof}

\section{Parallelogram uniform covering}  \label{sec:paracov}
The goal of this section is to prove the following upgraded version of Theorem \ref{thm:finite-slope}. It will be the key input for the next two sections.
\begin{prop}  \label{prop:uniform-conv}
For any $h>0$, $s\in\N$, and any bounded continuous continuous $f:\R^{\llbracket -s,s \rrbracket^2}\times \{0,1\}^{\llbracket -s,s \rrbracket^2} \to \R$, regarded as a function on $\R^{\Z^2}\times \{0,1\}^{\Z^2}$, we have
\[
\max_{a,b\in\Z,|a|,|b|<hn^{2/3}} \mu_{\langle 0,a \rangle_{}, \langle n,b \rangle_{}}(f), \min_{a,b\in\Z,|a|,|b|<hn^{2/3}} \mu_{\langle 0,a \rangle_{}, \langle n,b \rangle_{}}(f) \to \nu(f),\]
in probability.
\end{prop}
For simplicity of notations, below we write the proof for $\rho=1/2$, while the general $\rho$ case follows essentially verbatim.

We now explain our strategy.
We will take two families of vertices, $\fP_1$ and $\fP_2$, around the segment connecting $\langle 0,-hn^{2/3} \rangle$ and $\langle 0,hn^{2/3} \rangle$ and the segment connecting $\langle n,-hn^{2/3} \rangle$ and $\langle n,hn^{2/3} \rangle$, respectively.
Both $\fP_1$ and $\fP_2$ are finite, in the sense that their sizes do not increase as $n\to\infty$.
Then by Theorem \ref{thm:finite-slope}, when $n$ is large enough, with high probability, for any $u\in \fP_1$ and $v\in\fP_2$, $\mu_{u, v}(f)$ is close to $\nu(f)$.
We will show that with high probability, for any $|a|, |b| < hn^{2/3}$, the geodesic $\Gamma_{\langle 0, a\rangle, \langle n, b\rangle}$ is mostly covered by some $\Gamma_{u,v}$ with $u\in\fP_1$ and $v\in\fP_2$, thus $\mu_{\langle 0,a \rangle_{}, \langle n,b \rangle_{}}(f)$ is also close to $\nu(f)$.

The main task to do is to establish the covering statement.
To motivate our arguments, we start with the following attempt.
For $a^-<a^+$ and $b^-<b^+$, if the geodesics $\Gamma_{\langle 0, a^-\rangle, \langle n, b^-\rangle}$ and $\Gamma_{\langle 0, a^+\rangle, \langle n, b^+\rangle}$ coalescence near both ends, then they must mostly stay together; and by ordering of geodesics (Lemma \ref{l:ordering}), for any $a^-<a<a^+$ and $b^-<b<b^+$, the geodesic $\Gamma_{\langle 0, a\rangle, \langle n, b\rangle}$ must be covered by $\Gamma_{\langle 0, a^-\rangle, \langle n, b^-\rangle}$, except for a small portion.
By estimates on coalescence of geodesics (e.g. Proposition \ref{prop:coalesce}), if we let $b^+-b^-=a^+-a^-$ be in the order of $\delta_0 n^{2/3}$ (for some small $\delta_0>0$), the probability for $\Gamma_{\langle 0, a^-\rangle, \langle n, b^-\rangle}$ and $\Gamma_{\langle 0, a^+\rangle, \langle n, b^+\rangle}$ to stay disjoint within order $n$ distance from their endpoints is in the order of $\delta_0$.
Now we take $\fP_1$ and $\fP_2$ to be contained in the segment connecting $\langle 0,-hn^{2/3} \rangle$ and $\langle 0,hn^{2/3} \rangle$ and the segment connecting $\langle n,-hn^{2/3} \rangle$ and $\langle n,hn^{2/3} \rangle$, respectively.
Let these vertices split these two segments into $h\delta_0^{-1}$ many small segments, each of length $\delta_0 n^{-2/3}$.
By taking a union bound over all pairs of such small segments, we conclude that the probability of existing some $\Gamma_{\langle 0, a\rangle, \langle n, b\rangle}$ not being mostly covered (by one geodesic with two endpoints in $\fP_1$ and $\fP_2$) is upper bounded by $(\delta_0^{-1})^2\delta_0$, which is too large.

To resolve this issue, we need to get a better bound on the probability of the following event:
there exist some $a^-<a<a^+$ and $b^-<b<b^+$, such that the geodesic $\Gamma_{\langle 0, a\rangle, \langle n, b\rangle}$ is not mostly covered by any geodesic with endpoints in $\fP_1$ and $\fP_2$.
If this probability could be upper bounded by $\delta_0^{2+\epsilon}$ for some $\epsilon>0$ (rather than $\delta_0$), then by a union bound and sending $\delta_0\to 0$, the conclusion follows. 
Towards this, we need to take $\fP_1$ and $\fP_2$ larger (but still finite).
Instead of having them contained in $\LL_0$ and $\LL_n$, we let $\fP_1$ and $\fP_2$ have $h\delta_0^{-1} \times \delta_0^{-1}$ vertices in the rectangles $\{u: 0\le d(u) \le 2n/3, -2hn^{2/3}\le ad(u) \le 2hn^{2/3}\}$ and $\{u: 4n/3\le d(u) \le 2n, -2hn^{2/3}\le ad(u) \le 2hn^{2/3}\}$, respectively.
Fix some small $\kappa>0$.
Using ordering of geodesics (Lemma \ref{l:ordering}), and a union bound, the above task can roughly be reduced to proving the following statement.
For given $a^-, a^+$ and $b^-, b_+$ that are contained in $[-hn^{2/3}, hn^{2/3}]$ with $b^+-b^-=a^+-a^-$ in the order of $\delta_0 n^{2/3}$, 
the following event happens with probability in the order of at most $\delta_0^{2+\epsilon}$ for some $\epsilon>0$:
there exist $a^-<a<a^+$ and $b^-<b<b^+$, such that for any $u\in\fP_1$ and $v\in\fP_2$ in the same side of $\Gamma_{\langle 0, a\rangle, \langle n, b\rangle}$, $\Gamma_{\langle 0, a\rangle, \langle n, b\rangle}\cap \Gamma_{u,v}$ contains no vertex below $\LL_{2\kappa n}$.

Now let's consider the scenario where the above event happens.
Take any $v\in \fP_2$ that is within distance $\delta_0n^{2/3}$ to $\Gamma_{\langle 0, a\rangle, \langle n, b\rangle}$.
We find vertices $u_1, u_2, u_3, u_4, u_5$ in $\fP_1$, such that (1) they are between $\LL_{\kappa n}$ and $\LL_{2\kappa n}$; (2) these vertices are in the same side of $\Gamma_{\langle 0, a\rangle, \langle n, b\rangle}$ as $v$; (3) each is within distance $\delta_0n^{2/3}$ to $\Gamma_{\langle 0, a\rangle, \langle n, b\rangle}$.
Consider the geodesics from each of these vertices to $v$: these geodesics are disjoint from  $\Gamma_{\langle 0, a\rangle, \langle n, b\rangle}$ below $\LL_{2\kappa n}$, by the above event. We can show that (with high probability), any two geodesics cannot stay close to each other while being disjoint for a long distance. 
By choosing the vertices $u_5, u_4, u_3, u_2, u_1$ sequentially and in a multi-scale way (see Figure \ref{fig:uni-cov-pf} below for an illustration), we can actually find $\alpha_0$ with $\kappa<\alpha_0\in 2\kappa$, such that for $\Gamma_{u_i,v}$ with $i=1,2,3,4,5$ and $\Gamma_{\langle 0, a\rangle, \langle n, b\rangle}$, their intersections with $\LL_{\alpha_0 n}$ are far from each other (with distances in the order of at least $\delta_0^{1/150}n^{2/3}$).

However, using $\Gamma_{\langle 0, a\rangle, \langle n, b\rangle}$ and each $\Gamma_{u_i,v}$, (with high probability) one can construct a path from $\langle 0, a^-\rangle$ to $ \langle n, b^-\rangle$, and the difference between its passage time and $T_{\langle 0, a^-\rangle, \langle n, b^-\rangle}$ is at most in the order of $\delta_0^{1/2}n^{1/3}$.
Indeed, one can just mainly use the path of $\Gamma_{u_i, v}$, and switch to $\Gamma_{\langle 0, a\rangle, \langle n, b\rangle}$ only near $u_i$ and $v$, and switch to $\langle 0, a^-\rangle$ and $\langle n, b^-\rangle$ near the ends.
One can also just mainly use the path $\Gamma_{\langle 0, a\rangle, \langle n, b\rangle}$ and switch to $\langle 0, a^-\rangle$ and $\langle n, b^-\rangle$ near the ends.
This way we get in total $6$ paths from $\langle 0, a^-\rangle$ to $ \langle n, b^-\rangle$, each with total passage time at least $T_{\langle 0, a^-\rangle, \langle n, b^-\rangle}-\delta_0^{1/2}n^{1/3}$; and they intersect $\LL_{\alpha_0 n}$ at vertices far away from each other.
Now consider the optimal passage time from $\langle 0, a^-\rangle$ to $\langle n, b^-\rangle$ passing through $\langle \alpha_0 n, b'\rangle$, as a function of $b'$. This is roughly the sum of two independent point-to-line last-passage profiles (see Section \ref{ssec:multi-p} below).
Its scaling limit is known to behave like a Brownian motion, and the event that there are $6$ paths with near optimal passage times is reduced to that, for a Brownian motion in a compact interval one can find $6$ points such that their distances are at least in the order of $\delta_0^{1/150}$, and the Brownian motion values at these points are at least the maximum (of the Brownian motion) minus $\delta_0^{1/2}$.
This event has probability in the order of at most $(\delta_0^{1/2-1/300})^5=\delta_0^{5/2-1/60}$, which is smaller $\delta_0^{2+\epsilon}$ as needed (and this is also why we need to find $5$ alternative paths).

We now explain the organization of the remaining of this section.
We will first list some useful ingredients that will be useful in carrying out the above plan. The proofs of some of these ingredients are delayed to Section \ref{ssec:multi-p} and Section \ref{ssec:dis-path}.
Then we will define several events, each with a small probability. The main arguments are contained in the proof of Lemma \ref{lem:cov-geo} below, where we show that under the intersection of the complements of these events, every $\Gamma_{\langle 0, a\rangle, \langle n, b\rangle}$ is mostly covered by one geodesic in a finite family.
Finally we deduce Proposition \ref{prop:uniform-conv} using Lemma \ref{lem:cov-geo}.\\

\noindent\textbf{Ingredients:} 
The first one concerns continuity of the function $(a,b)\mapsto T_{\langle 0, a\rangle, \langle n, b\rangle}$.
\begin{lemma}  \label{lem:conti-passage-time}
There exist constants $c,C>0$ such that the following is true.
For $h>0$, $0<\theta < 1$, and $t>1$, we have
\[
\PP\left[ \max_{\substack{|a|, |a'|, |b|, |b'| < hn^{2/3} \\
|a-a'|,|b-b'|<\theta n^{2/3}}}
|T_{\langle 0, a\rangle, \langle n, b\rangle}-T_{\langle 0, a'\rangle, \langle n, b'\rangle}| > t\theta^{1/2-0.01}n^{1/3} + Ch\theta n^{1/3}
\right] < Che^{-ct}
\]
when $n$ is large enough (depending on $h,\theta,t$).
\end{lemma}
The proof of this lemma will be given in Section \ref{ssec:multi-p}.

We next state a bound on transversal fluctuations of geodesics.
It actually immediately follows from the results in Section \ref{ssec:elpps}, and we state it here mainly for the convenience of the proof of Proposition \ref{prop:uniform-conv}.
For vertices $u\le v$, and $0\le l \le d(v)-d(u)$, $t>1$, let $\cT^{u,v}_{l,t}$ be the event where $\Gamma_{u,v}$ below $u+\LL_l$ is not contained in a rectangle of width $2tl^{2/3}$, or $\Gamma_{u,v}$ above $v-\LL_l$ is not contained in a rectangle of width $2tl^{2/3}$ (see Figure \ref{fig:trans}).
Formally, we let $\Gamma_{u,v}$ be the event where there exists $w\in\Gamma_{u,v}$ with $d(u)\le d(w) \le d(u)+2l$ and $|ad(w)-ad(u)|\ge 2tl^{2/3}$,
or with $d(v)-2l\le d(w) \le d(v)$ and $|ad(w)-ad(v)|\ge 2tl^{2/3}$.
\begin{figure}[hbt!]
    \centering
\begin{tikzpicture}[line cap=round,line join=round,>=triangle 45,x=5cm,y=5cm]
\clip(-0.15,-0.15) rectangle (1.15,1.3);

\fill[line width=0.pt,color=green,fill=green,fill opacity=0.35]
(0.15,0.25) -- (0.25,0.15) -- (0.05,-0.05) -- (-0.05,0.05) -- cycle;
\fill[line width=0.pt,color=green,fill=green,fill opacity=0.35]
(0.9,1.2) -- (1.,1.1) -- (0.8,0.9) -- (0.7,1.) -- cycle;

\draw (0.5,-0.06) node[anchor=north east]{$\LL_l$};
\draw (1.13,0.6) node[anchor=north east]{$\LL_{n-l}$};

\draw [line width=.4pt] (-0.35,0.75) -- (0.75,-0.35);
\draw [line width=.4pt] (0.05,1.65) -- (1.35,0.35);

\draw (0,0) node[anchor=north]{$\boo$};
\draw (.95,1.15) node[anchor=south]{$\langle n,b\rangle$};

\draw [red] plot [smooth] coordinates {(0.95,1.15) (0.9,1.1) (0.87,1.04) (0.84,0.99) (0.78,0.93) (0.71,0.85) (0.68,0.78) (0.66,0.75) (0.6,0.69) (0.58,0.64) (0.52,0.6) (0.46,0.57) (0.39,0.55) (0.35,0.49) (0.31, 0.41) (0.27,0.36) (0.25,0.27) (0.23,0.19) (0.18,0.14) (0.14,0.09) (0.07,0.04) (0,0)};
\draw [fill=uuuuuu] (0.,0.) circle (1.0pt);
\draw [fill=uuuuuu] (0.95,1.15) circle (1.0pt);
\end{tikzpicture}
\caption{The complement of the event $\cT_{l,t}^{\boo,\langle n,b\rangle}$: the geodesic $\Gamma_{\boo,\langle n,b\rangle}$ is restricted within the green boxes with width $tl^{2/3}$, below $\LL_l$ or above $\LL_{n-l}$.}   \label{fig:trans}
\end{figure}

\begin{lemma}  \label{lem:trans-fluc}
For $h>0$, there exist constants $c,C>0$ such that the following is true.
For any $0\le l \le n$ large enough, and $|b|<hn^{2/3}$, $t>1$, we have $\PP[\cT^{\boo,\langle n, b\rangle}_{l,t}] < Ce^{-ct^3}$.
\end{lemma}
This lemma can be obtained by applying Corollary \ref{cor:trans-fluc-comb} twice, and we omit its proof.

Our next lemma establishes that, for a geodesic and a path with a `near-optimal' passage time, it is unlikely for them to stay together for a while but remain disjoint.

For any vertices $u\le v$, and $M,l\in\N, m\in\Z$ with $d(u)\le 2m < 2m+2Ml \le d(v)$, and a small enough parameter $c_0>0$, we denote $\cD_{M,l,m}^{u,v}$ as the following event (see Figure \ref{fig:disjoint}): there exists an up-right path $\gamma$ from $\LL_m$ to $\LL_{m+Ml}$,
such that 
\begin{itemize}
    \item $\gamma$ is disjoint from $\Gamma_{u,v}$.
    \item The passage time of $\gamma$ (i.e.\;$T(\gamma)$) is at least $4Ml-c_0Ml^{1/3}$.
    \item For each $i=0,1,\ldots,M$, $|ad(\Gamma_{u,v}\cap \LL_{m+il}) - ad(\gamma\cap \LL_{m+il})| < 2c_0l^{2/3}$.
\end{itemize}

\begin{lemma}  \label{lem:disjoint-paths}
There exist universal constants $c,C>0$ such that the following is true.
For any $M,l,n \in \N$ and $m,b\in \Z$ with $l>C$, $c_0<c$, $|b|\le n$, and $0\le m <m+Ml \le n$, we have $\PP[\cD_{M,l,m}^{\boo, \langle n, b\rangle}] < Ce^{-cM}$.
\end{lemma}

\begin{figure}[hbt!]
    \centering
\begin{tikzpicture}[line cap=round,line join=round,>=triangle 45,x=8cm,y=8cm]
\clip(-0.2,-0.2) rectangle (1.2,1.2);

\draw (0.62,0.02) node[anchor=north east]{$\LL_m$};
\draw (0.77,0.12) node[anchor=north east]{$\LL_{m+l}$};
\draw (0.82,0.22) node[anchor=north east]{$\cdots$};
\draw (1.12,0.42) node[anchor=north east]{$\LL_{m+Ml}$};

\draw [line width=.4pt] (-0.15,0.65) -- (0.65,-0.15);
\draw [line width=.4pt] (-0.05,0.75) -- (0.75,-0.05);
\draw [line width=.4pt] (0.15,0.95) -- (0.95,0.15);
\draw [line width=.4pt] (0.25,1.05) -- (1.05,0.25);

\draw (0,0) node[anchor=east]{$u$};
\draw (1.05,1.05) node[anchor=south east]{$v$};
\draw (0.2,0.3) node[anchor=east]{$u'$};
\draw (0.605,0.695) node[anchor=east]{$v'$};
\draw (0.45,0.5) node[anchor=north]{$\gamma$};
\draw (0.85,0.9) node[anchor=north]{$\Gamma_{u,v}$};
\draw [line width=.75pt,color=green] (0.609,0.695) -- (0.64,0.66);
\draw [line width=.75pt,color=green] (0.545,0.555) -- (0.51,0.59);
\draw [line width=.75pt,color=green] (0.2,0.3) -- (0.26,0.24);
\draw [line width=.75pt,color=green] (0.3,0.4) -- (0.33,0.37);

\draw [red] plot [smooth] coordinates {(1.05,1.05) (0.97,1.02) (0.91,0.97) (0.85,0.94) (0.78,0.89) (0.71,0.85) (0.68,0.78) (0.66,0.75) (0.6,0.69) (0.58,0.64) (0.52,0.6) (0.46,0.57) (0.39,0.55) (0.35,0.49) (0.31, 0.41) (0.24,0.38) (0.2,0.3) (0.16,0.23) (0.13,0.15) (0.06,0.11) (0.04,0.05) (0,0)};
\draw [blue] plot [smooth] coordinates { (0.64,0.66) (0.62,0.6) (0.58,0.56) (0.49,0.54) (0.44,0.5) (0.39,0.45) (0.34, 0.38) (0.29,0.33) (0.26,0.24)};
\draw [fill=uuuuuu] (0.,0.) circle (1.0pt);
\draw [fill=uuuuuu] (0.2,0.3) circle (1.0pt);
\draw [fill=uuuuuu] (0.605,0.695) circle (1.0pt);
\draw [fill=uuuuuu] (1.05,1.05) circle (1.0pt);
\end{tikzpicture}
\caption{The event $\cD_{M,l,m}^{u,v}$: each green segment has length $<c_0l^{2/3}$, and $T(\gamma)\ge 4Ml-c_0Ml^{1/3}$.}   \label{fig:disjoint}
\end{figure}

The last ingredient we need is to bound the probability of multiple peaks in the sum of two independent point-to-line profiles.

As done in previous sections, we denote $T_{u,v}^\bu = T_{u,v}-\xi(v)$ for any vertices $u\le v$ (i.e.\;remove the weight of the last vertex).
For any vertices $u\le v$, and $m\in\Z$ with $d(u)\le 2m \le d(v)$, and $\lambda,t>0$, we denote $\cM_{\lambda,t,m,g}^{u,v}$ as the following event:
there exist $-g\le b_1<b_2<b_3<b_4<b_5<b_6\le g$, with $b_2-b_1, b_3-b_2, b_4-b_3, b_5-b_4, b_6-b_5 \ge \lambda$, such that $T_{u,v} = T_{u,\langle m,b_1\rangle}^\bu + T_{\langle m,b_1\rangle,v}$, and
\[
T_{u,\langle m,b_i\rangle}^\bu + T_{\langle m,b_i\rangle,v}>T_{u,v} - t\lambda^{1/2},\quad \forall i\in \{2,3,4,5,6\}.
\]
\begin{lemma}  \label{lem:multi-peaks}
For $h>0$ and $0<\kappa < 1/2$, there exists a constant $C>0$ such that the following is true.
For any $\theta>0$, $0<t<1$,  $\kappa < \alpha < 1-\kappa$, $|\beta| < h$, 
we have $\PP[\cM_{\theta n^{2/3},t,\lfloor\alpha n\rfloor,hn^{2/3}}^{\boo, \langle n, \lfloor\beta n^{2/3}\rfloor\rangle}] < Ct^{5-0.01}$,
for $n$ large enough depending on $h, \theta, t, \alpha, \beta$.
\end{lemma}
Lemma \ref{lem:disjoint-paths} will be proved in Section \ref{ssec:dis-path}, and Lemma \ref{lem:multi-peaks} will be proved in Section \ref{ssec:multi-p}.

Assuming all the lemmas above, we now prove Proposition \ref{prop:uniform-conv}.
We set up the events to be used in the proof of Proposition \ref{prop:uniform-conv}, for which we first define the parameters.\\

\noindent\textbf{Parameters:}
From now on we fix $h$ in the statement of Proposition \ref{prop:uniform-conv}.
As indicated above, we will choose vertices  $u_1, u_2, u_3, u_4, u_5$ in $\fP_1$, in a multi-scale way.
Thus we define the scales as follows. We take a  small number $\delta>0$, and let $\delta_i=\delta^{100^{6-i}}$ for $i=0,1,2,3,4,5$.
So we have $0<\delta_0<\delta_1<\delta_2<\delta_3<\delta_4<\delta_5<\delta$. We also take small $\kappa>0$ and large $\hh$, and we can assume that $\delta$ is small enough depending on $\kappa$ and $\hh$, and $\hh$ is large enough depending on $h$. The values of the parameters $\delta, \kappa, \hh$ are to be determined, but we always ensure that $\delta^{-1}, \kappa^{-1}, \hh$ are integers.
Then there exists some integer $N$, such that if we denote $\sN=\{Nk^3:k\in \N\}$, for any $n\in \sN$ we have $\delta_0n, \delta_0 n^{2/3}, \hh n^{2/3}, \kappa n, \delta^{-1}$, and each $\delta_i^{-1}\delta_{i+1}$ for $i\in\{0,1,2,3,4\}$ are integers.
From now on we assume that $n \in \sN$, and is large enough depending on all these parameters. Only inside the proof of Proposition \ref{prop:uniform-conv} will we treat general large $n$.

Below we use $c,C>0$ to denote small and large enough constants, which can only depend on $\hh$ and $\kappa$, and the values may change from line to line.\\

\noindent\textbf{Events:} We take the two families of vertices as $\fP_1=\{\langle i\delta_0 n, j\delta_0 n^{2/3} \rangle : i,j \in \Z, 0\le i \le \delta_0^{-1}/3, |j|<4\hh\delta_0^{-1} \}$ and $\fP_2= \bn - \fP_1$.
Note that here we take $\fP_1$ and $\fP_2$ to be in rectangles with width in the order of $\hh n^{2/3}$ rather than $h n^{2/3}$, because the geodesics (that we will study) can potentially have large transversal fluctuations. 
Consider the following events.
\begin{itemize}
\item Let $\cT$ be the union of $\cT^{u,v}_{l,\delta^{-1}}$, for all $u\in\fP_1, v\in\fP_2$, and $l\in \delta_0n\Z$, $0\le l<d(v)-d(u)$.
By Lemma \ref{lem:trans-fluc} we have $\PP[\cT]<C\delta_0^{-5}e^{-c\delta^{-3}}$.
\item Let $\cT_*=\cT^{\langle 0, \hh n^{2/3} \rangle, \langle n, \hh n^{2/3}  \rangle}_{n,\hh } \cup \cT^{\langle 0, - \hh n^{2/3} \rangle, \langle n, - \hh n^{2/3}  \rangle}_{n,\hh }\cup \cT^{\langle 0, 3\hh n^{2/3} \rangle, \langle n, 3\hh n^{2/3}  \rangle}_{n,\hh } \cup \cT^{\langle 0, - 3\hh n^{2/3} \rangle, \langle n, - 3\hh n^{2/3}  \rangle}_{n,\hh }$.
In other words, $\cT_*^c$ is just the event where for each $j\in\{-3,-1,1,3\}$ the geodesic $\Gamma_{\langle 0, j\hh n^{2/3} \rangle, \langle n, j\hh n^{2/3}  \rangle}$ is contained in $\{u \in \Z^2: 0\le d(u)\le 2n, |ad(u)-2j\hh n^{2/3}| \le 2\hh n^{2/3}\}$.
By Lemma \ref{lem:trans-fluc} we have that $\PP[\cT_*]\to 0$ as $\hh\to\infty$, uniformly in $n$.
\item Let $\cF$ be the event where
\[
|T_{\langle i\delta_1 n, a\rangle, \langle j\delta_1 n, b\rangle}^\bu-T_{\langle i\delta_1 n, a'\rangle, \langle j\delta_1 n, b'\rangle}^\bu| > \delta_0^{1/2-0.02}n^{1/3}
\]
for some integers $0\le i<j \le \delta_1^{-1}$, and $|a|,|a'|,|b|,|b'|\le 4\hh n^{2/3}$ with $|a-a'|,|b-b'|\le\delta_0n^{2/3}$.
By applying Lemma \ref{lem:conti-passage-time} to this event with each fixed $i, j$ and taking a union bound, 
we have $\PP[\cF]<C\delta_1^{-2-2/3}e^{-c\delta_0^{-0.01}}$.
\item Let $\cD$ be the union of $\cD^{u,v}_{\delta^{-7}, l, m}$ for all $u \in \fP_1$, $v\in\fP_2$, $l\in\{\delta_in:i=1,2,3,4,5\}$, $m\in \delta_0n\Z$, such that $d(u)\le 2m < 2m+2\delta^{-7}l \le d(v)$.
Here we take $c_0$ to be small enough as required by Lemma \ref{lem:disjoint-paths}.
Then by applying Lemma \ref{lem:disjoint-paths} to each $\cD^{u,v}_{\delta^{-7}, l, m}$ and taking a union bound, we have $\PP[\cD]<C\delta_0^{-5}e^{-c\delta^{-7}}$.
\item Let $\cH$ denote the event where there exists some $m\in\delta_0n\Z$, $0\le m \le n$, and $l\in\{\delta_in:i=1,2,3,4,5\}$, $|a|, |b| < 4\hh n^{2/3}$, $|a-b|<\delta^{-6}l^{2/3}$, such that
\[
T_{\langle m, a \rangle, \langle m+\delta^{-7}l, b\rangle} < 4\delta^{-7}l - c_0\delta^{-6}l^{1/3},
\]
where $c_0$ is the same as in the event $\cD$.
By applying Proposition \ref{t:seg-to-seg} via splitting the lines $\LL_m$ and $\LL_{m+\delta^{-7}l}$ into segments of length $\delta_0n^{2/3}$, we have $\PP[\cH]<C\delta_0^{-3}e^{-c\delta^{-11/2}}$.
\item Let $\cM$ be the union of $\cM^{u,v}_{c_0(\delta_1n)^{2/3}, \delta_0^{1/2-0.03}\delta_1^{-1/3},\alpha n, 4\hh n^{2/3}}$, for all $u \in \fP_1 \cap \LL_0$, $v\in\fP_2 \cap \LL_n$, and $\alpha \in \delta_1 \Z$ with $\kappa<\alpha<1-\kappa$, and $c_0$ be the same as in the event $\cD$.
By Lemma \ref{lem:multi-peaks}, we have $\PP[\cM]<C\delta_0^{-2}\delta_1^{-1}(\delta_0^{1/2-0.03}\delta_1^{-1/3})^{5-0.01}<C\delta_0^{0.3}\delta_1^{-8/3}$.
\end{itemize}
We denote $\cE=\cT^c\cap\cT_*^c\cap\cF^c\cap\cD^c\cap\cH^c\cap\cM^c$.
These events are designed so that $\cE$ happens with high probability, and under $\cE$ we have covering of geodesics. 
\begin{lemma}  \label{lem:cov-geo}
Under $\cE$ the following holds:
for any $|a|, |b| < \hh n^{2/3}$, there exist $u\in\fP_1$ and $v\in\fP_2$, with $d(u)<4\kappa n$ and $d(v)>(1-4\kappa)n$, such that $\Gamma_{\langle 0,a\rangle, \langle n,b\rangle}$ is the same as $\Gamma_{u,v}$ between $\LL_{ 2\kappa n}$ and $\LL_{(1- 2\kappa) n}$.
\end{lemma}

\begin{proof}
Assume that $\cE$ holds, and fix $a,b$ such that $|a|, |b| < \hh n^{2/3}$. 
By ordering of geodesics (Lemma \ref{l:ordering}), $\Gamma_{\langle 0,a\rangle, \langle n,b\rangle}$ is between $\Gamma_{\langle 0, -\hh n^{2/3} \rangle, \langle n, -\hh n^{2/3}  \rangle}$ and $\Gamma_{\langle 0, \hh n^{2/3} \rangle, \langle n, \hh n^{2/3}  \rangle}$.
Then by $\cT_*^c$, we have
\begin{equation}  \label{eq:transabn}
\Gamma_{\langle 0,a\rangle, \langle n,b\rangle} \subset \{u\in \Z^2: 0\le d(u) \le 2n, |ad(u)| \le 4\hh n^{2/3}\}.    
\end{equation}

Let $b^+$ be the smallest number with $b^+\in \delta_0n^{2/3}\Z$ and $b^+\ge b$.
As indicated above, we now show that we can find $u^*\in\fP_1$ with $d(u^*)<4\kappa n$, such that there exists $u\in\Gamma_{\langle 0,a\rangle, \langle n,b\rangle}$ with $d(u)=d(u^*)$ and $ad(u) \le ad(u^*) \le ad(u)+2\delta_0n^{2/3}$,
and
$\Gamma_{u^*,\langle n ,b^+\rangle}$ intersects $\Gamma_{\langle 0,a\rangle, \langle n,b\rangle}$ before $\LL_{2\kappa n}$.

For this we argue by contradiction, and assume that no such $u^*$ exists.
We will sequentially find the vertices $u_5, u_4, u_3, u_2, u_1$ (as illustrated in Figure \ref{fig:uni-cov-pf}) and then use them to find some multiple peaks, thus get a contradiction with $\cM^c$.
The idea is to take each $u_i$ as the vertex in $\fP_1\cap \LL_{\alpha_i n}$ that is to the right of and closest to $\Gamma_{\langle 0,a\rangle, \langle n,b\rangle}$.
Here $\alpha_i$ are numbers to be chosen sequentially: given $u_{i+1}$, we find $\alpha_i$ such that the intersections of $\LL_{\alpha_i}$ with $\Gamma_{\langle 0,a\rangle, \langle n,b\rangle}$ and $\Gamma_{u_{i+1}, \langle n,b^+\rangle}$ are $c_0(\delta_{i+1} n)^{2/3}$ apart, using $\cD^c$.
Finally we consider the intersections of each $\Gamma_{u_i, \langle n,b^+\rangle}$ with $\LL_{\alpha_0 n}$: we can ensure that they are still $c_0(\delta_1 n)^{2/3}$ apart from each other, using transversal fluctuation bounds (from event $\cT^c$) and the fact that $\alpha_0-\alpha_i$ is chosen to be in the order of $\delta_i$.\\

\noindent\textbf{Sequential construction.} Let's start by choosing $u_5$.
We take $\alpha_5$ as the smallest number such that $\alpha_5\in \delta_5\Z$ and $\alpha_5>\kappa$, and take $u_5\in \fP_1\cap \LL_{\alpha_5n}$ being the first one on or to the right of $\Gamma_{\langle 0,a\rangle, \langle n,b\rangle}$.
In other words, we have
$0\le ad(u_5)-ad(\Gamma_{\langle 0,a\rangle, \langle n,b\rangle} \cap \LL_{\alpha_5n})<2\delta_0n^{2/3}$.
Then by \eqref{eq:transabn}, we have that $|ad(u_5)|\le 4\hh n^{2/3}$.
Consider the path $\Gamma_{u_5, \langle n,b^+\rangle}$.
Again by $\cT_*^c$ and ordering of geodesics (Lemma \ref{l:ordering}), it is between $\Gamma_{\langle 0, -3\hh n^{2/3} \rangle, \langle n, -3\hh n^{2/3}  \rangle}$ and $\Gamma_{\langle 0, 3\hh n^{2/3} \rangle, \langle n, 3\hh n^{2/3}  \rangle}$ and
\begin{equation}  \label{eq:transabns}
\Gamma_{u_5, \langle n,b^+\rangle} \subset \{u\in \Z^2: 0\le d(u) \le 2n, |ad(u)| \le 8\hh n^{2/3}\}.    
\end{equation}

For each $j \in \llbracket 0, \delta^{-7}\rrbracket$, we have
\[
ad(\LL_{(\alpha_5+j\delta_5)n}\cap \Gamma_{u_5, \langle n,b^+\rangle}) - ad(\LL_{(\alpha_5+j\delta_5)n}\cap \Gamma_{\langle 0,a\rangle, \langle n,b\rangle} ) \ge 0,
\]
by ordering of geodesics (Lemma \ref{l:ordering}).
We claim that there must exist $j_5 \in \llbracket 0, \delta^{-7}\rrbracket$, such that the left-hand side for $j=j_5$ is at least $2c_0(\delta_5n)^{2/3}$.
Indeed, otherwise we can show that the event $\cD^{\langle 0,a\rangle, \langle n,b\rangle}_{\delta^{-7},\delta_5n,\alpha_5n}$ holds with the path being $\Gamma_{u_5, w_5}$, where $w_5=\Gamma_{u_5, \langle n,b^+\rangle}\cap\LL_{(\alpha_5+\delta^{-7}\delta_5)n}$ (see Figure \ref{fig:uni-cov-pf}).
For this we just verify several things:
\begin{itemize}
    \item By the assumption above (that no such $u^*$ exists), $\Gamma_{u_5, \langle n,b^+\rangle}$ is disjoint from $\Gamma_{\langle 0,a\rangle, \langle n,b\rangle}$ before $\LL_{2\kappa n}$, thus $\Gamma_{u_5, w_5}$ is disjoint from $\Gamma_{\langle 0,a\rangle, \langle n,b\rangle}$ since by taking $\delta$ small enough depending on $\kappa$ we have $\alpha_5+\delta^{-7}\delta_5<2\kappa$.
    \item We have $|ad(w_5)|\le 8\hh n^{2/3}$ by \eqref{eq:transabns}.
By $\cT^c$ we have that \[|ad(u_5)-ad(w_5)|<2\delta^{-1}(\delta^{-7}\delta_5n)^{2/3} =  2\delta^{-17/3}(\delta_5n)^{2/3}.\]
Then we have $T_{u_5,w_5}\ge 4\delta^{-7}\delta_5n - c_0\delta^{-6}(\delta_5n)^{1/3}$ by $\cH^c$.
\end{itemize}
Thus the event $\cD^{\langle 0,a\rangle, \langle n,b\rangle}_{\delta^{-7},\delta_5n,\alpha_5n}$ holds, contradicting with $\cD^c$. So such $j_5$ must exist.

We next let $\alpha_4=\alpha_5+j_5\delta_5$,
and take $u_4\in \fP_1\cap \LL_{\alpha_4n}$ being the first one on or to the right of $\Gamma_{\langle 0,a\rangle, \langle n,b\rangle}$.
Using the same arguments we find $0\le j_4 \le \delta^{-7}$, such that
\[
|ad(\LL_{(\alpha_4+j_4\delta_4)n}\cap \Gamma_{\langle 0,a\rangle, \langle n,b\rangle} )-ad(\LL_{(\alpha_4+j_4\delta_4)n}\cap \Gamma_{u_4, \langle n,b^+\rangle})| \ge 2c_0(\delta_4n)^{2/3}.
\]
Then we let $\alpha_3=\alpha_4+j_4\delta_4$.
Similarly we find $j_3, j_2, j_1 \in \llbracket 0, \delta^{-7}\rrbracket$, and  $\alpha_2=\alpha_3+j_3\delta_3$, $\alpha_1=\alpha_2+j_2\delta_2$, $\alpha_0=\alpha_1+j_1\delta_1$, and vertices $u_3\in \fP_1\cap \LL_{\alpha_3n}$, $u_2\in \fP_1\cap \LL_{\alpha_2n}$, $u_1\in \fP_1\cap \LL_{\alpha_1n}$, such that for each $i=1,2,3$ we have
\begin{equation}  \label{eq:adic}
0\le ad(u_i) - ad(\LL_{\alpha_{i}n}\cap \Gamma_{\langle 0,a\rangle, \langle n,b\rangle} ) < 2\delta_0 n^{2/3},
\end{equation}
and
\begin{equation}  \label{eq:adif}
ad(\LL_{\alpha_{i-1}n}\cap \Gamma_{u_i, \langle n,b^+\rangle}) - ad(\LL_{\alpha_{i-1}n}\cap \Gamma_{\langle 0,a\rangle, \langle n,b\rangle} ) \ge 2c_0(\delta_in)^{2/3}.
\end{equation}
Note that $\LL_{(\alpha_i+j_i\delta_i)n}=\LL_{\alpha_{i-1}n}$ for each $i=1,2,3,4,5$, and that \eqref{eq:adic} and \eqref{eq:adif} also hold for $i=4,5$ as stated above.
See Figure \ref{fig:uni-cov-pf} for (some of) these constructed objects.\\

\begin{figure}[hbt!]
    \centering
\begin{tikzpicture}[line cap=round,line join=round,>=triangle 45,x=10cm,y=10cm]
\clip(-0.22,-0.15) rectangle (1.05,1.05);

\begin{scriptsize}
\draw (0.32,0.4) node[anchor=north]{$u_5$};
\draw (0.44,0.58) node[anchor=east]{$u_4$};
\draw (0.475,0.66) node[anchor=east]{$u_3$};
\draw (0.72,0.57) node[anchor=north]{$w_5$};

\draw (0.75,-0.12) node[anchor=west]{$\LL_{\kappa n}$};
\draw (0.05,-0.12) node[anchor=west]{$\LL_{0}$};
\draw (1.05,0.4) node[anchor=north east]{$\LL_{(\alpha_5+\delta^{-7}\delta_5) n}$};
\draw (1.05,0.17) node[anchor=north east]{$\LL_{\alpha_0n}$};
\draw (1.05,0.07) node[anchor=north east]{$\LL_{(\alpha_5+j_5\delta_5)n}=\LL_{\alpha_4n}$};

\draw (0.632,0.548) node[anchor=north]{$\langle \alpha_0n,b_5\rangle$};
\draw (0.553,0.627) node[anchor=north]{$\langle \alpha_0n,b_4\rangle$};
\draw (0.481,0.699) node[anchor=south]{$\langle \alpha_0n,b_0\rangle$};
\draw (0,0) node[anchor=east]{$\langle 0,a\rangle$};
\draw (-0.07,0.07) node[anchor=east]{$\langle 0,a^-\rangle$};
\draw (1.05,1.05) node[anchor=south east]{$\langle n,b\rangle$};
\draw (0.85,0.9) node[anchor=south east]{$\Gamma_{\langle 0,a\rangle,\langle n,b\rangle}$};
\draw (0.93,0.8) node[anchor=south east]{$\Gamma_{u_3,\langle n,b^+\rangle}$};
\draw (0.94,0.7) node[anchor=south east]{$\Gamma_{u_4,\langle n,b^+\rangle}$};
\draw (1.0,0.59) node[anchor=south east]{$\Gamma_{u_5,\langle n,b^+\rangle}$};
\end{scriptsize}
\draw [line width=.75pt] (-10,10) -- (10,-10);
\draw [line width=.75pt] (-0.5,1.2) -- (1.2,-0.5);
\draw [line width=.4pt] (-0.5,1.79) -- (1.79,-0.5);
\draw [line width=.4pt] (-0.5,1.68) -- (1.68,-0.5);
\draw [line width=.4pt] (-0.5,1.52) -- (1.52,-0.5);

\draw [red] plot [smooth] coordinates {(1.05,1.05) (0.97,1.02) (0.91,0.97) (0.85,0.94) (0.78,0.89) (0.71,0.85) (0.63,0.8) (0.59,0.78) (0.53,0.76) (0.48,0.7) (0.44,0.6) (0.39,0.57) (0.35,0.55) (0.33,0.49) (0.31, 0.41) (0.24,0.38) (0.2,0.3) (0.16,0.23) (0.13,0.15) (0.06,0.11) (0.04,0.05) (0,0)};
\draw [blue] plot [smooth] coordinates {(1.13,0.83) (1.07,0.76) (0.99,0.73) (0.94,0.67) (0.86,0.63) (0.79,0.6) (0.72,0.57) (0.64,0.55) (0.59,0.53) (0.53,0.49) (0.49,0.45) (0.4,0.42) (0.32, 0.4)};
\draw [blue] plot [smooth] coordinates {(1.08,0.87) (0.98,0.83) (0.91,0.8) (0.85,0.76) (0.79,0.73) (0.72,0.7) (0.64,0.66) (0.58,0.65) (0.53,0.61) (0.44,0.58)};
\draw [blue] plot [smooth] coordinates {(1.07,0.99) (0.99,0.96) (0.94,0.93) (0.89,0.89) (0.84,0.85) (0.75,0.82) (0.69,0.79) (0.64,0.76) (0.58,0.71) (0.55,0.69) (0.475,0.66)};
\draw [fill=uuuuuu] (0.,0.) circle (1.0pt);
\draw [fill=uuuuuu] (-0.07,0.07) circle (1.0pt);
\draw [fill=uuuuuu] (0.32, 0.4) circle (1.0pt);
\draw [fill=uuuuuu] (0.44, 0.58) circle (1.0pt);
\draw [fill=uuuuuu] (0.475, 0.66) circle (1.0pt);

\draw [fill=uuuuuu] (0.72,0.57) circle (1.0pt);

\draw [fill=red] (0.632,0.548) circle (1.0pt);
\draw [fill=red] (0.553,0.627) circle (1.0pt);
\draw [fill=red] (0.507,0.673) circle (1.0pt);
\draw [fill=red] (0.481,0.699) circle (1.0pt);

\end{tikzpicture}
\caption{An illustration of the geodesics $\Gamma_{u_i,\langle n,b^+\rangle}$ for $i=5,4,3$. Their intersections with $\LL_{\alpha_0 n}$ are separated by $c_0(\delta_1n)^{2/3}$.}   \label{fig:uni-cov-pf}
\end{figure}

\noindent\textbf{Multiple peaks event.}
We denote the intersections of $\LL_{\alpha_0n}$ with $\Gamma_{\langle 0,a\rangle, \langle n,b\rangle}$ and $\Gamma_{u_i, \langle n,b^+\rangle}$ to be $\langle \alpha_0n, b_0 \rangle$ and $\langle \alpha_0n, b_i \rangle$, for $i=1,2,3,4,5$.
We next lower bound the differences between these $b_i$.

From \eqref{eq:adif} we have that $b_1-b_0\ge c_0(\delta_1n)^{2/3}$.
We next show that $b_i-b_{i-1}\ge c_0(\delta_1n)^{2/3}$, for each $i=2,3,4,5$.
By $\cT^c$ and considering $\Gamma_{\langle 0,a\rangle, \langle n,b\rangle}$ above $\LL_{\alpha_in}$ and $\Gamma_{u_i, \langle n,b^+\rangle}$, and using \eqref{eq:adic}, we have 
\[b_i-b_0 < \delta_0n^{2/3} + 2\delta^{-1}(\alpha_0-\alpha_i)^{2/3}n^{2/3}< \delta_0n^{2/3}+ 2\delta^{-1}(2\delta^{-7}\delta_in)^{2/3},\]
for each $i=1,2,3,4,5$, where the last inequality is by $\alpha_0-\alpha_i \le \delta^{-7}\sum_{i'=1}^i \delta_{i'} < 2\delta^{-7}\delta_i$.
Similarly, by $\cT^c$ and considering $\Gamma_{\langle 0,a\rangle, \langle n,b\rangle}$ and $\Gamma_{u_i, \langle n,b^+\rangle}$ above $\LL_{\alpha_{i-1}n}$, and using \eqref{eq:adif}, we have 
\[b_i-b_0 \ge c_0(\delta_in)^{2/3}- 2\delta_0 n^{2/3} - 2\delta^{-1}(\alpha_0-\alpha_{i-1})^{2/3}n^{2/3} > c_0(\delta_in)^{2/3}- 2\delta_0 n^{2/3} - 2\delta^{-1}(2\delta^{-7}\delta_{i-1}n)^{2/3}, \]
for each $i=2,3,4,5$.
Thus we get that 
\[
b_i-b_{i-1}>c_0(\delta_in)^{2/3}- 3\delta_0 n^{2/3} - 4\delta^{-1}(2\delta^{-7}\delta_{i-1}n)^{2/3} > c_0(\delta_1n)^{2/3},
\]for each $i=2,3,4,5$.

Besides, since $|b_0|\le 2\hh n^{2/3}$ (by \eqref{eq:transabn}), we have that $-2\hh n^{2/3}\le b_0 < b_5 < (2\hh + 1)n^{2/3}$.

To obtain the multiple peaks event at these $b_i$, the remaining task is to bound the passage times through each $\langle \alpha_0 n, b_i\rangle$, from $\langle 0, a^-\rangle$ to $\langle n, b^-\rangle$, where $a^-, b^-$ are the largest numbers satisfying $a^-, b^-\in \delta_0n^{2/3}\Z$ and $a^-\le a$, $b^-\le b$.
Recall that we denote $T_{u,v}^\bu = T_{u,v}-\xi(v)$ for any vertices $u\le v$.
For each $i=1,2,3,4,5$, denote $u_i'=\Gamma_{\langle 0,a\rangle, \langle n,b\rangle}\cap \LL_{\alpha_in}$, and we have $|ad(u_i)|, |ad(u_i')|\le 4\hh n^{2/3}$ by \eqref{eq:transabn} and \eqref{eq:adic}. We then have
\[
\begin{split}
&T_{\langle 0,a^-\rangle, \langle \alpha_0n,b_i \rangle}^\bu + T_{\langle \alpha_0n,b_i \rangle, \langle n,b^-\rangle}
\\
\ge &T_{\langle 0,a\rangle, \langle \alpha_0n,b_i \rangle}^\bu + T_{\langle \alpha_0n,b_i \rangle, \langle n,b^+\rangle} - 2\delta_0^{1/2-0.02}n^{1/3}
\\
\ge &
T_{\langle 0,a\rangle, u_i'}^\bu + T_{u_i', \langle \alpha_0n,b_i \rangle}^\bu + T_{\langle \alpha_0n,b_i \rangle, \langle n,b^+\rangle} - 2\delta_0^{1/2-0.02}n^{1/3}
\\
\ge &
T_{\langle 0,a\rangle, u_i'}^\bu + T_{u_i, \langle \alpha_0n,b_i \rangle}^\bu + T_{\langle \alpha_0n,b_i \rangle, \langle n,b^+\rangle} - 3\delta_0^{1/2-0.02}n^{1/3}
\\
= &
T_{\langle 0,a\rangle, u_i'}^\bu + T_{u_i, \langle n,b^+\rangle} - 3\delta_0^{1/2-0.02}n^{1/3}
\\
\ge &T_{\langle 0,a\rangle, u_i'}^\bu + T_{u_i', \langle n,b\rangle} - 4\delta_0^{1/2-0.02}n^{1/3}
\\
= &
T_{\langle 0,a\rangle, \langle n,b\rangle} - 4\delta_0^{1/2-0.02}n^{1/3}
\\
\ge &T_{\langle 0,a^-\rangle, \langle n,b^-\rangle} - 5\delta_0^{1/2-0.02}n^{1/3},
\end{split}
\]
where the second inequality is by $T_{\langle 0,a\rangle, \langle \alpha_0n,b_i \rangle}^\bu\ge T_{\langle 0,a\rangle, u_i'}^\bu + T_{u_i', \langle \alpha_0n,b_i \rangle}^\bu$ which follows from the definition of passage times, and every other inequality is due to $\cF^c$.
Note that if $\langle \alpha_0n, b_0^- \rangle$ is the intersection of $\Gamma_{\langle 0,a^-\rangle, \langle n,b^-\rangle}$ with $\LL_{\alpha_0 n}$, then $-2\hh n^{2/3}\le b_0^-\le b_0$ by $\cT_*^c$ and ordering of geodesics (Lemma \ref{l:ordering}).
Then we have that $\cM^{\langle 0,a^-\rangle, \langle n,b^-\rangle}_{c_0(\delta_1n)^{2/3}, \delta_0^{1/2-0.03}\delta_1^{-1/3},\alpha_0 n, 4\hh n^{2/3}}$ holds with $b_0^-<b_1<b_2<b_3<b_4<b_5$.
Also note that $\alpha_0\ge \alpha_5>\kappa$, and \[
\alpha_0\le \alpha_5+\sum_{i=1}^5 \delta^{-7}\delta_i< \alpha_5 + 2\delta^{-7}\delta_5 \le \kappa+\delta_5+2\delta^{-7}\delta_5 < 1-\kappa.\]
Thus we get a contradiction with $\cM^c$.
Now we conclude that there exists $u^*\in\fP_1$ with $d(u^*)<4\kappa n$, such that there is $u\in\Gamma_{\langle 0,a\rangle, \langle n,b\rangle}$ with $d(u)=d(u^*)$ and $ad(u) \le ad(u^*) \le ad(u)+2\delta_0n^{2/3}$, and $\Gamma_{u^*,\langle n ,b^+\rangle}$ intersects $\Gamma_{\langle 0,a\rangle, \langle n,b\rangle}$ before $\LL_{2\kappa n}$.\\

\noindent\textbf{Final steps.}
Using the same arguments, we can find $v^*\in\fP_2$ with $d(v^*)>(2-4\kappa)n$, such that there is $v\in\Gamma_{\langle 0,a\rangle, \langle n,b\rangle}$ with $d(v)=d(v^*)$ and $ad(v) \le ad(v^*) \le ad(v)+2\delta_0n^{2/3}$, and $\Gamma_{u^*,v^*}$ intersects $\Gamma_{\langle 0,a\rangle, \langle n,b\rangle}$ after $\LL_{(1-2\kappa)n}$.
We now consider the geodesics $\Gamma_{\langle 0,a\rangle, \langle n,b\rangle}$, $\Gamma_{u^*,\langle n,b^+\rangle}$, and $\Gamma_{u^*,v^*}$, between $\LL_{2\kappa n}$ and $\LL_{(1-2\kappa)n}$.
By ordering of geodesics (Lemma \ref{l:ordering}), we have that either $\Gamma_{u^*,\langle n,b^+\rangle}$ is sandwiched between $\Gamma_{\langle 0,a\rangle, \langle n,b\rangle}$ and $\Gamma_{u^*,v^*}$, or $\Gamma_{u^*,v^*}$ is sandwiched between $\Gamma_{\langle 0,a\rangle, \langle n,b\rangle}$ and $\Gamma_{u^*,\langle n,b^+\rangle}$.
In the former case we have that $\Gamma_{u^*,\langle n,b^+\rangle}$ intersects $\Gamma_{\langle 0,a\rangle, \langle n,b\rangle}$ before $\LL_{2\kappa n}$ and after $\LL_{(1-2\kappa)n}$, so $\Gamma_{\langle 0,a\rangle, \langle n,b\rangle}$ is the same as $\Gamma_{u^*,\langle n,b^+\rangle}$ between $\LL_{ 2\kappa n}$ and $\LL_{(1- 2\kappa) n}$; 
in the later case we have that $\Gamma_{u^*,v^*}$ intersects $\Gamma_{\langle 0,a\rangle, \langle n,b\rangle}$ before $\LL_{2\kappa n}$ and after $\LL_{(1-2\kappa)n}$, so $\Gamma_{\langle 0,a\rangle, \langle n,b\rangle}$ is the same as $\Gamma_{u^*,v^*}$ between $\LL_{ 2\kappa n}$ and $\LL_{(1- 2\kappa) n}$. Thus the conclusion follows.
\end{proof}
We can now finish the proof of Proposition \ref{prop:uniform-conv} using Lemma \ref{lem:cov-geo}.
\begin{proof}[Proof of Proposition \ref{prop:uniform-conv}]
As stated above we write the proof for $\rho=1/2$ for simplicity of notations.

We now consider general $n$, i.e.\;not necessarily in $\sN$.
We let $n'$ be the largest number such that $n'\le n$ and $n'\in \sN$.
Then we have that $n'\to\infty$ and $n'/n\to 1$ as $n\to\infty$. 
We define $\cE'$ as $\cE$ for $n'$ instead of $n$, and $\fP_1'$, $\fP_2'$ as $\fP_1$, $\fP_2$ for $n'$ instead of $n$.

By Theorem \ref{thm:finite-slope}, as $n\to\infty$ we have
\[
\max_{u\in\fP_1',v\in\fP_2'} |\mu_{u,v}(f)-\nu(f)| \to 0,
\]
in probability. Thus by Lemma \ref{lem:cov-geo}, and that $f$ is a bounded function on $\R^{\llbracket -s, s\rrbracket^2}\times\{0,1\}^{\llbracket -s, s\rrbracket^2}$, we have
\[
\PP\left[\cE', \max_{a,b\in\Z,|a|,|b|<\hh {n'}^{2/3}} |\mu_{\langle 0,a \rangle, \langle n',b \rangle}(f) - \nu(f)| > 10\kappa\|f\|_\infty\right] \to 0.
\]
Denote $\cT'=\cT^{\langle 0, \lceil hn^{2/3}\rceil \rangle, \langle n, \lceil hn^{2/3}\rceil \rangle}_{n-n', \hh {n'}^{2/3}(n-n')^{-2/3}/2}\cup \cT^{\langle 0, -\lceil hn^{2/3}\rceil \rangle, \langle n, -\lceil hn^{2/3}\rceil \rangle}_{n-n', \hh {n'}^{2/3}(n-n')^{-2/3}/2}$.
By ordering of geodesics (Lemma \ref{l:ordering}), we have that $\Gamma_{\langle 0, a\rangle, \langle n, b\rangle}$ for $|a|, |b| < hn^{2/3}$ is sandwiched between $\Gamma_{\langle 0, -\lceil hn^{2/3}\rceil \rangle, \langle n, -\lceil hn^{2/3}\rceil \rangle}$ and $\Gamma_{\langle 0, \lceil hn^{2/3}\rceil \rangle, \langle n, \lceil hn^{2/3}\rceil \rangle}$; so assuming the complement of $\cT'$, we must have that for any $\Gamma_{\langle 0, a\rangle, \langle n, b\rangle}$ with $|a|, |b| < hn^{2/3}$, it intersects $\LL_{n'}$ at some vertex $\langle n', b'\rangle$ with 
\[|b'|\le \lceil hn^{2/3}\rceil+\hh {n'}^{2/3}/2 < \hh {n'}^{2/3},\] where the second inequality is by taking $\hh$ much larger than $h$.
Thus we have
\[
\PP\Big[\cE'\cap
{\cT'}^c,
\max_{a,b\in\Z,|a|,|b|<hn^{2/3}/2} |\mu_{\langle 0,a \rangle, \langle n,b \rangle}(f) - \nu(f)| > (10\kappa+2(n-n')/n)\|f\|_\infty\Big] \to 0.
\]
Then since $(n-n')/n\to 0$ as $n\to\infty$, we have
\begin{equation}  \label{eq:uni-cov-f-b}
\limsup_{n\to\infty}\PP\Big[
\max_{a,b\in\Z,|a|,|b|<hn^{2/3}/2} |\mu_{\langle 0,a \rangle, \langle n,b \rangle}(f) - \nu(f)| > 11\kappa\|f\|_\infty\Big] \le \limsup_{n\to\infty}\PP[\cT'] + \PP[{\cE'}^c].    
\end{equation}
By Lemma \ref{lem:trans-fluc}, we have $\limsup_{n\to\infty}\PP[\cT']=0$.
Also, by the discussion of the events $\cT$, $\cT_*$, $\cF$, $\cD$, $\cH$, $\cM$ before Lemma \ref{lem:cov-geo}, we have that $\lim_{\hh\to\infty}\limsup_{\delta\to 0}\limsup_{n\to\infty}\PP[{\cE'}^c] = 0$.
Thus we conclude that the left-hand side of \eqref{eq:uni-cov-f-b} equals $0$. Then since $\kappa$ can be arbitraily taken, the conclusion follows.
\end{proof}

For the next two subsections we prove Lemmas \ref{lem:conti-passage-time}, \ref{lem:disjoint-paths}, and \ref{lem:multi-peaks}.

\subsection{Continuity of passage times and multiple peaks}  \label{ssec:multi-p}
In this subsection we prove Lemma \ref{lem:conti-passage-time} and \ref{lem:multi-peaks}.
For both of them we use the convergence of the point-to-line profile to the Airy$_2$ process, which is a stationary ergodic process minus a parabola. 
Such convergence in the sense of finite dimensional distributions is from \cite{BF08,BP08}. 
Using the so-called slow decorrelation phenomenon, and proving equicontinuity of the point-to-line profile,
it also follows that the weak convergence holds in the topology of uniform convergence on compact sets \cite{basu2019temporal,FO17}.
More precisely, let $\cA_2$ denote the stationary Airy$_2$ process on $\R$, and let us define the stochastic process $\cL:\R\to \R$ by
\[\cL(x):=\cA_{2}(x)-x^2.\]
We quote the following result.
\begin{theorem}[\protect{\cite[Theorem 3.8]{basu2019temporal}}]  \label{thm:profile-a2}
Consider the function
\[
\cL_n:x\mapsto 2^{-4/3}n^{-1/3}\left(T_{\boo,\langle n, x(2n)^{2/3}\rangle}-4n\right),
\]
where we linearly interpolate between points in $(2n)^{-2/3}\Z$.
As $n\to\infty$, we have $\cL_n\to\cL$ weakly in the topology of uniform convergence on compact sets.
\end{theorem}
We shall also use the following (quantitative) comparison between the Airy$_2$ process, and a Brownian motion.

For $K\in \R, d>0$, let $\cB^{[K,K+d]}$ denote the law of a Brownian motion with diffusivity $2$ on $[K,K+d]$, taking value $0$ at $K$. Let $\cL^{[K,K+d]}$ denote the random function on $[K,K+d]$ defined by 
\[\cL^{[K,K+d]}(x):=\cL(x)-\cL(K), ~\forall x\in [K,K+d].\]
Let $C\big([K,K+d],\mathbb{R} \big)$ denote the space of all real-valued continuous functions defined on $[K,K+d]$ which vanish at $K$, with the topology of uniform convergence. The following result can be obtained from \cite{HHJ+}. 
\begin{theorem}[\protect{\cite[Theorem 1.1]{HHJ+}}]\label{thm:a2-comp-br}
There exists an universal constant $G>0$ such that the following holds.
For any fixed $M>0,$ there exists $a_0=a_0(M)$ such that for all intervals $[K,K+d]\subset [-M,M]$ and for all measurable $A \subset C\big([K,K+d],\mathbb{R} \big)$ with $0<\cB^{[K,K+d]}(A)=a\le a_0,$
\[\PP\left[\cL^{[K,K+d]} \in A\right] \leq a     e^{ G {M} ( \log a^{-1} )^{5/6} }.\]
\end{theorem}

Now we prove Lemma \ref{lem:conti-passage-time}.
We start with the following estimate on deviations when moving one endpoint.
\begin{lemma} \label{lem:conti-pass-one-side}
There are constants $c,C>0$ such that the following holds.
For any $h\in\R$, $0\theta<1$, and $t>1$, we have
\begin{equation}  \label{eq:probdiv}
\PP\left[ \max_{hn^{2/3}<b, b'<(h+1)n^{2/3}, |b-b'|<\theta n^{2/3}} |T_{\boo,\langle n, b\rangle} - T_{\boo,\langle n, b'\rangle}| >  t\theta^{1/2-0.01}n^{1/3} + C(|h|+1)\theta n^{1/3} \right] < Ce^{-ct}    
\end{equation}
for $n$ large enough (depending on $h,\theta,t$).
\end{lemma}
\begin{proof}
For any continuous function $f:\R\to \R$, we let
\[
\sM (f):= \max_{2^{-2/3}h\le x, x'\le 2^{-2/3}(h+1), |x-x'|\le 2^{-2/3}\theta} |f(x)-f(x')|.
\]
It is straightforward to check that $\sM$ is a continuous functional on the space of all continuous real-valued functions on $\R$, with the topology of uniform convergence on compact sets.

By Theorem \ref{thm:a2-comp-br}, $\sM (\cL)$ has continuous distribution since this is the case when $\cL$ is replaced by a Brownian motion.
Thus by Theorem \ref{thm:profile-a2}, as $n\to\infty$ we have
$\PP[\sM(\cL_n) > x ] \to \PP[\sM(\cL) > x ]$ for any $x>0$.
We note that the left-hand side of \eqref{eq:probdiv} is bounded by
\[
\PP[\sM(\cL_n) > 2^{-4/3}t\theta^{1/2-0.01} + 2^{-4/3}C(|h|+1)\theta ].
\]
Thus as $n\to\infty$, the $\limsup$ of the left-hand side of \eqref{eq:probdiv} is bounded by
\[
\PP[\sM(\cL) > 2^{-4/3}t\theta^{1/2-0.01} + 2^{-4/3}C(|h|+1)\theta ].
\]
We next show that this is bounded by $Ce^{-ct}$.
When $C>2$ we have $|x^2-x'^2|<2^{-4/3}C(|h|+1)\theta$ for all $x,x'$ with $2^{-2/3}h\le x, x'\le 2^{-2/3}(h+1)$ and $|x-x'|\le 2^{-2/3}\theta$.
Then by stationarity of $\cA_2$, we can bound this probability by
\[
\PP\left[\max_{0\le x, x'\le 2^{-2/3}, |x-x'|\le 2^{-2/3}\theta} |\cL(x)-\cL(x')| >  2^{-4/3}t\theta^{1/2-0.01} \right].
\]
Note that the event now only relies on $\cL^{[0,2^{-2/3}]}$.
Using modulus of continuity for Brownian motions and Theorem \ref{thm:a2-comp-br}, we can bound this by $Ce^{-ct}$ as desired.
\end{proof}
We can now prove Lemma \ref{lem:conti-passage-time} by using Lemma \ref{lem:conti-pass-one-side} repeatedly. 
\begin{proof} [Proof of Lemma \ref{lem:conti-passage-time}]
First, note that we have the following inequality for passage times: 
\[
T_{\langle 0,a \rangle,\langle n,b\rangle} - T_{\langle 0,a \rangle,\langle n,b'\rangle} \ge
T_{\langle 0,a' \rangle,\langle n,b\rangle}
- T_{\langle 0,a' \rangle,\langle n,b'\rangle}
\]
for any $a\le a', b\le b'$.
Indeed, if we take the geodesics $\Gamma_{\langle 0,a \rangle,\langle n,b'\rangle}$ and $\Gamma_{\langle 0,a' \rangle,\langle n,b\rangle}$, then they must intersect. By switching the paths after their first intersection, we get two up-right paths, from $\langle 0,a \rangle$ to $\langle n,b\rangle$ and from $\langle 0,a' \rangle$ to $\langle n,b'\rangle$, and the sum of their passage times equals $T_{\langle 0,a' \rangle,\langle n,b\rangle} + T_{\langle 0,a \rangle,\langle n,b'\rangle}$.
Thus we get the above inequality from the definition of last-passage times.

Using this inequality, for any $|a|,|a'|,|b|,|b'|<hn^{2/3}$ we have
\[
|T_{\langle 0, a\rangle, \langle n, b\rangle}-T_{\langle 0, a\rangle, \langle n, b'\rangle}| \le
|T_{\langle 0, -\lceil hn^{2/3}\rceil\rangle, \langle n, b\rangle}-T_{\langle 0, -\lceil hn^{2/3}\rceil\rangle, \langle n, b'\rangle}|
\vee |T_{\langle 0, \lceil hn^{2/3}\rceil\rangle, \langle n, b\rangle}-T_{\langle 0, \lceil hn^{2/3}\rceil\rangle, \langle n, b'\rangle}|,
\]
and
\[
|T_{\langle 0, a\rangle, \langle n, b'\rangle}-T_{\langle 0, a'\rangle, \langle n, b'\rangle}| \le
|T_{\langle 0, a\rangle, \langle n, -\lceil hn^{2/3}\rceil\rangle}-T_{\langle 0, a'\rangle, \langle n, -\lceil hn^{2/3}\rceil\rangle}|
\vee |T_{\langle 0, a\rangle, \langle n, \lceil hn^{2/3}\rceil\rangle}-T_{\langle 0, a'\rangle, \langle n, \lceil hn^{2/3}\rceil\rangle}|.
\]
By adding up these two inequalities and using the triangle inequality, we have
\[
\begin{split}
|T_{\langle 0, a\rangle, \langle n, b\rangle}-T_{\langle 0, a'\rangle, \langle n, b'\rangle}|
\le & |T_{\langle 0, -\lceil hn^{2/3}\rceil\rangle, \langle n, b\rangle}-T_{\langle 0, -\lceil hn^{2/3}\rceil\rangle, \langle n, b'\rangle}|
\vee |T_{\langle 0, \lceil hn^{2/3}\rceil\rangle, \langle n, b\rangle}-T_{\langle 0, \lceil hn^{2/3}\rceil\rangle, \langle n, b'\rangle}|\\
+&|T_{\langle 0, a\rangle, \langle n, -\lceil hn^{2/3}\rceil\rangle}-T_{\langle 0, a'\rangle, \langle n, -\lceil hn^{2/3}\rceil\rangle}|
\vee |T_{\langle 0, a\rangle, \langle n, \lceil hn^{2/3}\rceil\rangle}-T_{\langle 0, a'\rangle, \langle n, \lceil hn^{2/3}\rceil\rangle}|
\end{split}
\]
By symmetry, now it suffices to bound 
\[
\PP\left[ \max_{\substack{|b|, |b'| < hn^{2/3} \\
|b-b'|<\theta n^{2/3}}}
|T_{\langle 0, -\lceil hn^{2/3}\rceil\rangle, \langle n, b\rangle}-T_{\langle 0, -\lceil hn^{2/3}\rceil\rangle, \langle n, b'\rangle}| > \frac{1}{2}(t\theta^{1/2-0.01}n^{1/3} + Ch\theta n^{1/3})
\right].
\]
For this we split $\{\langle 0, b\rangle: |b|<hn^{2/3}\}$ into overlapping segments of length $n^{2/3}$, and apply Lemma \ref{lem:conti-pass-one-side} to each of them, and get the desired bound.
\end{proof}

We next prove Lemma \ref{lem:multi-peaks}.
Again, using Theorem \ref{thm:profile-a2} we reduce the point-to-line profiles to Airy$_2$ processes, and then by applying Theorem \ref{thm:a2-comp-br} we can just prove the result for Brownian motions.
\begin{proof}[Proof of Lemma \ref{lem:multi-peaks}]
Denote $\cL_{\alpha,\beta}:\R\to\R$ as the process given by
\[
\cL_{\alpha,\beta}(x):=\alpha^{1/3}\cL(\alpha^{-2/3}x) + (1-\alpha)^{1/3}\cL'((1-\alpha)^{-2/3}(x-2^{-2/3}\beta)),
\]
where $\cL'$ is an independent copy of $\cL$.
Denote
\[
\cL_{n,\alpha,\beta}(x):= 2^{-4/3}n^{-1/3}\left(T_{\boo,\langle \lfloor \alpha n\rfloor, x(2n)^{2/3} \rangle}^\bu + T_{\langle \lfloor \alpha n\rfloor, x(2n)^{2/3} \rangle, \langle n, \lfloor\beta n^{2/3}\rfloor \rangle}-4n\right),
\]
where we linearly interpolate between points in $(2n)^{-2/3}\Z$.
Using Theorem \ref{thm:profile-a2}, we can deduce that $\cL_{n,\alpha,\beta}\to\cL_{\alpha,\beta}$ as $n\to\infty$, weakly in the topology of uniform convergence on compact sets.

We let $\Omega$ be the set of all continuous function $f:\R\to\R$, such that there exist $-2^{1/3}h\le x_1<x_2<x_3<x_4<x_5<x_6\le 2^{1/3}h$, with $x_2-x_1, x_3-x_2, x_4-x_3, x_5-x_4, x_6-x_5 \ge 2^{-2/3}\theta$, and $x_1=\argmax_{[-2^{4/3}h,2^{4/3}h]}f$, and
\[
f(x_1) \le f(x_i) + 2^{-4/3}t\theta^{1/2},\quad \forall i=2,3,4,5,6. 
\]
It is straightforward to check that $\Omega$ is a closed set, in the space of all continuous function with the topology of uniform convergence on compact sets.
It is also straightforward to check that $\cM_{\theta n^{2/3},t,\lfloor\alpha n\rfloor,2hn^{2/3}}^{\boo, \langle n, \lfloor\beta n^{2/3}\rfloor\rangle}$ implies $\cL_{n,\alpha,\beta}\in \Omega$.
So by Theorem \ref{thm:profile-a2} we have
\[
\limsup_{n\to\infty}\PP[\cM_{\theta n^{2/3},t,\lfloor\alpha n\rfloor,2hn^{2/3}}^{\boo, \langle n, \lfloor\beta n^{2/3}\rfloor\rangle}]\le
\limsup_{n\to\infty}\PP[\cL_{n,\alpha,\beta}\in \Omega] \le \PP[\cL_{\alpha,\beta} \in \Omega].
\]
We just need to bound the right-hand side.
By Theorem \ref{thm:a2-comp-br}, we can consider the probability of a (two-sided) Brownian motion (with diffusivity $4$) belonging to $\Omega$.
By Lemma \ref{lem:appa:1} below this probability is bounded by $Ct^5$ for $C>0$ being a universal constant, so the conclusion follows.
\end{proof}
We finally bound the event on Brownian motions.
\begin{lemma}\label{lem:appa:1}
There exists a universal constant $C>0$, such that for any $t,\theta>0$, the following event holds with probability at most $Ct^5$.
For $W:[-2,2]\to\R$ being a two-sided Brownian motion, there are $-1<x_1<x_2<x_3<x_4<x_5<x_6<1$, with $x_2-x_1, x_3-x_2, x_4-x_3, x_5-x_4, x_6-x_5 > \theta$, such that $x_1=\argmax_{[-2,2]}W$, and
\[
W(x_1) < W(x_i) + t\theta^{1/2},\quad \forall i=2,3,4,5,6. 
\]
\end{lemma}
\begin{proof}
Fix $T_1 \in [-1,1]$, and let $\cE$ be the event where $W(T_1)=\max_{[-2,2]}W$.
For $i=2,3,4,5,6$, let $T_i=\min\{x\ge T_{i-1}+\theta:W(x)\ge W(x_1)-t\theta^{1/2}\}$.
It suffices to show that $\PP[T_6\le 1\mid \cE]<Ct^5$ for some universal constant $C>0$.
For $i=2,3,4,5,6$, conditioned on $\cE$ and the event $T_{i-1}\le 1$, and given the values of $T_{i-1}$ and $W(T_{i-1}) - W(T_1)$, the process $x\mapsto W(T_{i-1}+x) - W(T_1)$ on $[0,2-T_{i-1}]$ has the same law of $W'$, which is a Brownian motion on $[0,2-T_{i-1}]$ starting from $W'(0)=W(T_{i-1})-W(T_1)$ and conditioned to stay below zero (for $i=2$ this degenerates to a Brownian meander). Using the reflection principle we have that $\PP[\max_{[\theta, 2-T_{i-1}]}W' \ge -t\theta^{1/2}]<C't$ for some universal constant $C'>0$. So we have that $\PP[T_i \le 1\mid \cE, T_{i-1}\le 1] < C't$. Thus $\PP[T_6 \le 1\mid \cE] < (C't)^5$, which implies the conclusion.
\end{proof}

\subsection{Disjoint paths}  \label{ssec:dis-path}
In this subsection we prove Lemma \ref{lem:disjoint-paths}.
The idea is to show that for a path restricted to be close to another (deterministic) path for a while, its passage time is unlikely to be small (compared to that of a geodesic with the same endpoints).
We then use the FKG inequality to move from a deterministic path to a geodesic.
\begin{lemma}  \label{lem:max-small-seg}
For sufficiently small $c_0>0$, there is $c_1>0$, such that for $l\in\N$ large enough (depending on $c_0$) and any $r\in\Z$, we have
\[
\E\left[ \max_{a, b \in \llbracket 0, c_0l^{2/3}\rrbracket} T_{\langle 0,a\rangle, \langle l,r+b\rangle}\right] < 4l-c_1l^{1/3}.
\]
\end{lemma}
\begin{proof}
Take $u=\langle -\lfloor c_0^{3/2}l \rfloor, 0\rangle$ and $v=\langle l+\lfloor c_0^{3/2}l \rfloor, r'\rangle$, where $r'$ is the number in $\lfloor c_0l^{2/3}\rfloor \Z$ with $r\le r'< r+\lfloor c_0l^{2/3}\rfloor$.
Note that
\[
\E \left[\max_{a, b \in \llbracket 0, c_0l^{2/3}\rrbracket} T_{\langle 0,a\rangle, \langle l,r+b\rangle}\right] 
\leq  \E[T_{u,v}] -\E\left[\min_{a \in \llbracket 0, c_0l^{2/3}\rrbracket} T_{u, \langle -1,a\rangle}\right]-\E\left[\min_{b \in \llbracket 0, c_0l^{2/3}\rrbracket} T_{\langle l+1,r+b\rangle, v}\right].
\]
By Proposition \ref{t:seg-to-seg}, we have \[\E\left[\min_{a \in \llbracket 0, c_0l^{2/3}\rrbracket} T_{u, \langle -1,a\rangle}\right],\;\E\left[\min_{b \in \llbracket 0, c_0l^{2/3}\rrbracket} T_{\langle l+1,r+b\rangle, v}\right] \geq 4c_0^{3/2}l-Cc_0^{1/2}l^{1/3},\]
where $C>0$ is a universal constant. 
We also claim that for $l$ sufficiently large,
\begin{equation}  \label{eq:max-small-seg1}
  \E[T_{u,v}]\leq 4(l+2c_0^{3/2}l)-c_2l^{1/3},
\end{equation}
for some small universal constant $c_2>0$.
Let $C'>0$ be a large enough universal constant.
When $l^{-2/3}|r|>C'$, \eqref{eq:max-small-seg1} follows from \eqref{e:mean}.
When $l^{-2/3}|r|\le C'$, for each $l$ there are at most $3C'/c_0$ possible numbers $r'$ can take. For each of them, by Theorem \ref{thm:profile-a2} the corresponding $T_{u,v}$ after rescaling converges (as $l\to\infty$) to one point of the Airy$_2$ process, whose law is given by the GUE Tracy-Widom distribution.
Thus \eqref{eq:max-small-seg1} (for $l$ large enough) follows since the GUE Tracy-Widom distribution has negative expectation.
By choosing $c_0$ such that $2Cc_0^{1/2}<c_2/2$ and letting $c_1=c_2/2$, we complete the proof.
\end{proof}
For the next lemma, as before we denote $T_{u,v}^\bu = T_{u,v}-\xi(v)$ for any vertices $u\le v$.
\begin{lemma}  \label{lem:disjoint-path-fix}
For $l,M\in\N$ and any $r_0,\ldots,r_M \in \Z$, we have
\[
\PP\left[ \max_{ a_0,\ldots,a_M\in \llbracket 0, cl^{2/3}\rrbracket}
\sum_{i=0}^{M-1}T_{\langle il, r_i+a_i \rangle, \langle (i+1)l, r_{i+1}+a_{i+1} \rangle}^\bu \ge 4Ml - cMl^{1/3}
\right] < Ce^{-cM},
\]
for some universal constants $c,C>0$, when $l$ is large enough.
\end{lemma}
\begin{proof}
In this proof we let $c,C>0$ denote small and large enough universal constants, and their values can change from line to line.

Take $c_0, c_1>0$ such that Lemma \ref{lem:max-small-seg} holds.
For each $0\le i \le M-1$ we denote $S_i=\max_{a_i, a_{i+1} \in \llbracket 0, c_0l^{2/3}\rrbracket} T_{\langle il, r_i+a_i\rangle, \langle (i+1)l, r_{i+1}+a_{i+1}\rangle}^\bu$.
Then (by Lemma \ref{lem:max-small-seg}) we have each $\E[S_i] < 4l-c_1l^{1/3}$ when $l$ is large enough.

Next we apply Proposition \ref{t:seg-to-seg}.
When $|r_i-r_{i+1}|\le 0.9l$ we could directly apply it; and when $|r_i-r_{i+1}|> 0.9l$, the slope condition may not be satisfied, thus we use the fact that $T_{\langle il, r_i+a_i\rangle, \langle (i+1)l, r_{i+1}+a_{i+1}\rangle}^\bu < T_{\langle il, r_i+a_i\rangle, \langle (i+1)l+\lfloor 0.1l \rfloor, r_{i+1}+a_{i+1}\rangle}^\bu$, and upper bound the later using Proposition \ref{t:seg-to-seg}.
In either case we conclude that
$\PP[S_i > 4l + xl^{1/3}] < Ce^{-cx}$,
for any $x>0$.

Note that $S_i$ for each $i$ are independent.
Thus by a Bernstein type bound on the sum of independent random variables with exponential tails, we have 
\[
\PP\left[ \max_{a_0,\ldots,a_M\in \llbracket 0, c_0l^{2/3}\rrbracket}
\sum_{i=0}^{M-1}T_{\langle il, r_i+a_i \rangle, \langle (i+1)l, r_{i+1}+a_{i+1} \rangle}^\bu \ge 4Ml - \frac{c_1}{2}Ml^{1/3}
\right] < Ce^{-cM}.
\]
Then the conclusion follows.
\end{proof}
\begin{proof}[Proof of Lemma \ref{lem:disjoint-paths}]
Take any up-right path $\Gamma$ from $\boo$ to $\langle n,b\rangle$.
Denote $\cD_{\Gamma}$ as the following event: there exists an up-right path $\gamma$ from $\LL_m$ to $\LL_{m+Ml}$,
such that
\begin{itemize}
    \item $\gamma$ is disjoint from $\Gamma$.
    \item The passage time of $\gamma$ (i.e.\;$T(\gamma)$) is at least $4Ml-c_0Ml^{1/3}$.
    \item For each $i=0,1,\ldots,M$, $|ad(\Gamma\cap \LL_{m+il}) - ad(\gamma\cap \LL_{m+il})| < 2c_0l^{2/3}$.
\end{itemize}
Here $c_0>0$ is the same as in the definition of $\cD_{M,l,m}^{\boo, \langle n, b\rangle}$.
Now we consider the event $\Gamma_{\boo, \langle n, b\rangle}=\Gamma$.
Under this event we have $\cD_{M,l,m}^{\boo, \langle n, b\rangle} = \cD_\Gamma$.
Also, $\Gamma_{u,v}=\Gamma$ is a negative event of the field on $\Z^2\setminus \Gamma$, while $\cD_\Gamma$ is determined by the field on $\Z^2\setminus \Gamma$, and is a positive event of the field on $\Z^2\setminus \Gamma$.
By the FKG inequality we have
\[\PP[\cD_{M,l,m}^{\boo, \langle n, b\rangle}\mid \Gamma_{u,v}=\Gamma] = \PP[\cD_\Gamma\mid \Gamma_{u,v}=\Gamma] \le \PP[\cD_\Gamma].\]
By Lemma \ref{lem:disjoint-path-fix}, we have $\PP[\cD_{\Gamma}]<Ce^{-cM}$ when $c_0<c$ and $l>C$, for $c,C>0$ being universal constants.
By averaging over all $\Gamma$ we get the conclusion. 
\end{proof}

\section{Convergence of one point distribution}  \label{sec:one-point}
In this section we prove Theorems \ref{t:one-point-converg-infinite} and \ref{t:one-point-converg-finite}.
The general idea is to show that the law of the environment around a specific vertex in the geodesic is close to that of nearby vertices along the geodesic; and this is achieved by a coalescing argument. Then we use Proposition \ref{prop:uniform-conv} to argue that certain time average (of environments along the geodesic) is close to the stationary measure $\nu$.

To prove Theorem \ref{t:one-point-converg-infinite}, a key step would be to bound the total variation distance between $\xi\{\Gamma_\boo[i]\}, \Gamma_\boo-\Gamma_\boo[i]$ and $\xi\{\Gamma_\boo[i-r]\}, \Gamma_\boo-\Gamma_\boo[i-r]$ in a finite box, for any $i$ large and $r$ much smaller than $i$.
For this, we use translation invariance, and consider the environment around $\Gamma_{v[r]}[i-r]$ instead of $\Gamma_\boo[i-r]$, where
\[
v[r]=
\begin{cases}
    \langle \lfloor r/2\rfloor, 0\rangle_{} & r \text{ is even}; \\
    \langle \lfloor r/2\rfloor, 0\rangle_{} + (1,0) & r \text{ is odd}.
\end{cases}
\]
We define $v[r]$ this way so that there is always $d(v[r])=r$.
We show that with high probability $\Gamma_{v[r]}[i-r]=\Gamma_\boo[i]$, and in a finite box around this vertex the paths $\Gamma_{v[r]}$ and $\Gamma_\boo$ are the same.
Towards this we need the following estimate on coalescence of geodesics, which directly follows from Proposition \ref{prop:coal-semi-inf-1} and Lemma \ref{lem:semi-inf-trans}.
\begin{lemma}  \label{lem:coal-semi-inf}
There is a constant $C>0$, such that for any $r\in\N$, and $k>2$, we have 
$\PP[\Gamma_\boo\cap \LL_{\lfloor rk\rfloor} \neq \Gamma_{v[r]}\cap \LL_{\lfloor rk\rfloor}] < C\log(k)k^{-2/3}$.
\end{lemma}
\begin{proof}
Denote the intersections of $\Gamma_\boo$ and $\Gamma_{v[r]}$ with $\LL_r$ as $\langle r,b_r\rangle_{}$ and $\langle r,b_r'\rangle_{}$, respectively.
By Lemma \ref{lem:semi-inf-trans} and Proposition \ref{prop:coal-semi-inf-1}, there is a constant $C_0>0$ such that
\[\PP[|b_r|,|b_r'|\le C_0\log(k)r^{2/3}] > 1-C_0k^{-1},\] 
and
\[\PP[\Gamma_{\langle r, -\lfloor C_0\log(k)r^{2/3}\rfloor-1\rangle_{}}\cap \LL_{\lfloor rk\rfloor} \neq \Gamma_{\langle r,\lfloor C_0\log(k)r^{2/3}\rfloor+1\rangle_{}}\cap \LL_{\lfloor rk\rfloor}] < C_0^2 \log(k)(k-1)^{-2/3}.\]
Thus the conclusion follows by ordering of geodesics (Lemma \ref{l:ordering}).
\end{proof}

\begin{proof} [Proof of Theorem \ref{t:one-point-converg-infinite}]
Take any $s\in\N$ and any continuous function $f:\R^{\llbracket -s,s\rrbracket^2}\times \{0,1\}^{\llbracket -s,s\rrbracket^2}\to [0,1]$, regarded as a function on $\R^{\Z^2}\times \{0,1\}^{\Z^2}$. We need to show that 
$\lim_{i\to\infty}\E[f(\xi\{\Gamma_\boo[i]\}, \Gamma_\boo-\Gamma_\boo[i])] = \nu(f)$.

For $i, r\in\N$ and $k>2$ with $i-2s>2rk$, by Lemma \ref{lem:coal-semi-inf}, with probability at least $1-C\log(k)k^{-2/3}$ we have $\Gamma_\boo[j]=\Gamma_{v[r]}[j-r]$ for any $j\ge i-2s$; thus $\xi\{\Gamma_\boo[i]\}, \Gamma_\boo-\Gamma_\boo[i]$ and $\xi\{\Gamma_{v[r]}[i-r]\}, \Gamma_{v[r]}-\Gamma_{v[r]}[i-r]$ are the same in $\llbracket -s, s \rrbracket^2$.
Since $\xi\{\Gamma_{v[r]}[i-r]\}, \Gamma_{v[r]}-\Gamma_{v[r]}[i-r]$ have the same joint distribution as $\xi\{\Gamma_\boo[i-r]\}, \Gamma_\boo-\Gamma_\boo[i-r]$, we must have that
\[
|\E[f(\xi\{\Gamma_\boo[i]\}, \Gamma_\boo-\Gamma_\boo[i])] - \E[f(\xi\{\Gamma_\boo[i-r], \Gamma_\boo-\Gamma_\boo[i-r]\})] | \le C\log(k)k^{-2/3}.
\]
By averaging over $r\in \llbracket 0, i/4k\rrbracket$, we have (when $i>4s$)
\[
|\E[f(\xi\{\Gamma_\boo[i]\}, \Gamma_\boo-\Gamma_\boo[i])] - \E [\mu_{\Gamma_\boo[i-\lfloor i/4k\rfloor],\Gamma_\boo[i]} (f)] | \le C\log(k)k^{-2/3}.
\]
By Lemma \ref{lem:semi-inf-trans}, and Proposition \ref{prop:uniform-conv}, for any fixed $k>0$, we have $\mu_{\Gamma_\boo[i-\lfloor i/4k\rfloor],\Gamma_\boo[i]} (f) \to \nu(f)$ in probability as $i\to\infty$.
Thus we have that \[\limsup_{i\to\infty} |\E [f(\xi\{\Gamma_\boo[i]\}, \Gamma_\boo-\Gamma_\boo[i])] - \nu (f) | \le C\log(k)k^{-2/3}.\]
Since $k$ can be arbitrarily large, the conclusion follows.
\end{proof}

The proof of Theorem \ref{t:one-point-converg-finite} is similar.
Again we need the following estimate on coalescence of geodesics, 
which follows from Corollary \ref{cor:trans-fluc-comb} and Proposition \ref{prop:coalesce}.
Recall that we denote $\bn^{}=\bn^\rho=\langle n,0\rangle=\left(\left\lfloor\frac{2(1-\rho)^2n}{\rho^2+(1-\rho)^2}\right\rfloor, \left\lceil\frac{2\rho^2n}{\rho^2+(1-\rho)^2}\right\rceil\right)$ for any $n\in\Z$.
\begin{lemma}  \label{lem:coal-f}
There is a constant $C>0$, such that for any $r, n\in\N$ and $k>2$, with $n\ge 2rk$, we have $\PP[\Gamma_{\boo,\bn^{}}\cap \LL_{\lfloor rk\rfloor} \neq \Gamma_{v[r],\bn^{}+v[r]}\cap \LL_{\lfloor rk\rfloor}] < C\log(k)k^{-2/3}$, and $\PP[\Gamma_{\boo,\bn^{}}\cap \LL_{n-\lfloor rk\rfloor} \neq \Gamma_{v[r],\bn^{}+v[r]}\cap \LL_{n-\lfloor rk\rfloor}] < C\log(k)k^{-2/3}$.
\end{lemma}
\begin{figure}[hbt!]
    \centering
\begin{tikzpicture}[line cap=round,line join=round,>=triangle 45,x=7cm,y=7cm]
\clip(-0.55,-0.15) rectangle (1.55,1.2);

\draw (0.5,-0.06) node[anchor=north east]{$\LL_r$};
\draw (0.9,-0.06) node[anchor=north east]{$\LL_{\lfloor rk\rfloor}$};
\draw (1.23,0.5) node[anchor=north east]{$\LL_{n-r}$};

\draw [line width=.4pt] (-0.15,0.55) -- (0.55,-0.15);
\draw [line width=.4pt] (-0.15,0.95) -- (0.95,-0.15);
\draw [line width=.4pt] (-0.15,1.85) -- (1.25,0.45);

\draw (0,0) node[anchor=east]{$\boo$};
\draw (1.05,1.05) node[anchor=south east]{$\bn^{}+v[r]$};
\draw (0.95,0.95) node[anchor=west]{$\bn^{}$};
\draw (0.1,0.1) node[anchor=east]{$v[r]$};

\begin{scriptsize}
\draw (0.12,0.28) node[anchor=east]{$\langle r, -\lfloor C_0\log(k)r^{2/3}\rfloor-1\rangle_{}$};
\draw (0.77,0.93) node[anchor=east]{$\langle n-r, -\lfloor C_0\log(k)r^{2/3}\rfloor-1\rangle_{}$};
\draw (0.28,0.12) node[anchor=west]{$\langle r, \lfloor C_0\log(k)r^{2/3}\rfloor+1\rangle_{}$};
\draw (0.93,0.77) node[anchor=west]{$\langle n-r, \lfloor C_0\log(k)r^{2/3}\rfloor+1\rangle_{}$};
\end{scriptsize}

\draw [red] plot [smooth] coordinates {(1.05,1.05) (0.97,1.02) (0.91,0.97) (0.85,0.94) (0.78,0.89) (0.71,0.85) (0.68,0.78) (0.66,0.75) (0.6,0.69) (0.58,0.64) (0.52,0.6) (0.46,0.57) (0.39,0.55) (0.35,0.49) (0.31, 0.41) (0.27,0.36) (0.25,0.27) (0.23,0.19) (0.18,0.14) (0.14,0.09) (0.07,0.04) (0,0)};
\draw [red] plot [smooth] coordinates {(0.95,0.95) (0.89,0.92) (0.8,0.88) (0.73,0.85) (0.68,0.78) (0.66,0.75) (0.6,0.69) (0.58,0.64) (0.52,0.6) (0.46,0.57) (0.39,0.55) (0.35,0.49) (0.31, 0.41) (0.24,0.38) (0.2,0.3) (0.16,0.23) (0.13,0.19) (0.1,0.1)};

\draw [blue] plot [smooth] coordinates {(0.77,0.93) (0.73,0.89) (0.7,0.86) (0.66,0.78) (0.65,0.75) (0.6,0.69) (0.58,0.64) (0.52,0.6) (0.46,0.57) (0.39,0.55) (0.35,0.49) (0.32, 0.41) (0.31,0.31) (0.28,0.2) (0.28,0.12)};
\draw [blue] plot [smooth] coordinates {(0.93,0.77) (0.82,0.76) (0.78,0.75) (0.7,0.75) (0.65,0.74) (0.6,0.69) (0.58,0.64) (0.52,0.6) (0.46,0.57) (0.39,0.55) (0.35,0.49) (0.31, 0.41) (0.24,0.38) (0.18,0.3) (0.12,0.28)};
\draw [fill=uuuuuu] (0.,0.) circle (1.0pt);
\draw [fill=uuuuuu] (1.05,1.05) circle (1.0pt);
\draw [fill=uuuuuu] (0.1,0.1) circle (1.0pt);
\draw [fill=uuuuuu] (0.95,0.95) circle (1.0pt);

\draw [fill=uuuuuu] (0.28,0.12) circle (1.0pt);
\draw [fill=uuuuuu] (0.12,0.28) circle (1.0pt);
\draw [fill=uuuuuu] (0.93,0.77) circle (1.0pt);
\draw [fill=uuuuuu] (0.77,0.93) circle (1.0pt);

\end{tikzpicture}
\caption{An illustration of the proof of Lemma \ref{lem:coal-f}. 
The geodesics $\Gamma_{\boo,\bn^{}}$ and $\Gamma_{v[r],\bn^{}+v[r]}$ are sandwiched between $\Gamma_{\langle r, -\lfloor C_0\log(k)r^{2/3}\rfloor-1\rangle_{},\langle n-r, -\lfloor C_0\log(k)r^{2/3}\rfloor-1\rangle_{}}$ and $\Gamma_{\langle r, \lfloor C_0\log(k)r^{2/3}\rfloor+1\rangle_{},\langle n-r, \lfloor C_0\log(k)r^{2/3}\rfloor+1\rangle_{}}$.}   \label{fig:overlap}
\end{figure}
\begin{proof}
Since $n\ge 2rk$, we just show
$\PP[\Gamma_{\boo,\bn^{}}\cap \LL_{\lfloor rk\rfloor} \neq \Gamma_{v[r],\bn^{}+v[r]}\cap \LL_{\lfloor rk\rfloor}] < C\log(k)k^{-2/3}$, and by symmetry the other inequality would follow.

Denote the intersections of $\Gamma_{\boo,\bn^{}}$ and $\Gamma_{v[r],\bn^{}+v[r]}$ with $\LL_r$ as $\langle r,b_-\rangle_{}$ and $\langle r,b_-'\rangle_{}$, respectively;
and the intersections of $\Gamma_{\boo,\bn^{}}$ and $\Gamma_{v[r],\bn^{}+v[r]}$ with $\LL_{n-r}$ as $\langle n-r,b_+\rangle_{}$ and $\langle n-r,b_+'\rangle_{}$, respectively.
There is a constant $C_0>0$, such that
\[\PP[|b_-|,|b_-'|\le C_0\log(k)r^{2/3}],\;
\PP[|b_+|,|b_+'|\le C_0\log(k)r^{2/3}]
> 1-C_0k^{-1}\]
by Corollary \ref{cor:trans-fluc-comb};
and
\[
\begin{split}
&\PP[\Gamma_{\langle r, -\lfloor C_0\log(k)r^{2/3}\rfloor-1\rangle_{},\langle n-r, -\lfloor C_0\log(k)r^{2/3}\rfloor-1\rangle_{}}\cap \LL_{\lfloor rk\rfloor} \neq \Gamma_{\langle r, \lfloor C_0\log(k)r^{2/3}\rfloor+1\rangle_{},\langle n-r, \lfloor C_0\log(k)r^{2/3}\rfloor+1\rangle_{}}\cap \LL_{\lfloor rk\rfloor}] 
\\
\le &
\PP[\Gamma_{\langle r, -\lfloor C_0\log(k)r^{2/3}\rfloor-1\rangle_{},\langle n-r, -\lfloor C_0\log(k)r^{2/3}\rfloor-1\rangle_{}}\cap \LL_{\lfloor rk\rfloor} \neq \Gamma_{\langle r, -\lfloor C_0\log(k)r^{2/3}\rfloor-1\rangle_{},\langle n-r, \lfloor C_0\log(k)r^{2/3}\rfloor+1\rangle_{}}\cap \LL_{\lfloor rk\rfloor}] 
\\
&+\PP[\Gamma_{\langle r, -\lfloor C_0\log(k)r^{2/3}\rfloor-1\rangle_{},\langle n-r, \lfloor C_0\log(k)r^{2/3}\rfloor+1\rangle_{}}\cap \LL_{\lfloor rk\rfloor} \neq \Gamma_{\langle r, \lfloor C_0\log(k)r^{2/3}\rfloor+1\rangle_{},\langle n-r, \lfloor C_0\log(k)r^{2/3}\rfloor+1\rangle_{}}\cap \LL_{\lfloor rk\rfloor}] 
\\
< & C_0^2 \log(k)(k-1)^{-2/3},
\end{split}
\]
where the last inequality is by Proposition \ref{prop:coalesce}. Then the conclusion follows by ordering of geodesics (Lemma \ref{l:ordering}). See Figure \ref{fig:overlap} for an illustration.
\end{proof}

\begin{proof} [Proof of Theorem \ref{t:one-point-converg-finite}]
Take any $s\in\N$ and any continuous function $f:\R^{\llbracket -s,s\rrbracket^2}\times \{0,1\}^{\llbracket -s,s\rrbracket^2}\to [0,1]$, regarded as a function on $\R^{\Z^2}\times \{0,1\}^{\Z^2}$. We need to show that 
\[\lim_{n\to\infty}\E [f(\xi\{\Gamma_{\boo,\bn^{}}[\lfloor \alpha n\rfloor]\}, \Gamma_{\boo,\bn^{}}-\Gamma_{\boo,\bn^{}}[\lfloor \alpha n\rfloor])] = \nu(f).\]
Without loss of generality we assume that $\alpha\le 1$.
For $n, r\in\N$ and $k>2$ with $\alpha n-2s>2rk$ and $\alpha n+2s<2n-2rk$, by Lemma \ref{lem:coal-f} we have 
\[
\PP[\Gamma_{\boo,\bn^{}}[\lfloor\alpha n\rfloor + j]=\Gamma_{v[r],\bn^{}+v[r]}[\lfloor\alpha n\rfloor-r+j]\},\; \forall j \in \llbracket -2s, 2s \rrbracket ] \ge 1-C\log(k)k^{-2/3}.
\]
By translation invariance, $\xi\{\Gamma_{v[r],\bn^{}+v[r]}[\lfloor\alpha n\rfloor-r]\}, \Gamma_{v[r],\bn^{}+v[r]}-\Gamma_{v[r],\bn^{}+v[r]}[\lfloor\alpha n\rfloor-r]$ have the same joint distribution as $\xi\{\Gamma_{\boo,\bn^{}}[\lfloor\alpha n\rfloor-r]\}, \Gamma_{\boo,\bn^{}}-\Gamma_{\boo,\bn^{}}[\lfloor\alpha n\rfloor-r]$. So we must have that 
\[
|\E[f(\xi\{\Gamma_{\boo,\bn^{}}[\lfloor\alpha n\rfloor]\}, \Gamma_{\boo,\bn^{}}-\Gamma_{\boo,\bn^{}}[\lfloor\alpha n\rfloor])] - \E[f(\xi\{\Gamma_{\boo,\bn^{}}[\lfloor\alpha n\rfloor-r]\}, \Gamma_{\boo,\bn^{}}-\Gamma_{\boo,\bn^{}}[\lfloor\alpha n\rfloor-r])] | \le C\log(k)k^{-2/3}.
\]
By averaging over $r\in \llbracket 0, \alpha n/4k\rrbracket$, we have (when $\alpha n>4s$)
\[
|\E[f(\xi\{\Gamma_{\boo,\bn^{}}[\lfloor\alpha n\rfloor]\}, \Gamma_{\boo,\bn^{}}-\Gamma_{\boo,\bn^{}}[\lfloor\alpha n\rfloor])] - \E [\mu_{\Gamma_{\boo,\bn^{}}[\lfloor\alpha n\rfloor-\lfloor \alpha n/4k\rfloor],\Gamma_{\boo,\bn^{}}[\lfloor\alpha n\rfloor]} (f)] | \le C\log(k)k^{-2/3}.
\]
By Corollary \ref{cor:trans-fluc-comb} and Proposition \ref{prop:uniform-conv}, for fixed $k$ we have $\mu_{\Gamma_{\boo,\bn^{}}[\lfloor\alpha n\rfloor-\lfloor \alpha n/4k\rfloor],\Gamma_{\boo,\bn^{}}[\lfloor\alpha n\rfloor]} (f) \to \nu(f)$ in probability as $n\to\infty$.
Thus we have that \[\limsup_{i\to\infty} |\E[f(\xi\{\Gamma_{\boo,\bn^{}}[\lfloor\alpha n\rfloor]\}, \Gamma_{\boo,\bn^{}}-\Gamma_{\boo,\bn^{}}[\lfloor\alpha n\rfloor])] - \nu (f) | \le C\log(k)k^{-2/3}.\]
Then the conclusion follows since $k$ can be arbitrarily large.
\end{proof}

\section{Exponential concentration via counting argument}  \label{sec:exp-con}
Using a covering argument, we can prove the following exponential concentration of the empirical environment, for both finite and semi-infinite geodesics.
\begin{prop}  \label{prop:exp-concen-infinite}
For any $s\in\N$, and any bounded continuous $f:\R^{\llbracket -s,s\rrbracket^2}\times \{0,1\}^{\llbracket -s,s\rrbracket^2}\to \R$, regarded as a function on $\R^{\Z^2}\times \{0,1\}^{\Z^2}$, and any $\epsilon>0$, we have
\[
\PP[|\mu_{\boo;r}(f)-\nu(f)| > \epsilon] < Ce^{-cr},
\]
for $r$ large enough, and $c,C>0$ depending on $s,f,\epsilon$.
\end{prop}
\begin{prop}  \label{prop:exp-concen-finite}
Let $\{b_n\}_{n\in\N}$ be a sequence of integers such that $\lim_{n\to\infty} n^{-2/3}|b_n| < \infty$.
Then for any $s\in\N$, any bounded continuous $f:\R^{\llbracket -s,s\rrbracket^2}\times \{0,1\}^{\llbracket -s,s\rrbracket^2}\to \R$, regarded as a function on $\R^{\Z^2}\times \{0,1\}^{\Z^2}$, and any $\epsilon>0$, we have
\[
\PP[|\mu_{\boo,\langle n, b_n\rangle_{}}(f)-\nu(f)| > \epsilon] < Ce^{-cn},
\]
for $n$ large enough, and $c,C>0$ depending on $s,f,\epsilon$.
\end{prop}

From Proposition \ref{prop:exp-concen-infinite} we can deduce Theorem \ref{thm:semi-infinite}.
\begin{proof}[Proof of Theorem \ref{thm:semi-infinite}]
By Proposition \ref{prop:exp-concen-infinite}, for any bounded continuous $f:\R^{\llbracket -s, s \rrbracket^2}\times\{0,1\}^{\llbracket -s, s \rrbracket^2} \to \R$ (regarded as a function on $\R^{\Z^2}\times \{0,1\}^{\Z^2}$) and $\epsilon>0$, we have that $\sum_{r\in\N}\PP[|\mu_{\boo;r}(f)-\nu(f)| > \epsilon]<\infty$.
So almost surely, there exists some (random) $r_0$ such that $|\mu_{\boo;r}(f)-\nu(f)| \le \epsilon$ for any $r>r_0$.
Thus we have that $\mu_{\boo;r}(f)\to\nu(f)$ almost surely.
The conclusion follows by taking all $s\in\N$, and $f$ over
a countable dense subset of the space of continuous and compactly supported functions on $\R^{\llbracket -s,s\rrbracket^2}\times \{0,1\}^{\llbracket -s,s\rrbracket^2}$ with the uniform convergence topology.
\end{proof}
Using the same arguments we can deduce Theorem \ref{thm:finite} from Proposition \ref{prop:exp-concen-finite}. We omit the details.

To prove these exponential concentration bounds (Propositions \ref{prop:exp-concen-infinite} and \ref{prop:exp-concen-finite}), we cover the geodesics with  short finite ones, and use Proposition \ref{prop:uniform-conv}.

We take $m\in\N$ such that $m^{2/3}\in\Z$.
For each $i, j \in \Z$ we denote $L_{i,j}$ as the segment joining $\langle im, (2j-1)m^{2/3} \rangle_{}$ and $\langle im, (2j+1)m^{2/3} \rangle_{}$.
For each integer sequence $j_0,j_1,\ldots,j_k$, we let $P_{j_0,\ldots,j_k}$ be the collection of paths from $L_{0,j_0}$ to $L_{k,j_k}$, intersecting each $L_{i,j_i}$, $0\le i \le k$.
For any $k\in\N$ and $D>0$, we denote $P_{k,D}$ as the union of all $P_{j_0, j_1,\ldots,j_k}$ such that $j_0=0$ and $\sum_{i=1}^k (j_i-j_{i-1})^2 > Dk$. 
In words, $P_{k,D}$ contains all paths from $L_{0,0}$ to $\LL_{km}$ with `quadratic variation' $>Dk$.
We next upper bound the passage times of these paths.
\begin{lemma}  \label{lem:large-total-tran}
There exists $c_0>0$, such that when $m,k,D$ are large enough,
\[
\PP\bigg[\exists \gamma\in P_{k,D}, T(\gamma)> \frac{2km}{(1-\rho)^2+\rho^2} - (b_+-b_-)(\rho^{-1}-(1-\rho)^{-1}) - c_0Dkm^{1/3}\bigg] < e^{-c_0k},
\]
where $b_-,b_+\in\Z$ such that $\langle 0,b_-\rangle, \langle km, b_+\rangle_{}$ are the intersections of $\gamma$ with $\LL_0, \LL_{km}$, respectively.
\end{lemma}
\begin{proof}
First, there exist $c_1,C_1>0$ such that for $m$ large enough and any $j\in\Z$, $x>0$,
\[
\E\Bigg[\max_{\substack{\langle 0,b\rangle_{}\in L_{0,0},\\\langle m,b'\rangle_{}\in L_{1,j}}} T_{\langle 0,b\rangle_{},\langle m,b'\rangle_{}} +(b'-b)(\rho^{-1}-(1-\rho)^{-1}) \Bigg] < \frac{2m}{(1-\rho)^2+\rho^2}+(C_1-c_1j^2)m^{1/3},\]
\[
\PP\Bigg[\max_{\substack{\langle 0,b\rangle_{}\in L_{0,0},\\\langle m,b'\rangle_{}\in L_{1,j}}} T_{\langle 0,b\rangle_{},\langle m,b'\rangle_{}} +(b'-b)(\rho^{-1}-(1-\rho)^{-1}) > \frac{2m}{(1-\rho)^2+\rho^2}+(x-c_1j^2)m^{1/3}\Bigg] < C_1e^{-c_1x}.
\]
When $|j|<(\rho^2\wedge (1-\rho)^2)m^{1/3}$ these inequalities follow from Proposition \ref{t:seg-to-seg} and \eqref{e:mean}, and fundamental computations. When $|j|\ge (\rho^2\wedge (1-\rho)^2)m^{1/3}$ these inequalities can be obtained by applying \eqref{e:wslope} in Theorem \ref{t:onepoint} to each $T_{u,v}$ with $u\in L_{0,0}$ and $v\in L_{1,j}$ and taking a union bound.

Note that
\[
\max_{\gamma\in P_{j_0,j_1,\ldots,j_k}} T(\gamma) \le \sum_{i=1}^{k-1} \max_{\substack{u\in L_{i-1,j_{i-1}},\\ v\in L_{i,j_i}}} T_{u,v}^\bu + \max_{\substack{u\in L_{k-1,j_{k-1}},\\ v\in L_{k,j_k}}} T_{u,v}.
\]
Here $T_{u,v}^\bu = T_{u,v}-\xi(v)$ for any $u\le v \in \Z^2$.
Then by a Bernstein type estimate for independent random variables with exponential tails, we have
\begin{multline*}
\PP\left[
\max_{\gamma\in P_{j_0,j_1,\ldots,j_k}} T(\gamma) + (b_+-b_-)(\rho^{-1}-(1-\rho)^{-1}) > \frac{2km}{(1-\rho)^2+\rho^2} - \frac{c_1}{2} Dk m^{1/3}
\right]\\ < C_2e^{-c_2\sum_{i=1}^k(j_i-j_{i-1})^2},    
\end{multline*}
for any $D$ large (depending on $c_1,C_1$) and any integer sequence $j_0,\ldots,j_k$ with $j_0=0$, $\sum_{i=1}^k(j_i-j_{i-1})^2>Dk$. 
Here $c_2,C_2>0$ are constants, and $\langle 0,b_-\rangle, \langle km, b_+\rangle_{}$ are the intersections of $\gamma$ with $\LL_0,\LL_{km}$.
Summing over all such sequences $j_0, j_1,\ldots,j_k$, the right-hand side is bounded by
\[
 C_2e^{-c_2Dk/2}\left(\sum_{j\in\Z}e^{-c_2j^2/2}\right)^k.
\]
By taking $D$ large so that $e^{c_2D/4}>\sum_{j\in\Z}e^{-c_2j^2/2}$, we get the conclusion.
\end{proof}
We next prove Proposition \ref{prop:exp-concen-infinite}. The general idea is to upper bound the `quadratic variation' of the first $r$ steps of $\Gamma_\boo$, and use Proposition \ref{prop:uniform-conv} to show that the empirical environment between each $\LL_{im}$ and $\LL_{(i+1)m}$ is close to $\nu$, and use independence to deduce the exponential concentration.
\begin{proof}[Proof of Proposition \ref{prop:exp-concen-infinite}]
For any vertices $u\le v$, denote
\[
\mu_{u,v}^\bu := \frac{1}{|\Gamma_{u,v}|-1} \sum_{w\in \Gamma_{u,v}, w\neq v} \delta_{(\xi\{w\}, \Gamma_{u,v}-w)},
\]
i.e.\;it is the empirical environment along $\Gamma_{u,v}$, excluding the last vertex $v$.
Without loss of generality we assume that $0\le f\le 1$, and $\epsilon$ is small enough (depending on $s$ and $f$).

We first consider paths with small `quadratic variation'. Take $D>0$ and $m\in\N$ such that $m^{2/3}\in\Z$, and let them be large enough as required by Lemma \ref{lem:large-total-tran}. We also choose $m$ large enough such that
\begin{equation}  \label{eq:abmaxg}
\PP\left[\max_{|a|,|b|<\epsilon^{-2}m^{2/3}} |\mu_{\langle 0,a \rangle_{}, \langle m,b \rangle_{}}^\bu(f) - \nu(f)| > \epsilon^2\right] < \varepsilon,    
\end{equation}
by Proposition \ref{prop:uniform-conv}.
Here $\varepsilon$ is a small number depending on $D, \epsilon$ and to be determined.
Take any $k\in\N$ (also large enough as required by Lemma \ref{lem:large-total-tran}), and a sequence $j_0,\ldots,j_k$ such that $j_0=0$ and $\sum_{i=1}^k (j_i-j_{i-1})^2 \le Dk$.
We let $I'\subset \{1,\ldots,k\}$ be the collection of indices such that
$|j_i-j_{i-1}|<\epsilon^{-2}/2-1$ for each $i\in I'$. Then $|I'|>(1-\epsilon/2)k$, when $\epsilon$ is small enough (depending on $D$).
Next we let $I\subset I'$ such that for each $i\in I$, 
\[
\max_{u\in L_{i-1,j_{i-1}}, v\in L_{i,j_i}} |\mu_{u,v}^\bu(f) - \nu(f)| \le \epsilon^2.
\]
By \eqref{eq:abmaxg} we have $\PP[i\in I]>1-\varepsilon$ for each $i\in I'$.
Also note that $i_1\in I$ and $i_2\in I$ are independent for any $i_1, i_2\in I'$ with $i_1-i_2\ge 2$.
Then by a Chernoff bound and taking $\varepsilon$ small enough (depending on $D,\epsilon$), we can make  $\PP[|I'|-|I|>\epsilon^2 k]<(D+1)^{-2k}$.

Let $\gamma$ be the path consisting of the first $2km+1$ vertices of $\Gamma_\boo$;
i.e.\;$\gamma$ is the part of $\Gamma_\boo$ on and between $\LL_0$ and $\LL_{km}$.
Given that $\gamma\in P_{j_0,\ldots,j_k}$, and $|I'|-|I|\le \epsilon^2 k$, 
for any $r\in\llbracket 2km, 2(k+1)m\rrbracket$ we must have that $|\mu_{\boo;r}(f)-\nu(f)| \le \epsilon/2 + 2\epsilon^2 + 1/(k+1)$.
So when $k>\epsilon^{-2}$ and $\epsilon$ is large enough, we have
\[
\PP\left[\gamma\in P_{j_0,\ldots,j_k}, |\mu_{\boo;r}(f)-\nu(f)| > \epsilon \right] < (D+1)^{-2k}.
\]
Thus by summing over all sequences $j_0,\ldots,j_k$ with $j_0=0$, $\sum_{i=1}^k (j_i-j_{i-1})^2 \le Dk$, we have 
\[
\PP\left[\gamma\not\in P_{k,D}, |\mu_{\boo;r}(f)-\nu(f)| > \epsilon \right] < {\lfloor Dk\rfloor+k-1 \choose k-1} (D+1)^{-2k} < e^{-ck}
\]
for some $c>0$ depending on $D$.

Now it remains to bound $\PP[\gamma\in P_{k,D}]$.
By Lemma \ref{lem:large-total-tran}, we have
\begin{equation}  \label{eq:exp-concen-i-pf}
\PP[\gamma\in P_{k,D}] < e^{-c_0k}+\PP\bigg[T(\gamma)\le \frac{2km}{(1-\rho)^2+\rho^2} - b_+(\rho^{-1}-(1-\rho)^{-1}) - c_0Dkm^{1/3}\bigg],
\end{equation}
where $\langle km, b_+\rangle_{}$ is the intersection of $\Gamma_\boo$ with $\LL_{km}$, and recall that $c_0>0$ is a constant independent of $m,k,D$. 
When the event in the right-hand side of \eqref{eq:exp-concen-i-pf} happens, we must have that (at least) one of the following happens:
\begin{itemize}
    \item $|b_+|>km^{2/3}$,
    \item $\max_{|b|\le km^{2/3}} \bB(\langle km,b \rangle_{}, \langle km,0 \rangle_{})-b(\rho^{-1}-(1-\rho)^{-1})\ge c_0Dkm^{1/3}/3$,
    \item $T_{\boo,\langle km,0 \rangle_{}}^\bu\le \frac{2km}{(1-\rho)^2+\rho^2} - c_0Dkm^{1/3}/2$,
\end{itemize}
where $T_{u,v}^\bu = T_{u,v}-\xi(v)$ for any vertices $u\le v$ as before, and recall that $\bB$ is the Busemann function (defined in Section \ref{ssec:buseman}).
To see this, we assume the contrary, i.e.\;none of the above three events happen (while the event in the right-hand side of \eqref{eq:exp-concen-i-pf} happens). 
Then we must have $T_{\boo,\langle km,0 \rangle_{}}^\bu-T_{\boo,\langle km,0 \rangle_{}}^\bu> \bB(\langle km,b_+ \rangle_{}, \langle km,0 \rangle_{})$, which contradicts with Lemma \ref{lem:buse-opti}.

We claim that we can bound the probability of each of the three events by $C'e^{-c'k}$, for some $c',C'>0$ depending on $m,D$.
For the first event the bound is by Lemma \ref{lem:semi-inf-trans}.
For the second event, note that $b\mapsto \bB(\langle km,b \rangle_{}, \langle km,0 \rangle_{})-b(\rho^{-1}-(1-\rho)^{-1})$ is a (two-sided) centered random walk; for the third event, use Theorem \ref{t:onepoint}.

Finally, by sequentially choosing $D, \epsilon, \varepsilon, m$, and considering all large enough $k$ and each $r\in\llbracket 2km, 2(k+1)m\rrbracket$, the conclusion follows.
\end{proof}
We prove Proposition \ref{prop:exp-concen-finite} using a similar strategy.
\begin{proof}[Proof of Proposition \ref{prop:exp-concen-finite}]
The first half of this proof follows the same way as the proof of Proposition \ref{prop:exp-concen-infinite}.
We omit the details, and conclude that the following is true for any $D>0$, $\epsilon>0$, $m\in\N$ with $m^{2/3}\in\Z$, and $k\in \N$, such that $D,m$ are large enough as required by Lemma \ref{lem:large-total-tran}, $\epsilon$ is small enough depending on $D$, and $m$ is large enough depending on $D,\epsilon$.
Take any $k\in\N$ which is $>\epsilon^{-2}$ and large enough as required by Lemma \ref{lem:large-total-tran}, and take any $n\in \llbracket km, (k+1)m\rrbracket$.
Let $\gamma$ be the path from $\LL_0$ to $\LL_{km}$, consisting of the first $2km+1$ vertices of $\Gamma_{\boo,\langle n,b_n\rangle_{}}$.
Then we have
\[
\PP\left[\gamma \not\in P_{k,D}, |\mu_{\boo,\langle n,b_n\rangle_{}}(f)-\nu(f)| > \epsilon \right] < {\lfloor Dk\rfloor+k-1 \choose k-1} (D+1)^{-2k} < e^{-ck}    
\]
for some $c>0$ depending on $D$.
It remains to bound $\PP[\gamma\in P_{k,D}]$.
By Lemma \ref{lem:large-total-tran}, we have
\begin{equation}  \label{eq:exp-concen-f-pf2}
\PP[\gamma\in P_{k,D}] < e^{-c_0k}+\PP\bigg[T(\gamma)\le \frac{2km}{(1-\rho)^2+\rho^2} - b_+(\rho^{-1}-(1-\rho)^{-1}) - c_0Dkm^{1/3}\bigg],
\end{equation}
where $\langle km, b_+\rangle_{}$ is the intersection of $\Gamma_{\boo,\langle n,b_n\rangle_{}}$ with $\LL_{km}$. 
When the event in the right-hand side of \eqref{eq:exp-concen-f-pf2} happens, we must have that (at least) one of the following happens:
\begin{itemize}
    \item $\max_{b\in\Z}  
    T_{\langle km, b\rangle_{},\langle n,b_n\rangle_{}}
    -(b-b_n)(\rho^{-1}-(1-\rho)^{-1})\ge c_0Dkm^{1/3}/3$,
    \item $T_{\boo,\langle km,b_n\rangle_{}}\le \frac{2km}{(1-\rho)^2+\rho^2}
    -b_n(\rho^{-1}-(1-\rho)^{-1})
    -c_0Dkm^{1/3}/2$.
\end{itemize}
To see this, we assume the contrary, i.e.\;none of the above events happen. 
Then we must have
\begin{multline*}
T(\gamma) > T_{\boo,\langle n,b_n\rangle_{}} - T_{\langle km, b_+\rangle_{},\langle n,b_n\rangle_{}} 
\ge  T_{\boo,\langle km,b_n\rangle_{}}
- T_{\langle km, b_+\rangle_{},\langle n,b_n\rangle_{}} \\ >
\frac{2km}{(1-\rho)^2+\rho^2}
-b_+(\rho^{-1}-(1-\rho)^{-1}) - 5c_0Dkm^{1/3}/6, \end{multline*}
which contradicts with the event in the right-hand side of \eqref{eq:exp-concen-f-pf2}.
We claim that we can bound the probability of each of the two events by $C'e^{-c'k}$, for some $c',C'>0$ depending on $m,D$.
For the first event,
note that $n-km\le m$, then the bound can be obtained by taking a union bound over all up-right paths from $\LL_{km}$ to $\langle n,b_n\rangle_{}$ (there are at most $2^{2m}$ such paths, and the passage time of each is the sum of at most $2m+1$ i.i.d.\;$\Exp(1)$ random variables).
For the second event, apply Theorem \ref{t:onepoint}.
Thus the conclusion follows. 
\end{proof}

\vspace{0.05in}

\section*{Acknowledgement}
We thank Timo Sepp\"{a}l\"{a}inen and Pablo 
Ferrari for valuable conversations. We 
thank the organisers of the workshop on integrable
probability at the Open Online Probability School in
June 2020 for hosting a talk by LZ on joint work with AS,
which led to this collaboration.
We would also like to thank anonymous referees for carefully reading this paper and providing many valuable comments that help improve the expository.
AS was supported by NSF grants DMS-1352013 and
DMS-1855527, Simons Investigator grant and a MacArthur Fellowship.

\vspace{0.5in}
\bibliographystyle{halpha}
\bibliography{bibliography}

\end{document}